\documentclass[reqno,11pt]{amsart}
\usepackage{amsmath,amssymb,amsthm,graphicx,url,mathrsfs}
\usepackage{setspace} 
\usepackage{wrapfig}
\usepackage{enumitem}
\usepackage{mathtools}
\usepackage{xparse}
\usepackage[usenames,dvipsnames]{xcolor}
\usepackage[colorlinks=true,linkcolor=Red,citecolor=Green]{hyperref}
\usepackage[super]{nth}
\usepackage[open, openlevel=2, depth=3, atend]{bookmark}
\hypersetup{pdfstartview=XYZ}
\usepackage[font=footnotesize]{caption}
\usepackage{a4wide}

\makeatletter
\def\l@subsection{\@tocline{2}{0pt}{1pc}{4.6em}{}}
\renewcommand{\tocsubsection}[3]{\indentlabel{\@ifnotempty{#2}{\hspace*{2em}\makebox[2em][l]{\ignorespaces#1 #2. \hfill}}}#3}
\makeatother

\numberwithin{equation}{section}

\def\Im{\operatorname{Im}}
\def\Re{\operatorname{Re}}
\def\ov{\overline}
\def\11{{\rm 1~\hspace{-1.4ex}l} }
\def\R{\mathbb R}
\def\C{\mathbb C}
\def\Z{\mathbb Z}
\def\N{\mathbb N}

\def\T{\mathbb T}

\def\lg{\langle}
\def\rg{\rangle}
\def\pa{\partial}
\newcommand{\Id}{{\rm Id}}

\newcommand{\noi}{\noindent}

\newcommand{\Op}{\operatorname{Op}}
\newcommand{\tr}{\operatorname{Tr}}
\newcommand{\w}{\mathrm{w}}

\theoremstyle{plain}

\newtheorem{thm}{Theorem}
\newtheorem{prop}{Proposition}[section]
\newtheorem{cor}[prop]{Corollary}
\newtheorem{lemma}[prop]{Lemma}
\newtheorem{definition}[prop]{Definition}

\theoremstyle{remark}
\newtheorem{remark}[prop]{Remark}
\theoremstyle{definition}
\newtheorem{rem}{Remark}[section]

\theoremstyle{remark}

\newtheorem{ex}[prop]{Example}

\numberwithin{equation}{section}

\definecolor{darkred}{rgb}{0.7,0.1,0.1}
\definecolor{darkblue}{rgb}{0,0,0.7}
\title[Sharp resolvent estimate for the Baouendi-Grushin operator and applications]{Sharp resolvent estimate for the damped-wave Baouendi-Grushin operator and applications}
\author{Victor Arnaiz and Chenmin Sun}
\address[V. Arnaiz]{CNRS, Universit\'e Paris-Saclay. Laboratoire de Math\'ematiques d'Orsay, and Laboratoire de Math\'ematiques Jean Leray, Universit\'e de Nantes. UMR CNRS 6629, 2
	rue de la Houssini\`ere, 44322 Nantes Cedex 03, France.}

\email{victor.arnaiz@univ-nantes.fr}

\address[C.~ Sun]{CNRS, Universit\'e Paris-Est Cr\'eteil, Laboratoire d'Analyse et de Math\'ematiques appliqu\'ees, UMR  8050 du CNRS, 
	94010 Cr\'eteil cedex, France.}
\email{chenmin.sun@cnrs.fr}

\begin{document}    
	
	\begin{abstract}
		In this article we study the semiclassical resolvent estimate for the non-selfadjoint Baouendi-Grushin operator on the two-dimensional torus $\mathbb{T}^2=\mathbb{R}^2/(2\pi\mathbb{Z})^2$ with H\"older dampings. The operator is subelliptic degenerating along the vertical direction at $x=0$.
		We exhibit three different situations: 
		(i) the damping region verifies the geometric control condition with respect to both the non-degenerate Hamiltonian flow and the vertical subelliptic flow; (ii) the undamped region contains a horizontal strip; (iii) the undamped part is a  horizontal line. In all of these situations, we obtain sharp resolvent estimates. Consequently, we prove the optimal energy decay rate for the associated damped waved equations. For (i) and (iii), our results are in sharp contrast to the Laplace resolvent since the optimal bound is governed by the quasimodes in the subelliptic regime. While for (ii), the optimality is governed by the quasimodes in the elliptic regime, and the optimal energy decay rate is the same as for the classical damped wave equation on $\mathbb{T}^2$. 
		
		Our analysis contains the study of adapted two-microlocal semiclassical measures, construction of quasimodes and refined Birkhoff normal-form reductions in different regions of the phase-space. Of independent interest, we also obtain the propagation theorem for semiclassical measures of quasimodes microlocalized in the subelliptic regime.
	\end{abstract}
	
	\maketitle
	
	\begin{spacing}{0.8}
		\tableofcontents
	\end{spacing}
	
	
	
	
	
	
	\section{Introduction}

\subsection{Motivation} 
After the fundamental work of Lars H\"ormander, the theory of subelliptic operators has been developed since  the late 60's. Subelliptic operators are rooted from physics (such as kinetic theory) and are linked to probability theory (such as the Mallivian calculus). From the PDE point of view, the presence of subellipticity brings some degeneracy which may dramatically change the behavior of solutions. 

In the past decade, typical control problems (observability, controllability, stabilization) for evolution PDEs driven by subelliptic operators have been widely investigated (see \cite{BeCa1},\cite{BeCa2},\cite{BDE},\cite{DA},\cite{Ko} and the references therein), with a particular focus on parabolic equations (heat equations). Only very recently, such problems  have been addressed for the wave \cite{Let} and some particular types of Schr\"odinger equations \cite{BuSun},\cite{FeLe} (see also the survey \cite{LetS}), with completely different mechanism compared to parabolic PDEs, as the study for wave and Schr\"odinger equations relies heavily on the study of propagations for high-energy wave-packets (quasimodes) at fine scales.

In this article, we address the question of  stabilization, namely the energy decay rate 
associated with a damped wave equation given in terms of the  subelliptic Laplacian known as \textit{Baouendi-Grushin operator}, via the study of the resolvent of the related damped-wave operator.  Our results exhibit different situations where the subellipticity and the damping properties are responsible for the optimal energy decay rates. Our analysis reveals fine concentration properties of quasimodes for the damped-wave operator. This  completes the program of the observability and stabilization for the Baouendi-Grushin Schr\"odinger and wave equations previously investigated in \cite{BuSun} and \cite{LeS20}. In fact, obtaining the observability for the Schr\"odinger equation is easier than getting the sharp resolvent estimate for the damped-wave operator, since not only the geometry of the damping region, but the profile and regularity of the damping term play a role in the latter problem.

\subsection{Setup}
Let  $\mathbb{T}^2= \mathbb{R}^2/(2\pi\mathbb{Z})^2$ be the two dimensional flat torus, so that the open rectangle $(-\pi,\pi)_x\times(-\pi,\pi)_y$ is a fundamental domain of $\mathbb{T}^2$. 
We consider the Baouendi-Grushin operator defined on $\mathbb{T}^2$ by
\begin{equation} 
	\label{e:Baouendi-Grushin}
	\Delta_G:=\partial_x^2+V(x)\partial_y^2
\end{equation}
where the function $V(x)$ is smooth, $2\pi$-periodic, satisfying
$$ 
V(0)=V'(0)=0, \quad V''(0)>0,
$$
and moreover $V(x)>0$ for all $x\in[-\pi,\pi]\setminus \{0\}$. 

Without loss of generality, we assume that $V''(0)=2$ in this article. Typically, near the vertical axis $x=0$, the Baouendi-Grushin operator can be written by Taylor expansion as$$\Delta_G=\partial_x^2 + \big( x^2 + \mathcal{O}(x^3) \big) \partial_y^2.$$ 
A concrete example is given by $V(x)=4\sin^2\big(\frac{x}{2}\big)$.

Our main interest in this work is the study of the damped wave equation
\begin{align}\label{e:damped_wave_equation}
	\begin{cases}
		&\!\!\!  \partial_t^2 u -\Delta_{G}u+b(x,y)\partial_tu=0,\\ &\!\!\!(u,\partial_tu)|_{t=0}=(u_0,v_0)
	\end{cases}
\end{align}
with damping $b(x,y)\geq 0$, and initial data $(u_0,v_0)$ lying in some sufficiently regular Sobolev space precised below. Our main goal is to study how the geometry of the \emph{damping region} 
$$
\Omega:= \big \{ (x,y)\in\T^2:b(x,y)>0 \big \}
$$ and the regularity of $b$ affect the energy decay rate of the solutions to \eqref{e:damped_wave_equation}, and what is the influence of the subellipticity stemmed from the Baouendi-Grushin operator in this decay rate, in comparison with the elliptic Laplacian $\Delta$ on the torus (see \cite{AL14}, \cite{BH07}, \cite{DK20}, \cite{K19}).

The well-posedness of \eqref{e:damped_wave_equation} can be viewed as a consequence of the  Hille-Yosida theory (see \cite[Appendix A.3]{LeS20} for details) and the solution $(u(t),\partial_tu(t))$ is given by a semi-group $e^{t\mathcal{A}}$ generated by
\begin{align}\label{matrixform}
	\mathcal{A}=
	\left(\begin{matrix}
		0 &1\\
		\Delta_G &-b
	\end{matrix}
	\right),
\end{align}
with domain
\begin{align*}
	D(\mathcal{A}):=H_G^2(\T^2)\times H_G^1(\T^2),
\end{align*}
where the associated Sobolev spaces $H_G^k$ are defined in \eqref{e:seminorms}.

Moreover, using the equation and integration by parts, we see that the energy of the solution $u(t)$, defined as
\begin{align} 
	\label{e:energy}
	E[u](t)=\frac{1}{2}\int_{\T^2} \big( |\partial_tu|^2+|\nabla_G u|^2 \big) dxdy,
\end{align}
where $\nabla_G=\big( \partial_x,V(x)^{\frac{1}{2}}\partial_y \big)$, verifies that 
$$ \frac{d}{dt}E[u](t)=-\int_{\T^2} b(x,y)|\partial_tu|^2dxdy\leq 0.
$$
We are interested in the study of the sharp energy dacay rates of $E[u](t)$ given by \eqref{e:energy} in three cases regarding the damping region $\Omega$:	
\begin{enumerate}		
	\item \emph{Geometric control case}: The open set $\Omega$ satisfies the elliptic and subelliptic geometric control condition, given by Definitions \ref{EGCC} and \ref{SGCC} below.
	\smallskip
	
	\item \emph{Widely undamped case}: The undamped set $b^{-1}(0) = \T^2 \setminus \Omega$ consists of a horizontal strip of the form $b^{-1}(0) = \T_x\times I_y$ for a  strict closed interval $I_y = [\alpha,\beta]\subset\T_y$.
	\smallskip
	
	\item \emph{Narrowly undamped case:} The set $b^{-1}(0)$ is a horizontal line $b^{-1}(0) = \T_x \times \{ y_0 \}$ for $y_0 \in \mathbb{T}_y$. 
\end{enumerate}

With respect to the regularity of $b$, we will assume that $b \in W^{k_0,\infty}(\T^2)$ for some sufficiently large $k_0 \in \mathbb{N}$,  and, similar to \cite[Eq. (2.13)]{AL14}  and \cite[Lemma 3.1]{BH07}, we will consider the following condition:
\begin{equation} 
	\label{e:B-H_condition}
	|\nabla b|\leq Cb^{1-\sigma},\quad |\nabla^2b|\leq Cb^{1-2\sigma}
\end{equation}
for some $\sigma <\frac{1}{4}$. These conditions imply that $b$ vanishes smoothly. Moreover, in connection with \cite{DK20}, \cite{K19}, we will sharpen some of our results in some particular cases when $b$ vanishes like  $y^\nu$ near the vanishing set $b^{-1}(0)$, for some larger $\nu$ (see precisely \eqref{e:particular_b} and \eqref{e:finite_type_condition} below).  


Before stating our main results, we first recall the relation between damped waves  and resolvent estimates. As we will see, in all the situations (1), (2), and (3) above, the energy decays at polynomial rate\footnote{Here the polynomial rate means a rate like $t^{-p}$ for some $p>0$, not necessary an integer. We respect this terminology from the existing literature.  }, so we restrict our definitions to this case.

\begin{definition}
	Let $\alpha > 0$. We say that the system \eqref{e:damped_wave_equation} is \emph{stable at rate $t^{-\frac{1}{\alpha} }$} if there exists $C>0$ such that for every initial data $(u_0,u_1)\in D(\mathcal{A})$,
	$$ (E[u](t))^{\frac{1}{2}}\leq Ct^{-\frac{1}{\alpha}}\|(u_0,u_1)\|_{D(\mathcal{A})}.
	$$
	We say that the rate $t^{-\frac{1}{\alpha}}$ is optimal if
	moreover
	$$ \limsup_{t\rightarrow+\infty}\hspace{0.1cm}t^{\frac{1}{\alpha}}\cdot\!\!\!\!\!\!\!\!\!\sup_{ \substack{(u_0,u_1)\in D(\mathcal{A})\\ (u_0,u_1)\notin\mathrm{Ker}(\mathcal{A}) } }\frac{(E[u](t))^{\frac{1}{2}}}{\|(u_0,u_1)\|_{D(\mathcal{A})}}>0.
	$$
\end{definition} 
Let us now consider, for any $k \geq 1$, the Grushin Sobolev spaces $H_G^k$ defined by the seminorms \eqref{e:seminorms} below.  Let $\mathscr{H} := H_G^1\times L^2$ be the energy space, we set $\dot{\mathcal{A}} := \mathcal{A} \vert_{\dot{\mathscr{H}}}$, where the homogeneous Hilbert space $\dot{\mathscr{H}}$ is defined by $\dot{\mathscr{H}} = (\ker \mathcal{A})^\perp \subset \mathscr{H}$. One can see, from an abstract theorem of Borichev-Tomilov\footnote{We will recall the precise statement in Appendix \ref{a:borichev_tomilov}.} \cite{BoT}, that with $D(\mathcal{A}) =  H^2_G \times H_G^1$,  the system \eqref{e:damped_wave_equation} is stable at rate $t^{-\frac{1}{\alpha}}$ if and only if\footnote{For the specific operator $\mathcal{A}$, the condition $i\R\cap\sigma(\dot{\mathcal{A}})=\emptyset$ is satisfied as a consequence of the unique continuation property, see \cite{LeS20}. } $i\R\cap \sigma(\dot{\mathcal{A}})=\emptyset$ and the following resolvent estimate holds:
\begin{align}\label{resolvent1}
	\| ( \mathcal{A}-i\lambda\mathrm{Id} )^{-1} \|_{\mathcal{L}(\mathscr{H})}\leq C|\lambda|^{\alpha},\text{ as } |\lambda|\rightarrow+\infty, \quad \lambda\in\R.
\end{align}
On the other hand, the resolvent estimate \eqref{resolvent1} (c.f. \cite{AL14} or Appendix \ref{a:borichev_tomilov}) is equivalent to the following semiclassical resolvent estimate:
\begin{align}\label{resolvent2} 
	\|(-h^2\Delta_G-1\pm ihb)^{-1}\|_{\mathcal{L}(L^2)}\leq Ch^{-\alpha-1},\text{ uniformly in  } 0<h\ll 1.
\end{align}

From now on, we concentrate ourselves on the resolvent estimate of type \eqref{resolvent2}. 
Our main results will show that there is a qualitative difference between the widely undamped case (2) and the narrowly undamped case (3), appearing as a jump in the sharp exponent $\alpha$ in \eqref{resolvent2}. This strong difference of behaviors also arises in the case of the Laplacian $\Delta$ on the torus (see \cite{K19} and \cite{DK20}).
More precisely, in case (2), our resolvent bound is the same as the optimal resolvent bound for Laplacian, hence the wave-energy for the subelliptic problem decays at the same order as the classical one.   
While in cases (1) and (3), the optimal resolvent bounds are larger than the resolvent bounds for the Laplacian, meaning that the energy of waves under subellipticiy decays slower than that of  waves subject to the classical equation.


\subsection{Main results} In this  section we state the main results of this article.

\subsubsection{Geometric control condition}

We first consider the case in which suitable adapted geometric control conditions to the problem \eqref{e:damped_wave_equation} hold. Let us first recall that the classical \textit{geometric control condition} for a continuous damping  $b:M\rightarrow\R_+$ defined on a compact Riemannian manifold $(M,g)$ states that there exist $T,c > 0$ such that 
\begin{equation}
	\label{e:classic_GCC}
	\inf_{\rho \in S^*M} \int_0^T b(\gamma_\rho(t)) dt \geq c,
\end{equation}
where $\gamma_\rho : \R \to M$ denotes the geodesic orbit issued from the point $\rho$ of the co-sphere bundle $S^*M$. This is a necessary and sufficient condition \cite{BLR92}, \cite{BG97}, \cite{RT74} for the \textit{uniform stabilization} of the elliptic damped-wave equation
$\partial_t^2u-\Delta_gu+b\partial_tu=0$
on $M$, that is, 
\begin{align}\label{uniform} 
	E[u](t) \leq Ce^{-\kappa t} E[u](0),
\end{align}
for some constants $C,\kappa> 0$. Moreover, the optimal decay rate $\kappa$ is related to the averaging property of the damping along the geodesic flow on the co-sphere bundle $S^*M$ and the real part of finitely many eigenvalues of $\mathcal{A}$ (possibly empty) on the right side of the line
$$\Re(z)=-\lim_{T\rightarrow\infty}\inf_{\rho\in S^*M}\frac{1}{T}\int_{0}^T b(\gamma_\rho(t))dt, 
$$
by \cite[Thm. 2]{Leb93}.

In our setting, the geodesic flow is replaced by the Hamiltonian flow generated by the principal symbol of the Baouendi-Grushin operator $-\Delta_G$, and subellipticity of this operator comes into play, leading to a weaker geometric control condition which will only provide a polynomial stabilization rate (for smoother initial data), or equivalently a resolvent estimate of type \eqref{resolvent2}. 

First of all, the principal symbol $p(x,y,\xi,\eta) := \xi^2+V(x)\eta^2$ of the operator $-\Delta_G$ generates a Hamiltonian flow on the phase space $T^*\T^2$ which we denote by $\phi_t^{\mathrm{e}}$. We define the \textit{elliptic geometric control condition} corresponding to the Hamiltonian flow $\phi_t^{\mathrm{e}}$:
\begin{definition}\label{EGCC} 
	We say that $b$ satisfies the elliptic geometric control condition (EGCC) if 
	$$ \liminf_{T\rightarrow\infty}\frac{1}{T}\int_0^Tb\circ \pi \circ\phi_t^{\mathrm{e}}(\rho) \, dt>0
	$$
	for all $\rho\in T^*\T^2$ such that $p(\rho)=1$, where $\pi : T^* \mathbb{T}^2 \to \mathbb{T}^2$ is the canonical projection.
\end{definition}

\begin{rem}
	Notice that the (EGCC) condition is weaker than the classical geometric control condition, since the time required for the orbit $\pi \circ \phi_t^{\mathrm{e}}(\rho)$ to reach the set $\Omega$ is not uniformly bounded in $\rho \in p^{-1}(1)$, and moreover the average of $b$ along any orbit is not uniformly bounded from below (compare with \eqref{e:classic_GCC}). We remark that the set $p^{-1}(1)$ is not compact.
\end{rem}

We next define the vertical flow $\phi_t^{\mathrm{v}}: \big(y,\eta \big)\mapsto \big(y+\frac{t\eta}{|\eta|}, \eta\big)$ on $\mathbb{T}_y \times (\R_\eta \setminus \{ 0 \})$ and the subelliptic flow $\phi_t^{\rm s}$ on the cone $\mathcal{C} := \{ 0 \} \times \mathbb{T}_y \times (\R_\eta \setminus \{0 \})$ by $\phi_t^{\rm s} = \Id_{x=0} \otimes \phi_t^{\rm v}$. This yields to the following notion of \textit{subelliptic geometric control condition}.
\begin{definition}\label{SGCC} 
	We say that $b$ satisfies the subelliptic geometric control condition (SGCC), if there exists $c_1>0$ such that 
	$$ \liminf_{T\rightarrow\infty}\frac{1}{T}\int_0^Tb\circ\pi_1 \circ \phi_t^{\mathrm{s}}(\rho_1) \, dt\geq c_1
	$$	
	for all $\rho_1\in\mathcal{C}$, where $\pi_1 (0, y, \eta) := (0,y)$.
\end{definition}

\begin{ex}
	Consider a function $0 \leq b_1 \in C^{\infty}(\T^2)$ such that:
	\begin{align*}
		b_1(x,y)=
		\begin{cases}
			\, 0, &   \text{if } |(x,y)|<\frac{\pi}{8},\\
			\, 1, &   \text{if }  |(x,y)|>\frac{\pi}{4},
		\end{cases}
	\end{align*}
	in the fundamental domain $(-\pi,\pi)^2$. 
	It can be easily verified that $b_1$ satisfies both (SGCC) and (EGCC). 
\end{ex}

\begin{ex}
	Next, consider the smooth damping $b_2(x,y)=b_2(y)\in C^{\infty}(\T)$ such that
	$$ b_2(y)=\begin{cases}
		\, 0, & \text{if } |y|<\frac{\pi}{4},\\
		\, 1, & \text{if } |y|\geq \frac{\pi}{2},
	\end{cases} 
	$$
	in the fundamental domain $(-\pi,\pi)$. In this case, only condition (SGCC) is verified by $b_2$. Since the classical trajectory 
	$$ x(t)=x_0+2t\; (\text{mod } 2\pi), \; \xi(t)=1,\; y(t)=0,\; \eta(t)=0
	$$
	on $p^{-1}(1)$ never encounters the damped region $\{b_2>0\}$, then (EGCC) fails. The damping $b_2(y)$ is the prototype that will be considered in Theorems \ref{t:main_theorem} and  \ref{t:sharp_in_1}. 
\end{ex}

\begin{ex}				Finally, consider a damping $0\leq b_3\in C^{\infty}(\T^2)$ such that near $x=0$, 
	$b_3(x,y)=x^2.
	$
	Then condition (EGCC) is satisfied by $b_3$, but (SGCC) is not. This type of damping requires further analysis and is not studied in the present work.
\end{ex}

We are now in position to state our first result. 
\begin{thm}\label{t:EGCC}
 Let $\sigma>0$ be sufficiently small and $k_0$ sufficiently large\footnote{We will require enough regularity on $b$ to use pseudodifferential symbolic calculus when necessary along the proof.}. Let $b \in W^{k_0,\infty}(\mathbb{T}^2)$ satisfy condition \eqref{e:B-H_condition}. Assume also that $b\geq 0$ satisfies (EGCC), and (SGCC). Then there exist $C_0>0$ and $h_0\in(0,1)$ such that, for all $0<h\leq h_0$,
	\begin{align}\label{resolvent:EGCC}
		\|(-h^2\Delta_G-1\pm ihb)^{-1}\|_{\mathcal{L}(L^2)}\leq C_0h^{-2}.
	\end{align}
	Consequently, the system \eqref{e:damped_wave_equation} is stable at rate $t^{-1}$. If moreover $\mathrm{supp} (b) \cap \{ x = 0 \} \neq \{0\}_x\times\mathbb{T}_y$, then there exists $C_1>0$ such that
	\begin{align}\label{resolvent:EGCClower}
		\|(-h^2\Delta_G-1\pm ihb)^{-1} \|_{\mathcal{L}(L^2)}\geq C_1h^{-2},
	\end{align}
	and in this case, the stable rate $t^{-1}$ for the associated damped wave equation is optimal.
\end{thm}

\begin{rem}
	There exist non-degenerate trajectories spiraling around the degenerate line $x=0$. For example, assuming that $V(x)=x^2$ near $x=0$, consider the family of trajectories
	$$ x_{\epsilon}(t)=\epsilon \sin(2t/\epsilon) \text{ \textnormal{mod} }2\pi, \quad y_{\epsilon}(t)=\epsilon(t-\epsilon/4\sin(4t/\epsilon) ) \text{ \textnormal{mod} }2\pi.
	$$
	Roughly speaking, (SGCC) allows to control not only the propagation along the subelliptic flow $\phi_t^{\mathrm{s}}$, but also these non-degenerate trajectories of $\phi_t^{\mathrm{e}}$ close to the degenerate line $x=0$. The classical trajectories in the compact regime, away from the vertical trajectory, are then controlled by condition (EGCC).
\end{rem}	
\begin{rem}
	The analogue of the vertical flow $\phi_t^{\mathrm{v}}$
	can be intrinsically defined on more general manifolds with sub-Riemannian metric, such as three-dimensional contact manifolds (\cite{CdVHT}), quotient of H-type groups (\cite{FeLe}), etc. We propose as an open question to obtain sharp damped-wave decay rate for the above-mentioned general compact sub-Riemannian manifolds, under (EGCC) and (SGCC).  
\end{rem}


\subsubsection{Rectangular damping}
In the second set of main results, we assume that $b=b(y)$ depends only on the $y$ variable. Notice that in cases (2) and (3) the (EGCC) condition is violated.
We first state our results in case (2), in correspondence with \cite[Thm. 2.6]{AL14} for the damped-wave equation with elliptic Laplacian on the torus.

\begin{thm}
	\label{t:main_theorem}
	Let $\sigma>0$ be sufficiently small and $k_0$ sufficiently large. Let $b \in W^{k_0,\infty}(\mathbb{T}^2)$ satisfy \textnormal{(2)} and condition \eqref{e:B-H_condition}. Then there exist $h_0>0$, $\delta_0=\delta_0(\sigma) > 0$, $\delta_0\rightarrow 0$ as $\sigma\rightarrow 0$, and $C_0>0$ such that, for all $0<h\leq h_0$:
	\begin{align}\label{resolvent:generalupper}		
		\|(-h^2\Delta_G-1\pm ihb)^{-1}\|_{\mathcal{L}(L^2)}\leq C_0h^{-2-\delta_0}.
	\end{align}
	Consequently, the system \eqref{e:damped_wave_equation} is stable at rate $t^{-\frac{1}{1+\delta_0}}$. Moreover, there exists $C_1>0$ such that
	\begin{align}\label{resolvent:generallower}
		\|(-h^2\Delta_G-1\pm ihb)^{-1} \|_{\mathcal{L}(L^2)}\geq C_1h^{-2}.
	\end{align}
\end{thm}

\begin{rem}\label{ObSt}
	In \cite{BuSun}, N.~Burq and the second author have shown the observability for the Schr\"odinger equation\footnote{Though the theorem in \cite{BuSun} is stated for the Schr\"odinger equation associated to a simpler Baouendi-Grushin operator $\Delta_{G_0}:=\partial_x^2+x^2\partial_y^2$ on a bounded domain with Dirichlet boundary condition, essentially the same proof would lead to \eqref{Ob}. }:
	\begin{align}\label{Ob} 
		\|u_0\|_{L^2(\T^2)}^2\leq C_{T}\int_0^T\|b^{1/2}\mathrm{e}^{it\Delta_G}u_0\|_{L^2(\T^2)}^2dt
	\end{align}
	for some sharp time threshold $T>T_0$.
	Together with an abstract theorem (\cite[Thm. 2.3]{AL14}), the observability \eqref{Ob} leads to a stable rate $t^{-\frac{1}{2}}$ for  \eqref{e:damped_wave_equation}. More precisely, the Schr\"odinger observability \eqref{Ob} is essentially equivalent to the semi-classical observability estimate of the type (\cite[Thm. 4]{BuZ}):
	\begin{align}\label{semiob} 
		\|u\|_{L^2}\leq C\|b^{\frac{1}{2}}u\|_{L^2}+\frac{C}{h^2}\|(-h^2\Delta_G-1)u\|_{L^2},
	\end{align}
	which would only lead to the upper bound $O(h^{-3})$ for the norm of the non-selfadjoint resolvent $(-h^2\Delta_G-1+ihb)^{-1}$ (see also \cite[Prop. B.1, Cor. B.2]{LeS20}). Though it is not clear to the authors how to deduce from such upper bound the Schr\"odinger observability by an abstract argument, 
	the improvement of the $O(h^{-3})$ bound for the resolvent estimate of the damped-wave operator requires more work than proving the observability inequalities \eqref{Ob} or \eqref{semiob}.  Precisely, for more regular damping\footnote{For the bound $O(h^{-3})$ it is sufficient that $b \in L^\infty$.} $b(y)$,  Theorem \ref{t:main_theorem}  improves the resolvent bound up to $O(h^{-2-\delta_0})$ (and the stable rate of \eqref{e:damped_wave_equation}), which is almost sharp.  
\end{rem}

In connection with \cite{DK20} and \cite{K19}, we consider the following particular example regarding the regularity of $b$ in case (2), which allows us to reach the sharp resolvent estimate for the damped-wave Baouendi-Grushin operator.
Let $\nu >4$, $0 < \rho < 1$, $y_0 > 0 $, such that $y_0 + \rho < \pi$, we assume that $b$ has the form:
\begin{equation}
	\label{e:particular_b}
	b(y) = \left \lbrace \begin{array}{ll}
		0, & \text{if } \vert y \vert \leq y_0, \\[0.2cm]
		(\vert y \vert -  y_0)^{\nu}, & \text{if } y_0 \leq \vert y \vert < y_0 + \rho, \\[0.2cm]
		c(\vert y \vert), & \text{if } y_0 + \rho < \vert y \vert < \pi,
	\end{array} \right.
\end{equation}
for some positive and non-vanishing $c $ such that $b\in W^{4,\infty}(\T)$.

\begin{thm}[Widely undamped case]
	\label{t:sharp_in_1}
	Let $\nu>4$ and assume that $b \in W^{4,\infty}(\T)$ satisfies \eqref{e:particular_b}. Then there exist $h_0>0$ sufficiently small and $C_0 >0$ such that, for all $0<h\leq h_0$, we have:
	$$
	C_0^{-1} h^{-2 - \frac{1}{\nu+2}}  \leq  \|(-h^2\Delta_G-1\pm ihb)^{-1}\|_{\mathcal{L}(L^2)} \leq C_0h^{-2 - \frac{1}{\nu+2}}.
	$$
	Consequently, the system \eqref{e:damped_wave_equation} is stable at rate $t^{-\frac{\nu+2}{\nu+3}}$, which is optimal.
\end{thm}
\begin{rem}
	The optimal stable rate matches with the case for the flat Laplacian. More precisely,
	with the same $b(y)$, in \cite{DK20} and \cite{K19}, the authors proved that
	$$ \|(-h^2\Delta-1\pm ihb)^{-1}\|_{\mathcal{L}(L^2)}\sim h^{-2-\frac{1}{\nu+2}}.
	$$
	The optimality is saturated by elliptic quasimodes constructed in \cite{K19}.
\end{rem}

In case (3), our main result can be regarded as a subelliptic version of  \cite[Thms. 1.6 and 1.7]{LeL17}. We assume that $b$ satisfies the following \textit{finite-type} condition:
for $\nu \geq 6$ :
\begin{equation}
	\label{e:finite_type_condition}
	b(y) = \left \lbrace \begin{array}{ll}
		0, & \text{if } y = y_0, \\[0.2cm]
		\vert y -  y_0\vert^{\nu}, & \text{if } \vert y - y_0 \vert \leq \rho, \\[0.2cm]
		c(\vert y \vert), & \text{if } \vert y - y_0 \vert > \rho,
	\end{array} \right.
\end{equation}
for some positive and non-vanishing $c$ such that $ b\in W^{4,\infty}(\T)$.

\begin{thm}[Narrowly undamped case]
	\label{t:sharp_in_2}
	Let $ \nu \geq 6$ and assume that $b \in W^{4,\infty}(\mathbb{T})$ satisfies \eqref{e:finite_type_condition}. Then there exist $h_0>0$ sufficiently small and $C_0 >0$ such that, for all $0<h\leq h_0$, we have:
	\begin{align}
		\label{e:narrow_resolvent}
		C_0^{-1} h^{-2 + \frac{1}{\nu+1}}  \leq  \|(-h^2\Delta_G-1\pm ihb)^{-1}\|_{\mathcal{L}(L^2)} \leq C_0h^{-2 + \frac{1}{\nu+1}}.
	\end{align}
	Consequently, the system \eqref{e:damped_wave_equation} is stable at rate $t^{-\frac{\nu+1}{\nu}}$, which is optimal.
\end{thm}
The regularity assumption $\nu\geq 6$ is only technical, which is relevant for the proof of Proposition \ref{TH} in the narrowly undamped case ($j=3$). More important remarks are in order:
\begin{rem} 
	In the narrowly undamped case, the norm of the resolvent is larger than in the case of the elliptic Laplacian:
	$$ \|(-h^2\Delta-1\pm ihb)^{-1}\|_{\mathcal{L}(L^2)}\sim h^{-2+\frac{2}{\nu+2}},
	$$
	proved in \cite[Thms. 1.6 and 1.7]{LeL17}. This is in sharp contrast to the widely undamped situation \textnormal{(2)}, where the resolvent norm is the same as for the elliptic Laplacian. We emphasize the gain with respect to $O(h^{-2})$ in both cases, and in comparison with Theorems \ref{t:main_theorem} and \ref{t:sharp_in_1}.  
\end{rem}

\begin{rem}
	In \cite{BG20}, the authors prove uniform stabilization \eqref{uniform} for the classical damped-wave equation on $\T^2$ in a related geometrical situation (although with non-continuous damping) where only finitely many geodesics miss the damping region. More precisely, the damping region is a disjoint union of polygons and there exist undamped geodesics that follows for some time one of the sides of a polygon on the left, and for some other time one of the sides of a polygon on the right. 
\end{rem}

\begin{rem}
	From our proofs based on the averaging method with respect to the principal  selfadjoint part of the Baouendi-Grushin operator, it is most likely that in cases \textnormal{(2)} and \textnormal{(3)} one can relax the assumption that $b$ depends only on the vertical variable $y$, assuming only conditions \eqref{e:B-H_condition}, \eqref{e:particular_b}, and \eqref{e:finite_type_condition} for the average of $b$ along the horizontal variable $x$. 
\end{rem}

\subsection{Propagation of singularities for subelliptic quasimodes}

As a byproduct of our analysis, we are able to describe the structure of semiclassical measures associated with (a sequence of) quasimodes $(\psi_h)_{h>0}$ satisfying
\begin{align}\label{quasimodeso(h2)}
	( -h^2\Delta_G - 1) \psi_h =o_{L^2}(h^2),\quad \|\psi_h\|_{L^2}=1.
\end{align}
We will describe the concentration properties of the probability density $d\nu_h:=|\psi_h(x,y)|^2dxdy$ as $h\rightarrow 0$. Roughly speaking, any semiclassical limit of the lift of $(d\nu_h)_{h>0}$ to the phase space is decomposed into two parts, a measure $\mu_0$ capturing the mass on the compact part of the energy level $p^{-1}(1)$, and another one giving account of the energy trapped in the closure of the  region $\{ ( x, y, \xi, \eta) \in p^{-1}(1): \vert \eta \vert \geq R \}$, as $R \to \infty$. Let us define 
$$
\T_y\times\ov{\R_\eta \setminus \{0 \}} := (\T_y\times\mathbb{S}^0_\omega)\bigsqcup \, (\T_y\times( \R_\eta \setminus \{0 \})).
$$ 
Notice that the vertical flow $\phi_t^{\rm v}$ extends to $\T_y\times\ov{\R_\eta \setminus \{0 \}}$ as $\phi_t^{\rm v}(y, \omega) = (y + t\omega, \omega)$.

\begin{thm}\label{t:measurelimit} 
	Let $(\psi_h)_{h>0}$ be a sequence of quasimodes satisfying \eqref{quasimodeso(h2)}. Then there exist a subsequence $h_n\rightarrow 0$, positive Radon measures $\mu_0$ on $T^*\T^2$ and $\ov{\mu}_0$ on $\T_y\times \overline{\R_\eta \setminus \{ 0 \}}$, such that for any symbol $a_0 \in C_c^{\infty}(T^*\T^2)$ and any symbol $a_1=a_1(y,\eta,\theta)$, homogeneous of degree $0$ at infinity in $\eta$ and compactly supported in $\theta$, we have
	\begin{align*}
		\lim_{n\rightarrow+\infty} \big \langle\Op_{h_n}^{\w}(a_0+\widetilde{a}_1)\psi_{h_n},\psi_{h_n} \big \rangle_{L^2(\T^2)}=\int_{T^*\T^2}(a_0+a_1|_{ \theta =0})d\mu_0+\int_{\T_y\times\ov{\R_\eta \setminus \{0 \}} } \, \overline{a}_\infty \, d\ov{\mu}_0,
	\end{align*}
	where 
	\begin{equation}
		\label{e:closure_of_a}
		\overline{a}_\infty(y, \eta) = \left \lbrace \begin{array}{ll} \displaystyle a_\infty(y,\omega,0), & \eta = \omega  \in \mathbb{S}_\omega^0, \\[0.5cm]
			\displaystyle a_\infty\left( y, \frac{\eta}{\vert \eta \vert}, \eta \right), & \eta \in \R_\eta \setminus \{ 0 \},
		\end{array} \right.
	\end{equation}
	and we denote:
	\begin{equation}
		\label{e:further_notations_symbols}
		\widetilde{a}_1(y,\eta)=a_1(y,\eta,h\eta), \quad \text{ and }\quad  a_\infty \Big(y, \frac{\eta}{\vert \eta \vert}, \eta\Big) = \lim_{s\rightarrow+\infty}a_1\big(y,s\eta,\eta\big), \quad \text{for } \eta \neq 0.
	\end{equation}
	Moreover, these measures $\mu_0$ and $\overline{\mu}_0$ satisfy the following properties:
	\begin{itemize}
		\item The sum of total masses of $\mu_0$ and $\ov{\mu}_0$ is $1$.
		\smallskip
		
		\item The measure $\mu_0$ is invariant along the elliptic flow $\phi_t^{\mathrm{e}}$, while $\ov{\mu}_0$ is invariant along the vertical flow $\phi_t^{\mathrm{v}}$.
	\end{itemize}
\end{thm} 

\begin{rem}
	Theorem \ref{t:measurelimit} can be viewed as a generalization of the measure-invariance result  (part 1 of Theorem B) in \cite{CdVHT} to $o(h^2)$ quasimodes, in a much simpler geometric setting.  
\end{rem}



\subsection{Localization of spectrum and pseudo-spectrum}

In Section \ref{s:quasimodes}, following the main idea in \cite{Ar20}, we give a novel construction of quasimodes whose energy concentrates asymptotically on the subelliptic regime (see Section \ref{s:structure_proof} below) and that saturate the lower bounds on the resolvent \eqref{resolvent:EGCClower} in Theorem \ref{t:EGCC}, \eqref{resolvent:generallower} in Theorem \ref{t:main_theorem}, and the lower bound of \eqref{e:narrow_resolvent} in Theorem \ref{t:sharp_in_2}. Moreover, in Section \ref{s:compact_quasimodes} we construct quasimodes in the compact part of the energy level $p^{-1}(1)$. These quasimodes saturate the resolvent estimate of Theorem \ref{t:sharp_in_1}.

In view of \cite[Eq. (2.8), Prop. 2.8]{AL14}, which can be faithfully adapted to our setting, the resolvent estimate \eqref{resolvent1} (or equivalently the semiclassical resolvent \eqref{resolvent2}) implies the existence of constants $C> 0$ and $s_0 > 0$ such that there are no eigenvalues of $\mathcal{A}$ in the region 
$$
\vert \zeta \vert \geq s_0, \quad \vert \Re \zeta \vert \leq \frac{1}{C \vert \Im \zeta \vert^{\frac{1}{\alpha}}},
$$
together with the resolvent estimate
$$
\Vert (\mathcal{A} - \zeta \, \Id)^{-1} \Vert_{\mathcal{L}(\mathscr{H})} \leq C \vert \Im \zeta \vert^{\frac{1}{\alpha}}.
$$

On the other hand, we emphasize that the corresponding lower bounds on the resolvent, although showing that the estimate of the resolvent is optimal, do not imply an analogous localization property for the spectrum of the operator due to non-selfadjointness. To explain this phenomenon, let us connect the damped-wave Baouendi-Grushin operator with a related non-selfadjoint semiclassical operator  sharing some relevant spectral properties. In \cite[Eq. (7)]{ArR20} (see also \cite[Eq. (1.15)]{Ar20}), it is shown the  semiclassical resolvent estimate analogous to \eqref{resolvent2}, for a semiclassical harmonic oscillator $\widehat{H}_h = \frac{1}{2} \sum_{j=1}^d \omega_j ( -\partial_{x_j}^2 + x_j^2)$ perturbed by a subprincipal non-selfadjoint operator $h \widehat{P}_h$. This resolvent estimate is sharp, which is shown in \cite[Thm. 2]{Ar20}, proving the existence of quasimodes concentrating on the strip of the complex plane given by \cite[Eq. (1.15)]{Ar20}. However, assuming further hypothesis of analyticity on the symbol $P$ of $\widehat{P}_h$, in \cite[Thm. 7]{ArR20} (see also \cite{AsL}) it is shown that the spectrum of the semiclassical operator is separated from the real axis by a larger gap than the one predicted by the resolvent estimate. In other words, the spectrum lies deep inside the pseudo-spectrum, given essentially by the set of points of the complex plane $\{ P(z) \, : \, z \in T^*\R^d \}$ where the H\"ormander bracket condition $\{ \Re P, \Im P \}(z) \neq 0$ holds. 

In our setting, in the subelliptic regime the vertical flow plays the role of  the flow generated by the subprincipal (selfadjoint) perturbation, and the pseudo-spectrum appears since this vertical flow is transversal to the level sets of the damping term. We omit a further study on the precise  localization of the spectrum of the damped-wave Baouendi-Grushin operator since this is irrelevant for the sharp resolvent estimate \eqref{resolvent2}. We refer to \cite{DeSZ04} for a discussion of this non-selfadjoint phenomenon and for references to one dimensional examples. 

\subsection{Structure of the article and strategy of proof}
\label{s:structure_proof}
The proof of Theorems \ref{t:EGCC}, \ref{t:main_theorem}, \ref{t:sharp_in_1}, and \ref{t:sharp_in_2} will be divided into two parts of analysis: The proof of the upper bounds and the construction of quasimodes to show the lower bounds. 

The upper bound analysis will be organized in a rather unified way, by contradiction. Roughly speaking, we use diverse semiclassical  analysis tools to deal with different regimes of quasimodes, localized at semiclassical scale $|-\Delta_G|\sim h^{-2}$:
\begin{itemize}
	\item[(a)] Subelliptic regime: This corresponds to $h^{-1}\ll |D_y|\lesssim h^{-2}$.
	In Section \ref{subellipticPropagation}, we show that the subelliptic quasimodes cannot concentrate in the undamped region. We develop a completely different analysis, compared to the earlier work for the Schr\"odinger observability \cite{BuSun} and the self-adjoint resolvent estimate obtained in \cite{LeS20}. Our analysis is inspired by the proof of \cite[Thm. 3]{ArR20} via the study of semiclassical measures associated to subprincipal non-selfadjoint perturbations of the semiclassical harmonic oscillator. The main idea relies on obtaining a first invariance property of the semiclassical measure by the Hamiltonian flow of the principal part of the operator (a semiclassical harmonic oscillator $-\partial_x^2 + \eta^2 x^2$ with frequency $\vert \eta \vert$ given by the Fourier variable $\eta \leftrightarrow D_y$) and, in a second scale, to study the interaction between real (subprincipal self-adjoint remainders) and imaginary (damping term remainders) parts of the subprincipal symbols. In our setting, to capture the proper subelliptic scales at which the classical properties emerge, we introduce and systematically study the propagation properties of suitable two microlocal semiclassical measures adapted to the subelliptic regime of the phase-space. Roughly speaking, we split the semiclassical measure in a scalar measure capturing the regime $h^{-1}\ll |D_y| \ll h^{-2}$, and an operator-valued measure capturing the regime $\vert D_y \vert \sim h^{-2}$. This technique is inspired by \cite{F00} and the series of works \cite{AFM15}, \cite{AL14}, \cite{ALM16}, \cite{ALM16(2)}, \cite{Mac_Rive16}, \cite{Mac_Riv18}, \cite{Mac_Riv19}, where the authors fruitfully develop and apply (two-microlocal) semiclassical measures techniques to spectral problems involving high-energy asymptotics of completely integrable systems. As a by-product, we accomplish the proof of Theorem \ref{t:measurelimit}.

	\item[(b)] Compact regime: This corresponds to $|D_y|\lesssim h^{-1}$, and it is treated in Section \ref{normalformimproved}. The major difficulty appears in the horizontal trapped regime $|D_y|\ll h^{-1}$, when the elliptic geometric control condition (EGCC) is violated (in Theorem \ref{t:main_theorem}, Theorem \ref{t:sharp_in_1} and Theorem \ref{t:sharp_in_2}). We refine the Birkhoff normal form transform used in \cite{BuSun} and \cite{LeS20} to transform the quasimodes to the quasimodes associated with the flat Laplacian. The main difficulty compared to the self-adjoint resolvent estimate in \cite{LeS20} is that the Fourier truncation does not commute with the damping. As a result, one needs to choose appropriately the second semiclassical parameter, in order to retain the width of quasimodes after localization, and at the same time to produce acceptable remainders when performing the Birkhoff normal form. To prove the sharp upper bound in Theorem \ref{t:sharp_in_1}, successive normal form reductions are needed. 
	
\end{itemize}

To prove lower bounds in Theorems \ref{t:EGCC}, \ref{t:main_theorem}, \ref{t:sharp_in_1}, and \ref{t:sharp_in_2}, we construct suitable quasimodes:
\begin{itemize}
	\item[(a)] Subelliptic regime: The lower bounds in Theorem \ref{t:EGCC},  Theorem \ref{t:main_theorem}, and Theorem \ref{t:sharp_in_2} will be obtained by constructing quasimodes in the subelliptic regime. This is done in Section \ref{s:quasimodes}. We use the same idea developed in \cite{Ar20} (and inspired by \cite{Sj10}) consisting of averaging a solution to the time-dependent non-selfadjoint problem in a small interval of time to construct the desired quasimode. Instead of using propagation of wave-packets, as in \cite{Ar20}, here we consider the ground state of the quantum harmonic oscillator $- \partial_x^2 + \eta^2 x^2$, and propagate it by the vertical flow in the $y$ variable, via the inverse Fourier transform $\eta \mapsto y$. As we will see, this propagation is slighlty perturbed (at subprincipal scale) by the Taylor approximation of $V(x)$ near zero, so that a normal form reduction in terms of the harmonic oscillator is required to deal with this perturbation. We use for this reduction a well-known finite-dimensional version of the averaging method summarized in Section \ref{s:finite_dimensional_averaging} of the Appendix. In Section \ref{s:subelliptic_quasimodes} we construct quasimodes microlocalized away from the damping region. In this case the propagation follows the vertical flow, which at quantum level is unitary. On the other hand, in Section \ref{s:damped_quasimodes} we consider propagation within the damping region, so we have to deal with the effect of the damping term on the solution to the time-dependent equation. This effect essentially consists of a dramatic change in the $L^2$ norm of the solution due to friction, which manifests via multiplication by the exponential of the integral of $b(y)$ along the vertical flow. We obtain the desired quasimode by averaging the time-dependent solution in a small interval of time  ( possibly of length tending to zero as $h \to 0^+$) and renormalizing by a constant.
	
	\item[(b)] Compact regime: The lower bound in Theorem \ref{t:sharp_in_1} will be obtained by quasimodes in the horizontal trapped regime, in Section \ref{s:compact_quasimodes}. The idea is relatively simple. We start by the quasimodes that saturated the resolvent estimate for the flat-Laplacian on $\T^2$ (constructed in \cite{K19}), which are morally localized in the regime $|D_y|\leq O(1)$ and $|D_x|\sim h^{-1}$. Then we perform a Birkhoff normal form (inverse to the one used to prove the upper bound) to transform such quasimodes to the desired quasimodes of the Baouendi-Grushin damped wave operator which will saturate the lower bound of Theorem \ref{t:sharp_in_1}.
\end{itemize}


\subsection*{Acknowledgments} The first author has received funding from the European Research Council (ERC) under the European Union's Horizon 2020 research and innovation program (grant agreement No. 725967), and has been partially supported by grant MTM2017-85934-C3-3-P (MINECO, Spain). The second author is partially supported by the ANR grant ODA (ANR-18-CE40- 0020-01) and he would like to thank ICERM of Brown University for the kind hospitality, where the end of this research was carried out. We also thank Gabriel Rivi\`ere for his comments on a draft version of this article. Finally, we are grateful for anonymous referees.

\section{Notations and Preliminaries}

	\subsection{Global Notations} 
	We clarify the following reserved class of notations  throughout this article:
	\begin{itemize}
		
		\item We use the notations: $D_x=\frac{1}{i}\partial_x$, $D_y=\frac{1}{i}\partial_y$, so $\Delta_G=-D_x^2-V(x)D_y^2$. By the hypothesis of $V$ and the Morse Lemma, there exists a smooth function $W$ on $(-\pi,\pi)$ such that $V(x)=W(x)^2$. Note that $W(x)$ is not $2\pi$-periodic but $W(-\pi)=-W(\pi)$.  
		Sometimes we will refer to the operator $\Delta_{G_0} := \partial_x^2 + x^2 \partial_y^2$ acting on functions of $\R^2$.
		\item $b,b_0,b_1,b_2,b_3$ are always reserved as damping terms, which belong to $W^{k_0,\infty}(\T^2)$.
		Moreover, $\psi_h,\psi_{1,h},\ldots$ are reserved as quasimodes. We will also denote by $u_h,v_h,w_h$, quasimodes in some specific regimes.
		
		\item $\chi,\chi_0,\chi_1,\ldots$ are reserved as cutoff functions with specific support properties, depending on the context.
		
		\item $h,\hbar,\widetilde{h}$ are reserved as semi-classical parameters. We use the Weyl quantization $\Op_h^\w(\cdot)$, $\Op_1^\w(\cdot)$ throughout this article.  
		\item  For $h$-dependent quantities $A,B$, we use $A\lesssim B$ ($A\gtrsim B$) to stand for $A\leq CB$ ($A\geq C^{-1}B$) for some uniform constant $C$ which is \emph{independent} of the semi-classical parameter $h$. Sometimes we use subscripts $C_{\epsilon},C_R, C_N,\ldots$ to stand for inequalities with bounds depending on parameters $\epsilon,R,N,\ldots$. We use the notation  $A\ll B$ to stand for $\lim_{h \to 0} A/B = 0$.
		We mean by $O(h^{\alpha}), O_X(h^{\alpha}), o(h^{\alpha}), o_{X}(h^{\alpha}), O(1),O_X(1), o(1),o_X(1)$ to stand for asymptotic for quantities or asymptotic for $X$-norm of functions (operators) as $h\rightarrow 0$. The notations $O(h^{\infty}), O_X(h^{\infty})$ should stand for quantities (norms) that are bounded by $C_Nh^N$ for all $N\in\N$, as $h\rightarrow 0$. 
		\item We denote by $\nabla_G=(\partial_x,V(x)^{\frac{1}{2}}\partial_y)$ the Grushin-gradient.  We define the Grushin Sobolev space $H_G^k$ for $k\geq 1$ via the homogeneous semi-norms:
		\begin{align}
			\label{e:seminorms}
			\|u\|_{\dot{H}_G^k(\T^2)}:=\sum_{\mathcal{X}_1,\cdots,\mathcal{X}_k\in \{\partial_x,\;V(x)^{\frac{1}{2}}\partial_y\}}\|\mathcal{X}_1\cdots\mathcal{X}_ku\|_{L^2(\T^2)}.
		\end{align}
		That is, $\Vert u \Vert_{H_G^k(\T^2)} := \Vert u \Vert_{L^2(\T^2)} + \Vert u \Vert_{\dot{H}_G^k(\T^2)}$.
		\item The $L^2$ inner product of $f,g$ is given by $\lg f,g\rg_{L^2}$.
		
	\end{itemize}

	\subsection{Reduction to the study of quasimodes}
	Throughout this article, we will only focus on the study of the resolvent $(-h^2\Delta_G-1+ihb)^{-1}$, since the same argument works for the negative sign. 
	To better distinguish the different conditions on the damping terms, we consider
	$b_0(x,y)$ satisfying (EGCC) and (SGCC) in Theorem \ref{t:EGCC}, $b_1(y)$ the damping in Theorem \ref{t:main_theorem} satisfying \eqref{e:B-H_condition}, $b_2(y)$ the damping in Theorem \ref{t:sharp_in_1} satisfying \eqref{e:particular_b}, and $b_3(y)$ the damping in Theorem \ref{t:sharp_in_2} satisfying \eqref{e:finite_type_condition}. 
	We will prove the upper bounds in Theorems \ref{t:EGCC},  \ref{t:main_theorem},  \ref{t:sharp_in_1}, and \ref{t:sharp_in_2} by a uniform scheme of contradiction argument. 
	That is to say, if the upper bounds do not hold,  there exist sequences $h_n \rightarrow 0$ and $(\psi^{(j)}_{h_n})\subset L^2(\T^2)$, such that
	\begin{equation} 
		\label{e:quasimode_to_contradiction}
		\|\psi^{(j)}_{h_n}\|_{L^2}= 1,\quad  (-h_n^2 \Delta_G + i h_nb_j -1 )\psi^{(j)}_{h_n} = r_h^{(j)} =o_{L^2}(h_n^{2}\delta^{(j)}_{h_n}),\quad j=0,1,2,3.
	\end{equation}
	Respectively, $\delta_h^{(0)}=1$ for Theorem \ref{t:EGCC}, $\delta_h^{(1)}=h^{\delta_0}$ in Theorem \ref{t:main_theorem}, $\delta_h^{(2)} = h^{\frac{1}{\nu+2}}$ for Theorem \ref{t:sharp_in_1} and $\delta_h^{(3)}=h^{-\frac{1}{\nu+1}}$ for Theorem \ref{t:sharp_in_2}.
	The aim is therefore to reach a contradiction with respect to $b_j$ and $\psi_{h_n}^{(j)}$. In the sequel, we will drop the subindex in $h_n$ and write simply $h$, and the limit 
	$\lim_{h\rightarrow 0}$ should be understood as $\lim_{n\rightarrow \infty}$, possibly through a subsequence. Also, we set $P_{h,b}=-h^2\Delta_G+ihb$.
	\begin{definition}
		The (sequence of) solutions $(\psi_h)$ of $$(P_{h,b}-1)\psi_h=o_{L^2}(h^2\delta_h)$$ are called \emph{the $o(h^2\delta_h)$ $b$-quasimodes}. 
	\end{definition}
	
	We now fix cutoffs $\chi_0,\chi_1$ such that
	\begin{align} 
		\label{e:cut_off_0}
		& \chi_0\in C_c^{\infty}((-1,1)),  \hspace*{1.7cm} \chi_0(\eta) =1,\; \text{for } |\eta|\leq \frac{1}{2}, \\[0.2cm]
		\label{e:cut_off_1}
		&  \chi_1 \in C_c^{\infty}\left( \frac{1}{2}\leq |\eta|\leq 2 \right),  \quad \;  \, \chi_1(\eta)=1,  \;  \text{for } \frac{3}{4}\leq |\eta|\leq \frac{3}{2},
	\end{align}
	and present some a priori estimates for the $b$-quasimodes.
	\begin{lemma}\label{apriori} 
		Assume that $b$ satisfies \eqref{e:B-H_condition} and $\delta_h\leq h^{-\frac{1}{\nu+1}}$. For any $o(h^2\delta_h)$ $b$-quasimode $(\psi_h)$, we have the a priori estimates:
		\begin{align}\label{apriori1}
			& \|h\nabla_G\psi_h\|_{L^2}^2=1+o(h^{2}\delta_h),\quad \|b^{1/2}\psi_h\|_{L^2}=o(h^{\frac{1}{2}}\delta_h^{\frac{1}{2}}).
		\end{align}
	 Moreover, if $\widetilde{h}=\widetilde{h}(h)$ is another semiclassical parameter, i.e. $\widetilde{h}(h) \to 0$ as $h\rightarrow 0$, then for any $\widetilde{h}$-pseudodifferential operator $A_{\widetilde{h}}$ of order $0$,
		$$ \|b^{\frac{1}{2}}A_{\widetilde{h}}\psi_h\|_{L^2}=o(h^{\frac{1}{2}}\delta_h^{\frac{1}{2}})+o(\widetilde{h}).
		$$
	\end{lemma}
	\begin{proof}
		Multiplying the equation \eqref{e:quasimode_to_contradiction} by $\ov{\psi}_h$ and computing
		$$ 
		\Re\lg (-h^2 \Delta_G + i hb -1 ) \psi_h,\psi_h\rg_{L^2},\quad \Im\lg (-h^2 \Delta_G + i hb -1 )\psi_h,\psi_h\rg_{L^2},
		$$
		we get the desired bounds in \eqref{apriori1}. 
		
		 For the last assertion, we write
		$$ b^{\frac{1}{2}}A_{\widetilde{h}}\psi_h=A_{\widetilde{h}}(b^{\frac{1}{2}}\psi_h)+[b^{\frac{1}{2}},A_{\widetilde{h}}]\psi_h.		$$
		The $L^2$ norm of the first term on the right hand side is $o(h^{\frac{1}{2}}\delta_h^{\frac{1}{2}}).$ To treat the second term, from the symbolic calculus, 
		$$ [b^{\frac{1}{2}},A_{\widetilde{h}}]=\frac{\widetilde{h}}{i}\mathrm{Op}_{\widetilde{h}}^{\w}(\{b^{\frac{1}{2}},a\} )+O_{\mathcal{L}(L^2)}(\widetilde{h}^2),
		$$
		where $a$ is the principal symbol of $A_{\widetilde{h}}$. By \eqref{e:B-H_condition}, there exists $C>0$ such that
		$$ |\{b^{\frac{1}{2}},a\}|\leq Cb^{\frac{1}{2}-\sigma}.
		$$
		Therefore, by the sharp G$\mathring{\mathrm{a}}$rding inequality,
		$$\langle Cb^{1-2\sigma}\psi_h,\psi_h\rangle_{L^2}\geq 
		\|\mathrm{Op}_{\widetilde{h}}^{\w}(\{b^{\frac{1}{2}},a\} )
		\psi_h\|_{L^2}^2-O(\widetilde{h}),
		$$
		thus $\widetilde{h}\|\mathrm{Op}_{\widetilde{h}}^{\w}(\{b^{\frac{1}{2}},a \} )\psi_h \|_{L^2}\leq C\widetilde{h}\|b^{\frac{1}{2}-\sigma}\psi_h\|_{L^2}+O(\widetilde{h}^{\frac{3}{2}})$.
		Notice that, by interpolation $\|b^{\frac{1}{2}-\sigma}\psi_h\|_{L^2}\leq \|b^{\frac{1}{2}}\psi_h\|_{L^2}^{1-2\sigma}\|\psi_h\|_{L^2}^{2\sigma}=o(h^{\frac{1-2\sigma}{2}}\delta_h^{\frac{1-2\sigma}{2}})$, we have in particular that $\|[b^{\frac{1}{2}},A_{\widetilde{h}}]\psi_h\|_{L^2}=o(\widetilde{h})$. This completes the proof of Lemma \ref{apriori}. 
	\end{proof}
	Consequently,  $\chi_1(-h^2\Delta_G)\psi_h$ are still $o(h^2\delta_h)$ quasimodes and $(1-\chi_1(-h^2\Delta_G))\psi_h=o_{L^2}(1)$. Therefore, in the sequel, we always assume that $o(h^2\delta_h)$ quasimodes $(\psi_h)$ satisfy $\psi_h=\chi_1(-h^2\Delta_G)\psi_h$, without loss of generality.
	
	Next we collect some subelliptic a priori estimates:
	
	\begin{lemma}\label{subellipticapriori}
		For any (regular) function $f$ on $\T^2$, 
		$$ \||D_y|f \|_{L^2}\leq C_0\|\Delta_G f\|_{L^2}+\|f\|_{L^2}.
		$$
		In particular, for the $o(h^2\delta_h)$ $b$-quasimodes $\psi_h$ satisfying \eqref{e:quasimode_to_contradiction}, we have 
		\begin{equation}
			\label{e:subelliptic_compactness} 
			\| h^2 |D_y| \psi_h\|_{L^2(M)}\leq O(1). 
		\end{equation}
	\end{lemma}
	\begin{proof}
	Pick a bump function $\chi(x)$ with support sufficiently close to $x=0$ and $\chi\equiv 1$ near $x=0$. We divide $f$ by two pieces $f_1:=\chi f$ and $f_2:=(1-\chi)f$.   Since $-\Delta_G$ is elliptic outside $x=0$,  we have
	$$ \|f_2\|_{H^2}\lesssim \|\Delta_G((1-\chi) f)\|_{L^2}+\|(1-\chi)f\|_{L^2}\lesssim \|\Delta_Gf\|_{L^2}+\|\nabla_Gf\|_{L^2}+\|f\|_{L^2}.
	$$	
	Observe that $\|\nabla_Gf\|_{L^2}\leq \|\Delta_G f\|_{L^2}^{\frac{1}{2}}\|f\|_{L^2}^{\frac{1}{2}}$, we conclude the estimate for $f_2=(1-\chi)f$. 
		
	It remains to deal with $f_1$ which is supported near $x=0$. Consider $f_{\pm}:=\Pi_{\pm}f_1$, where $\Pi_{+}$ is the Fourier projection in $y$ such that $\mathcal{F}_y\Pi_+=\mathbf{1}_{n\geq 0}$ and $\Pi_-=\mathrm{Id}-\Pi_+$. With this notation, we have $D_yf_{\pm}=\pm|D_y|f_{\pm}$. Note that $D_y$ commutes with $\Delta_G$, we may prove the desired inequality separately for $f_{\pm}$ and then use Plancherel's identity to sum them up. Recall that $W(x)=V(x)^{\frac{1}{2}}$. For $\mathcal{X}_1=D_x, \mathcal{X}_2=W(x)D_y$, the vector field $[\mathcal{X}_1,\mathcal{X}_2]=-iW'(x)D_y$ is non-degenerate on the support of $f_{\pm}$, we have
		$$ \|D_yf_{\pm}\|_{L^2(M)}\lesssim \|[\mathcal{X}_1,\mathcal{X}_2]f_{\pm}\|_{L^2(M)}\leq \|\mathcal{X}_1\mathcal{X}_2f_{\pm}\|_{L^2(M)}+\|\mathcal{X}_2\mathcal{X}_1f_{\pm}\|_{L^2(M)}\leq \|f_{\pm}\|_{\dot{H}_G^2(M)}. 
		$$ 
		From the equivalent norm on $H_G^2$ (Lemma \ref{lemmaB1}), the proof is complete.
	\end{proof}

	The third a priori estimate concerns the localization property in the $x$ variable when $|D_y|\gg h^{-1}$:
	\begin{lemma}\label{apriori3}
		Let $\epsilon_1>0$ be sufficiently small. Let $\delta_1>0$ be a small constant such that $$\min_{\delta_1<|x|\leq\pi}|W(x)|>4\epsilon_1.
		$$
		Then for any $N\in\N$, there exists $C_N>0$, such that for all $f_h\in L^2(M)$, we have
		\begin{align*}
			& \big\|(1-\chi_0(\delta_1^{-1}x))(1-\chi_0(\epsilon_1hD_y))\chi_1(-h^2\Delta_G)f_h\big\|_{L^2} \\[0.2cm]
			& +\big\|(1-\chi_0(\delta_1^{-1}x))(1-\chi_0(\epsilon_1hD_y))h\pa_x \chi_1(-h^2\Delta_G)f_h \big\|_{L^2}   \\[0.2cm] 
			& \hspace*{5cm} \leq  C_Nh^N\big(\|f_h\|_{L^2}+\|h\nabla_Gf_h\|_{L^2} \big).
		\end{align*}
	\end{lemma}
	\begin{proof}
		The key point is that the support of $1-\chi_0(\delta_0^{-1}x)$ is contained in the classical forbidden region of the operator $-h^2\partial_x^2+h^2\partial_y^2W(x)^2$. The proof is essentially the same as the proof of Lemma 2.3 in \cite{LeS20}, hence we omit the details.
	\end{proof}


\section{Propagation of semiclassical measures in the subelliptic regime}\label{subellipticPropagation}

In this section we study the sub-ellitpic regime $h^{-1}\ll |D_y|\leq C_0h^{-2}$. In order to microlocalize quasi-modes into this regime, we denote, for any sequences $R\gg 1$,  $h\ll 1$ obeying $R\leq \frac{1}{C_0h}$ (this constraint comes from Lemma \ref{subellipticapriori}), $$\Upsilon_h^{R} := \Big(1 - \chi_0\Big( \frac{hD_y}{R}\Big)\Big)\chi_1(-h^2\Delta_G).$$
Throughout this section, $P_{h,b}=-h^2\Delta_G+ihb,$ and we assume that $\psi_h$ solves the equation (slightly more general than \eqref{e:quasimode_to_contradiction})
\begin{align}\label{e:quasimode_equation}
	(P_{h,b}-\zeta_h)\psi_h=r_h=o_{L^2}(h^2\delta_h),
\end{align}
where $\zeta_h = 1 + ih \beta_h$ with $0 \leq \beta_h \leq Ch$. The parameter $\delta_h$ will be fixed differently in this section, according to the specific situations of the damping $b$.
The goal of this section is to show that the quasimodes \emph{cannot} concentrate in the subelliptic regime.
Firstly, for the widely undamped case and the geometric control case, we will prove the following propagation theorem:
\begin{prop}\label{Measureformulationj=012}
	Assume that $\delta_h\leq 1$ and $\zeta_h=1+ih\beta_h$ for some $0\leq\beta_h\leq Ch$. Assume that $(\psi_h)$ satisfies: $$\|\psi_h\|_{L^2}=1,\quad 
	(P_{h,b}-\zeta_h)\psi_h=r_h=o_{L^2}(h^2\delta_h).
	$$
	Then up to the extraction of a subsequence, there exist positive Radon measures $\mu_0$ on $T^*\T^2$ and $\ov{\mu}_0$ on $\T_y\times \overline{\R_\eta \setminus \{ 0 \}}$, such that for any symbols $a_0\in C_c^{\infty}(T^*\T^2)$ and $a_1=a_1(y,\eta,\theta)$, homogeneous of degree $0$ at infinity in $\eta$ and compactly supported in $\theta$, we have
	\begin{align*}
		\lim_{h\rightarrow 0} \big \langle\Op_{h}^{\w}(a_0+\widetilde{a}_1)\psi_{h},\psi_{h}\big \rangle_{L^2(\T^2)}=\int_{T^*\T^2}(a_0+a_1|_{\theta=0})d\mu_0+\int_{\T_y\times \overline{\R_\eta \setminus \{ 0 \}} } \,  \overline{a}_{\infty} \,d\ov{\mu}_0,
	\end{align*}
	where we use the notations \eqref{e:closure_of_a} and \eqref{e:further_notations_symbols}. These measures satisfy the following properties:
	\begin{itemize}
		\item The sum of the total masses of $\mu_0,\ov{\mu}_0$ is $1$.
		\smallskip
		
		\item The measure $\mu_0$ is invariant along the elliptic flow $\phi_t^{\mathrm{e}}$.
		\smallskip
		\item $\langle\mu_0,b\rangle=\lg\ov{\mu}_0,b\rg=0$.
	\end{itemize}
	Furthermore:
	\begin{itemize}
		\item If $\int_{\T^2}b>0$ and $b=b_j$ for $j\in\{0,1,2\}$, we have $\ov{\mu}_0=0$. 
		\smallskip
		
		\item If $b=0$ and $\beta_h=o(h)$, the measure $\ov{\mu}_0$ is invariant along the vertical flow $\phi_t^{\mathrm{v}}$.
	\end{itemize}
\end{prop}
Note that the above proposition contains Theorem \ref{t:measurelimit}. It has the following immediate corollary: 
\begin{cor}
	\label{p:negative_sub-ellitpic_quasimodes}
	Let $j\in\{0,1,2\}$ and $\psi_{h}^{(j)}$ be $o(h^2\delta_h^{(j)})$ $b_j$-quasimodes. 	Then up to extracting a subsequence, $$\lim_{R\rightarrow\infty}\lim_{h\rightarrow 0}\Vert \Upsilon_h^{R} \psi_h^{(j)} \Vert_{L^2(\T^2)} = 0.$$
\end{cor}
By a contradiction argument, we complete the proof of the upper bound in Theorem \ref{t:EGCC} as follows:
\begin{proof}[Proof of Theorem \ref{t:EGCC}]
	By contradiction, assume that $\psi_{h}^{(0)}$ satisfies \eqref{e:quasimode_to_contradiction}. Applying Proposition \ref{Measureformulationj=012} to $\psi_h^{(0)}$, up to a subsequence, we obtain semi-classical measures $\mu_0$ and $\ov{\mu}_0$, with total mass $1$, and $\ov{\mu}_0=0$ (as $b_0$ satisfies (SGCC)). By the invariance of $\mu_0$ along $\phi_t^{\mathrm{e}}$, we have
	$$ 0=\langle b_0,\mu_0\rangle=\frac{1}{T}\int_0^T\langle b_0\circ \pi_1 \circ \phi_t^{\mathrm{e}},\mu_0\rangle dt
	$$
	for all $T>0$. Letting $T\rightarrow\infty$, we obtain that 
	$$ \Big\langle\mu_0, \liminf_{T\rightarrow\infty}\frac{1}{T}\int_0^Tb_0\circ \pi_1 \circ \phi_t^{\mathrm{e}}dt  \Big\rangle=0.
	$$
	As $b_0$ satisfies (EGCC), the function $[b_0]:=\liminf_{T\rightarrow\infty}\frac{1}{T}\int_0^Tb_0\circ \pi_1 \circ \phi_t^{\mathrm{e}}dt$ is strictly positive on $\rho\in p^{-1}(1)$. This implies that $\mu_0$=0, which contradicts to the fact that $\mu_0,\ov{\mu}_0$ have total mass $1$.
\end{proof}

For the narrowly damped case, we will prove in Section \ref{s:narrowly_undamped_measures}:
\begin{cor}
	\label{p:negative_sub-ellitpic_quasimodes(2)}
	Assume that there exist $\psi_h \in L^2(\T^2)$ with $\Vert \psi_h \Vert_{L^2(\T^2)} = 1$ and $\zeta_h = 1 + ih \beta_h$ with $0 \leq \beta_h \leq Ch^{1-\frac{1}{1+\nu}}$, such that
	\begin{equation}
		\label{e:quasimode_equation(2)}
		(P_{h,b_3} - \zeta_h )\psi_h = r_h =o_{L^2}\big( h^{2-\frac{1}{1+\nu}} \big).
	\end{equation}
	Then up to extracting a subsequence, $$\lim_{R\rightarrow\infty}\lim_{h\rightarrow 0}\Vert \Upsilon^{R}_h \psi_h \Vert_{L^2(\T^2)} = 0.$$
\end{cor}
The strategy to prove Proposition \ref{Measureformulationj=012} and Corollary \ref{p:negative_sub-ellitpic_quasimodes(2)} is based on the study of two-microlocal semiclassical measures and previous works on the theory of non-selfadjoint perturbations of semiclassical harmonic oscillators \cite{ArR20}. 
Before constructing the measures, it is more convenient to reduce the problem completely in the subelliptic regime near $x=0$.  The argument of reduction below works for $\delta_h\leq h^{-\frac{1}{1+\nu}}$. 

First,  as in the proof of Lemma \ref{apriori} (here we are dealing with more general quasimodes with $\zeta_h=1+ih\beta_h$), we have
\begin{align}\label{dampedregionweak}
	\|b^{\frac{1}{2}}\psi_h\|_{L^2(\T^2)}=O(\beta_h^{\frac{1}{2}}).
\end{align}
Then the new quasimodes $\psi_{h,\mathrm{s}}:=\Upsilon^{R}_h \psi_h$ solve equations
\begin{align}\label{eq:uh}
	(P_{h,b}-\zeta_h)\psi_{h,\mathrm{s}}=\Upsilon_h^{R}r_h+ih[b,\Upsilon_h^{R}]\psi_h.
\end{align}
Note that the commutator term can be written as $$ih[b^{1/2},[b^{1/2},\Upsilon_h^{R}]]+2ih[b^{1/2},\Upsilon_h^{R}]b^{1/2},$$ thanks to the condition \eqref{e:B-H_condition} and the fact that $\sigma<\frac{1}{4}$. Using \eqref{dampedregionweak}, we deduce that $ih[b,\Upsilon_h^{R}]\psi_h=O_{L^2}(h^{2}\beta_h^{\frac{1}{2}})$. In particular, by the assumption of $\beta_h\leq Ch^{1-\frac{1}{1+\nu}}$, we deduce that $(P_{h,b}-\zeta_h)\psi_{h,\mathrm{s}}=o_{L^2}(h^2\delta_h)$. To localize near $x=0$, we take another cutoff $\chi_{\delta}(x)$ such that supp$\chi_{\delta}'\subset \mathrm{supp}(1-\chi_{0}(\delta\cdot))$. Replacing $\psi_{h,\mathrm{s}}$ by $\chi_{\delta}(x)\psi_{h,\mathrm{s}}$, we still have
\begin{equation} 
	\label{e:localized_x_quasimode}
	(P_{h,b}-\zeta_h)(\chi_{\delta}\psi_{h,\mathrm{s}})=o_{L^2}(h^2\delta_h),
\end{equation}
since $\psi_h$ is $O_{L^2}(h^{\infty})$ on the support of $\chi_{\delta}'(x)$, thanks to Lemma \ref{apriori3}. Moreover, \eqref{dampedregionweak} holds by replacing  $\psi_{h}$ to $\chi_{\delta}\psi_{h,\mathrm{s}}$.
Therefore, we will work directly with $\chi_{\delta}(x)\psi_{h,\mathrm{s}}$ which is supported on a \emph{fixed} compact set of $M_0:=\R_x\times\T_y$. This allows us to construct two-microlocal semiclassical measures in the following subsections on the manifold $M_0$ to simplify arguments. With abuse of notation, we denote again $\psi_{h,\mathrm{s}}=\chi_{\delta}(x)\Upsilon_h^{R}\psi_h$ and this function satisfies the same quasimode equation \eqref{e:quasimode_equation} as $\psi_h$.

\vspace{1cm}

We next define the concept of $(-h^2 \Delta_{G_0})$-oscillation for a sequence $(u_h) \subset L^2(M_0)$, meaning that the energy of the sequence concentrates semiclassically around the level set of the Hamiltonian $H_{G_0}=\xi^2+x^2\eta^2$ given by the principal symbol of $-h^2 \Delta_G$.
Recall that $\chi_0$ was introduced in \eqref{e:cut_off_0}. 

\begin{definition}
	A sequence $(u_h) \subset L^2(M_0)$ with $\Vert u_h \Vert_{L^2(M_0)} = 1$ is called $(-h^2 \Delta_{G_0})$-oscillating if, given any bump function $\chi \in \mathcal{C}_c^\infty(\R)$ equal to one near zero and satisfying
	$$ \chi_0|_{\mathrm{supp}(\chi)}\equiv 1 ,
	$$
	there holds
	$$
	\lim_{\lambda \to \infty} \limsup_{h \to 0^+} \left \Vert   \chi_{\lambda} ( -h^2\Delta_{G_0} - 1 )  u_h \right \Vert_{L^2(M_0)} = 1,
	$$
	where $\chi_{\lambda}(\cdot):=\chi(\lambda^{-1}\cdot)$ and $\Delta_{G_0}=-D_x^2 -x^2D_y^2$.
\end{definition}

\begin{lemma}\label{oscillating} 
	The quasimode $\psi_{h}$ is $(-h^2\Delta_{G_0})$-oscillating. More precisely,
	$$ (-h^2\Delta_{G_0}-1)\psi_{h}=o_{L^2}(1),
	$$
	as $h\rightarrow 0$.  
\end{lemma}
\begin{proof}
	 Without loss of generality, we normalize the $L^2$ norm of $\psi_{h}$ to be 1.		From \eqref{dampedregionweak} and \eqref{e:localized_x_quasimode}, we have
			$$ (-h^2\Delta_{G_0}-1)\psi_{h}=(x^2-W(x)^2)h^2D_y^2\psi_{h}+O_{L^2}(h\beta_h^{\frac{1}{2}}).
			$$
			Since $-h^2\Delta_{G_0}-1$ is invertible when acting on $\big(1-\chi_{\lambda}(h^2\Delta_{G_0}-1)\big)\psi_{h}$,
			it suffices to show that
			\begin{align}\label{comparison1} 
				\lim_{h\rightarrow 0}\|(x^2-W(x)^2)h^2D_y^2\psi_{h}\|_{L^2}=0.
			\end{align}
			We mimic the proof in \cite{BuSun}.  For any $R\gg 1$, we break $\psi_{h}$ as $\psi_{h}\mathbf{1}_{|x|<R^{-\frac{1}{2}}}+\psi_{h}\mathbf{1}_{|x|\geq R^{-\frac{1}{2}}}$. Since near $x=0$, $x^2-W(x)^2=O(x^3)$, we have for any $h>0$,
			$$ \big\|(x^2-W(x)^2)h^2D_y^2\psi_{h} \big\|_{L^2(|x|<R^{-\frac{1}{2}})}\leq CR^{-\frac{1}{2}}\|W(x)^2h^2D_y^2\psi_{h}\|_{L^2}\leq CR^{-\frac{1}{2}},
			$$
			thanks to Lemma \ref{lemmaB1} and the fact that $\|h^2\Delta_G\psi_{h}\|_{L^2}=O(1)$.
			Therefore, it suffices to show that for any fixed $R\gg 1$,
			\begin{align*} \|x^2h^2D_y^2\psi_{h}\|_{L^2(|x|\geq  R^{-1/2})}=O(h^{\infty}).
			\end{align*}
			Since $x^2h^2D_y^2$ does not change $\mathrm{WF}_h(\psi_{h})$, we only need to show that
			\begin{align}\label{claim}  \|\psi_{h}\|_{L^2(|x|\geq  R^{-1/2})}=O(h^{\infty}).
			\end{align}
			For a bump function $\chi$, denote by $\widetilde{\chi}(x):=\chi(R^{1/2}x)$. Multiplying by $(1-\widetilde{\chi}(x))^2\ov{\psi}_{h}$ the quasimode equation $(-h^2\Delta_G-1)\psi_{h}=O_{L^2}(h\beta_h^{\frac{1}{2}})$ and using integration by parts, we get
			\begin{align}\label{claim1}
				\int_{\R\times\T}(1-\widetilde{\chi}(x))^2\big( |h\partial_x\psi_{h}|^2+W(x)^2|h\partial_y\psi_{h}|^2 -|\psi_{h}|^2 \big)dxdy & \\[0.2cm]
				& \hspace*{-8cm} =-hR^{1/2}\int_{\R\times\T}2(1-\widetilde{\chi}(x))\chi'(R^{1/2}x)h\partial_x\ov{\psi}_{h}\cdot\psi_{h}dxdy+O(h\beta_h^{\frac{1}{2}}). \label{claim2}
			\end{align}
			Since on the support of $1-\widetilde{\chi}$, $|x|\gtrsim R^{-1/2}$, 
			combining with the definition of $\psi_{h}$, we have 
			$$ \int_{\R\times\T}(1-\widetilde{\chi}(x))^2W(x)|h\partial_y\psi_{h}|^2dxdy\gtrsim R^{-1}\int_{\R\times\T}(1-\widetilde{\chi}(x))^2R^2|\psi_{h}|^2dxdy.
			$$
			Thus \eqref{claim1} is bounded from below by a (small) constant times
			$$ \int_{\R\times\T}(1-\widetilde{\chi}(x))^2 \big( |h\partial_x\psi_{h}|^2+R|\psi_{h}|^2 \big)dxdy,
			$$
			while \eqref{claim2} is bounded from above by $O(R^{1/2}h)$. In particular, this shows that
			$$ \|(1-\widetilde{\chi}(x))\psi_{h}\|_{L^2}\leq O(h).
			$$
			Repeating the argument above, we obtain \eqref{claim}. This completes the proof of Lemma 
			\ref{oscillating}.
	
\end{proof}

\subsection{Construction of two-microlocal semiclassical measures}

In order to define properly two-microlocal semiclassical measures, we need suitable classes of test symbols:
\begin{definition}\label{s:two-microlocal_measures}
	We say that $a \in \mathcal{C}^\infty \big(\R^5_{\sigma,y,\xi,\eta,\theta}\big)$ belongs to the class $\mathbf{S}_{c}^0(\R^5)$ if it has compact support in the variables $(\sigma,y,\xi,\theta)$, and is homogeneous of degree zero at infinity in the variable $\eta$, that is there exists $a_\infty \in \mathcal{C}_c^\infty\big( \R^4_{\sigma,y,\xi,\theta} \times \mathbb{S}^0_\omega\big)$ and $R > 0$ such that
	$$
	a(\sigma,y,\xi,\eta,\theta) = a_\infty\left( \sigma,y ,\xi,\frac{\eta}{\vert \eta \vert},\theta \right), \quad \text{for } \vert \eta \vert \geq R.
	$$
	We say that $\widetilde{a} \in \mathcal{C}^\infty\big( \R^3_{y,\eta,\theta} \big)$ belongs to $\mathbf{S}_{c}^0(\R^3)$ if it is compactly supported in $(y,\theta)$ and is homogeneous of degree zero at infinity in the variable $\eta$, that is, there exists $\widetilde{a}_\infty \in \mathcal{C}_c^\infty\big(\R^2_{y,\eta} \times \mathbb{S}_\omega^0\big)$ and $R > 0$ such that
	$$
	\widetilde{a}(y,\eta,\theta) = \widetilde{a}_\infty\left(y,\frac{\eta}{\vert \eta \vert},\theta \right), \quad \text{for } \vert \eta \vert \geq R.
	$$
\end{definition}  
For our problem, symbols are periodic in $y$, and we denote by $\mathbf{S}_c^0(T^*M_0\times\R)$, $\mathbf{S}_c^0(T^*\T\times\R)$ the respective classes. 

		For any $a \in \mathbf{S}_{c}^0(T^*M_0\times\R)$ and $\widetilde{a} \in \mathbf{S}_c^0(T^*\T\times\R)$, we define respectively the associated symbols
		\begin{align}\label{bfa1} 
			\mathbf{a}^h_R(x,y,\xi,\eta) & :=  (1 - \chi_R(h\eta)) a\left(hx\vert \eta \vert, y,\xi, h\eta, h^2 \eta \right), \\[0.2cm]
			\widetilde{\mathbf{a}}^h_R(y,\eta) & := (1 - \chi_R(h\eta)) \widetilde{a}\left(y,h\eta, h^2 \eta \right). \label{bfa1tilde}
		\end{align}
Correspondingly, the \textit{Wigner distribution} $I_R^h$ on $\mathbf{S}_{c}^0(T^*M_0\times\R)$ and on $\mathbf{S}_c^0(T^*\T\times\R)$ are defined respectively by:
\begin{align*}
	I_R^h(a) & := \big \langle \Op_1^{\mathrm{w}}(\mathbf{a}^h_R) u_h , u_h \big \rangle_{L^2(M_0)}, \\[0.2cm]
	I_R^h(\widetilde{a}) & := \big \langle \Op_1^\w(\widetilde{\mathbf{a}}^h_R) u_h , u_h \big \rangle_{L^2(M_0)}.
\end{align*}
We are going to show that, under suitable limiting process, the Wigner distributions converge to adapted two-microlocal semiclassical measures:  
\begin{remark}
			The intuition for introducing the symbol $\mathbf{a}_R^h(x,y,\xi,\eta)$ with  $(hx|\eta|,h\xi)$ as horizontal dual variables comes from the noncompacity of the energy level set $\xi^2 + x^2\eta^2 = 1$, distinctive of subelliptic operators. Precisely, when $\vert \eta \vert$ grows at scale $h^{-1}\ll \vert \eta \vert \leq h^{-2}$, $x$ concentrates at scale $\vert h \eta \vert^{-1}$, so we need to balance the variable $x$ according to such scaling to see propagation along the horizontal direction in the semiclassical limit. In other words, the semiclassical parameter measuring concentration in $x$ near zero is precisely $\vert h \eta \vert^{-1}$. This parameter depending on $\eta$ causes some new issues concerning symbolic calculus, as we will see, but we will show that it is still possible to define suitable semiclassical measures as limits of the Wigner distribution defined in terms of this family of symbols. Indeed, we consider this generalization very natural since it captures in a very precise way the different scales of oscillation in the subelliptic regime.
\end{remark}

\begin{remark}\label{idea2microlocal}
For symbols $\mathbf{a}_R^h$ as defined above, we are not able to perform directly the semiclassical symbolic calculus, since the remainders concerning the horizontal variables $(x,\xi)$ may be of order $O(1)$ as $h \to 0^+$ (they do not decrease in the semiclassical limit), due to the fact that this calculus is critical with respect to the uncertainty principle. We thus need to further split $\mathbf{a}_R^h$ into two parts. For the part where $\vert \eta \vert$  is relatively small, $h^{-1}\ll |\eta|\ll h^{-2}$, we can still use standard semiclassical symbolic calculus in $(x,\xi)$ ($x$ concentrates at scale $\gg h$ and $\xi$ oscillates at size $h^{-1}$, so this calculus is subcritical with respect to the uncertainty principle); in this regime we will show that the corresponding contribution in $I_R^h(\cdot)$ defines a scalar-valued semiclassical measure in the successive limits $h\rightarrow 0$, $R\rightarrow\infty$. While for the part where $|\eta|\sim h^{-2}$, we will show that the corresponding contribution in $I_R^h(\cdot)$ defines an operator-valued semiclassical measure in the semiclassical limit. 
\end{remark} 
\begin{prop}
	\label{p:semiclassical_measures}
	Let $(u_h) \subset L^2(M_0)$ be a $L^2$-normalized sequence, supported on a fixed compact set $K_0\subset M_0=\R_x\times\T_y$. Then, modulo a subsequence, there exist non-negative Radon measures $\mu_1 \in \mathcal{M}_+\big( \R_\sigma \times \mathbb{T}_y \times \R_\xi \times \mathbb{S}^0_\omega \big)$ and $M_2 \in \mathcal{M}_+\big( \mathbb{T}_y \times (\R_\eta \setminus \{0 \}); \mathcal{L}^1(L^2(\R_x)) \big)$, where $M_2$ takes values in the space of trace operators $\mathcal{L}^1(L^2(\R_x))$ such that, for every $a \in \mathbf{S}_c^0(T^*M_0\times\R)$,
	\begin{align}\label{measureformula1} 
		\lim_{R \to \infty} \lim_{h \to 0^+} I_R^h(a)  & = \int_{\R_\sigma \times \mathbb{T}_y \times \R_\xi \times \mathbb{S}^0_\omega} \, a(\sigma,y,\xi,\omega,0) \mu_1(d\sigma,dy,d\xi,d\omega)  \\[0.2cm]
		& \quad  +\operatorname{Tr} \left[ \int_{\mathbb{T}_y \times (\R_\eta \setminus \{0 \})}  a_\infty^{\w} \left( x \vert \eta \vert  ,y, D_x, \frac{\eta}{\vert \eta \vert},\eta \right) M_2( dy,d\eta ) \right]. \label{measureformula2} 
	\end{align}
	Moreover, there exist non-negative Radon measures $\overline{\mu}_1 \in \mathcal{M}_+(\T_y \times \mathbb{S}_\omega^0)$ and $\overline{\mu}_2 \in \mathcal{M}_+(\T_y \times (\R_\eta \setminus \{0 \}))$ such that, for the same subsequence and every $\widetilde{a} \in \mathbf{S}_c^0(T^*\mathbb{T} \times \R)$, 
	\begin{align}\label{measureformula3} 
		\lim_{R \to \infty} \lim_{h \to 0^+} I_R^h(\widetilde{a}) = \int_{\T_y \times \mathbb{S}^0_\omega} \widetilde{a}(y,\omega,0) \overline{\mu}_1(dy,d\omega) + \int_{\T_y \times (\R_\eta \setminus \{0 \})}  \widetilde{a}_\infty\left( y, \frac{\eta}{\vert \eta \vert}, \eta \right) \overline{\mu}_2(dy,d\eta).
	\end{align}
	In addition, if the sequence $(u_h)$ is $(-h^2 \Delta_{G_0})$-oscillating, then
	\begin{align*}
		\int_{\R_\sigma \times \R_\xi} \mu_1(d\sigma,y,d\xi,\omega)  = \overline{\mu}_1(y,\omega); \quad 
		\operatorname{Tr} M_2(y,\eta) = \overline{\mu}_2(y,\eta).
	\end{align*}
	Consequently, up to a subsequence,
	\begin{align}\label{totalmasssubelliptic}  \lim_{R\rightarrow\infty}\lim_{h\rightarrow 0}\|\Upsilon_h^{R}u_h\|_{L^2}^2=\int_{\T_y\times \mathbb{S}_\omega^0}d\ov{\mu}_1(y,\omega)+\int_{\T_y\times ( \R_\eta \setminus \{ 0 \})}d\ov{\mu}_2(y,\eta).
	\end{align}
\end{prop}

The proof of Proposition \ref{p:semiclassical_measures} is carried out in several steps.  As explained in Remark 	\ref{idea2microlocal}, for any $\epsilon>0$, we split the symbol $\mathbf{a}_R^h$ as two parts:
\begin{align*}
	\mathbf{a}_{h,\epsilon,R}(x,y,\xi,\eta)  & :=   \chi_{\epsilon}(h^2\eta)(1-\chi_R(h\eta)) a\left(hx\vert \eta \vert, y,h\xi,h\eta, h^2 \eta \right),\\ \mathbf{a}_{h,R}^{\epsilon}(x,y,\xi,\eta) & :=(1 - \chi_{\epsilon}(h^2\eta))(1-\chi_R(h\eta)) a\left(hx\vert \eta \vert, y,h\xi,h\eta, h^2 \eta \right).
\end{align*}
Correspondingly, the functional $I_R^h(a)$ is decomposed as $I_R^h(a) = I_{h,R,\epsilon}^1(a) + I_{h,R,\epsilon}^2(a)$, where
\begin{align*}
	I_{h,R,\epsilon}^1(a) & = \Big \langle \Op_1^\w \big(   \mathbf{a}_{h,\epsilon,R}(x,y,\xi,\eta) \big) u_h, u_h \Big \rangle_{L^2(M_0)}, \\[0.2cm]
	I_{h,R,\epsilon}^2(a) & =
	\Big\langle 
	\Op_1^{\w}(\mathbf{a}_{h,R}^{\epsilon})u_h,u_h
	\Big\rangle_{L^2(M_0)} 
	= \left \langle \Op_1^{\w,(y,\eta)} \Op_1^{\w,(x,\xi)}  \big(   \mathbf{a}_{h,R}^{\epsilon}(h x, y, h^{-1}\xi, \eta)  \big) U_h, U_h \right \rangle_{L^2(M_0)},
\end{align*}
where $U_h(x,y) := h^{1/2}u_h(h x,y)$.
When obtaining the semiclassical limit, we will take the successive limit $h\rightarrow 0$, $\epsilon\rightarrow 0$ and finally $R\rightarrow\infty$, so one may keep in mind that the hierarchy of small parameters $$h\ll \epsilon\ll R^{-1}\ll 1,\quad Rh<\epsilon.$$
Recall the definition of the special symbol class $\mathbf{S}_{1,1;1^-,0}^0$ in \eqref{condition}.  It is straightforward to verify that $\mathbf{a}_{h,\epsilon,R}\in \mathbf{S}_{1,1;1^-,0}^0(T^*M_0)$. On the other hand, the symbol $\mathbf{a}_{h,R}^{\epsilon}$ does not enjoy composition property as good symbolic calculus in $(x,\xi)$ variable. Nevertheless, due to the support property $|\eta|\leq O(h^{-2})$, we have 
$ \partial_{x}^j\partial_{\xi}^k\mathbf{a}_{h,R}^{\epsilon}\lesssim h^{-j+k}
$.

\begin{proof}[Proof of Proposition \ref{p:semiclassical_measures}]
	We begin by the scalar-valued measure.	
	
	\noi
	$\bullet$ {\bf Step 1: Existence of the scalar-valued measure.} The existence of non-negative  Radon measure associated to $I_{h,R,\epsilon}^1$ 
	follows from the standard argument of \cite{G91} (see also Chapter 4 of \cite{Z12}). It consists of
	verifying that along any convergent subsequence of $I_{h,R,\epsilon}^1$ in the triple limit $h\rightarrow 0,\epsilon\rightarrow 0,R\rightarrow\infty$, the limit $I^1(\cdot)$ defines a continuous, non-negative linear functional on a dense subset of  $C_c(\R_{\sigma,y,\xi,\theta}^4\times\mathbb{S}_{\omega}^0)$. As $\mathbf{S}_c^0(T^*M_0\times\R)$ is dense in $C_c(\R_{\sigma,y,\xi,\theta}^4\times\mathbb{S}_{\omega}^0)$, it suffices to check the above conditions for symbols.
	
	Using Proposition \ref{L2boundedness}, we have
	
	\begin{equation}
		\label{e:I_1_bound_from_above}
		\begin{array}{rl}
			\displaystyle \vert I_{h,R,\epsilon}^1(a) \vert  & \displaystyle \leq C_0 \Vert a \Vert_{L^\infty} +C_1(h^{1/2}+\epsilon^{1/2}).
		\end{array}
	\end{equation}
	Moreover, by Corollary \ref{Garding} and following the argument in Lemma 1.2 of \cite{G91}\footnote{The argument consists of applying Corollary \ref{Garding} to $\sqrt{\mathbf{a}_{h,\epsilon,\R}+\delta}$, taking the triple limit and finally letting $\delta\rightarrow 0$. }, we deduce that for $a\geq 0, a\in \mathbf{S}_{c}^0(T^*M_0\times\R)$,
	\begin{align}
		\liminf_{R\rightarrow\infty}\liminf_{\epsilon\rightarrow 0}\liminf_{h\rightarrow 0}I_{h,R,\epsilon}^1(a)   \geq 0.
	\end{align}
	Then there exists a subsequence $(u_h)$ and a non-negative Radon measure $\mu_1 \in \mathcal{M}_+\big( \R_\sigma \times \mathbb{T}_y \times \R_\xi \times \mathbb{S}^0_\omega \times\R_{\theta}\big)$ such that (after extraction of a subsequence in each limit procedure):
	$$
	\lim_{R \to + \infty} \lim_{\epsilon \to 0}  \lim_{h \to 0^+} I_{h,R,\epsilon}^1(a)  =  \int_{\R_\sigma \times \mathbb{T}_y \times \R_\xi \times \mathbb{S}^0_\omega} \, a(\sigma,y,\xi,\omega,\theta) \mu_1(d\sigma,dy,d\xi,d\omega,d\theta).
	$$
	Observe that $|h^2\eta|\lesssim \epsilon$ for the symbol $\mathbf{a}_{h,\epsilon,R}(x,y,\xi,\eta)$, so in the limit, the measure $\mu_1$ can be only supported on $\theta=0$. For this reason, we regard $\mu_1\in\mathcal{M}_+(\R_{\sigma}\times\T_y\times\R_{\xi}\times\mathbb{S}_{\omega}^0)$ only.

	Now, if $\widetilde{a} \in \mathbf{S}_c^0(T^*\T\times\R)$, one shows similarly that there exists a subsequence $(u_h) \subset L^2(M_0)$ (which can be taken the same as before) and a measure $\overline{\mu}_1$ so that the distribution
	$$
	I_{h,R,\epsilon}^1(\widetilde{a}) = \Big \langle \Op_1^{\w} \big(   \chi_\epsilon (h^2 \eta)  \widetilde{\mathbf{a}}_R^h(y,\eta) \big) u_h, u_h \Big \rangle_{L^2(M_0)}
	$$
	satisfies
	$$
	\lim_{R \to + \infty} \lim_{\epsilon \to 0}  \lim_{h \to 0^+} I_{h,R,\epsilon}^1(\widetilde{a}) = \int_{\T_y \times \mathbb{S}_\omega^0} \widetilde{a}(y , \omega,0) \overline{\mu}_1(dy,d\omega).
	$$
	Moreover, if the sequence $(u_h)$ is $(-h^2\Delta_{G_0})$-oscillating, we take
	\begin{equation}
		\label{e:compact_symbol}
		a_{\lambda}(\sigma,y,\xi,\eta,\theta) := \chi_{\lambda}^2( \sigma^2 + \xi^2 - 1) \widetilde{a}(y,\eta,\theta).
	\end{equation}
	
	Notice that $\widetilde{a} \in \mathbf{S}_c^0(T^*\T\times\R)$ implies that $a_{\lambda} \in \mathbf{S}_c^0(T^*M_0\times\R)$ for each $\lambda>0$.
	Then by definition,
	\begin{align}\label{alambda1}
		\lim_{R\rightarrow+\infty}\lim_{\epsilon\rightarrow 0}\lim_{h\rightarrow 0^+}I_{h,R,\epsilon}^1(a_{\lambda})=\int_{\T_y\times\mathbb{S}_{\omega}^0}\widetilde{a}(y,\omega,0)\chi_{\lambda}(\xi^2+\sigma^2-1 )^2\ov{\mu}_1(dy,d\omega).
	\end{align}
	From the symbolic calculus, we deduce that\footnote{Here the point is that $\chi_{\epsilon}(h\eta)\widetilde{a}_R^h(y,\eta)$ is \emph{independent} of $x$ and $\xi$ variable, hence the dangerous derivative in $x$ does not appear in the formal symbolic product $\chi_{\epsilon}(h\eta)\widetilde{\mathbf{a}}_R^h(y,\eta,\eta')_{\#}\chi_R^2(x^2\eta^2+\xi^2-1)$. } 
	\begin{align*}
		\label{e:compact_to_homogeneous}
		I_{h,R,\epsilon}^1(a_{\lambda}) & = \Big \langle \Op_h^\w \big( \chi_\epsilon(h\eta)\widetilde{\mathbf{a}}_R^h \big)\chi_{\lambda}(-h^2\Delta_{G_0}-1) u_h, \chi_{\lambda}(-h^2\Delta_{G_0}-1) u_h  \Big \rangle_{L^2(M_0)}  \\[0.2cm]
		& \quad + O(1/R +h/\epsilon+\epsilon).
	\end{align*}
	Finally, using that $(u_h)$ is $(-h^2\Delta_{G_0})$-oscillating, we conclude, modulo a  subsequence $(u_h)$ from the one taken to define $\mu_1$ that
	$$
	\lim_{\lambda\rightarrow+\infty}\lim_{R \to + \infty} \lim_{\epsilon \to 0}  \lim_{h \to 0^+} I_{h,R,\epsilon}^1(a_{\lambda}) = \lim_{R \to + \infty} \lim_{\epsilon \to 0}  \lim_{h \to 0^+} I_{h,R,\epsilon}^1(\widetilde{a}),
	$$
	and hence by \eqref{alambda1},
	$$
	\int_{\R_\sigma \times \R_\xi} \mu_1(d\sigma,y,d\xi,\omega) = \overline{\mu}_1(y,\omega).
	$$

	\noi
	$\bullet$ {\bf Step 2: Existence of the operator-valued measure.}
	
	\begin{lemma}\label{existenceoperator-valued}
		There exist a subsequence $(u_h)$ and a measure $M_2 \in \mathcal{M}_+\big( \mathbb{T}_y \times (\R_\eta \setminus \{ 0 \}); \mathcal{L}^1(L^2(\R_x)) \big)$, where $M_2$ takes values in the space of trace operators $\mathcal{L}^1(L^2(\R_x))$,
		such that, for every symbol $a \in \mathbf{S}_c^0(T^*M_0\times\R)$,
		\begin{align*}
			\lim_{R \to + \infty} \lim_{\epsilon \to 0} \lim_{h \to 0^+} I_{h,R,\epsilon}^2(a) & = \operatorname{Tr} \left[ \int_{\mathbb{T}_y \times (\R_\eta \setminus \{0 \})} a_\infty^{\w}\left( x \vert \eta \vert  ,y, D_x, \frac{\eta}{\vert \eta \vert},\eta \right) M_2( dy,d\eta ) \right].
		\end{align*}
	\end{lemma}
	\begin{proof}
		We loosely follow \cite{G91}.

		\noi
		$\bullet$ {\bf Testing against $\mathcal{K}(L^2(\R_x))$-valued symbols in  the product form.} First take $D \subset \mathbf{S}_c^0(T^*\mathbb{T} \times \R)$ a countable dense set. Fix $\widetilde{a} \in D$ and take, for $R > 0$ sufficiently large,
		\begin{equation}
			\label{e:partial_operator}
			\widehat{A}_{h,\epsilon,R} = \Op^{\w,(y,\eta)}_{h^2}\Big(  \widetilde{a}_\infty\Big( y, \frac{\eta}{\vert \eta \vert},\eta\Big) \left(1 - \chi_R\big( \frac{\eta}{h} \big)\right) (1 - \chi_\epsilon(\eta))\Big).
		\end{equation}
		Consider the sequence of linear forms on the space of compact operators $\mathcal{K}(L^2(\R_x))$,
		$$
		\begin{array}{ccl}
			L_{h,\epsilon,R} : \mathcal{K}(L^2(\R_x)) & \longrightarrow & \; \; \mathbb{C} \\
			K & \longmapsto & \langle \widehat{A}_{h,\epsilon,R} K U_h,U_h \rangle_{L^2(M_0)}.
		\end{array}
		$$
			   As $\widehat{A}_{h,\epsilon,r}$ can be viewed as $\mathrm{Op}_{\widetilde{h}}(\widetilde{b})$ for some symbol $\widetilde{b}\in S^{0,0}$ and $\widetilde{h}=\max\{h^2\epsilon^{-1},hR^{-1}\}$, by Corollary \ref{AppcorA2} (still valid for $y\in\T$), we have
				$$
				\vert L_{h,\epsilon,R} K \vert \leq (C_0\|\widetilde{a}_{\infty}\|_{L^{\infty}}+C_1(\widetilde{a})(\epsilon^{-\frac{1}{2}}h+h^{\frac{1}{2}}R^{-\frac{1}{2}})) \Vert K \Vert_{\mathcal{L}(L^2)}, \quad \forall h, \epsilon,R,
				$$
				where $C_0>0$ is independent of $\widetilde{a}_{\infty}$. Note that in the regime $h\ll \epsilon\ll R^{-1}, Rh<\epsilon$,
				the sequence $L_{h,\epsilon,R}$ is equicontinuous with respect to the weak-$\star$ topology of $\mathcal{K}(L^2(\R_x))$.

		By the Banach-Alaoglu Theorem, there exists an element $M_2(\widetilde{a}) \in \mathcal{L}^1(L^2(\R_x))$ such that by further extracting of subsequences,
		\begin{equation}
			\label{first-limit}
			\forall K \in \mathcal{K}(L^2(\R_x)), \quad \lim_{R \to +\infty} \lim_{\epsilon \to 0} \lim_{h \rightarrow 0} \langle \widehat{A}_{h,\epsilon,R} K U_{h}, U_{h} \rangle = \tr(KM_2(\widetilde{a})).
		\end{equation}
		By a diagonal extraction, we may assume that the subsequence is the same for any $\widetilde{a} $ in the dense subset $D$. We can extend $M_2(\cdot)$ as a linear operator
		$$
		M_2: \widetilde{D} \longrightarrow \mathcal{L}^1(L^2(\R_x)),
		$$
		where $\widetilde{D}$ is the vector space spanned by $D$. We claim that $M_2$ is bounded in $\mathcal{C}_c(\T_y \times (\R_\eta \setminus \{0 \}))$. Indeed, for any $\widetilde{a} \in D$,
		$$
		\vert \tr (KM_2(\widetilde{a})) \vert \leq C \sup \vert \widetilde{a}_\infty \vert \Vert K \Vert.
		$$
		By continuity and the Riesz theorem, $M_2$ can be extended to an element of $\mathcal{M}(\T_y \times (\R_\eta \setminus \{0 \}); \mathcal{L}^1(L^2(\R_x))$ such that
		\begin{align}\label{test-produce}
			\lim_{R \to + \infty} \lim_{\epsilon \to 0} \lim_{h \to 0} \big \langle \widehat{A}_{h,\epsilon,R} K U_{h},U_{h}  \big \rangle_{L^2(M_0)} = \int_{\T_y \times ( \R_{\eta} \setminus \{0 \})}  \widetilde{a}_\infty \left(y,\frac{\eta}{\vert \eta \vert}, \eta \right) \tr (K M_2(dy,d\eta)),
		\end{align}
		for any $\widetilde{a}_\infty \in \mathcal{C}_c^\infty(T^*\T \times \mathbb{S}^0)$ and any $K \in \mathcal{K}(L^2(\R_x))$. Moreover, we have $M_2 \in \mathcal{M}_+(\T_y \times (\R_\eta \setminus \{ 0 \}); \mathcal{L}^1(L^2(\R_x))$ (c.f. \cite[Lemma 1.2]{G91}).
		
		\vspace{0.3cm}
		
		\noi
		$\bullet$ {\bf Testing against general $\mathcal{K}(L^2(\R_x))$-valued symbols.} For a general $\mathcal{K}(L^2(\R_x))$-valued symbol, we will view it as the limit of finite combination of $\mathcal{K}(L^2(\R_x))$-valued symbols of the product form.   Fix an orthonormal basis $\{ \varphi_k(x) \}_{k \in \mathbb{N}_0}$ of $L^2(\R_x)$. For any $k \in \mathbb{N}_0$, denote by $\Pi_k$ the orthogonal projection on the linear space spanned by the first $k$ vectors of this basis and $\pi_j$ the orthogonal projection onto the space spanned by $\varphi_j$. Note that $\Pi_k=\sum_{j\leq k}\pi_j$. For $a\in\mathbf{S}_c^0$, set
		\begin{align*}
			\widetilde{\mathbf{a}}_{R,\epsilon,h}(y,\eta)  := \Op_1^{\w,(x,\xi)} \Big( (1 - \chi_\epsilon (\eta))\big(1-\chi_R\big(\frac{\eta}{h}\big)\big) a\left( x|\eta|, y, \xi, \frac{\eta}{h},\eta \right)  \Big), \quad
			\widetilde{\mathbf{A}}_{h,\epsilon,R}  :=\mathrm{Op}_{h^2}^{\w,(y,\eta)}(\widetilde{\mathbf{a}}_{R,\epsilon,h}).
		\end{align*}
		By the Calderon-Vaillancourt theorem for $S^{0,0}$ symbols (Proposition \ref{CaVai}), operators $\widetilde{\mathbf{a}}_{R,\epsilon,h}(y,\eta)$ are uniformly bounded in $h, \epsilon,R$ on $\T_y \times (\R_\eta \setminus \{0 \})$ as operators on $L^2(\R_x)$.  
		Notice also that 
		$$
		\Pi_k \, \widetilde{\mathbf{a}}_{R,\epsilon,h}(y,\eta) \, \Pi_k = \sum_{j,j' \leq k} \widetilde{\mathbf{a}}^{j,j'}_{R,\epsilon,h}(y,\eta)E_{j,j'},
		$$
		where $E_{j,j'}=\langle\cdot,\varphi_{j'}\rangle_{L_x^2}\;\varphi_j$,
		\begin{align*}
			\widetilde{\mathbf{a}}^{j,j'}_{R,\epsilon,h}  :=\langle \widetilde{\mathbf{a}}_{R,\epsilon,h}(y,\eta)\varphi_j,\varphi_j'\rangle_{L^2} = a_\infty^{j,j'}\Big( y, \frac{\eta}{\vert \eta \vert},\eta\Big) \Big(1 - \chi_R\Big( \frac{\eta}{h} \Big)\Big) (1 - \chi_\epsilon(\eta)),
		\end{align*}
		with $a_\infty^{j,j'} \in \mathcal{C}_c^\infty(T^*\mathbb{T} \times \mathbb{S}^0)$ and $E_{j,j'} \in \mathcal{K}(L^2(\R_x))$. 
		Note that for $h,\epsilon$ small enough obeying $h\ll\epsilon$, we can write alternatively
		$$ \Pi_k\widetilde{\mathbf{a}}_{R,\epsilon,h}\Pi_k=(1-\chi_{\epsilon}(\eta))\Big(1-\chi_R\Big(\frac{\eta}{h}\Big)\Big)\Pi_ka_{\infty}^{\w}\Big(x|\eta|,y,D_x,\frac{\eta}{|\eta|},\eta \Big)\Pi_k.
		$$
		Therefore,
		\begin{align}\label{formula}
			\Pi_ka_{\infty}^{\w}\Big(x|\eta|,y,D_x,\frac{\eta}{|\eta|},\eta \Big)\Pi_k=
			\sum_{j,j'\leq k}a_{\infty}^{j,j'}\Big(y,\frac{\eta}{|\eta|},\eta\Big)E_{j,j'}.
		\end{align}
		By \eqref{test-produce}, we deduce that for any $k\in\N$,
		\begin{align}\label{Test-finiteproduct} 
			\lim_{R\rightarrow\infty}\lim_{\epsilon\rightarrow 0}
			\lim_{h\rightarrow 0}
			\langle \pi_k\widetilde{\mathbf{A}}_{h,\epsilon,R}\pi_k U_h,U_h  \rangle_{L^2(M_0)}=&\int_{\T_y\times (\R_{\eta}\setminus\{0\} )} \sum_{j,j'\leq k}a_{\infty}^{j,j'}\Big(y,\frac{\eta}{|\eta|},\eta\Big)\mathrm{Tr}(E_{j,j'}M_2(dy,d\eta) ) \notag \\
			=&\mathrm{Tr}\Big[\int_{\T_y\times(\R_{\eta}\setminus\{0\} )}\Pi_k a_{\infty}^{\w}\Big(x|\eta|,,y,D_x,\frac{\eta}{|\eta|},\eta \Big)\Pi_kM_2(dy,d\eta)
			\Big].
		\end{align}

		To conclude, we need to pass to the limit in $k$ of the above identity.  This requires some uniform controls in $R,\epsilon$ and $h$. 
		Note that for any positive integer $N_0>0$, we have 
		\begin{align}\label{boundedness} 
			\sum_{|\beta|\leq N_0}\sup_{(y,\eta)\in T^*\T}\|\partial_{y,\eta}^{\beta}\widetilde{\mathbf{a}}_{R,\epsilon,h}(y,\eta)\|_{\mathcal{L}(L^2(\R_x))}+
			\sup_{k\in\N,(y,\eta)\in T^*\T}\|\partial_{y,\eta}^{\beta}\pi_k\widetilde{\mathbf{a}}_{R,\epsilon,h}\pi_k(y,\eta)\|_{\mathcal{L}(L^2(\R_x))}<\infty,
		\end{align}
		uniformly in $h,\epsilon$ and $R$,
		we deduce that 
		\begin{equation}
			\label{approx-compact}
			\lim_{k \rightarrow \infty} \sum_{|\alpha|\leq  N_0-1} \sup_{(y,\eta) \in T^*\T} \big \Vert  \partial_{y,\eta}^{\alpha} \big( \widetilde{\mathbf{a}}_{R,\epsilon,h}(y,\eta) - \pi_k \, \widetilde{\mathbf{a}}_{R,\epsilon,h}(y,\eta)  \, \pi_k \big) \big \Vert_{\mathcal{L}(L^2(\R_x))} = 0,
		\end{equation}
		uniformly in $h,\epsilon$ and $R$.
		To see this, note that the above limit holds for all fixed $(y,\eta)\in T^*\T$, varying in a compact set. Together with the uniform boundedness property \eqref{boundedness} with one more derivative, we deduce \eqref{approx-compact}.
		
		Consequently, by the Calder\'on-Vaillancourt theorem applied to operator-valued symbols, there exists an absolute constant $C_0>0$, such that
		\begin{align*}
			\big \vert \big \langle (\widetilde{\mathbf{A}}_{h,\epsilon,R} - \pi_k \, \widetilde{\mathbf{A}}_{h,\epsilon,R} \, \pi_k)U_{h},U_{h} \big \rangle_{L^2(M_0)} \big \vert \\[0.2cm]
			& \hspace*{-3cm} \leq C \sum_{|\alpha|\leq C_0} \sup_{(y,\eta) \in \R^2} \big \Vert  \partial_{y,\eta}^{\alpha} \big( \widetilde{\mathbf{a}}_{R,\epsilon,h}(y,\eta) - \pi_k \, \widetilde{\mathbf{a}}_{R,\epsilon,h}(y,\eta)  \, \pi_k \big) \big \Vert_{\mathcal{L}(L^2(\R_x))},
		\end{align*}
		hence
		$$
		\limsup_{h \rightarrow 0} \big \vert \big \langle (\widetilde{\mathbf{A}}_{h,\epsilon,R} - \pi_k \widetilde{\mathbf{A}}_{h,\epsilon,R} \pi_k)U_{h},U_{h} \big \rangle_{L^2(M_0)} \big \vert \rightarrow 0, \quad \text{as } k \rightarrow \infty.
		$$
		This allows us to exchange the order of limit $\lim_{k\rightarrow\infty}$ and $\lim_{R\rightarrow\infty}\lim_{\epsilon\rightarrow 0}\lim_{h\rightarrow 0}$ on the left hand side of \eqref{Test-finiteproduct} when taking the limit $k\rightarrow\infty$. Consequently, the triple limit$$	\lim_{R \to + \infty} \lim_{\epsilon \to 0} \lim_{h \to 0^+} I_{h,R,\epsilon}^2(a)$$ is equal to the right hand side of \eqref{Test-finiteproduct}.
		
		Finally, by dominated convergence, we deduce that the limit when $k\rightarrow\infty$ of the right hand side of \eqref{Test-finiteproduct} is
		\begin{align*}
			\lim_{k\rightarrow\infty}\mathrm{Tr}\Big[\int_{\T_y\times(\R_{\eta}\setminus\{0\} )}\Pi_k a_{\infty}^{\w}\Big(x|\eta|,,y,D_x,\frac{\eta}{|\eta|},\eta \Big)\Pi_kM_2(dy,d\eta)
			\Big] \\[0.2cm]
			& \hspace*{-4cm}  = \operatorname{Tr} \left[ \int_{\mathbb{T}_y \times (\R_\eta \setminus \{ 0 \})}  a_{\infty}^{\w}\Big( x \vert \eta \vert  ,y, D_x, \frac{\eta}{\vert \eta \vert},\eta \Big) M_2( dy,d\eta ) \right].
		\end{align*}
		This completes the proof of Lemma \ref{existenceoperator-valued}.
	\end{proof}
	\noi
	$\bullet$ {\bf Step 3: No loss of mass for $(-h^2\Delta_{G_0})$-oscillating sequences.} Finally, for any $\widetilde{a} \in \mathbf{S}_c^0(T^*\T\times\R)$ depending only on $y,\eta$, write
	$$
	I_{h,R,\epsilon}^2(\widetilde{a}) = \left \langle \Op_1^{\w,(y,\eta)}   \big( (1 - \chi_\epsilon (h^2\eta))  \widetilde{\mathbf{a}}_R^h(y, \eta)  \big) U_h, U_h \right \rangle_{L^2(M_0)}.
	$$
	By the same argument as in Step 1, we obtain the existence of a subsequence and a measure $\overline{\mu}_2 \in \mathcal{M}(\T_y \times (\R_\eta \setminus \{ 0 \}))$ such that
	$$
	\lim_{R \to \infty} \lim_{\epsilon \to 0} \lim_{h \to 0^+}  I_{h,R,\epsilon}^2(\widetilde{a}) = \int_{\T_y \times (\R_\eta \setminus \{0 \})}  \widetilde{a}_\infty\left(y, \frac{\eta}{\vert \eta \vert}, \eta \right) \overline{\mu}_2(dy,d\eta).
	$$
	
	Moreover, if the sequence $(u_h)$ is $(-h^2\Delta_{G_0})$-oscillating, we take $a_{\lambda}$ given by \eqref{e:compact_symbol}, where $\widetilde{a} \in \mathbf{S}_c^0(T^*\T\times\R)$ implies that $a_{\lambda} \in \mathbf{S}_c^0(T^*M_0\times\R)$. We have
	\begin{align}\label{I2alambda} 
		I_{h,R,\epsilon}^2(a_{\lambda}) & = \left \langle \Op_h^{\w,(y,\eta)} \Big( (1 - \chi_\epsilon(h\eta))\widetilde{\mathbf{a}}_R^h(y,\frac{\eta}{h}) \mathrm{Op}_h^{\w,(x,\xi)} ( \chi_{\lambda}^2(H_{G_0}-1) )
		\Big) u_h, u_h \right \rangle_{L^2(M_0)}.
	\end{align}
	On the one hand, we have
	$$
	\Op_1^{\w,(x,\xi)}( \chi_\lambda^2( \xi^2 + x^2 \eta^2))=\chi_{\lambda}^2(D_x^2+x^2\eta^2)+\mathcal{O}_{\mathcal{L}(L^2)}(\lambda^{-1}),
	$$
	and then
	\begin{align}\label{alambda2}
		\lim_{R\rightarrow\infty}\lim_{\epsilon\rightarrow 0}\lim_{h\rightarrow 0}I_{h,R,\epsilon}^2(a_{\lambda})=\mathrm{Tr}\Big[\int_{\T_y\times(\R_{\eta} \setminus \{0 \})}\widetilde{a}_{\infty}\Big(y,\frac{\eta}{|\eta|},\eta\Big)\chi_{\lambda}^2(D_x^2+x^2\eta^2)M_2(dy,d\eta) \Big]+O(\lambda^{-1}).
	\end{align}
	
	On the other hand, using the symbolic calculus for operator-valued pseudo-differential operators\footnote{Here we only encounter compositions and commutators between a scalar-valued p.d.o (in $(y,\eta)$ variables) and a compact operator-valued p.d.o.  }, we can simplify the quantization in \eqref{I2alambda} as
	\begin{align*}
		\Op^{\w,(y,\eta)}_1 \Big( (1 - \chi_\epsilon(h^2\eta))\widetilde{\mathbf{a}}_R^h(y,\eta)  \Op_h^{\w,(x,\xi)} \big( \chi^2_{\lambda}( \xi^2 +x^2 (h\eta)^2 - 1) \Big) & \\[0.2cm]
		& \hspace{-10.8cm} =\Op^{\w,(y,\eta)}_{h^2} \Big( (1 - \chi_\epsilon(\eta))\widetilde{\mathbf{a}}_R^h(y,\eta/h^2) \Big) \Big[ \Op^{\w,(y,\eta)}_{h} \Op_h^{\w,(x,\xi)}(\chi_{\lambda}( \xi^2 +x^2 \eta^2 - 1)) \Big]^2  + O(h)+O(h^2/\epsilon).
	\end{align*}
	
	Finally, using the fact that $(u_h)$ is $(-h^2\Delta_{G_0})$-oscillating, we counclude, modulo a subsequence from the one taken to define $M_2$, that
	$$
	\lim_{\lambda\rightarrow\infty}\lim_{R \to + \infty} \lim_{\epsilon \to 0}  \lim_{h \to 0^+} I_{h,R,\epsilon}^2(a_{\lambda}) = \lim_{R \to + \infty} \lim_{\epsilon \to 0}  \lim_{h \to 0^+} I_{h,R,\epsilon}^2(\widetilde{a}).$$
	Combining with \eqref{alambda2}, this implies that
	$$
	\tr M_2(y,\eta) = \overline{\mu}_2(y,\eta).
	$$
	Finally, by Lemma \ref{subellipticapriori}, the measure $\ov{\mu}_2$ is supported on $|\eta|\leq C$ for some uniform constant $C>0$. By definition $\Upsilon_h^R=\big(1-\chi_0\big(\frac{hD_y}{R}\big)\big)$ and the fact that $1-\chi_R\equiv 1$ on supp$(1-\chi_0(R^{-1}\cdot))$, the total mass of the sequences $\Upsilon_h^Ru_h$ is captured by $\ov{\mu}_1$ and $\ov{\mu}_2$, 
	hence we obtain \eqref{totalmasssubelliptic}.
	
	We emphasize that in the end we can take the same subsequence $(u_h)$ to define $\mu_1$, $\overline{\mu}_1$, $M_2$ and $\overline{\mu}_2$. This concludes the proof of Proposition \ref{p:semiclassical_measures}. 
\end{proof}

\subsection{Quasimodes in the widely undamped case }\label{subsection3.1}
In this subsection, we prove Proposition \ref{Measureformulationj=012}. Since the proof is long, we will divide it into several paragraphs and lemmas.

\noi
$\bullet$ {\bf  The compact part of the measure.} By the standard existence theorem of semi-classical measures (p.100, Theorem 5.2 of \cite{Z12}), there exists a Radon measure $\mu_0$ on $T^*\T^2$ such that, up to a subsequence, for any compactly supported symbol $a_0(x,y,\xi,\eta)$,
$$ \lim_{h\rightarrow 0}\lg\Op_h^{\w}(a_0)\psi_h,\psi_h\rg_{L^2}=\int_{T^*\T^2}a_0 \, d\mu_0.
$$
Moreover, supp$(\mu_0)\subset\Sigma_1:=\{\xi^2+V(x)\eta^2=1\}$. However, due to the subellipticity (non-compactness of $\Sigma_1$), $\mu_0$ is not the total mass. 

For any symbol $a_1(y,\eta,\theta)$, compactly supported in $y,\theta$ and homogeneous of degree $0$ in $\eta$, we decompose
$ a_1=\chi_R(\eta)a_1+(1-\chi_R(\eta))a_1.
$
Hence $$\lim_{R\rightarrow\infty}\lim_{h\rightarrow 0}\langle\Op_h^{\w}(\chi_R(\eta)a_1)\psi_h,\psi_h\rangle_{L^2}= \lim_{R\rightarrow\infty}\int_{T^*\T^2}\chi_R(\eta)a_1(y,\eta,0)d\mu_0=\int_{T^*\T^2}a_1|_{\theta=0}d\mu_0.$$ For the contribution of $(1-\chi_R(\eta))a_1$, without loss of generality, we may replace $\psi_h$ by $\psi_{2,h}:=\Upsilon_h^R\psi_h$, which is of course $(-h^2\Delta_{G_0})$-oscillating, thanks to Lemma \ref{oscillating} (with the same proof). Therefore, by Proposition \ref{p:semiclassical_measures}, up to a subsequence,
$$ \lim_{R\rightarrow+\infty}\lim_{h\rightarrow 0}\lg\Op_h^{\w}((1-\chi_R(\eta))a_1 )\psi_h,\psi_h\rg_{L^2}=\int_{\T_y\times\mathbb{S}_{\omega}^0}a_1(y,\omega,0)d\ov{\mu}_1+\int_{\T_y\times( \R_\eta \setminus \{0 \})} a_1\Big(y,\frac{\eta}{|\eta|},\eta\Big)d\ov{\mu}_2.
$$
Defining $\T_y\times\ov{\R_\eta \setminus \{ 0 \}} := (\T_y\times\mathbb{S}_\omega^0)\bigsqcup \, (\T_y\times (\R_\eta \setminus \{0 \}))$, let $\ov{\mu}_0=\ov{\mu}_1|_{\T_y\times\mathbb{S}_\omega^0}+\ov{\mu}_2|_{\T_y\times\mathbb{R}_{\eta}}$, then using the notation \eqref{e:closure_of_a} we obtain the formula
$$ \lim_{R\rightarrow\infty}\lim_{h\rightarrow 0}\lg\Op_h^{\w}((1-\chi_R(\eta) )a_1)\psi_h,\psi_h\rg_{L^2}=\int_{\T_y\times\ov{\R_\eta \setminus \{ 0 \}}}  \, \overline{a}_{\infty}\, d \ov{\mu}_0.
$$

To prove that the total masses of $\mu_0$ and $\ov{\mu}_0$ sum $1$, we write $\psi_h=\Upsilon_h^R\psi_h+(1-\Upsilon_h^R)\psi_h$. Again by Proposition \ref{p:semiclassical_measures},
$$ \lim_{R\rightarrow\infty}\lim_{h\rightarrow 0}\|\Upsilon_h^R\psi_h\|_{L^2}^2=\int_{\T_y\times\ov{\R_\eta \setminus \{0 \}}}d\ov{\mu}_0.
$$
Since 
$$ \lim_{R\rightarrow\infty}\lim_{h\rightarrow 0}\|(1-\Upsilon_h^R)\psi_h\|_{L^2}^2=\lim_{R\rightarrow\infty}\int_{T^*\T^2}\chi_0\Big(\frac{\eta}{R}\Big)d\mu_0=\int_{T^*\T^2}d\mu_0,
$$
and
$$ \lim_{h\rightarrow 0}\lg\Upsilon_h^R\psi_h,(1-\Upsilon_h^R)\psi_h\rg_{L^2}=\int_{T^*\T^2}\Big(1-\chi_0\Big(\frac{\eta}{R}\Big)\Big)\chi_0\Big(\frac{\eta}{R}\Big)d\mu_0,
$$
which converges to $0$ as $R\rightarrow\infty$, thanks to the dominating convergence theorem, we have
\begin{align}\label{totalmass}  \int_{T^*\T^2}d\mu_0+\int_{\T_y\times\ov{\R_\eta \setminus \{0 \}}}d\ov{\mu}_0=\lim_{R\rightarrow\infty}\lim_{h\rightarrow 0}\|\psi_h\|_{L^2}^2=1.
\end{align}
The invariance of $\mu_0$ along $\phi_t^{\mathrm{e}}$ follows from the standard argument:
for any $a_0\in C_c^{\infty}(T^*\T^2)$, using the quasimode equation and the symbolic calculus, we have
$$ 0=\lim_{h\rightarrow 0}\lg\frac{i}{h}[-h^2\Delta_G,\Op_h^{\w}(a_0)]\psi_h,\psi_h\rg_{L^2}=\int_{T^*\T^2}\{\xi^2+V(x)\eta^2,a_0\}d\mu_0. 
$$

\noi
$\bullet$ {\bf  The subelliptic part of the measure.} Set $\psi_{h,\mathrm{s}}=\chi_{\delta}(x)\Upsilon_h^R\psi_h$. Recall from \eqref{e:localized_x_quasimode} that
$$ (P_{h,b}-\zeta_{h})\psi_{h,\mathrm{s}}=o_{L^2}(h^2\delta_h).
$$
We denote by $\mu_1,M_2,\ov{\mu}_1,\ov{\mu}_2$ semiclassical measures in Proposition \ref{p:semiclassical_measures}, applied to a subsequence of $\psi_{h,\mathrm{s}}$ (still denoted by $\psi_{h,\mathrm{s}}$).

Note that by Proposition \ref{p:semiclassical_measures},
\begin{align}\label{e:support_M_2_0}
	\mathrm{Tr}M_2(y,\eta)=\ov{\mu}_2(y,\eta).
\end{align}
Then, taking an orthonormal basis of $L^2(\R_x)$ given by eigenfunctions $\{\varphi_k(\eta,x), \vert \eta\vert (2k+1)\}_{k \in \mathbb{N}_0}$ of the operator $ -\partial_x^2 + x^2\eta^2$, we obtain
\begin{align}
	\int_{\T_y \times (\R_\eta \setminus \{0 \})}  \overline{\mu}_2(dy,d\eta) & =  \sum_{k \in \mathbb{N}_0} \int_{\T_y \times (\R_\eta \setminus \{0 \})}  \Big \langle M_2(dy,d\eta) \varphi_{k}(\eta,\cdot) , \varphi_{k}(\eta,\cdot) \Big \rangle_{L^2(\R_x)} \\[0.2cm]
	\label{e:trace_support}
	& =  \int_{\T_y \times (\R_\eta \setminus \{0 \})} \operatorname{Tr} m_2(y,\eta) \nu_2(dy,d\eta),
\end{align}
where $m_2(y,\eta) \in L^1(d\nu_2;\mathcal{L}^1(L^2(\R_x)))$ is the Radon-Nikodym derivative of $M_2$ so that $M_2 = m_2 \nu_2$ (see \cite[Proposition A.1]{G91}). 
Next, we are going to show some support properties of the measures:
\begin{lemma}\label{supportmeasure} 
	\begin{align}
		\mathrm{(i)} \hspace{0.3cm}	\label{e:support_mu_1}
		& \operatorname{supp} \mu_1 \subset \{ (\sigma,y,\xi,\omega) \in \R_\sigma \times \mathbb{R}_y  \times \R_\xi \times \mathbb{S}_\omega^0 \, : \, \sigma^2 + \xi^2 = 1, \; b(0,y) = 0 \}; \\
		\mathrm{(ii)} 
		\hspace{0.3cm}
		\label{e:support_M_2}
		& \operatorname{supp} M_2 \subset \{ (y,\eta) \in \mathbb{R}_y \times (\R_\eta \setminus \{0 \}) \, : \, b(0,y) = 0,\; |\eta|\leq C_1\};\\
		\mathrm{(iii)}\hspace{0.3cm}
		&  \mathrm{Tr}_{L^2(\R_x)}\big((D_x^2+\eta^2x^2)m_2(y,\eta)\big) \in L^1(d\nu_2), \text{ and }\notag\\
		& \int_{\T_y\times\R_{\eta}\setminus\{0\}}\mathrm{Tr}_{L^2(\R_x)}\big((D_x^2+\eta^2x^2)m_2(y,\eta)\big)\nu_2(dy,d\eta)=\int_{\T_y\times\R_{\eta}\setminus\{0\}} \ov{\mu}_2(dy,d\eta). \label{Tracemass}
	\end{align} 
\end{lemma}

\begin{proof}
			We first show that $\mu_1$ is supported on $\sigma^2+\xi^2=1$. This is a consequence of the $-h^2\Delta_{G_0}$-oscillating property of the quasimode $\psi_{h}$ as well as $\psi_{h,\mathrm{s}}$. By Lemma \ref{oscillating}, for any $\lambda>0$, we deduce that
			\begin{align}\label{mu1ellipticpart}  \lim_{R\rightarrow\infty}\lim_{\epsilon\rightarrow 0}\lim_{h\rightarrow 0}\big\langle\big(1-\chi_{\lambda}(-h^2\Delta_{G_0}-1)\big)\chi_{\epsilon}(h^2D_y)\psi_{h,\mathrm{s}},\psi_{h,\mathrm{s}}\big\rangle=0,
			\end{align}
			since $-h^2\Delta_{G_0}-1$ is elliptic on the support of $(1-\chi_{\lambda}(-h^2\Delta_{G_0}-1))$. On the other hand, by definition of $\mu_1$, for any $\lambda>0$,
			$$ \lim_{R\rightarrow\infty}\lim_{\epsilon\rightarrow 0}\lim_{h\rightarrow 0}\big\langle
			\chi_{\lambda}(-h^2\Delta_{G_0}-1)\chi_{\epsilon}(h^2D_y)\psi_{h,\mathrm{s}},\psi_{h,\mathrm{s}}
			\big\rangle_{L^2}=\big\langle \chi_{\lambda}(\sigma^2+\xi^2-1),
			\mu_1\big\rangle.
			$$
			Adding the two equalities above and letting $\lambda\rightarrow 0$, we deduce that $\mathrm{supp} \, \mu_1\subset\{\sigma^2+\xi^2=1\}$.
			
			The proof for the operator-valued measure $M_2$ is similar. We have for any $\epsilon>0, R>0, \lambda>0$,
			\begin{align}\label{ellipticity}   \lim_{h\rightarrow 0}\big\langle\big(1-\chi_{\lambda}(-h^2\Delta_{G_0}-1)\big)\big(1-\chi_{\epsilon}(h^2D_y)\big)\psi_{h,\mathrm{s}},\psi_{h,\mathrm{s}}\big\rangle_{L^2}=0.
			\end{align}		
			By rescaling $\Psi_h(x,y)=h^{\frac{1}{2}}\psi_{h,\mathrm{s}}(hx,y)$, we have	
			\begin{align*}  
				\lim_{R\rightarrow\infty}\lim_{\epsilon\rightarrow 0}\lim_{h\rightarrow 0}&
				\big\langle\chi_{\lambda}(-h^2\Delta_{G_0}-1)(1-\chi_{\epsilon}(h^2D_y))(-h^2\Delta_{G_0})\psi_{h,\mathrm{s}},\psi_{h,\mathrm{s}}
				\big\rangle_{L^2} \\	
				=&\lim_{R\rightarrow\infty}\lim_{\epsilon\rightarrow 0}\lim_{h\rightarrow 0}	\big\langle\mathrm{Op}_1^{\w,(y,\eta)}\big(\chi_{\lambda}(D_x^2+x^2h^4\eta^2-\mathrm{Id})(D_x^2+x^2h^4\eta^2)(1-\chi_{\epsilon}(h^2\eta) )\big)\Psi_h,\Psi_h 
				\big\rangle_{L^2}\\
				=&\mathrm{Tr}_{L^2(\R_x)}\Big[\int_{\T_y\times\R_{\eta}\setminus\{0 \}}
				\chi_{\lambda}(D_x^2+x^2\eta^2-1)(D_x^2+x^2\eta^2)M_2(dy,d\eta)
				\Big].
			\end{align*}
			Replacing $-h^2\Delta_{G_0}\psi_{h,\mathrm{s}}$ by $\psi_{h,\mathrm{s}}+o_{L^2}(1)$ on the left hand side of the equality above and adding  \eqref{ellipticity}, we deduce that for any $\lambda>0$,
			\begin{align*}  
				\lim_{R\rightarrow\infty}\lim_{\epsilon\rightarrow 0}\lim_{h\rightarrow 0}&
				\big\langle(1-\chi_{\epsilon}(h^2D_y))\psi_{h,\mathrm{s}},\psi_{h,\mathrm{s}}
				\big\rangle_{L^2} =\int_{T_y\times\R_{\eta}\setminus\{0\}}\ov{\mu}_2(dy,d\eta) \\
				=&\mathrm{Tr}_{L^2(\R_x)}\Big[\int_{\T_y\times\R_{\eta}\setminus\{0 \}}
				\chi_{\lambda}(D_x^2+x^2\eta^2-1)(D_x^2+x^2\eta^2)M_2(dy,d\eta)
				\Big].
			\end{align*}
			Letting $\lambda\rightarrow \infty$ on the right hand side, we obtain \eqref{Tracemass}. The fact that $|\eta|\leq C_1$ on supp$M_2$ is a direct consequence of the definition of the quasimode $\psi_{h,\mathrm{s}}$ and	Lemma \ref{subellipticapriori}.	
			
			Next we show that the measures $\mu_1,M_2$ do not see the support of $b(0,y)$.	Recall the definition of $\psi_{h,\mathrm{s}}:=\chi_{\delta}(x)\Upsilon_h^R\psi_h$, the formula \eqref{measureformula3} still holds for any symbol $a(y)$ depending only on $y$ variable. In particular,
			\begin{align}\label{limitb0y} \lim_{R\rightarrow\infty}\lim_{h\rightarrow 0}\langle b(0,y)\psi_{h,\mathrm{s}},\psi_{h,\mathrm{s}}\rangle_{L^2}=\int_{\T_y\times\mathbb{S}_{\omega}^0}b(0,y)\ov{\mu}_1(dy,d\omega)+\int_{\T_y\times (\R_{\eta}\setminus\{0\})}b(0,y)\mathrm{Tr}(M_2)(dy,d\eta).
			\end{align}
			Therefore, it suffices to show that the left hand side vanishes. By Taylor expansion $b(x,y)=b(0,y)+O(x)$ and recalling the estimate
			\eqref{claim},  we obtain that
			$$ \langle b\psi_{h,\mathrm{s}},\psi_{h,\mathrm{s}}\rangle_{L^2}=\langle b(0,y)\psi_{h,\mathrm{s}},\psi_{h,\mathrm{s}}\rangle_{L^2}+O(R^{-\frac{1}{2}})+O(h^{\infty}).
			$$
			Taking the limit $h\rightarrow 0$ first and then $R\rightarrow\infty$, the right hand side converges to the same limit \eqref{limitb0y}. In view of \eqref{dampedregionweak}, we deduce that $b(0,y)\ov{\mu}_1=b(0,y)\mathrm{Tr}(M_2)=0$. 
			This completes the proof of Lemma \ref{supportmeasure}.	
\end{proof}
\noi
$\bullet$ {\bf Propagation of semiclassical measures.}
By Lemma \ref{supportmeasure}, we see that $\mu_1$ or $M_2$ do not vanish identically. We aim at finding a contradiction to this fact, showing that $\mu_1 = 0$ and $M_2 = 0$.

We proceed in several steps. 
	{Let $a\in \mathbf{S}_c^0(T^*M_0\times\R)$ as in Definition \ref{s:two-microlocal_measures}. We set three different symbols:
		\begin{align}
			\label{e:a}
			\mathbf{a}_R^h(x,y,\xi,\eta) & := (1-\chi_R(h\eta)) a(hx \vert \eta \vert,y,h\xi,h\eta,h^2\eta), \\[0.2cm]
			\mathbf{b}_R^h(x,y,\xi,\eta) & := \frac{(1 - \chi_R(h\eta))}{h\vert \eta \vert} a(x h\vert\eta\vert, y, h\xi, h\eta, h^2\eta), \\[0.2cm]
			\label{e:c}
			\mathbf{c}_R^h(x,y,\xi,\eta) & := h\vert \eta \vert (1 - \chi_R(h\eta)) a(hx\vert \eta \vert,y,h\xi,h\eta,h^2\eta).
		\end{align}
The reason to consider multiplication by powers of $h \vert \eta \vert$ is to renormalize the Wigner equation at different scales for which we will obtain different propagation or invariance equations for the semiclassical measures. Since our subelliptic semiclassical parameter $h \vert \eta \vert$ depends on the variable $\eta$, it is natural that our scalings of the Wigner measure depend on this parameter as well.

First, we show that the measures are invariant along the horizontal flow:
\begin{lemma}[Measure-invariance along the horizontal flow]\label{horizontal} 
	For $a\in\mathbf{S}_c^0(T^*M_0\times\R)$, we have
	\begin{align}\label{horizontalinvarianemu1} 
		&  \int_{\R_\sigma \times \mathbb{T}_y \times \R_\xi \times S_\omega^0}  \big( 2\xi \partial_\sigma - 2 \sigma \partial_\xi \big) a(\sigma,y,\xi,\omega,0)  \mu_1(d\sigma, dy, d\xi, d\omega)  = 0,
	\end{align}
	and
	\begin{align}\label{horizontalinvarianceM2} 
		& \operatorname{Tr} \int_{\mathbb{T}_y  \times (\R_\eta\setminus \{0 \})}   \left[ D_x^2 + \eta^2 x^2, a_{\infty}\left( x\vert \eta \vert,y,D_x,\frac{\eta}{\vert \eta \vert},\eta \right) \right]_{L^2(\R_x)} M_2(dy, d \eta)   = 0.
	\end{align}
\end{lemma}
Consequently, $\mu_1(\sigma,y,\xi,\omega)$ is invariant by the harmonic-oscillator flow 
\begin{align}\label{flowhorizontal} 
	\vartheta_t(\sigma,y,\xi,\omega) = (\sigma \cos(t) + \xi \sin(t),y,  -\xi \cos(t) + \sigma \sin(t),\omega),
\end{align}
and $M_2(y,\eta)$ is invariant by the quantum flow $\mathrm{e}^{it (- \partial^2_x + \eta^2 x^2)}$. In particular, for each $\eta\neq 0$, we can choose an orthonormal basis $(\varphi_k(\eta,x))_{k\in\N}$ is of $L^2(\R_x)$ that diagonalizes $D_x^2+\eta^2x^2$ and $M_2(dy,d\eta)$ at the same time. 
\begin{proof}[Proof of Lemma \ref{horizontal}]
	We will take
	$$ \mathbf{b}_{h,\epsilon,R}(x,y,\xi,\eta)=\chi_{\epsilon}(h^2\eta)\mathbf{b}_R^h(x,y,\xi,\eta)=\chi_{\epsilon}(h^2\eta)(1-\chi_R(h\eta))a(hx|\eta|,y,h\xi,h\eta,h^2\eta)
	$$
	and
	$$ h\mathbf{a}_{h,R}^{\epsilon}=h(1-\chi_{\epsilon}(h^2\eta))(1-\chi_R(h\eta))a(hx|\eta|,y,h\xi,h\eta,h^2\eta) 
	$$
	as test symbols for the scalar-valued measure and operator-valued measure, with respectively.
	
Applying the quasimode equation \eqref{e:quasimode_equation}, we obtain that, for any symbol $\mathbf{b}$, 
			\begin{align}\label{commutatorwithb} 
				\frac{i}{h}\big\langle[-h^2\Delta_G,\mathrm{Op}_1^{\w}(\mathbf{b})]\psi_{h,\mathrm{s}},\psi_{h,\mathrm{s}}
				\big\rangle_{L^2}=&-2\Re\big\langle\mathrm{Op}_1^{\w}(\mathbf{b})(b-\beta_h) \psi_{h,\mathrm{s}},\psi_{h,\mathrm{s}}
				\big\rangle_{L^2}\notag
				\\&-\frac{2}{h}\Im\langle\Op_1^{\w}(\mathbf{b})\psi_{h,\mathrm{s}},r_h\rangle_{L^2}.
			\end{align}
			We first claim that for any $\mathbf{b}\in\{\mathbf{b}_{h,\epsilon,R},h\mathbf{a}_{h,R}^{\epsilon} \}$  
			$$ \text{right hand side of \eqref{commutatorwithb}}=O(h^{\frac{1}{2}}),
			$$
			uniformly in $\epsilon,R$ as $h\rightarrow 0$.
			
			When $\mathbf{b}=\mathbf{b}_{h,\epsilon,R}$, $\mathrm{Op}_1^{\w}(\mathbf{b})$ is uniformly bounded on $L^2$ (Proposition \ref{L2boundedness}), hence the right hand side of \eqref{commutatorwithb} is bounded by $O(h^{\frac{1}{2}})+o(h\delta_h)$, thanks to \eqref{dampedregionweak} and the fact that $\beta_h=O(h)$. When $\mathbf{b}=h\mathbf{a}_{h,R}^{\epsilon}$, we need to work with the rescaled quasimode
			\begin{align}\label{scalingPsi} \Psi_h(x,y):=h^{\frac{1}{2}}\psi_{h,\mathrm{s}}(hx,y)
			\end{align}
			and write the right hand side of \eqref{commutatorwithb} as
			\begin{align*} -2h\Re\Big\langle \mathrm{Op}_1^{\w}(\mathbf{b}_{h,R}^{\epsilon}(hx,y,h^{-1}\xi,\eta))(&b(hx,y)-\beta_h)\Psi_h,\Psi_h
				\Big\rangle_{L^2}\\ -
				&2\Im\Big\langle\mathrm{Op}_1^{\w}(\mathbf{b}_{h,R}^{\epsilon}(hx,y,h^{-1}\xi,\eta))\Psi_h,h^{\frac{1}{2}}r_h(hx,y)\Big\rangle_{L^2}
				.
			\end{align*}
			By the $L^2$-boundedness of $\mathrm{Op}_1^{\w}(\mathbf{b}_{h,R}^{\epsilon}(hx,y,h^{-1}\xi,\eta))$ (as it belongs to the class $S^{0,0}$) and \eqref{commutatorwithb}, this error is still of order $O(h^{\frac{3}{2}})+o(h^2\delta_h)$.

	The information obtained in the last paragraph merely comes from the quasimode equation. To obtain the propagation of measure, we need compute the commutators on the left hand side of \eqref{commutatorwithb}  directly\footnote{By the periodic extension procedure described in Appendix \ref{AppendixA}, it suffices to compute the commutator for operators acting on $\mathcal{D}'(\R^2)$ instead of $\mathcal{D}'(M_0)$.  }.  To treat $[-h^2\Delta_G,\mathrm{Op}_1^{\w}(\mathbf{b}_{h,\epsilon,R})]$, we apply the symbolic calculus developed in Appendix \ref{AppendixA}.
	More precisely, applying Lemma \ref{commutatorformula} to $q(x,y,\xi,\eta)=\mathbf{b}_{h,\epsilon,R}(x,y,\xi,\eta)$, we get
	$$ [-h^2\Delta_G,\mathrm{Op}_1^\w(\mathbf{b}_{h,\epsilon,R})]=\frac{2h^2}{i}\mathrm{Op}_1^\w(\xi\partial_x\mathbf{b}_{h,\epsilon,R}-x\eta^2\partial_{\xi}\mathbf{b}_{h,\epsilon,R})+\mathcal{O}_{\mathcal{L}(L^2)}(h^2),
	$$
	where we have used Proposition \ref{L2boundedness} to control the operator norms of the remainders appearing in the formula of Lemma \ref{commutatorformula}. Note also that
	\begin{align*}
		h\xi\partial_{x}\mathbf{b}_{h,\epsilon,R}-xh\eta^2\partial_{\xi}\mathbf{b}_{h,\epsilon,R}(x,y,\xi,\eta)=\chi_{\epsilon}(h^2\eta)(1-\chi_R(h\eta))[\xi\partial_{\sigma}a-\sigma\partial_{\xi}a]|_{(hx|\eta|,y,h\xi,h\eta,h^2\eta)}.
	\end{align*}
	Therefore,
	\begin{align}\label{invariance1}
		\lim_{R \to \infty} \lim_{\epsilon \to 0} \lim_{h \to 0^+} \frac{i}{h} \big \langle [-h^2\Delta_G,\mathrm{Op}_1^\w(\mathbf{b}_{h,\epsilon,R})]\psi_{h,\mathrm{s}}, \psi_{h,\mathrm{s}} \big \rangle_{L^2} \notag\\[0.2cm]
		& \hspace*{-5cm} =  \int_{\R_\sigma \times \mathbb{T}_y \times \R_\xi \times \mathbb{S}_\omega^0}  \big( 2\xi \partial_\sigma - 2 \sigma \partial_\xi \big) a(\sigma,y,\xi,\omega,0)  \mu_1(d\sigma, dy, d\xi, d\omega).
	\end{align}
	Combining \eqref{commutatorwithb}, 
	this proves \eqref{horizontalinvarianemu1}.

To deal with the operator-valued measure, we will test it against $\mathbf{b}=h\mathbf{a}_{h,R}^{\epsilon}(x,y,\xi,\eta)$. Set
			$$ \mathbf{d}_{h,R}^{\epsilon}(x,y,\xi,h^2\eta):=\mathbf{a}^{\epsilon}_{h,R}(hx,y,\frac{\xi}{h},\eta)=(1-\chi_{\epsilon}(h^2\eta))a(h^2x|\eta|,y,\xi,h\eta,h^2\eta),
			$$
			
			where we omit the cutoff $1-\chi_R(h\eta)$ since $h\ll \epsilon/R$. By rescaling, we have 
			\begin{align*}
				\big \langle [-h^2 \Delta_G, \Op_1^\w( h\mathbf{a}_{h,R}^\epsilon) ] &\psi_{h,\mathrm{s}}, \psi_{h,\mathrm{s}} \big \rangle_{L^2}\\[0.2cm] =  &\Big \langle \Big[ - \partial_x^2 - \frac{V(hx)}{h^2} h^4 \partial_y^2,\quad \Op_1^{\w,(y,\eta)} \Op_{1}^{\w,(x,\xi)}  h\mathbf{a}_{h,R}^\epsilon( hx, y ,\frac{\xi}{h}, \eta) \Big] \Psi_h, \Psi_h  \Big \rangle_{L^2},
			\end{align*}
			where the rescaled quasimode $\Psi_h$ is given by \eqref{scalingPsi}. 
			Applying Lemma \ref{commutatorformula2} to $q(x,y,\xi,\eta)=\mathbf{d}^{\epsilon}_{h,R}(x,y,\xi,\eta)$, we obtain 
			\begin{align} 
				\big \langle [-h^2 \Delta_G, \Op_1( \mathbf{b}_{h,R}^\epsilon) ] \psi_{h,\mathrm{s}}, \psi_{h,\mathrm{s}} \big \rangle_{L^2}	 \notag 
				& \\[0.2cm]
				& \hspace*{-4cm} =  h \Big \langle  \Op_{1}^{\w, (y,\eta)} \Big[ - \partial_x^2 + x^2 (h^2\eta)^2, \Op_1^{\w,(x,\xi)} (\mathbf{d}_{h,R}^\epsilon(x, y,  \xi, h^2\eta)) \Big]_{L^2(\R_x)} \Psi_h, \Psi_h \Big \rangle_{L^2} \notag  \\[0.2cm]
				& \hspace*{-4cm} \quad + h \Big \langle \Op_{1}^{\w, (y,\eta)} \Big[  \Big( \frac{V(hx)}{h^2} - x^2 \Big) (h^2\eta)^2, \Op_1^{\w,(x,\xi)} (\mathbf{d}_{h,R}^\epsilon(x, y, \xi, h^2\eta)) \Big]_{L^2(\R_x)} \Psi_h,\Psi_h \Big \rangle_{L^2}\notag \\[0.2cm]
				& \hspace*{-4cm} \quad + \frac{h^3}{i} \left \langle   \Op_{1}^{\w, (y,\eta)} \Op_1^{\w,(x,\xi)} (2V(hx)\eta \partial_y \mathbf{d}_{h,R}^\epsilon(x,y, \xi, h^2\eta)) \Psi_h, \Psi_h \right \rangle_{L^2} + O(h^4). \label{propagationhorizontalcommutator}
			\end{align}
			Note that $V(hx)=h^2x^2+O(h^3x^3)$ and $\Psi_{h}=O_{L^2}(h^{\infty})$ when $|hx|\gg R^{-\frac{1}{2}}$ (see \eqref{claim1}), we have $V(hx)/h^2-x^2=\mathcal{O}(R^{-\frac{1}{2}}x^2)$. Since $h^4x^2\partial_y^2\Psi_h=O_{L^2}(1)$, we deduce that
			$$h \Big \langle \Op_{1}^{\w, (y,\eta)} \Big[  \Big( \frac{V(hx)}{h^2} - x^2 \Big) (h^2\eta)^2, \Op_1^{\w,(x,\xi)} (\mathbf{d}_{h,R}^\epsilon(x, y, \xi, h^2\eta)) \Big]_{L^2(\R_x)} \Psi_h,\Psi_h \Big \rangle_{L^2}=O(hR^{-\frac{1}{2}}).
			$$
			Dividing by $h$ in \eqref{propagationhorizontalcommutator},  
			we obtain
			\begin{align*}
				\lim_{R \to + \infty} \lim_{\epsilon \to 0^+}  \lim_{h \to 0^+} \frac{i}{h} \big \langle [-h^2\Delta_G,\Op_1^{\w}(\mathbf{b}_{h,R}^{\epsilon})] \psi_{h,\mathrm{s}}, \psi_{h,\mathrm{s}} \big \rangle_{L^2} & \\[0.2cm]
				& \hspace*{-8cm} \quad = i\operatorname{Tr} \int_{\mathbb{T}_y \times (\R_\eta \setminus \{ 0 \})} \left[ D_x^2 + \eta^2 x^2,  a_{\infty}\left( x\vert \eta \vert,y,D_x, \frac{\eta}{\vert \eta \vert},\eta\right) \right]_{L^2(\R_x)} M_2(dy,d \eta).
			\end{align*}
			Combining \eqref{commutatorwithb}, we get \eqref{horizontalinvarianceM2}, hence we complete the proof of Lemma \ref{horizontal}.
\end{proof}

It remains to study the propagation along the vertical direction.
\begin{lemma}\label{lemmaverticalpropagation}
	If $b\equiv 0$ and $\beta_h=o(h)$, then $\mu_1, M_2$ are invariant along the flow $\vartheta_t, \mathrm{e}^{it(x^2+D_x^2)}$, with respectively. If $\int_{\T}b(0,y)dy>0$, then $\mu_1=0, M_2=0$. 
\end{lemma}

\begin{proof}
	Consider $\mathbf{c}^h_R$ given by \eqref{e:c} with $a \geq 0$. More precisely, take $a = a(y)$, and we  consider  $$\mathbf{c}_{h,\epsilon,R}
	=h|\eta|\chi_{\epsilon}(h^2\eta)(1-\chi_R(h\eta))a(y)
	.$$
	Note that $\mathbf{c}_{h,\epsilon,R}$ is not uniformly bounded in $L^2$ (with bound $O(\epsilon h^{-1})$), we will however make use of the sign of the symbol.
From the symbolic calculus and the sharp  G$\mathring{\mathrm{a}}$rding inequality (the regularity of $b$ is sufficient to use it) applied to the $h$-pseudo-differential calculus in $(y,\eta)$ variables, we have
			\begin{align}\label{SharpGarding}
				\Re\big \langle \Op_1^\w \big( h\vert \eta \vert a(y) (1-\chi_R( h\eta)) \chi_\epsilon(h^2 \eta)\big)b \psi_{h,\mathrm{s}}, \psi_{h,\mathrm{s}} \big \rangle_{L^2}\notag & \\[0.2cm] 
& \hspace*{-6cm} = \Big\langle ab\sqrt{h|D_y|(1-\chi_R(hD_y))\chi_{\epsilon}(h^2D_y) }\psi_{h,\mathrm{s}},
				\sqrt{h|D_y|(1-\chi_R(hD_y))\chi_{\epsilon}(h^2D_y) }\psi_{h,\mathrm{s}}
				\Big\rangle_{L^2}\notag \\[0.2cm]
				 & \hspace*{-6cm} \quad - O(h\epsilon^{-1}) -O(R^{-1})\notag \\[0.2cm]
				 & \hspace*{-6cm} \geq -Ch\epsilon^{-1}-R^{-1}.
			\end{align}
			By Lemma \ref{commutatorformula}, 
			\begin{align}\label{commutatorvertical1} \frac{i}{h}[-h^2\Delta_{G},\mathrm{Op}_1^{\w}(\mathbf{c}_{h,\epsilon,R}) ]=\mathrm{Op}_1^{\w}(2V(x)h^2\eta|\eta|a'(y)(1-\chi_R(h\eta))\chi_{\epsilon}(h^2\eta) )+\mathcal{O}_{\mathcal{L}(L^2)}(h).
			\end{align}
			Applying  \eqref{commutatorwithb} to $\mathbf{b}=\mathbf{c}_{h,\epsilon,R}$ we get
			\begin{align*}
				\frac{i}{h}\big\langle [-h^2\Delta_G,\mathrm{Op}_1^{\w}(\mathbf{c}_{h,\epsilon,R})]\psi_{h,\mathrm{s}},\psi_{h,\mathrm{s}}
				\big\rangle_{L^2} & \\[0.2cm]
				& \hspace*{-3cm} = \frac{2 \beta_h}{h}\Re \big \langle \Op_1^\w \big( h^2 \vert \eta \vert a(y) \left(1-\chi_R\left(h \eta \right) \right) \chi_\epsilon(h^2\eta)\big) \psi_{h,\mathrm{s}}, \psi_{h,\mathrm{s}}\big \rangle_{L^2}  \\[0.2cm]
				&  \hspace*{-3cm} \quad -2\Re\big\langle\mathrm{Op}_1^{\w}(h|\eta|a(y)(1-\chi_R(h\eta)) \chi_{\epsilon}(h^2\epsilon)  )b\psi_{h,\mathrm{s}},\psi_{h,\mathrm{s}}  \big\rangle_{L^2}+o(h\delta_h).
			\end{align*}
			Plugging \eqref{SharpGarding}, \eqref{commutatorvertical1} into the above equality, and using the hypotheses 
			\begin{equation}
				\label{e:hypothesis_beta_h}
				\limsup_{h \to 0^+} \frac{\beta_h}{h} \leq C_0,
			\end{equation}  
			we obtain that
			\begin{align}\label{operatordy}
				\big\langle\Op_1^{\w}\big(2V(x)h^2\eta|\eta|&\cdot a'(y)(1-\chi_R(h\eta))\chi_{\epsilon}(h^2\eta)
				\big)\psi_{h,\mathrm{s}},\psi_{h,\mathrm{s}}
				\big\rangle_{L^2} & \notag \\[0.2cm] 
				& \hspace*{-2cm} \leq   2C_0\Re \big \langle \Op_1^\w \big( h^2 \vert \eta \vert a(y) \left(1-\chi_R\left(h \eta \right) \right) \chi_\epsilon(h^2\eta)\big) \psi_{h,\mathrm{s}}, \psi_{h,\mathrm{s}}\big \rangle_{L^2} +Ch\epsilon^{-1}+CR^{-1}.
			\end{align}
	Note that $V(x)=x^2+O(x^3)$, then left hand side of the inequality above is
	$$ \big\langle\Op_1^{\w}\big(2x^2h^2\eta|\eta|\cdot a'(y)(1-\chi_R(h\eta))\chi_{\epsilon}(h^2\eta)
	\big)\psi_{h,\mathrm{s}},\psi_{h,\mathrm{s}}
	\big\rangle_{L^2}+O(R^{-\frac{1}{2}}),
	$$
	and the first term on the right hand side of the inequality above is bounded by $O(\epsilon)$ as $h^2|\eta|\leq O(\epsilon)$ on supp$(\chi_{\epsilon}(h^2\eta))$. Taking the triple limit $h\rightarrow 0$, $\epsilon\rightarrow 0$, $R\rightarrow\infty$, we deduce that
	$$
	\int_{\R_\sigma \times \mathbb{T}_y \times \R_\xi \times \mathbb{S}^0_\omega} \, \omega \sigma^2 \partial_y a(y) \mu_1(d\sigma,dy,d\xi,d\omega)  \leq 0.
	$$
	
	Let us assume that $\int_{\T}b(0,y)dy>0$. By averaging along the Hamiltonian flow generated by $\sigma^2+\xi^2$, this implies that
	\begin{equation}
	\label{e:vertical_invariance_equation}
	\int_{\R_\sigma \times \mathbb{T}_y \times \R_\xi \times \mathbb{S}^0_\omega} \, \omega (\xi^2 + \sigma^2) \partial_y a(y + t\omega) \mu_1(d\sigma,dy,d\xi,d\omega)  \leq 0,
	\end{equation}
	for every $t \in \R$, which is equivalent to the fact that
	$$
	\frac{d}{dt}\left(  \int_{\R_\sigma \times \mathbb{T}_y \times \R_\xi \times \mathbb{S}^0_\omega} \,  (\xi^2 + \sigma^2) a(y + t\omega) \mu_1(d\sigma,dy,d\xi,d\omega) \right) \leq 0, 
	$$
	for every $t \in \R$. Then
	$$
	\int_{\R_\sigma \times \mathbb{T}_y \times \R_\xi \times \mathbb{S}^0_\omega} \,  (\xi^2 + \sigma^2) a(y + t\omega) d\mu_1 \leq \int_{\R_\sigma \times \mathbb{T}_y \times \R_\xi \times \mathbb{S}^0_\omega} \,  (\xi^2 + \sigma^2) a(y) d\mu_1.
	$$
	Applying this inequality to $a(y) = b(0,y)$, and using  \eqref{e:support_mu_1}, we see that, for every $T > 0$,
	$$
	\int_{-T}^T \int_{\R_\sigma \times \mathbb{T}_y \times \R_\xi \times \mathbb{S}^0_\omega} \,  (\xi^2 + \sigma^2) b(0,y + t\omega) \mu_1(d\sigma,dy,d\xi,d\omega)dt = 0,
	$$
	which implies that $\mu_1 = 0$ by the hypothesis (1) or (2) on $b$ (geometric control case or widely undamped case).  If $b\equiv 0$, the inequality \eqref{operatordy} is indeed and equality, and we deduce that $\langle\partial_ya,\mu_1\rangle=0$.
	
	Next we treat the measure $M_2$. It turns out to be more convenient to work with $\eta$ with a constant sign. To do this, we pick another cutoff  $\chi_+\in C^{\infty}((0,\infty))$ such that $\chi_+\cdot (1-\chi)=(1-\chi)\mathbf{1}_{\eta>0}$. Set $\chi_{\epsilon,+}:=\chi_+(\cdot/\epsilon)$.  Now we test the commutator with the symbol
			$$ \mathbf{d}_{h,R}^{\epsilon,+}(x,y,\xi,\eta):=h\chi_{\epsilon,+}(h^2\epsilon)\mathbf{c}_{h,R}^{\epsilon}(x,y,\xi,\eta)	=h^2|\eta|\chi_{\epsilon,+}(h^2\eta)(1-\chi_{\epsilon}(h^2\eta))a(y),	
			$$
			where $a = a(y)\geq 0$ depends only in $y$. Note that in the regime $h^2|\eta|\gtrsim \epsilon$ and $h\ll \epsilon/R$, we no longer need the cutoff $(1-\chi_R(h\eta))$ in the symbol. 
			
			Applying  \eqref{commutatorwithb} to $\mathbf{b}=\mathbf{d}_{h,R}^{\epsilon,+}$ we get
			\begin{align*}
				\frac{i}{h}\big\langle [-h^2\Delta_G,\mathrm{Op}_1^{\w}(\mathbf{d}_{h,R}^{\epsilon,+})]\psi_{h,\mathrm{s}},\psi_{h,\mathrm{s}}
				\big\rangle_{L^2}
				& \\[0.2cm]
				& \hspace*{-2cm} = 2\beta_h\Re \big \langle \Op_1^\w \big( h^2|\eta| a(y)  (1-\chi_\epsilon(h^2\eta))\big) \psi_{h,\mathrm{s}}, \psi_{h,\mathrm{s}}\big \rangle_{L^2}  \\[0.2cm]
				& \hspace*{-2cm} \quad -2\Re\big\langle\mathrm{Op}_1^{\w}(h^2|\eta|a(y) (1-\chi_{\epsilon}(h^2\eta) ) )b\psi_{h,\mathrm{s}},\psi_{h,\mathrm{s}}  \big\rangle_{L^2}+o(h\delta_h).
			\end{align*}
			
			By rescaling and using Lemma \ref{commutatorformula2}, since $\mathbf{d}_{h,R}^{\epsilon}$ is independent of $x$ and $\xi$, we have
			\begin{align*}
				\frac{i}{h}\langle [-h^2\Delta_G,\mathrm{Op}_1^{\w}(\mathbf{d}_{h,R}^{\epsilon,+})]&\psi_{h,\mathrm{s}},\psi_{h,\mathrm{s}}\rangle_{L^2}
				\\[0.2cm]
				& \hspace*{-2cm} = \frac{i}{h}\big\langle [-\partial_x^2-V(hx)h^2\partial_y^2,\mathrm{Op}_1^{\w}(\mathbf{d}_{h,R}^{\epsilon,+})]\Psi_h,\Psi_h
				\big\rangle_{L^2}& \\[0.2cm]
				& \hspace*{-2cm} = h\big\langle
				\mathrm{Op}_1^{\w}\big(2V(hx)h^2\eta^2 a'(y)\chi_{\epsilon,+}(h^2\eta)(1-\chi_{\epsilon}(h^2\eta)) \big)\Psi_h,\Psi_h 
				\big\rangle_{L^2}. 
			\end{align*}
			Note that $V(hx)=h^2x^2+O(h^3x^3)$ and $\Psi_{h}=O_{L^2}(h^{\infty})$ when $|hx|\gg R^{-\frac{1}{2}}$ (see \eqref{claim1}), we have $V(hx)/h^2-x^2=\mathcal{O}(R^{-\frac{1}{2}}x^2)$. Since $h^4x^2\partial_y^2\Psi_h=O_{L^2}(1)$, we obtain that
			\begin{align*}
				2h\big\langle
				\mathrm{Op}_1^{\w}\big(x^2h^2\eta^2&\chi_{\epsilon,+}(h^2\eta) a'(y)
				(1-\chi_{\epsilon}(h^2\eta) ) \big)
				\Psi_h,\Psi_h\rangle_{L^2} & \\[0.2cm]
				 & \hspace*{-1cm}	  \leq 2 \beta_h \Re\big \langle  \Op_1^{\w} \big(h^2|\eta| a  \chi_{\epsilon,+}(h^2\eta)(1-\chi_\epsilon(h^2\eta))\big) \Psi_h, \Psi_h \big \rangle_{L^2} \\[0.2cm]
				  & \hspace*{-1cm} \quad -2\Re\big\langle
				\mathrm{Op}_1^{\w}\big(h^2|\eta|a\chi_{\epsilon,+}(h^2\eta)(1-\chi_{\epsilon}(h^2\eta) ) \big)b\psi_{h,\mathrm{s}},\psi_{h,\mathrm{s}}
				\big \rangle_{L^2}
				+o(h\delta_h)+O(hR^{-\frac{1}{2}}). 
			\end{align*}
			Using the symbolic calculus in $(y,\eta)$ variable, we have
			\begin{align*} \big\langle\mathrm{Op}_1^{\w}\big(h^2|\eta|a\chi_{\epsilon,+}(h^2\eta)&(1-\chi_{\epsilon}(h^2\eta) \big)b\psi_{h,\mathrm{s}},\psi_{h,\mathrm{s}}\big\rangle_{L^2} & \\[0.2cm]
			& \hspace*{-3cm} = \Big\langle ab\sqrt{h^2|D_y|\chi_{\epsilon,+}(h^2D_y)(1-\chi_{\epsilon}(h^2D_y) ) }\psi_{h,\mathrm{s}},
				\sqrt{h^2|D_y|\chi_{\epsilon,+}(h^2D_y)(1-\chi_{\epsilon}(h^2D_y) ) }\psi_{h,\mathrm{s}}
				\Big\rangle_{L^2}\\[0.2cm]
				&  \hspace*{-3cm} \quad +O(h^2/\epsilon)+O(h/R)\geq -Ch^2/\epsilon-Ch/R,
			\end{align*}
			where we use the fact that the first term on the right hand side is non-negative. Therefore,
			\begin{align*}
				\big\langle
				\mathrm{Op}_1^{\w}\big(x^2h^2\eta^2\chi_{\epsilon,+}(h^2\eta) a'(y)
				(1-\chi_{\epsilon}(h^2\eta) )   \big)\Psi_h,\Psi_h
				\big\rangle_{L^2}	& \\[0.2cm]
				 & \hspace*{-3cm} \leq  \frac{\beta_h}{h} \Re\big \langle  \Op_1^{\w} \big(h^2|\eta|\chi_{\epsilon,+}(h^2\eta) a  (1-\chi_\epsilon(h^2\eta))\big) \Psi_h, \Psi_h \big \rangle_{L^2} \\[0.2cm] 
				& \hspace*{-3cm} \quad + O(h/\epsilon)+O(R^{-\frac{1}{2}})+o(\delta_h). 
			\end{align*}
	Using \eqref{e:hypothesis_beta_h}, taking the triple limit $h \to 0^+, \epsilon\rightarrow 0^+,R\rightarrow+\infty$ of the above inequality, we obtain
	\begin{equation}
		\label{e:intermediate}
		C_0  \int_{\mathbb{T}_y  \times (\R_\eta \setminus \{0 \})}  \eta\mathbf{1}_{\eta>0}a(y) \ov{\mu}_2(dy,d\eta)  \geq  \mathrm{Tr}\Big[\int_{\mathbb{T}_y \times (\R_\eta \setminus \{0 \})}   \Op_1^{\w,(x,\xi)} \left( x^2 \eta^2\mathbf{1}_{\eta>0} \cdot  a'(y) \right) M_2(dy,d \eta)\Big],
	\end{equation}
	where $\ov{\mu}_2(y,\eta) = \operatorname{Tr}M_2(y,\eta)$.
	Here we note that \eqref{e:intermediate} makes sense thanks to the support property of $M_2$ (Lemma \ref{supportmeasure}). Moreover, though $a(y)(1-\chi_{\epsilon}(h\eta)))$ does not contain the compact part in $L_x^2$, we can still get the right hand side of \eqref{e:intermediate} by passing to the triple limit $\displaystyle{\lim_{R\rightarrow\infty}\lim_{\epsilon\rightarrow 0}\lim_{h\rightarrow 0}}$, since the operator valued measure $M_2$ already contains a compact operator on $L^2(\R_x)$. 
	
	Now we use the same trick of averaging to replace the symbol $x^2\eta^2$ by $x^2\eta^2+D_x^2$. Indeed, by writing $M_2(y,\eta)=m_2(y,\eta)\nu_2(y,\eta)$ for some trace class operator-valued function $m_2(y,\eta)$ and some scalar measure $\nu_2$, we can write
	\begin{align*}
		\mathrm{Tr}\int_{\T_y\times (\R_{\eta} \setminus \{0 \})}x^2\eta^2\cdot \mathbf{1}_{\eta>0}a'(y)M_2(dyd\eta)=\int_{\T_y\times (\R_{\eta} \setminus \{0 \})}\mathbf{1}_{\eta>0}a'(y)\mathrm{Tr}_{L_x^2}[x^2\eta^2m_2(y,\eta)]\nu_2(dyd\eta).
	\end{align*}
	Set $U_{\eta}(t)=\mathrm{e}^{it(D_x^2+x^2\eta^2)}$. Using the invariance $U_{\eta}(t)^*m_2U_{\eta}(t)=m_2$ proved in Lemma \ref{horizontal},  we have
	\begin{align*}
		\mathrm{Tr}_{L_x^2}[x^2\eta^2m_2(y,\eta)] & =\mathrm{Tr}_{L_x^2}[U_{\eta}(t)^*(x^2\eta^2)U_{\eta}(t)m_2(y,\eta)] \\[0.2cm]
		& =\frac{|\eta|}{2\pi}\mathrm{Tr}_{L_x^2}\Big[\Big(\int_0^{\frac{2\pi}{|\eta|}}U_{\eta}(t)^*(x^2\eta^2)U_{\eta}(t)dt\Big) m_2(y,\eta) \Big].
	\end{align*}
	By Lemma \ref{averagingLemma},
	$$ \frac{|\eta|}{2\pi}\int_0^{\frac{2\pi}{|\eta|}}U_{\eta}(t)^*(x^2\eta^2)U_{\eta}(t)dt=\frac{1}{2}(D_x^2+x^2\eta^2),
	$$
	thus
	$$ \mathrm{Tr}_{L_x^2}\int_{\T_y\times (\R_{\eta} \setminus \{0 \})} x^2\eta^2\mathbf{1}_{\eta>0} a'(y)M_2(dy,d\eta)=\frac{1}{2}\mathrm{Tr}_{L_x^2}\int_{\T_y\times (\R_{\eta} \setminus \{0 \})}(x^2\eta^2+D_x^2)\mathbf{1}_{\eta>0}a'(y)M_2(dy,d\eta),
	$$
	and we get
	\begin{equation}
		\label{e:intermediate'}
		2C_0 \int_{\mathbb{T}_y  \times (\R_\eta \setminus \{0 \}) } \, \eta\mathbf{1}_{\eta>0} a(y) \ov{\mu}_2(dy,d\eta)  \geq  \operatorname{Tr} \int_{\mathbb{T}_y \times (\R_\eta \setminus \{0 \})}  (x^2\eta^2+D_x^2)\mathbf{1}_{\eta>0} a'(y) M_2(dy,d \eta).
	\end{equation}
	Writing $M_2(dy,d\eta)=m_2(y,\eta)\nu_2(dy,d\eta)$ and $\ov{\mu}_2(dy,d\eta)=\mathrm{Tr}(m_2(y,\eta))\nu_2(dy,d\eta)$, and using the fact that $x^2\eta^2+D_x^2\geq |\eta|$, we obtain that
	\begin{align*}
		\int_{\mathbb{T}_y \times (\R_\eta \setminus \{0 \})}  \mathbf{1}_{\eta>0}\mathrm{Tr}_{L_x^2}(&(x^2\eta^2+D_x^2)m_2) a'(y) \nu(dy,d \eta)\\ \leq &2C_0 
		\int_{\mathbb{T}_y \times (\R_\eta \setminus \{0 \})} 
		\mathrm{Tr}_{L_x^2}((x^2\eta^2+D_x^2)m_2)\mathbf{1}_{\eta>0}a(y)\nu_2(dy,d\eta).
	\end{align*}
	The above inequality holds by replacing $a$ to $a(\cdot+t)$, we deduce that
	\begin{align*}
		\frac{d}{dt}\Big(\mathrm{e}^{-2C_0t}
		\int_{\mathbb{T}_y \times (\R_\eta \setminus \{0 \})} 
		\mathrm{Tr}_{L_x^2}((x^2\eta^2+D_x^2)m_2)\mathbf{1}_{\eta>0}a(y+t)\nu_2(dy,d\eta)
		\Big)\leq 0.
	\end{align*}
	We assume that $\int_{\T}b(0,y)dy>0$, by choosing $a(y)=b(0,y)$, the inequality above yields
	$$ \int_{\T_y\times(\R_{\eta}\setminus\{0\})}\mathrm{Tr}_{L_x^2}((x^2\eta^2+D_x^2)m_2)\mathbf{1}_{\eta>0}b(0,y+t)\nu_2(dy,d\eta)\leq 0,\quad \forall t\in\R,
	$$
	thanks to the support property of $M_2$ (Lemma \ref{supportmeasure}). Taking the integral of the inequality above in $t\in[0,2\pi]$, we deduce that
	$$ \Big(\int_{\T}b(0,y)dy\Big)\int_{\T_y\times(\R_{\eta}\setminus\{0\})}\mathrm{Tr}_{L_x^2}((x^2\eta^2+D_x^2)m_2)\mathbf{1}_{\eta>0}\nu_2(dy,d\eta)\leq 0.
	$$
	Since $M_2$ is non-negative, this implies that $M_2\mathbf{1}_{\eta>0}=0$. Similarly from the same argument, we have $M_2\mathbf{1}_{\eta<0}=0$.
	
	Finally, we remark that if $b\equiv 0$ and $\beta_h=o(h)$, we get \eqref{e:intermediate} with  $0$ on the left hand side, and this implies that $\partial_yM_2\geq 0$ in the distributional sense. By periodicity of $y\in\T_y$, we have $\partial_yM_2= 0$.
	
	The proof of Lemma \ref{lemmaverticalpropagation} is now complete. Consequently,
	we finish the proof of Proposition \ref{Measureformulationj=012}.
\end{proof}





\subsection{Quasimodes in the narrowly undamped case}
		\label{s:narrowly_undamped_measures}

		

		In this section we prove Corollary \ref{p:negative_sub-ellitpic_quasimodes(2)}. The main idea is to refine the last part of the proof of  Proposition \ref{Measureformulationj=012}, concerning the propagation by the vertical flow, in the case in which the profile of the damping term $b(y)$ is explicit and given by $\vert y - y_0 \vert^\nu$ near the point $y_0$. In this case, we can detect the obstruction to propagate the semiclassical measures in the vertical direction for quasimodes of larger width $o(h^{2 - \frac{1}{1+\nu}})$, leading to the much better resolvent estimate of Theorem \ref{t:sharp_in_2} in the subelliptic regime. To this aim, we will construct suitable two-microlocal semiclassical measures at scales depending on the coefficient $\nu$, capturing the precise scales of the Wigner equation in which the profile of $b(y)$ near $y_0$ and the vertical propagation interact. 
		
		\begin{proof}[Proof of Corollary \ref{p:negative_sub-ellitpic_quasimodes(2)}]
			
			 Analogously to the proof of Corollary \ref{p:negative_sub-ellitpic_quasimodes}, we assume that there exists a quasimode $\psi_{1,h} = \chi_\delta(x) \Upsilon_h^{R} \psi_h$ of width $o\big( h^{2-\frac{1}{1+\nu}} \big)$ to reach a contradiction. 
			
			We next study the two-microlocal measures constructed in Proposition \ref{p:semiclassical_measures} with a further two-microlo\-calization near the point $y_0$ at semiclassical scale $ \vert \eta \vert^{-\frac{1}{1+\nu}}$ (compare with \cite[Theorem 1]{F00}). Since $\delta := \frac{1}{1+\nu} < 1$, this second microlocalization near $y_0$ holds in a (semiclassical) scale below the critical one imposed by the uncertainty-principle ($\vert y - y_0 \vert \sim h \lesssim \vert \eta \vert^{-1}$), so no new operator-valued measures will appear at (semiclassical) scale $\vert y-y_0 \vert \sim \vert \eta \vert^{-\delta}$.
			
			Let us consider symbols $a(\sigma,y,y',\xi,\eta,\eta')$ to be compactly supported in $(\sigma,y,\xi,\eta')$ and homogeneous of degree zero at infinity in $(y',\eta)$. Our symbols are periodic in the variable $y$. We define:
			\begin{align*}
				\mathbf{a}^1_{h,R,\epsilon,\rho}(x,y,\xi,\eta) & := (1- \chi_R(\eta)) (1 - \chi_\epsilon(h\eta)) \chi_\rho\left( (y-y_0)\vert \eta \vert^\delta \right)
				a\left( x\vert \eta \vert, y, (y-y_0) \vert \eta \vert^\delta, \xi, \eta, h\eta \right), \\[0.2cm]
				\mathbf{a}^2_{h,R,\epsilon,\rho}(x,y,\xi,\eta) & := (1- \chi_R(\eta)) (1 - \chi_\epsilon(h\eta)) \left(1-\chi_\rho\left( (y-y_0)\vert \eta \vert^\delta \right) \right)
				a\left( x\vert \eta \vert, y,  (y-y_0)\vert \eta \vert^\delta, \xi, \eta, h\eta \right), \\[0.2cm]
				\mathbf{a}^3_{h,R,\epsilon,\rho}(x,y,\xi,\eta) & := (1- \chi_R(\eta)) \chi_\epsilon(h\eta) \chi_\rho\left(  (y-y_0)\vert \eta \vert^\delta \right)
				a\left( x\vert \eta \vert, y,  (y-y_0)\vert \eta \vert^\delta, \xi, \eta, h\eta \right), \\[0.2cm]
				\mathbf{a}^4_{h,R,\epsilon,\rho}(x,y,\xi,\eta) & := (1- \chi_R(\eta)) \chi_\epsilon(h\eta) \left(1-\chi_\rho\left(  (y-y_0)\vert \eta \vert^\delta \right) \right)
				a\left( x\vert \eta \vert, y,  (y-y_0)\vert \eta \vert^\delta, \xi, \eta, h\eta \right),
			\end{align*}
			and consider the Wigner distributions
			$$
			I^{j}_{h,R,\epsilon,\rho}(a) = \big \langle \Op^\w_h (\mathbf{a}^j_{h,R,\epsilon,\rho}) \psi_{1,h}, \psi_{1,h} \big \rangle_{L^2(M)}, \quad j = 1,2,3,4.
			$$
			Taking limits through subsequences (c.f. \cite[Theorem 1]{F00}) $h \to 0^+$, $\epsilon \to 0^+$, $\rho \to +\infty$, and $R \to +\infty$, we find operator valued measures $M_2^1 \in \mathcal{M}_+(\R_\theta \times (\R_\eta \setminus \{0 \}); \mathcal{L}^1(L^2(\R_x))$, $M_2^{2,1} \in \mathcal{M}_+(\mathbb{S}_\omega^0 \times (\R_\eta \setminus \{0 \}); \mathcal{L}^1(L^2(\R_x))$, $M_2^{2,2}  \in \mathcal{M}_+(\R_y \times (\R_\eta \setminus \{0 \}); \mathcal{L}^1(L^2(\R_x))$ such that:
			$$
			I^{1}_{h,R,\epsilon,\rho}(a) \to \operatorname{Tr} \int_{\R_\theta  \times (\R_\eta \setminus \{0 \})}  \Op_1^{\w,(x,\xi)} \left(  a \left( x \vert \eta \vert, y_0, \theta, \xi, \frac{\eta}{\vert \eta \vert},\eta \right)\right)  M_2^1(d\theta,d\eta),
			$$
			and 
			\begin{align*}
				I^{2}_{h,R,\epsilon,\rho}(a)  & \to \operatorname{Tr} \int_{\mathbb{S}^0_{\omega} \times  (\R_\eta \setminus \{0 \})}  \Op_1^{\w,(x,\xi)} \left( a \left(x \vert \eta \vert, y_0, \omega, \xi, \frac{\eta}{\vert \eta \vert} ,\eta \right)\right)  M_2^{2,1}(d\omega,d\eta) \\[0.2cm]
				& \quad + \operatorname{Tr} \int_{\mathbb{R}_y  \times (\R_\eta \setminus \{0 \})} \mathbf{1}_{y \neq y_0} \Op_1^{\w,(x,\xi)} \left( a_\infty\left( x \vert \eta \vert, y, \frac{y-y_0}{\vert y - y_0 \vert},  \xi,\frac{\eta}{\vert \eta \vert},\eta \right)\right)  M_2^{2,2}(dy,d\eta).
			\end{align*}
			Notice in particular that $M_2^{2,2} = \mathbf{1}_{y \neq y_0} M_2$ for the operator valued measure $M_2$ given by Proposition \ref{p:semiclassical_measures}, so this measure captures the energy of the sequence $(\psi_{1,h})$ which does not concentrate at the point $y_0$; on the other hand, the measure $M_2^{1}$ captures concentration of the energy at scale $\vert \eta \vert^{-\delta}$, while $M_2^{2,1}$ captures concentration at scale $\vert \eta \vert^{-\delta} \ll \vert y - y_0 \vert \ll 1$. 
			
			Similarly, there exist $\mu_1^1 \in \mathcal{M}_+(\R_\sigma \times \R_\theta \times \R_\xi \times \mathbb{S}_\omega^0)$, $\mu^{2,1}_1 \in \mathcal{M}_+(\R_\sigma \times \mathbb{S}_{\omega_1}^0 \times \R_\xi \times \mathbb{S}_{\omega_2}^0)$, and $\mu_1^{2,2} \in \mathcal{M}_+(\R_\sigma \times \R_y \times \R_\xi \times \mathbb{S}_\omega^0)$ such that, modulo the extraction of subsequences,
			\begin{align*}
				I_{h,R,\epsilon,\rho}^3(a)  \to  \int_{\R_\sigma \times \mathbb{R}_\theta \times \R_\xi \times \mathbb{S}^0_\omega} \, a(\sigma,y_0,\theta,\xi,\omega,0) \mu^1_1(d\sigma,d\theta,d\xi,d\omega),
			\end{align*}
			and 
			\begin{align*}
				I_{h,R,\epsilon,\rho}^4(a)  & \to \int_{\R_\sigma \times \mathbb{S}^0_{\omega_1} \times \R_\xi \times \mathbb{S}^0_{\omega_2}} \, a(\sigma,y_0,\omega_1,\xi,\omega_2,0) \mu^{2,1}_1(d\sigma,d\omega_1,d\xi,d\omega_2) \\[0.2cm]
				& \quad + \int_{\R_\sigma \times \R_y \times \R_\xi \times \mathbb{S}^0_{\omega}} \mathbf{1}_{y \neq y_0} \, a\left( \sigma,y, \frac{y-y_0}{\vert y - y_0 \vert},\xi,\omega,0 \right) \mu^{2,2}_1(d\sigma,dy,d\xi,d\omega).
			\end{align*}
			As before $\mu_1^{2,2} = \mathbf{1}_{y \neq y_0} \mu_1$ for the measure $\mu_1$ obtained in Proposition \ref{p:semiclassical_measures}. To study the properties of these semiclassical measures associated to the sequence $\psi_{1,h}$, we again consider the Wigner equations, for $j = 1,2,3,4$,
			\begin{equation}
				\label{e:distribution_equation_ft}
				\left \langle [-h^2 \Delta_G, \Op^\w_h(\mathbf{a}_{h,R,\epsilon,\rho}^j) ] \psi_h ,\psi_h \right \rangle_{L^2(M)} = 2i h \left \langle \Op^\w_h(\mathbf{a}_{h,R,\epsilon,\rho}^j( b - \beta_h)) \psi_h ,\psi_h \right \rangle_{L^2(M)} + r(h) + O(h^2/R).
			\end{equation}
			We assume that at least one of the measures $M_2^1,M_2^{2,1},M_2^{2,2}$, $\mu_1^1,\mu_1^{2,1},\mu_1^{2,2}$ does not vanish, to reach a contradiction.
			
			By mimicking the first part of the proof of Corollary \ref{p:negative_sub-ellitpic_quasimodes}, we see that $M_2^1$, $M_2^{2,1}$, $M_2^{2,2}$ satisfy \eqref{e:trace_support} (with obvious substitutions) and are invariant by the flow $e^{it (-\partial^2_x + \eta^2 x^2)}$, while $\mu_1^1$, $\mu_1^{2,1}$, $\mu_1^{2,2}$  satisfy that its support is contained in $\{ \sigma^2 + \xi^2 = 1 \}$, and are invariant by the flow $\vartheta_t$. 
			
The difference with respect to Section \ref{subsection3.1}, and the reason why we need to introduce a further two-microlocalization near the point $y_0$ at scale $\vert \eta \vert^{-\delta}$, is to capture the precise obstruction to propagate the semiclassical measures in the vertical direction due to the profile of the damping term $b(y) = \vert y - y_0 \vert^\nu$ near $y_0$. The semiclassical calculus concerning the new variable $(y-y_0)\vert \eta \vert^\delta$ is contained in the one developed in the previous sections (and estimates concerning this variable are even easier since we stay at semiclassical scales under the critical regime for the uncertainty principle). Moreover, the subelliptic calculus in the variable $x \vert \eta \vert$ has been exhaustively studied in the previous sections. It is essential that we weight the concentration scale near $y_0$ in terms of the $\eta$ variable (governing the subelliptic scale), since $\eta$ controls the vertical propagation velocity and hence the hitting velocity towards the profile of the damping term.			

From now on we concentrate  ourselves on the novelties of the analysis arising in this case. Assuming that the symbol $a$ is invariant by the flow $\vartheta_t$, using \eqref{e:distribution_equation_ft} and the explicit expression $b(y) = \vert y - y_0 \vert^\nu$ near $y_0$, noting also that $\nu \delta = 1 - \delta$, we find for $\vert \eta \vert^{1-\delta} \mathbf{a}^3_{h,R,\epsilon,\rho}$ that 
			\begin{align}
			\label{e:wigner_scalar_narrow}
				\left \langle \Op^{\w,(x,\xi)}_h \Op_1^{\w,(y,\eta)} \left( \vert y - y_0 \vert^\nu \vert h \eta \vert^{\delta \nu}  \mathbf{a}^3_{h,R,\epsilon,\rho}(x, y,\xi, h\eta)  \right) \psi_{1,h}, \psi_{1,h} \right \rangle_{L^2} & \\[0.2cm]
				&  \hspace*{-8cm}  =  \left \langle \Op^{\w,(x,\xi)}_h \Op_1^{\w,(y,\eta)} \left(  2 x^2 h \eta \vert h \eta \vert \cdot \eth_y \mathbf{a}^3_{h,R,\epsilon,\rho}(x,y,\xi, h\eta) \right) \psi_{1,h}, \psi_{1,h} \right \rangle_{L^2} \notag \\[0.2cm]
				& \hspace*{-7.5cm} + O(\rho^{-1}) + O(R^{-\delta}) + O(\epsilon^{1-\delta}) + O (r_h h^{-1} ) + O(h/R), \notag
			\end{align}
			where 
			$$
			\eth_y\mathbf{a}^3_{h,R,\epsilon,\rho}(x,y,\xi, h\eta) := (1- \chi_R(\eta)) \chi_\epsilon(h\eta) \chi_\rho\left(  (y-y_0)\vert \eta \vert^\delta \right)
			\partial_{y'}a\left( x\vert \eta \vert, y,  (y-y_0)\vert \eta \vert^\delta, \xi, \eta, h\eta \right).
			$$ 
	 Similarly as in Section \ref{subsection3.1}, the reason to multiply the symbol $\mathbf{a}_{h,r,\epsilon,\rho}^3$ by $\vert \eta \vert^{1-\delta}$ is to renormalize the equation in terms of the (variable) semiclassical parameter $h \vert \eta \vert$. Precisely, taking limits in both sides of the Wigner equation yields in this case:
			$$
			\int_{\R_\sigma \times \mathbb{R}_\theta \times \R_\xi \times \mathbb{S}^0_\omega} \big( \vert \theta \vert^\nu -  \omega \sigma^2  \partial_\theta \big) a(\sigma,y_0,\theta,\xi,\omega,0) \mu_1^1(d\sigma,d\theta,d\xi,d\omega)  = 0.
			$$
	Using next the invariance of $\mu_1^1$ by the flow $\vartheta_t$, this is equivalent to the equation:
			\begin{equation}
			\label{e:differential_narrow_equation}
			\int_{\R_\sigma \times \mathbb{R}_\theta \times \R_\xi \times \mathbb{S}^0_\omega} \big( \vert \theta \vert^\nu - \omega (\xi^2 + \sigma^2) \partial_\theta \big) a(\sigma,y_0,\theta,\xi,\omega,0) \mu_1^1(d\sigma,d\theta,d\xi,d\omega)  = 0.
			\end{equation}
This gives us a new invariance property of $\mu_1^1$ in the vertical direction with respect to the profile $\vert \theta \vert^\nu$ (compare with \eqref{e:vertical_invariance_equation}). We aim at showing that this invariance property is not possible unless $\mu_1^1$ vanishes.	
	
			 To this aim, we integrate the above differential equation for $\mu_1^1$ by considering, for any given $a$, the symbol
			$$
			 \widetilde{a}(\sigma,y,y',\xi,\eta,\eta') := a ( \sigma,y,y' + t \operatorname{sgn}(\eta) H_{G_0}, \xi,\eta,\eta') e^{\int_0^t ( \vert y' + s  \operatorname{sgn}(\eta) H_{G_0}\vert^\nu)  ds},
			$$
where recall that $H_{G_0}(x,\xi,\eta) = \xi^2 + x^2\eta^2$.
			Plugging this symbol in the differential equation \eqref{e:differential_narrow_equation}, integrating in the interval $[0,t]$, and using that $\operatorname{supp} \mu_1^1 \subset \{ \sigma^2 + \xi^2 = 1\}$, we obtain for any $t \in \R$:
			$$
			\int_{\R_\sigma \times \mathbb{R}_\theta \times \R_\xi \times \mathbb{S}^0_\omega}  a(\sigma,y_0,\theta,\xi,\omega,0) d \mu_1^1 = \int_{\R_\sigma \times \mathbb{R}_\theta \times \R_\xi \times \mathbb{S}^0_\omega}  a(\sigma,y_0,\theta + t \omega,\xi,\omega,0) e^{\int_0^t ( \vert \theta + s \omega \vert^\nu)  ds} d \mu_1^1.
			$$
			Therefore, taking $a = 1$, we get, for all $t > 0$,
			$$
			\int_{\R_\sigma \times \mathbb{R}_\theta \times \R_\xi \times \mathbb{S}^0_\omega}  d\mu_1^1 = \int_{\R_\sigma \times \mathbb{R}_\theta \times \R_\xi \times \mathbb{S}^0_\omega}e^{\int_0^t  \vert \theta + s \omega \vert^\nu  ds} d \mu_1^1,
			$$
			which implies, since $\mu_1^1$ is a finite measure, that $\mu_1^{1} = 0$.
			
			On the other hand, looking at $M_2^1$,  we take $ \vert \eta \vert^{1- \delta} \mathbf{a}^1_{h,R,\epsilon,\rho}$ and get the Wigner equation:
			\begin{align}
			\label{e:wigner_operator_narrow}
				h \left \langle \Op_1^{\w,(x,\xi)} \Op_1^{\w,(y,\eta)} \big( (\vert y - y_0 \vert^\nu \vert h \eta \vert^{\delta \nu} - \beta_h h^{\delta-1} \vert h^2 \eta \vert^{1- \delta} )  \mathbf{a}^1_{h,R,\epsilon,\rho}(hx, y, \xi, h\eta) \big) \Psi_h, \Psi_h \right \rangle_{L^2} & \\[0.2cm]
				&  \hspace*{-13cm}  =  h \left \langle \Op_1^{\w,(x,\xi)} \Op_1^{\w,(y,\eta)} \Big(  2 x^2 (h^2 \eta) \vert h^2 \eta \vert \cdot \partial_y \mathbf{a}^1_{h,R,\epsilon,\rho}(hx, y,\xi,h\eta) \Big) \Psi_h, \Psi_h \right \rangle_{L^2} \notag \\[0.2cm]
				& \hspace*{-12.5cm} +  O(h \rho^{-1}) + O( r_h h^{-1 + \delta}) + O(h^{1+\delta}/R)), \notag
			\end{align}
			where $\Psi_h(x,y) = h^{1/2}\psi_{1,h}(hx,y)$. Defining $C_0 = \lim_{h \to 0^+} \beta_h h^{1-\delta}$, dividing both sides by $h$, using that $r_h = o(h^{2 -\delta})$, taking limits, and using the invariance property of $M_2^1$ by $e^{it (-\partial^2_x + \eta^2 x^2)}$, we obtain
			\begin{align*}
				\operatorname{Tr} \int_{\R_\theta  \times (\R_\eta \setminus \{0 \})}  (\vert \theta \vert^\nu - C_0 \vert \eta \vert^{\frac{\nu}{1+\nu}} ) \Op_1^{\w,(x,\xi)} \left(  a \left( x \vert \eta \vert, y_0, \theta, \xi, \frac{\eta}{\vert \eta \vert},\eta \right)\right)  M_2^1(d\theta,d\eta) & \\[0.2cm]
				& \hspace*{-11cm} = \operatorname{Tr} \int_{\mathbb{R}_\theta \times (\R_\eta \setminus \{0 \})}  \Op_1^{\w,(x,\xi)} \left(  H_{G_0}  \cdot \partial_{\theta} a\left( x\vert \eta \vert,y_0,\theta,\xi,\frac{\eta}{\vert \eta \vert},\eta \right) \right) M^1_2(d\theta,d \eta).
			\end{align*}
			Let us denote $\kappa(\theta,\eta) :=  \vert \theta \vert^\nu - C_0 \vert \eta \vert^{\frac{\nu}{1+\nu}} $.  To integrate the above differential equation for the measure $M_2^1$, we consider the symbol, given $a$:
			$$
			\widetilde{a}(\sigma,y,y',\xi,\eta,\eta') := a ( \sigma,y,y' + t  \operatorname{sgn}(\eta) H_{G_0}, \xi,\eta,\eta') e^{\int_0^t \kappa( y' + s  \operatorname{sgn}(\eta) H_{G_0}, \eta) ds}.
			$$
			Plugging this symbol in the previous identity, taking $a = 1$, using  \eqref{e:trace_support},  and integrating from $0$ to $t$, we get:
			\begin{align*}
				\operatorname{Tr} \int_{\R_\theta  \times (\R_\eta \setminus \{0 \})}  M_2^1(d\theta,d\eta)  = \operatorname{Tr} \int_{\mathbb{R}_\theta \times (\R_\eta \setminus \{0 \})}   \Op_1^{\w,(x,\xi)} \left( e^{\int_0^t \kappa( \theta + s  \operatorname{sgn}(\eta) H_{G_0},\eta)   ds}\right) M^1_2(d\theta,d \eta).
			\end{align*}
			Denoting $\overline{\mu}_2^1 := \operatorname{Tr} M_2^1$ and $M_2^1(\theta,\eta) = m_2^1(\theta,\eta) \nu_2^1(\theta,\eta)$, where more precisely we write $m_2^1(\theta,\eta) = \mathbf{m}_2^1(\vert \eta \vert^{-1} (D_x^2 + x^2 \eta^2), \theta,\eta)$ the Radon-Nikodym derivative of $M_2^1$, we obtain, for any $t > 0$,
			\begin{align*}
				\int_{\mathbb{R}_\theta \times (\R_\eta \setminus \{0 \})}  \overline{\mu}^1_2(d\theta,d \eta) &  =  \sum_{k \in \mathbb{N}_0} \int_{\mathbb{R}_\theta \times (\R_\eta \setminus \{0 \})}   e^{\int_0^t \kappa( \theta + s  \operatorname{sgn}(\eta) H_{G_0},\eta)    ds}  \langle M_2^1(d\theta,d\eta) \varphi_k(\eta,x), \varphi_k(\eta,x) \big \rangle_{L^2(\R_x)} \\[0.2cm]
				& = \sum_{k \in \mathbb{N}_0} \int_{\mathbb{R}_\theta \times (\R_\eta \setminus \{0 \})}    e^{\int_0^t\kappa( \theta + s  \eta  \lambda_k,\eta)}    ds \, \mathbf{m}_2^1(\lambda_k, \theta,\eta) \nu_2^1(d\theta,d\eta).
			\end{align*}
			From this, we deduce, since $M_2^1$ is a finite measure, $\lambda_k = 2k+1$, and $$
		 \kappa(\theta + s  \eta  \lambda_k, \eta)  =  \big \vert \theta + s \eta  \lambda_k  \big \vert^\nu - C_0 \vert \eta \vert^{\frac{1}{1+ \nu}} \gtrsim \vert s \eta \vert^\nu,
$$ 
as $\vert s \vert \to \infty$  uniformly in compact sets $0 < \vert \eta \vert \leq C$ and $\vert \theta \vert \leq C$, that $M_2^1 = 0$.
			
			Finally, we consider 
			\begin{align*}
				\frac{1}{\vert y - y_0 \vert^\nu}  \mathbf{a}^j_{h,R,\epsilon,\rho}(x,y,\xi,\eta), \quad j=2,4.
			\end{align*}
 In this regime, $\vert y - y_0 \vert \gg \vert \eta \vert^{-1}$, so the term involving the multiplication by the damping term $b(y)$ in the Wigner equation is in this regime much larger than the term involving the vertical propagation, and larger than the remaining terms as well.	Plugging this symbol into the Wigner equation \eqref{e:wigner_scalar_narrow} (for $j=4$) and respectively into \eqref{e:wigner_operator_narrow} (for $j=2$), we get:
			\begin{align}
				& \operatorname{Tr} \int_{\mathbb{S}^0_{\omega} \times  (\R_\eta \setminus \{0 \})}  \vert \omega \vert \, \Op_1^{\w,(x,\xi)} \left( a \left(x \vert \eta \vert, y_0, \omega, \xi, \frac{\eta}{\vert \eta \vert} ,\eta \right)\right)  M_2^{2,1}(d\omega,d\eta) = 0,  \\[0.2cm]
				& \label{e:second_support} \operatorname{Tr} \int_{\mathbb{R}_y  \times (\R_\eta \setminus \{0 \})} \mathbf{1}_{y \neq y_0}   \frac{ \vert y - y_0 \vert^\nu}{\vert y - y_0 \vert^\nu} \Op_1^{\w,(x,\xi)} \left( a\left( x \vert \eta \vert, y, \frac{y-y_0}{\vert y - y_0 \vert},  \xi,\frac{\eta}{\vert \eta \vert},\eta \right)\right)  M_2^{2,2}(dy,d\eta) = 0, \\[0.2cm]
				& \int_{\R_\sigma \times \mathbb{S}^0_{\omega_1} \times \R_\xi \times \mathbb{S}^0_{\omega_2}} \vert \omega \vert a(\sigma,y_0,\omega_1,\xi,\omega_2,0) \mu^{2,1}_1(d\sigma,d\omega_1,d\xi,d\omega_2) = 0, \\[0.2cm]
				& \label{e:fourth_support} \int_{\R_\sigma \times \R_y \times \R_\xi \times \mathbb{S}^0_{\omega}} \mathbf{1}_{y \neq y_0} \, \frac{ \vert y - y_0 \vert^\nu}{\vert y - y_0 \vert^\nu} a\left( \sigma,y, \frac{y-y_0}{\vert y - y_0 \vert},\xi,\omega,0 \right) \mu^{2,2}_1(d\sigma,dy,d\xi,d\omega) = 0.
			\end{align}
Actually, \eqref{e:second_support} is also direct consequence of \eqref{e:support_M_2} since $M_2^{2,2} = {\bf 1}_{y\neq y_0} M_2$, and similarly \eqref{e:fourth_support} is direct consequence of \eqref{e:support_mu_1} since $\mu_1^{2,2} = {\bf 1}_{y \neq y_0} \mu_1$.	Therefore, we obtain $M_2^{2,1}, M_2^{2,2,}, \mu_1^{2,1}, \mu_1^{2,2} = 0$, which is a contradiction, and the proof is complete.
		\end{proof}

\section{Quasimodes in the compact regime
		}\label{normalformimproved}
		
		In this section, we consider the rectangular-shaped dampings $b_j=b_j(y),j=1,2,3$ and prove the upper bounds for Theorem \ref{t:main_theorem}, Theorem \ref{t:sharp_in_1} and Theorem \ref{t:sharp_in_2}. Recall that for $o(h^2\delta_h^{(j)})$ quasimodes $(\psi_h^{(j)})$, satisfying \eqref{e:quasimode_to_contradiction}, we have shown in the last section that
		$$ \lim_{R\rightarrow\infty}\lim_{h\rightarrow 0}\|\Upsilon_h^R \psi_h\|_{L^2}=0.
		$$
		Consequently, from Proposition \ref{Measureformulationj=012}, the subelliptic semiclassical measures $\ov{\mu}_0^{(j)}$ vanish. To obtain a contradiction, we need to show that the compact part of the semiclassical measures $\mu_0^{(j)}$ vanish.

		\subsection{Elliptic geometric control condition}
		\begin{prop}\label{GCC} 
			For $j\in\{1,2,3\}$, we have $\mu_0^{(j)}\mathbf{1}_{\eta\neq 0}=0$.
		\end{prop}
	\begin{proof}
		Thanks to Proposition \ref{Measureformulationj=012}, we know that $\mu_0^{(j)}$ are invariant along the elliptic flow $\phi_t^{\mathrm{e}}$. The identity $\mu_0^{(j)}\mathbf{1}_{\eta\neq 0}=0$ is a direct consequence of the following stronger statement:
				\begin{lemma}\label{dynamical}  	 Let $\Omega_0=\mathbb{T}_x\times (l_1,l_2)$.  There exists $c_0=c_0(\Omega_0)>0$, such that for any $0<|\eta_0|<\infty$ and $(x_0,y_0;\xi_0,\eta_0)\in p^{-1}(1)$,
					\begin{align}\label{GCCj}
						\liminf_{T\rightarrow\infty}\frac{1}{T}\int_0^T\mathbf{1}_{\Omega_0}(\phi_s^{\mathrm{e}}(x_0,y_0;\xi_0,\eta_0))ds\geq c_{0}>0.
					\end{align}
				\end{lemma}
				\begin{proof} 	The flow $\phi_s^{\mathrm{e}}$ is given by the ODEs on $T^*\T^2$:
					$$ \dot{x}=2\xi,\; \dot{\xi}=-V'(x)\eta^2,\; \dot{y}=2V(x)\eta,\; \eta(s)=\eta_0\neq 0.
					$$
					We have the first integrals
					\begin{align}\label{1integral}   \frac{1}{4}|\dot{x}(s)|^2+V(x(s))\eta_0^2=1,\quad y(s)=2\eta_0\int_0^sV(x(s'))ds' \; (\text{mod } 2\pi).
					\end{align}	
					We call the projection $x(s),y(s)$ horizontal flow and vertical flow, respectively. The lower bound $c_0$ in \eqref{GCCj} looks a bit strange since it does not depend on $\eta_0$ as long as $\eta_0\neq 0$. In fact, although the vertical velocity of the flow might be very slow, it turns out that along the flow, the dynamics will spend a relatively long time in $\Omega_0$, within a period of the vertical flow $s\mapsto y(s)$. Thanks to \eqref{1integral}, the vertical velocity does not change  sign, so the key point is to show that the increasement of $y(s)$ within  a ``period'' (when it is periodic) of the horizontal is comparable to the period of the horizontal flow. 
					
					Set $V_m:=\max_{x\in\mathbb{T}}V(x)$. Let $T_{\pi}$ be the period of the vertical flow $y(s)$. 
					We split the argument in two cases, according to the size of $\eta_0$.
					
					\noi
					$\bullet$ {\bf Case 1:} $0< \eta_0<\frac{1}{\sqrt{V_m}}$.				
					
					Since $V_m\eta_0^2<1$, the velocity of the horizontal flow $x\mapsto x(s)$ does not change sign and it is periodic with period 
					$$ \tau_{\pi}=\int_{-\pi}^{\pi}\frac{dx}{\sqrt{1-V(x)\eta_0^2}}.
					$$ 	
					By the Taylor expansion of $V(x)$ near $x=0$, we can choose $\sigma\in(0,\pi/2)$, small enough (independent of $\eta_0$), such that on $(-2\sigma,2\sigma)$, 
					$x\mapsto V(x)$ is increasing on $(0,2\sigma)$ and decreasing on $(-2\sigma,0)$. Note that along the flow $x(s)$, the time spent in $(-\sigma,\sigma)$ is 
					$$ \tau_{\sigma}:=\int_{-\sigma}^{\sigma}\frac{dx}{\sqrt{1-V(x)\eta_0^2}}.
					$$
					The choice of $\sigma$ leads to
					$$ \tau_{\sigma}\leq \int_{\sigma<|x|\leq 2\sigma}\frac{dx}{\sqrt{1-V(x)\eta_0^2}}<\tau_{\pi}-\tau_{\sigma}. 
					$$
					Set $V_{\sigma}:=\min_{x\notin [-\sigma,\sigma]}V(x)>0$ (independent of $\eta_0$), we deduce that the increasement $L$ of $y(s)$ within a horizontal period $\tau_{\pi}$ satisfies 
					$$  2\eta_0\tau_{\pi}V_m\geq L\geq 2\eta_0V_{\sigma}(\tau_{\pi}-\tau_{\sigma})>\eta_0\tau_{\pi}V_{\sigma}.  
					$$
					Therefore, the vertical period $T_{\pi}$ is bounded above by$ \frac{2\pi}{\eta_0\tau_{\pi}V_{\sigma}}$, and the flow will spend at least $\frac{|\Omega_0|}{2\eta_0\tau_{\pi}V_m}$ proportion of time $T_{\pi}$ in $\Omega_0$. We obtain \eqref{GCCj} with $c_0\geq \frac{|\Omega_0|V_{\sigma}}{4\pi V_m}>0$.

					\noi
					$\bullet$ {\bf Case 2:}
					$ \frac{1}{\sqrt{V_m}}\leq \eta_0<\infty$.
					
					In this case, the horizontal flow $x(s)$ cannot bypass critical points of the potential well and is confined between two maxima of $V(x)\mathbf{1}_{V(x)\eta_0^2\leq 1}$. Set
					$$ x_-(\eta_0):=\max\{-\pi\leq x<0: V(x)\eta_0^2=1\},\quad x_+(\eta_0):=\min\{0<x\leq\pi: V(x)\eta_0^2=1\}.
					$$  
					The existence of $x_-(\eta_0),x_+(\eta_0)$ is ensured by the assumption on $V(x)$ and the fact that $\eta_0>\frac{1}{\sqrt{V_m}}$. Without loss of generality, we assume that $x(s)$ is confined on the interval $[x_-(\eta_0),x_+(\eta_0)]$, since between other two maxima the argument is similar.

					Take $\sigma_1=\min\{\frac{1}{2}\min\{|x_-|,|x_+|\},\sigma\}$, where $\sigma$ is given  in Case 1 (independent of $\eta_0$). Then $V(x)$ is increasing in $(0,2\sigma_1)$ and decreasing in $(-2\sigma_1,0)$. Note that $\sigma_1$ may depend on $\eta_0$ when $\eta_0$ is relatively large, and in this case $|x_-(\eta_0)|\sim |x_+(\eta_0)|\sim \frac{1}{\eta_0}$.  
					
					If $V'(x_-)\neq 0$ and $V'(x_+)\neq 0$, the horizontal flow $x(s)$ is still periodic with period
					$$ \tau=\int_{x_-}^{x_+}\frac{dx}{\sqrt{1-V(x)\eta_0^2}}.
					$$ Arguing as in Case 1, we deduce that the increasement $L$ of $y(s)$ within the horizontal period $\tau$ satisfies
					$$ 2\eta_0\tau \max_{[x_-,x_+]}V(x)\geq L\geq \eta_0\tau\min_{[x_-,x_+]\setminus (-\sigma_1,\sigma_1)}V(x).
					$$
					Observing that 
					$$ \frac{\min_{[x_-,x_+]\setminus (-\sigma_1,\sigma_1)}V(x)}{\max_{[x_-,x_+]}V(x)}
					$$
					is bounded from below, uniformly in $\eta_0$, we obtain \eqref{GCCj}.
					
					If $V'(x_-)=0$ or $V'(x_+)=0$, the horizontal flow will converge to a critical point (probably one of the two critical points if $V'(x_-)=V'(x_+)=0$). Note that this situation can only happen when $\sigma_1=\sigma$, thus $\frac{1}{\sqrt{V_m}}\leq \eta_0\leq C_0$, for some uniform constant $C_0>0$, depending only $V$. Consequently $\sigma_1$ is \emph{independent} of $\eta_0$. Since along the flow $x(s)$, the trajectory can pass the interval $(-\sigma_1,\sigma_1)$ at most twice, spending at most
					$$ \tau_{\sigma_1}:=2\int_{-\sigma_1}^{\sigma_1}\frac{dx}{\sqrt{1-V(x)\eta_0^2}}
					$$
					in such an interval, we deduce that during any period $2\tau_{\sigma_1}$, the increasement $L$ of $y(s)$ satisfies
					$$ 4\eta_0\tau_{\sigma_1}V_m\geq L\geq 2\eta_0\tau_{\sigma_1}V_{\sigma_1},
					$$ 
					with $V_{\sigma_1}=\min_{x\notin (-\sigma_1,\sigma_1)}V(x)$ independent of $\eta_0$. Thus we verify \eqref{GCCj} with $c_0=\frac{|\Omega_0|V_{\sigma_1}}{4\pi V_m}>0$.
					This completes the proof of Lemma \ref{dynamical}.

				\end{proof}	
				
				Finally, to complete the proof of Proposition \ref{GCCj}, it suffices to choose $J:=(l_1,l_2)$ be such that $b_j>\frac{\|b_j\|_{L^{\infty}}}{10}$ on $J$ and let $\Omega_0=\T_x\times J$. For $j\in\{1,2,3\}$, by the invariance of $\mu_0^{(j)}$ along the flow $\phi_{\mathrm{s}}^{\mathrm{e}}$ and Lemma \ref{dynamical}, we deduce that $\mu_0^{(j)}\mathbf{1}_{\eta\neq 0}=0$.	
	\end{proof}
	

		Combining Proposition \ref{Measureformulationj=012} and Proposition \ref{GCCj}, to arrive at a contradiction to \eqref{e:quasimode_to_contradiction}, it suffices to show that
		\begin{align*}
			\mu_0^{(j)}\mathbf{1}_{\eta=0}=0,\quad j\in\{1,2,3\}.
		\end{align*}
		In other words, we would like to prove: for any sufficiently small $\epsilon_0>0$,
		\begin{align}\label{compactregime:goal} 
			\lim_{h\rightarrow 0}\|\chi_0(\epsilon_0^{-1}hD_y)\psi_h^{(j)}\|_{L^2}=0,\quad j=1,2,3.
		\end{align}
		By Lemma \ref{apriori}, $\chi_0(\epsilon_0^{-1}hD_y)\psi_h^{(j)}$ are still $o(h^2\delta_h)$ $b_j$-quasimodes, so with abuse of notation, throughout this section, we denote simply $\psi_h^{(j)}=\chi_0(\epsilon_0^{-1}hD_y)\psi_h^{(j)}$. 
		
		Below, we separate the analysis for \emph{transversal high frequency} (TH) part and the \emph{transversal low frequency} part (TL). More precisely, with the second semiclassical parameter $\hbar=h^{\frac{1}{2}}\delta_h^{\frac{1}{2}}$, we write
		\begin{align}\label{THTL} 
			\psi_h^{(j)}=u_h^{(j)}+v_h^{(j)},\quad u_h^{(j)}=\chi_0(\hbar D_y)\psi_h^{(j)},\; v_h^{(j)}=(1-\chi_0(\hbar D_y))\psi_h^{(j)}.
		\end{align}

		\begin{lemma}\label{localization} 
			Let $j\in\{1,2,3\}$ and $\hbar=h^{\frac{1}{2}}(\delta_h^{(j)})^{\frac{1}{2}}$. Then $u_h^{(j)}, v_{h}^{(j)}$ are still $o(h^2\delta_h^{(j)})$ $b_j$-quasimodes. In particular,
			$$\|b_j^{\frac{1}{2}}u_h^{(j)}\|_{L^2}+\|b_j^{\frac{1}{2}}v_h^{(j)}\|_{L^2}=o(\hbar). $$
		\end{lemma} 	
		\begin{proof}
		This is exactly Lemma 3.2 in \cite{S21} and we invite the reader to refer \cite{S21} for more details. Here we only sketch the proof.
		It suffices to show that
		$$ih[\chi_0(\hbar D_y),b_j(y)]\psi_h^{(j)}=o_{L^2}(h\hbar^2)=o_{L^2}(h^2\delta_h).$$ 
		The point is to write the commutator $[\chi_0(\hbar D_y),b_j]$ as
		$$ b_j^{\frac{1}{2}}[\chi_0(\hbar D_y),b_j^{\frac{1}{2}}]+[\chi_0(\hbar D_y),b_j^{\frac{1}{2}}]b_j^{\frac{1}{2}}.
		$$
		Thanks to the assumption on $b_j$, each commutator above provides a $O(\hbar)$ factor and the damped term $b_j^{\frac{1}{2}}\psi_h^{(j)}$ will provide $o_{L^2}(\hbar)$. The result follows.
	\end{proof}


\subsection{Analysis for the transversal high frequencies (TH)}

\begin{prop}\label{TH}
	For $j\in\{1,2,3\}$, we have
	$$ \|v_h^{(j)}\|_{L^2}=o(1).
	$$
\end{prop}
\begin{proof}
	The proof is based on the positive commutator method.  Since the proof is uniform for all $j=1,2,3$, we will simply write $v_h=v_h^{(j)}$ and $b=b_j$ below. We first claim that
			\begin{align}\label{order1dampederror}
				\|b^{\frac{1}{2}}h\partial_yv_h\|_{L^2}=o(\hbar).
			\end{align}	
			Indeed, by definition, we can write  $v_h=\widetilde{\chi}_0(hD_y)v_h$ for some bump function $\widetilde{\chi}_0\in C_c^{\infty}$, we have
			$$ b^{\frac{1}{2}}h\partial_yv_h=h\partial_y\widetilde{\chi}_0(hD_y)(b^{\frac{1}{2}}v_h)+[h\partial_y,\widetilde{\chi}_0(hD_y)]v_h=h\partial_y\widetilde{\chi}_0(hD_y)(b^{\frac{1}{2}}v_h)+O_{L^2}(h).
			$$
			By Lemma \ref{localization} and the obvious fact that $\hbar\ll h$, we obtain \eqref{order1dampederror}. 
	
	Take a cutoff $\chi_3(y)$ such that $\chi_3(y)\equiv 1$ in a neighborhood of $\{y: b(y)=0\}$ and $\mathrm{supp}(\chi_3')\subset\{y: b(y)\gtrsim 1\}$. 
	We compute the inner product $([h^2\Delta_{G}+1,\chi_3(y)y\partial_y]v_h,v_h)_{L^2}$ in two ways. On the one hand, breaking the commutator and using the fact that $v_h$ are $o(h^2\delta_h)$ $b$-quasimodes, we have
	\begin{align}\label{commutator1} 
		\lg[h^2\Delta_G+1,\chi_3(y)y\partial_y]v_h, v_h\rg_{L^2}& =\lg\chi_3(y)y\partial_yv_h, ihbv_h+o_{L^2}(h^2\delta_h) \rg_{L^2} \\[0.2cm]
		& \quad  + \lg ihbv_h+o_{L^2}(h^2\delta_h),\partial_y(\chi_3(y)yv_h) \rg_{L^2}\notag \\[0.2cm]
		& = o(h\delta_h)+O(h)\|b^{1/2}\partial_yv_h\|_{L^2}\|b^{1/2}v_h\|_{L^2} \\[0.2cm]
		& = o(\hbar^2).
	\end{align}
Note that in the estimate above, we have used the fact that $\|\partial_yv_h\|_{L^2}=O(h^{-1})$, thanks to the definition of $v_h$, and \eqref{order1dampederror}.

	On the other hand, we compute the commutator directly as
	\begin{align*}
		[h^2\Delta_{G}+1,\chi_3(y)y\partial_y]& = 2\chi_3(y)V(x)(h\partial_y)^2+2y\chi_3'(y)V(x)(h\partial_y)^2 \\[0.2cm]
		& \quad + h^2\chi_3''(y)yV(x)\partial_y+2h^2\chi_3'(y)V(x)\partial_y.
	\end{align*}
Using the relation above and
			doing the integration by part, we have (recall that $V(x)=W(x)^2$)
			\begin{align*}
				\big\langle [h^2\Delta_G+1,\chi_3(y)y\partial_y]v_h,v_h
				\big\rangle_{L^2}=&
				\big\langle
				2\chi_3(y)V(x)h^2\partial_y^2v_h,v_h
				\big\rangle_{L^2}
				+\big\langle
				2y\chi_3'(y)V(x)h^2\partial_y^2v_h,v_h
				\big\rangle_{L^2}\\[0.2cm]
				& +h\big\langle
				(\chi_3''(y)yV(x)+2\chi_3'(y)V(x) )h\partial_yv_h,v_h
				\big\rangle_{L^2}\\[0.2cm]
				=&-2	\|\chi_3(y)^{1/2}W(x)h\partial_yv_h\|_{L^2}^2-2h
				\big\langle
				V(x)h\partial_yv_h,\chi_3'(y)v_h
				\big\rangle_{L^2}\\[0.2cm]
				&-2\big\langle
				y\chi_3'(y)W(x)^2h\partial_yv_h,h\partial_yv_h
				\big\rangle_{L^2}
				-2h\big\langle
				V(x)h\partial_yv_h,(\chi_3'(y)y)'v_h
				\big\rangle_{L^2}\\[0.2cm]
				&+h\big\langle
				(\chi_3''(y)yV(x)+2\chi_3'(y)V(x) )h\partial_yv_h,v_h
				\big\rangle_{L^2}.
			\end{align*}		
			Therefore,
			\begin{align}\label{chi3} 
				\|\chi_3(y)^{1/2}W(x)h\partial_yv_h\|_{L^2}^2  \leq & Ch\|W(x)h\partial_yv_h\|_{L^2}\|v_h\|_{L^2}\notag \\[0.2cm]  & +  C|\lg \chi_3'(y)W(x)h\partial_yv_h,W(x)h\partial_yv_h\rg_{L^2}|
			 +o(\hbar^2). 
			\end{align}	
	On supp$(\chi_3')$ $b\gtrsim 1$, so the inner products on the right hand side is bounded by
	$ C\|b^{\frac{1}{2}}h\partial_yv_h\|_{L^2}^2.
	$
	Since on supp$(1-\chi_3)$, $b\gtrsim 1$, we obtain from \eqref{chi3},\eqref{order1dampederror} and Young's inequality $AB\leq \sigma A^2+C_{\sigma}B^{2}$ applied to $A=\|W(x)h\partial_yv_h\|_{L^2},B=h\|v_h\|_{L^2}$ that
	\begin{align}\label{chi3'}
		\|W(x)h\partial_yv_h\|_{L^2}^2\leq o(\hbar^2)+C\|b^{\frac{1}{2}}h\partial_yv_h\|_{L^2}^2\leq o(\hbar^2).
	\end{align}
We will need another geometric control estimate:
			\begin{lemma}\label{geomtricestimate}
				Assume that $w_h\in L^2$ such that $w_h=\widetilde{\chi}(hD_y)\widetilde{\chi}_1(h^2\Delta_G)w_h$  for some cutoffs $\widetilde{\chi}_0\in C_c^{\infty}(\R)$ and $\chi_1\in C_c^{\infty}((\frac{1}{100},\infty\big))$. Let  $\omega_1\subset\T^2$ be an open subset such that $\omega_1$ verifies the  (EGCC)
				for the flow $\phi_t^{\mathrm{e}}$. Then there exists $C>0$ such that for all $0<h\ll 1$,
				$$ \|w_h\|_{L^2(\T^2)}\leq C\|w_h\|_{L^2(\omega_1)}+\frac{C}{h}\|(h^2\Delta_G+1)w_h\|_{L^2(\T^2)}.
				$$					
			\end{lemma}
			\begin{proof}
				The proof follows from a standard argument of the propagation of semiclassical meausres. We argue by contradiction. If the desired inequality is untrue, then by normalization, we have
				$$ \|w_h\|_{L^2}=1,\quad \|w_h\|_{L^2(\omega_1)}=o(1),\quad (h^2\Delta_G+1)w_h=o_{L^2}(h).
				$$
				Up to extracting subsequences, we assume that $\mu$ is a associated semiclassical measure of $w_h$. By using the equation and symbolic calculus, we deduce that for any symbol $a(z,\zeta)\in C^{\infty}(\T^2)$,
				$$ 0=\lim_{h\rightarrow 0}\big\langle\frac{i}{h}[h^2\Delta_G+1,\mathrm{Op}_h^{\w}(a)]w_h,w_h\big\rangle_{L^2}=\int_{T^*\T^2}\{\xi^2+V(x)\eta^2,a\}(z,\zeta)\mu(dz,d\zeta).
				$$
				This shows that $\mu$ is invariant along the Hamiltonian flow $\phi_t^{\mathrm{e}}$ of $\xi^2+V(x)\eta^2$. By assumption, $\mu\mathbf{1}_{\omega_1\times\T^2}=0$ and $\mu(T^*\T^2)=1$. This contradicts the geometric control condition on $\omega_1$. The proof of Lemma \ref{geomtricestimate} is now complete.
			\end{proof}
	Consequently, we deduce that the mass cannot concentrate on $x=0$:
	\begin{cor}\label{propagation:horizontal}
		We have 
		$$ \|h\hbar^{-1}\partial_yv_h\|_{L^2}\leq C\|h\hbar^{-1}W(x)\partial_yv_h\|_{L^2}+O(\hbar^{\frac{1}{2}}),
		$$
		as $h\rightarrow 0$.
	\end{cor}

	\begin{proof}
		Let $z_h=h\hbar^{-1}\partial_yv_h$. Since $v_h=(1-\chi_0(\hbar D_y))\chi_0(\epsilon_0^{-1}hD_y)\psi_h$, $z_h$ satisfies the equation
		\begin{align}\label{eq:z} (h^2\Delta_{G}+1-ihb(y))z_h=o_{L^2}(h\hbar)+ih^2\hbar^{-1}[\partial_y,b(y)]v_h=o_{L^2}(h\hbar).
		\end{align}
		Multiplying by $\ov{z}_h$ and taking the imaginary part, we get $\|b^{\frac{1}{2}}z_h\|_{L^2}=o(\hbar^{\frac{1}{2}})$. In particular,
		\begin{align}\label{oh} 
			(h^2\Delta_G+1)z_h=O_{L^2}(h\hbar^{\frac{1}{2}}).
		\end{align}
		Pick $\sigma_0>0$, small enough and set $\omega_1:=\{x\in\T:|x|\geq \sigma_0 \}\times\T_y$. 
		We observe that the trajectory of $\phi_t^{\mathrm{e}}$ passing through any point $(x_0,y_0;\xi_0,\eta_0)\in p^{-1}(1)$ with $|\eta_0|\leq\epsilon_0$ will leave immediately the line $x=0$, provided that $\sigma_0>0$ is small enough. Consequently, $\omega_1$ satisfies (EGCC) for the flow $\phi_t^{\mathrm{e}}$.
		Applying Lemma \ref{geomtricestimate}, we complete the proof of Corollary \ref{propagation:horizontal}. 
	\end{proof}
	
	Now we finish the proof of Proposition \ref{TH}. Indeed, by \eqref{chi3'} and Lemma \ref{propagation:horizontal},
			$$ \|h\partial_yv_h\|_{L^2}\leq C\|W(x)h\partial_yv_h\|_{L^2}+o(\hbar)\leq o(\hbar).
			$$
			By definition, the $y$-Fourier mode of $v_h$ $\mathcal{F}_y(v_h)(\cdot,n)$ vanishes when $|n|\lesssim \hbar^{-1}$, we deduce that 
			\begin{align}\label{passtoL2} \frac{h}{\hbar}\|v_h\|_{L^2}\leq C\|h\partial_yv_h\|_{L^2}\leq o(\hbar),
			\end{align}
			hence $\|v_h\|_{L^2}\leq  o(\hbar^2h^{-1})=o(\delta_h)$.
			Recalling that $\hbar=h^{\frac{1}{2}}\delta_h^{\frac{1}{2}}$, we conclude for the cases $j=1,2$ (since $\delta_h^{(1)},\delta_h^{(2)}\leq 1$).
			
			For the case $j=3$, additional analysis is needed, as $\delta_h^{(3)}=h^{-\frac{1}{\nu+1}}\gg 1$. We still denote $v_h=v_h^{(3)}$.
			The idea is to start from the better bound $\|h\partial_yv_h\|_{L^2}\leq o(\hbar)$ instead of $O(1)$ to
			refine the analysis from \eqref{commutator1} to \eqref{chi3'}. The key point is to improve the bound $\|b^{\frac{1}{2}}h\partial_yv_h\|_{L^2}$.
			
			To do so, we first remark that since $v_h=\widetilde{\chi}(hD_y)v_h$ for some 
			cutoff $\widetilde{\chi}\in C_c^{\infty}(\R)$, 
			the remainder $\widetilde{r}_h=o_{L^2}(h\hbar^2)$ in the equation $(-h^2\Delta_G-1+ihb(y))v_h=\widetilde{r}_h$ can be replaced by $\widetilde{\chi}(hD_y)\widetilde{r}_h+\widetilde{r}_{1,h}$ such that $\widetilde{r}_{1,h}=o_{L^2}(h\hbar^2)$ and $h\partial_y\widetilde{r}_{1,h}=o_{L^2}(h\hbar^2)$. Then for $w_h:=h\partial_yv_h$, it satisfies the equation
			$$ (-h^2\Delta_G-1+ihb(y))w_h=o_{L^2}(h\hbar^2)-ih^2b'(y)v_h=o_{L^2}(h\hbar^2).
			$$
			We deduce as in the proof of Lemma \ref{apriori} that 
			\begin{align}\label{improvedbhdyv}  \|b^{\frac{1}{2}}w_h\|_{L^2}=\|b^{\frac{1}{2}}h\partial_yv_h\|_{L^2}=o(\hbar)\|w_h\|_{L^2}^{\frac{1}{2}}=o(\hbar)\|h\partial_yv_h\|_{L^2}^{\frac{1}{2}}.
			\end{align}
			Plugging this into \eqref{commutator1}, we have
			$$ \big\langle
			[h^2\Delta_G+1,\chi_3(y)y\partial_y]v_h,v_h
			\big\rangle_{L^2}
			=o(\hbar^{2})\|h\partial_yv_h\|_{L^2}^{\frac{1}{2}}.
			$$ 
			Therefore, in the estimate \eqref{chi3}, we can replace $o(\hbar^2)$ on the right hand side by $o(\hbar^{2})\|h\partial_yv_h\|_{L^2}^{\frac{1}{2}}$ and obtain the improvement
			$$ \|W(x)h\partial_yv_h\|_{L^2}^2\leq o(\hbar^{2})\|h\partial_yv_h\|_{L^2}^{\frac{1}{2}}+C\|b^{\frac{1}{2}}h\partial_yv_h\|_{L^2}^2+Ch^2.
			$$
			Using Corollary \ref{propagation:horizontal} and \eqref{improvedbhdyv}, we have
			$$ \|h\partial_yv_h\|_{L^2}\leq o(\hbar)\|h\partial_yv_h\|_{L^2}^{\frac{1}{4}}+o(\hbar)\|h\partial_yv_h\|_{L^2}^{\frac{1}{2}}+O(h)+O(\hbar^{\frac{3}{2}}).
			$$
			This implies that
			$$ \|h\partial_yv_h\|_{L^2}\leq o(\hbar^{\frac{4}{3}}).
			$$

			Coming back to \eqref{passtoL2}, we finally obtain that
			$$ \|v_h\|_{L^2}\leq C\hbar h^{-1}\|h\partial_yv_h\|_{L^2}\leq o(\hbar^{2+\frac{1}{3}}h^{-1})=o(\delta_h^{\frac{7}{6}}h^{\frac{1}{6}})=o(h^{\frac{1}{6}-\frac{7}{6(\nu+1)}})=o(1),
			$$
			thanks to the fact that $\nu\geq 6$.
			The proof of Proposition \ref{TH} is now complete.
\end{proof}

\subsection{Analysis for transversal low frequencies (TL)}

It remains to analyze the TL part $u_h^{(j)}=\chi_0(\hbar D_y)\psi_h^{(j)}$, where we recall that $\hbar=h^{\frac{1}{2}}\delta_h^{\frac{1}{2}}$. The idea is to perform an averaging method (normal form reduction) to average the operator $-h^2\Delta_G-1+ihb(y)$ to recover the flat damped-wave operator $-h^2\Delta-1+ihb(y)$. This idea was implemented in \cite{BuSun} to prove the observability for the Grushin-Schr\"odinger equation, and in \cite{LeS20} to prove the self-adjoint resolvent estimate for Baouendi-Grushin operators. In all these works, quasimodes were localized in the regime $|D_y|\ll h^{-\epsilon}$. However, when dealing with the damped Baouendi-Grushin operator, which is non-selfadjoint, we can only localize $|D_y|$ to a much bigger size $O(\hbar^{-1})$, in order to keep the original width of quasimodes after localization. Since the new scale $\hbar^{-1}$ is much greater than $h^{-\epsilon}$, supplementary error terms appear after the operator conjugation, which cannot be absorbed directly to the remainder of width $o_{L^2}(h^2\delta_h)$. For this reason, we need an iterated normal form to  finally reduce the original damped wave operator to a $x$-free operator.

\subsubsection{Normal form reduction}

\begin{prop}\label{prop:reduction}
	Let $\delta_h\in[h^{\frac{1}{\nu+2}},h^{-\frac{1}{\nu+1}}]$. There exist operators $F_{h,0}, F_{h,1},F_{h,2}$, a constant $M=\int_{\T}V(x)dx$ and a bounded smooth function $\ov{r}_2:\R\rightarrow\R$, such that
	\begin{itemize}
		\item $F_{h,0}$ is selfadjoint;
		\item $\|hF_{h,j}\|_{\mathcal{L}(L^2)}\leq O(h^{\frac{1}{2}}\delta_h^{-\frac{1}{2}})$, $j=1,2$.
	\end{itemize}
	Moreover, for any $o(h^2\delta_h)$ $b$-quasimodes $u_h$, the normal form transformations
	$$ u_{h,3}:= e^{ihF_{h,2}}\circ e^{ihF_{h,1}}\circ e^{ihF_{h,0}}(\chi_0(\hbar D_y)u_h),
	$$
	satisfy:
	\begin{itemize}
		\item[$\mathrm{(a)}$] $\big(h^2D_x^2+Mh^2D_y^2+\ov{r}_2(hD_x)h^4D_y^4-1+ihb(y)\big)u_{h,3}=o_{L^2}(h^2\delta_h)$.
		\smallskip
		
		\item[$\mathrm{(b)}$]  $\|b^{\frac{1}{2}}u_{h,3}\|_{L^2}=o(h^{\frac{1}{2}}\delta_h^{\frac{1}{2}}),\; \|u_{h,3}\|_{L^2}\sim \|u_h\|_{L^2}=1$.
		\smallskip
		
		\item[$\mathrm{(c)}$] $u_{h,3}=\widetilde{\chi}_0(\hbar D_y)u_{h,3}+O_{H^{\infty}}(h^{\infty})$ for some cutoff $\widetilde{\chi}_0$ such that $\widetilde{\chi}_0\chi_0=\chi_0$.
	\end{itemize}
\end{prop}

The proof is based on the Taylor expansion of the conjugation of operators:
\begin{align}\label{Conjugaision} 
	e^{ihF}Pe^{-ihF}=&\sum_{k=0}^{N-1}\frac{1}{k!}\mathrm{ad}_{ihF}^k(P)+\mathcal{O}_N(\mathrm{ad}_{ihF}^{N}(P)),\quad N\in\N,
\end{align}
where the adjoint representation is given by $\mathrm{ad}_A(B):=[A,B]$. Before computing the conjugation, we need a Lemma which says that the normal form transforms will preserve the localization property as well as the decay in the damping region:
\begin{lemma}\label{preserving}
	Assume that $m_1,m_2\in\N$ such that $\max\{1,m_2-1\}\leq m_1\leq m_2$ and $\chi\in C_c^{\infty}(\R)$. Let $q\in C_c^{\infty}(\T_x\times\R_{\xi})$. Let $(f_h)$ be a sequence of functions such that $\|f_h\|_{L^2}=O(1)$  and $$ \|b^{\frac{1}{2}}f_h\|_{L^2}=o(\hbar) \text{ as } h\rightarrow 0.$$ Then for $G_{h}=\Op_h^{\w}(q(x,\xi))\chi(\hbar D_y)h^{m_1}D_y^{m_2}$ and $f_{h,1}:=e^{iG_h}(\chi_0(\hbar D_y)f_h)$, we have $f_{h,1}=\widetilde{\chi}_0(\hbar D_y)f_{h,1}$ for some $\widetilde{\chi}_0\in C_c^{\infty}(\R)$. Furthermore,
	$$ \|b^{\frac{1}{2}}f_{h,1}\|_{L^2}=o(\hbar).
	$$
\end{lemma}
\begin{proof}
	We only consider the situation where $\delta_h=h^{\delta}$ (thus $\hbar=h^{\frac{1+\delta}{2}}$) for some $\delta<\frac{1}{4}$. For smaller $\delta_h$, the proof is easier. 	Applying the last assertion in Lemma \ref{apriori} to $\widetilde{h}=\hbar$, we have $\|b^{\frac{1}{2}}\chi_0(\hbar D_y)f_h\|_{L^2}=o(\hbar)$. Hence without loss of generality, we abuse the notation by denoting $f_h=\chi_0(\hbar D_y)f_h$. Since $[G_h,D_y]=0$, any Fourier multiplier $m(\hbar D_y)$ commutes with $e^{iG_h}$. Therefore, for any $\widetilde{\chi}_0\in C_c^{\infty}(\R)$ such that $\widetilde{\chi}_0\equiv 1$ on supp$(\chi_0)$, we have $f_{h,1}=\widetilde{\chi}_0(\hbar D_y)f_{h,1}$. 
	
	To prove the second assertion, it suffices to show that $$e^{-iG_h}b^{\frac{1}{2}}e^{iG_h}f_h=o_{L^2}(\hbar).$$
	We write $G_h=\mathrm{Op}_h^{\w}(q(x,\xi))h^{m_1}\hbar^{-m_2}L_{\hbar},$ with $L_{\hbar}=(\hbar D_y)^{m_2}\chi(\hbar D_y)$.  Using the Taylor expansion and the symbolic calculus  (as $b^{\frac{1}{2}}\in W^{2,\infty}(\T)$ thanks to \eqref{e:B-H_condition}), we have
	\begin{align*}
		e^{-iG_h}b^{\frac{1}{2}}e^{iG_h}=&b^{\frac{1}{2}}-ih^{m_1-\frac{m_2(1+\delta)}{2}}\mathrm{Op}_h^{\w}(q(x,\xi))[L_{\hbar},b^{\frac{1}{2}}]\\&+O_{\mathcal{L}(L^2)}(h^{m_1+1-\frac{m_2(1+\delta)}{2}} )+O_{\mathcal{L}(L^2)}(h^{2m_1-(m_2-1)(1+\delta)}),
	\end{align*}
where the first error of big $O$ comes from the commutator $[\mathrm{Op}_h^{\w}(q),b^{\frac{1}{2}}]$, and the second error of big $O$ comes from the higher order commutators with at least two $G_h$.	Since $m_1\geq m_2-1$ and $\delta<\frac{1}{4}$, we have
	$$ e^{-iG_h}b^{\frac{1}{2}}e^{iG_h}f_{h}=-ih^{m_1-\frac{m_2(1+\delta)}{2}}{\mathrm{Op}_h^{\w}(q(x,\xi))  } [L_{\hbar},b^{\frac{1}{2}}]f_h+o_{L^2}(\hbar).
	$$
	As $[L_{\hbar},b^{\frac{1}{2}}]=O_{\mathcal{L}(L^2)}(h^{\frac{1+\delta}{2}})$, the only case we cannot conclude from the symbolic calculus is when $m_1=1,m_2=2$, since we can only control it by $O(h^{\frac{1-\delta}{2}})$. In order to gain an extra $o(h^{\delta})$, we need to make use of the damping $b$. Indeed,  we rewrite the term $ih^{-\delta}[L_{\hbar},b^{\frac{1}{2}}]$ using symbolic calculus as
	$$ h^{\frac{1-\delta}{2}}\Op_{\hbar}^{\w}(\{\chi(\eta)\eta^{m_2},b^{\frac{1}{2}}(y)\})+O_{\mathcal{L}(L^2)}(h).
	$$  
	Thanks to \eqref{e:B-H_condition}, 
	$$ |\{\chi(\eta)\eta^{m_2},b^{\frac{1}{2}}(y) \}|\leq Cb^{\frac{1}{2}-\sigma}.
	$$
Therefore, by the sharp G$\mathring{\mathrm{a}}$rding inequality, interpolation $\|b^{\frac{1}{2}-\sigma}f_h\|_{L^2}\leq o(\hbar^{1-2\sigma})$ as in the proof of the last assertion in Lemma \ref{apriori}, 
we conclude that
			$ih^{-\delta}[L_{\hbar},b^{\frac{1}{2}}]f_h=o_{L^2}(\hbar)$, and consequently $h^{m_1-\frac{m_2(1+\delta)}{2}}\|\mathrm{Op}_h^{\w}(q)[L_{\hbar},b^{\frac{1}{2}}]f_h\|_{L^2}=o(\hbar).$
	This completes the proof of Lemma \ref{preserving}.
\end{proof}

		\begin{proof}
			With abuse of notation, we assume that $u_h=\chi_0(\hbar D_y)u_h$. We will construct three normal form transforms $e^{ihF_{h,0}}, e^{ihF_{h,1}}$ and $e^{ihF_{h,2}}$. Each time the symbol $hF_{h,j}$ will be of the form $$\Op_h^{\w}(q_j(x,\xi))\chi(\hbar D_y)h^{m_1}D_y^{m_2},\quad \max\{0,m_2-1\}\leq m_1\leq m_2.$$ Therefore, the assertions (b), (c) of Proposition \ref{prop:reduction} are consequences of Lemma \ref{preserving}. Below we concentrate on the construction of normal form transforms:
			
			\noi
			$\bullet$ {\bf Step 1: The first normal form reduction.} We will choose the first conjugate operator $e^{ihF_{h,0}}$ with $F_{h,0}=\Op_h^\w(q_0(x,\xi))\chi(\hbar D_y)D_y^2$ for some $\chi\in C_c^{\infty}(\R)$ such that $\chi\equiv 1$ on supp$(\chi_0)$. Note that $F_{h,0}$ is self-adjoint\footnote{Unlike in \cite{BuSun} and \cite{LeS20}, the operators $hF_{h,0}$ are not uniformly unbounded with respect to $h$.}.
			The transformed quasimodes will be denoted by $u_{h,1}:=e^{ihF_{h,0}}u_h$, with $u_h=\chi_0(\hbar D_y)u_h$ for simplicity. Note that $\|u_{h,1}\|_{L^2}=\|u_h\|_{L^2}$ since $e^{ihF_{h,0}}$ is unitary.
			
			In order to find out the symbol $q_0(x,\xi)$, 
			we need to compute the conjugate operator $e^{ihF_{h,0}}(P_{h,b}-1)e^{-ihF_{h,0}}$ using \eqref{Conjugaision} with $N=3$
			and the symbolic calculus, and this gives
			\begin{align}\label{normalform1} 
				e^{ihF_{h,0}}(P_{h,b}-1)e^{-ihF_{h,0}}u_{h,1}
				:=&(h^2D_x^2+V(x)h^2D_y^2-1+ihb)u_{h,1}+ih[F_{h,0},h^2D_x^2]u_{h,1}\notag\\
				   &  -\frac{h^2}{2}[F_{h,0},[F_{h,0},h^2D_x^2]]u_{h,1}
				  +ih[F_{h,0},V(x)h^2D_y^2]u_{h,1} \notag \\
				   & - h^2[F_{h,0},b]u_{h,1}-\frac{ih^3}{2}[F_{h,0},[F_{h,0},b]]u_{h,1}\notag \\  &+O_{L^2}(h^6\hbar^{-6})+O_{L^2}(h^4\hbar^{-3}),
			\end{align}
where the errors in big $O$ comes from the third order commutator $\mathrm{ad}_{ihF_{h,0}}^3(P_{h,b}-1)$. To see this more precisely, we remark that $ihF_{h,0}$ is of the form $i(h\hbar^{-2})A_hB_{\hbar}$, where $A_h=\mathrm{Op}_h^{\w}(q_0)$ and $B_{\hbar}=\chi(\hbar D_y)(\hbar D_y)^2$. Since $[B_{\hbar},A_h]=0, [B_{\hbar},h^2\Delta_G]=0$, we have
\begin{align*} \mathrm{ad}_{ihF_{h,0}}^3(P_{h,b}-1)&= i(h\hbar^{-2})^3\mathrm{ad}_{A_hB_{\hbar}}^3(-h^2\Delta_G)-h(h\hbar^{-2})^3\mathrm{ad}_{A_hB_{\hbar}}(b)\\[0.2cm] &= ih^3\hbar^{-6}\mathrm{ad}_{A_h}^3(-h^2\Delta_G)B_{\hbar}^3-h^4\hbar^{-6}\mathrm{ad}_{A_hB_{\hbar}}(b).
\end{align*}
The first term on the right hand side is $O_{\mathcal{L}(L^2)}(h^6\hbar^{-6})$ while the second term on the right hand side is $O_{\mathcal{L}(L^2)}(h^3\hbar^{-3})$.
			As $h^{\frac{1}{\nu+2}}\leq \delta_h\leq h^{-\frac{1}{\nu+1}}$ and $\nu>4$, all the errors in big $O$ are indeed $o_{L^2}(h^2\delta_h)$.

			The first line on the right hand side of \eqref{normalform1} is the principal term, by comparing the principal symbol in front of $h^2D_y^2$, we require that $q_0$ solves the first cohomological equation
			$$ 2\xi\partial_xq_0(x,\xi)=V(x)-\frac{1}{2\pi}\int_{\T}V.
			$$  
		Therefore, with $M=\frac{1}{2\pi}\int_{-\pi}^\pi V(z)dz$, we solve the above cohomological equation by setting 
			$$ q_0(x,\xi)=\frac{\chi_1(\xi)}{2\xi}\int_{-\pi}^x(V(z)-M)dz-\frac{\chi_1(\xi)}{2\xi}\cdot\frac{1}{2\pi}\int_{-\pi}^{\pi}dx\int_{-\pi}^x(V(z)-M)dz,
			$$
			where $\chi_1\in C_c^{\infty}(\R)$ is given by \eqref{e:cut_off_1}\footnote{Here we can choose any cutoff with support away from zero, but to sake the notation, we make the choice of cutoff $\chi_1$. }. Note that with the second term depending only on $\xi$, we have \begin{align}\label{zeroaverage}
				\int_{\T}q_0(x,\xi)dx=0.\end{align}
			By the exact Weyl symbolic calculus, we have
			\begin{align}\label{error1}
				ih[F_{h,0},h^2D_x^2] = -\mathrm{Op}_h^\w(2\xi\partial_xq_0)\chi(\hbar D_y)h^2D_y^2.
			\end{align}
			Using the cohomological equation
			$$
			2 \xi \partial_x q_0 = \chi_1(\xi) (V(x) - M), 
			$$ 
			we have
			\begin{align*} 
				(h^2D_x^2+V(x)h^2D_y^2-1+ihb(y))u_{h,1}+ih[F_{h,0},h^2D_x^2]u_{h,1}-\frac{h^2}{2}[F_{h,0},[F_{h,0},h^2D_x^2]]u_{h,1} &  \notag \\[0.2cm]
				& \hspace*{-11cm} = (h^2D_x^2+Mh^2D_y^2-1+ihb(y))u_{h,1} - \frac{1}{2}\Op^\w_h(  \{ q_0 , \chi_1(\xi)(V(x)-M) \}) h^4 D_y^4 u_{h,1}\notag \\[0.2cm]
				& \hspace*{-11cm}  \quad -\frac{ih^3}{2}[F_{h,0},[F_{h,0},b ] u_{h,1}  +o_{L^2}(h^{2}\delta_h),
			\end{align*}
			where we used the fact that $ (1 - \chi_1(hD_x)) u_{h,1} = O(h^\infty)$.
		We claim that 
			\begin{align}\label{doublecommutator}
			h^3[F_{h,0},[F_{h,0},b] ]u_{h,1}=o_{L^2}(h^2\delta_h).	
			\end{align}		
					 Recall that $F_{h,0}=\hbar^{-2}A_hB_{\hbar}$ with $[A_h,B_{\hbar}]=0$. Using the Jacobi identity $[A,[B,C]]+[B,[C,A]]+[C,[A,B]]=0$ and the fact that $$[A_h,b]=O_{\mathcal{L}(L^2)}(h),\quad [B_{\hbar},b]=O_{\mathcal{L}(L^2)}(\hbar),$$ we deduce that
		$$ h^3[F_{h,0},[F_{h,0},b]]u_{h,1}=h^3\hbar^{-4}A_h^2[B_{\hbar},[B_{\hbar},b]]u_{h,1}+\underbrace{O_{L^2}(h^4\hbar^{-3})}_{o_{L^2}(h^2\delta_h) },
		$$	
		where the last error in big $O$ comes from the terms where there is at least one $[A_h,b]$. Since the principal symbol of $-\frac{1}{\hbar^2}[B_{\hbar},[B_{\hbar},b]]$ satisfies
		$$ |\{\chi(\eta)\eta^2 \{\chi(\eta)\eta^2,b \} \} |\leq Cb^{\frac{1}{2}-2\sigma},
		$$
		the same manipulation as in the proof of last assertion in Lemma \ref{apriori} (using the sharp G$\mathring{\mathrm{a}}$rding inequality and the interpolation) yields $\|[B_{\hbar},[B_{\hbar},b]]u_{h,1}\|_{L^2}\leq O(\hbar^{3-4\sigma})+O(\hbar^{\frac{3}{2}})$. Consequently, 
		$h^3[F_{h,0},[F_{h,0},b]]u_{h,1}=o_{L^2}(h^2\delta_h).
		$
	
			Next,
			\begin{align}\label{error2} 
				ih[F_{h,0},V(x) h^2D_y^2]u_{h,1} = \mathrm{Op}_h^\w( \{ q_0,  V(x) \}) )\chi(\hbar D_y)h^4D_y^4 v_h^{(1)}+o_{L^2}(h^{2}\delta_h),
			\end{align}
			and
			\begin{align}\label{error3} 
				-h^2[F_{h,0},b(y)]u_{h,1} &= 2ih^2\Op_h^\w(q_0)\chi(\hbar D_y)b'(y)D_yu_{h,1}+o_{L^2}(h^{2}\delta_h),
			\end{align}
			where we have used $[\chi(\hbar D_y),b]v_h^{(1)}=O_{L^2}(\hbar^{\infty})$, thanks to the support property of $\chi,\chi_0$ and the fact that $u_{h,1}=\chi_0(\hbar D_y)e^{ihF_{h,0}}u_h$.
		
		Set
			\begin{align}\label{R1} 
				R_{1,h}& :=2h^2\Op_h^{\w}(q_0)\chi(\hbar D_y)b'(y)D_y,\\[0.2cm]
				R_{2,h} & := \frac{1}{2} \mathrm{Op}_h^\w( \{ q_0, \chi_1(\xi)( V(x)+ M) \})\chi(\hbar D_y)h^4D_y^4. \label{R2} 
			\end{align}
		In summary, with
			\begin{align}\label{Lh1} 
				\mathcal{L}_{h,1} & := h^2D_x^2+Mh^2D_y^2-1+ihb(y)+\underbrace{ iR_{1,h}
				}_{\text{anti-selfadjoint remainder}}
				+ \underbrace{R_{2,h}}_{\text{self-adjoint remainder}},
			\end{align}
			we have
			$$ \mathcal{L}_{h,1}u_{h,1}=g_{h,1}=o_{L^2}(h^{2}\delta_h),\quad b^{\frac{1}{2}}u_{h,1}=o_{L^2}(h^{\frac{1}{2}}\delta_h^{\frac{1}{2}}).
			$$

			\noi
			$\bullet$ {\bf Step 2: Successive normal form reductions.} In order to average the anti-selfadjoint remainder $iR_{1,h}$, we introduce
			$$ F_{h,1}:=\Op_h^\w(q_1)\chi(\hbar D_y)b'(y)D_y, 
			$$
			where $q_1(x,\xi)$ is complex-valued. Set $u_{h,2}:=e^{ihF_{h,1}}u_{h,1}$. Note that $$\|hF_{h,1}\|_{\mathcal{O}(L^2)}\leq h\hbar^{-1}=h^{\frac{1}{2}}\delta_h^{-\frac{1}{2}},$$ the exponential $e^{ihF_{h,1}}$ is well-defined as a perturbation of identity. In particular, $\|u_{h,2}\|_{L^2}\sim \|u_{h,1}\|_{L^2}$. 
			Mimic the proof of Lemma \ref{preserving}, we deduce that (with $m_1=m_2=1, L_{\hbar}=\chi(\hbar D_y)b'(y)\hbar D_y$)
			\begin{align}\label{dampeddecay} 
				\|b^{\frac{1}{2}}u_{h,2}\|_{L^2}=o(\hbar).
			\end{align}
			Also, with slightly wider cutoff $\widetilde{\chi}$, we have 
			\begin{align}\label{WFhbar} 
				(1-\widetilde{\chi}(\hbar D_y))u_{h,2}=O_{H^{\infty}}(h^{\infty}).
			\end{align}
			By conjugation with $\mathcal{L}_{h,1}$ and the symbolic calculus, we have
			\begin{align*}
				e^{ihF_{h,1}}\mathcal{L}_{h,1}e^{-ihF_{h,1}}u_{h,2}=&(h^2D_x^2+Mh^2D_y^2-1+ihb(y)+iR_{1,h}+R_{2,h})u_{h,2}+i[hF_{h,1},h^2D_x^2]u_{h,2}\\[0.2cm]
				& +i[hF_{h,1},Mh^2D_y^2+ihb(y)+iR_{1,h}+R_{2,h}]u_{h,2}+O_{L^2}(h^4\hbar^{-3}+h^3+h^6\hbar^{-6}).
			\end{align*}
			The error inside the big $O$ is $o_{L^2}(h^2\delta_h)$. Moreover, 
			$$[hF_{h,1},Mh^2D_y^2+ihb(y)+iR_{1,h}+R_{2,h}]u_{h,2}=o_{L^2}(h^2\delta_h).
			$$
			Defining
			$$ q_1(x,\xi)=\frac{i\chi_1(\xi)}{\xi}\int_{-\pi}^xq_0(z,\xi)dz,
			$$
			by \eqref{zeroaverage} and the cohomological equation
			$$ 2\xi\partial_xq_1=2i\chi_1(\xi)q_0(x,\xi)
			$$
			and the fact that $\|(1-\chi_1(hD_x))u_{h,2}\|_{L^2}=O(h^{\infty})$,
			we deduce that
			$$ \mathcal{L}_{h,2}u_{h,2}=o_{L^2}(h^2\delta_h),
			$$
			where
			$$ \mathcal{L}_{h,2}:=h^2D_x^2+Mh^2D_y^2-1+ihb(y)+R_{2,h}.
			$$
			
			Finally, to average the self-adjoint remainder $R_{2,h}$, we define
			$$ F_{h,2}:=\Op_h^{\w}(q_2)\chi(\hbar D_y)h^2D_y^4,
			$$
			where the symbol $q_2(x,\xi)$ is real-valued.
			We define $u_{h,3}=e^{ihF_{h,2}}u_{h,2}$, by the symbolic calculus, we have similarly as in the previous argument
			\begin{align*}
				e^{ihF_{h,2}}\mathcal{L}_{h,2}e^{-ihF_{h,2}}u_{h,3}  & = \mathcal{L}_{h,2}u_{h,3}+ih[F_{h,2},\mathcal{L}_{h,2}]u_{h,3}+o_{L^2}(h^{2}\delta_h)\\[0.2cm]
				& =   (h^2D_x^2+Mh^2D_y^2-1+ihb(y))u_{h,3} \\[0.2cm]
				& \quad + 
				\frac{1}{2} \mathrm{Op}_h^\w( \{ q_0, \chi_1(\xi)(V(x) + M) \})\chi(\hbar D_y)h^4D_y^4u_{h,3} \\[0.2cm]
				&  \quad +\frac{i}{h}[\Op_h^{\w}(q_2),h^2D_x^2]\chi(\hbar D_y)h^4D_y^4u_{h,3} + o_{L^2}(h^2\delta_h),
			\end{align*}
			thanks to the fact that $\delta_h\geq h^{\frac{1}{\nu+2}}$ and $\nu>4$.
			Denoting  $r_2(x,\xi) =  \frac{1}{2} \{ q_0, \chi_1(\xi)( V(x) + M) \}$, we choose
			$$  q_1(x,\xi)=\frac{\chi_1(\xi)}{2\xi} \int_{-\pi}^x  (r_2(z,\xi) - \overline{r}_2(\xi)) dz, \quad \overline{r}_2(\xi) := \frac{1}{2\pi} \int_{-\pi}^\pi r_2(x,\xi) dx,
			$$
			which solves the cohomological equation $ 2\xi\partial_xq_1(x,\xi)=\chi_1(\xi)(r_2(x,\xi) - \overline{r}_2(\xi))$.  Therefore, with 
			$$ \mathcal{L}_{h,3}:=h^2D_x^2+Mh^2D_y^2 +  \overline{r}_2(hD_x) h^4 D_y^4 -1+ihb(y), 
			$$
			we have
			$$ \mathcal{L}_{h,3}u_{h,3}=o_{L^2}(h^{2}\delta_h),\quad b^{\frac{1}{2}}u_{h,3}=o_{L^2}(\hbar),
			$$
			and $u_{h,3}=(1-\widetilde{\chi}(\hbar D_y) )u_{h,3}+O_{H^{\infty}}(h^{\infty})$. This completes the proof of Proposition \ref{prop:reduction}.
		\end{proof}

		Now we are ready to finish the proof of the upper bound in Theorem \ref{t:sharp_in_2}:
		\begin{proof}[Proof of Upper bounds in Theorem \ref{t:sharp_in_2}]
			Denote by $u_{h,3}^{(3)}$ be the transformed quasimodes of $u_h^{(3)}$ given by Proposition \ref{prop:reduction} with $\delta_h=h^{-\frac{1}{\nu+1}}$ and $\hbar=h^{\frac{\nu}{2(\nu+1)}}$. Since
			$$ \|\ov{r}_2(hD_x)h^4D_y^4u_{h,3}^{(1)}\|_{L^2}\lesssim (h/\hbar)^4=O(h^{\frac{2(\nu+2)}{\nu+1}})=o(h^2\delta_h).
			$$
			Applying Lemma \ref{LL}, we get $\|u_{h,3}^{(3)}\|_{L^2}=o(1)$, hence $\|u_h^{(3)}\|_{L^2}=o(1)$. This implies finally that $\|\psi_h^{(3)}\|_{L^2}=o(1)$ and this is a contradiction.
			
		\end{proof}
		
		To prove upper bounds in Theorem \ref{t:main_theorem} and Theorem \ref{t:sharp_in_1}, we cannot directly apply Lemma \ref{AL} and Lemma \ref{Holder}, since the self-adjoint perturbation $\ov{r}_2(hD_x)h^4D_y^4u_{h,3}^{(j)}$ cannot be viewed as a remainder of size $o_{L^2}(h^2\delta_h)$, as $\delta_h=o(1)$. 
		
		\subsubsection{1D resolvent estimate for the H\"older damping}
		
		To simplify the presentation, we only prove the upper bound in Theorem \ref{t:sharp_in_1} by applying Lemma \ref{Holder}, since the same argument applies to prove the upper bound in Theorem \ref{t:main_theorem}, by applying Lemma \ref{AL} instead. Recall the definition $u_{h,3}$ in Proposition \ref{prop:reduction}, and without loss of generality, we normalize the constant $M=1$ in Proposition \ref{prop:reduction}. Recall that $u_{h,3}$ satisfies
		the equation
		$$ (-h^2\Delta-1+ihb_2(y)+\ov{r}_2(hD_x)h^4D_y^4\widetilde{\chi}_0(\hbar D_y) )u_{h,3}=g_h=o_{L^2}(h^{2+\delta}),
		$$ 
		with $\delta=\log_h\delta_h^{(2)}=\frac{1}{\nu+2}$, $\hbar=h^{\frac{1+\delta}{2}}$.
		
		Set $w_{h,E}(y)=\mathcal{F}_xu_{h,3}(n,y)$, where $\mathcal{F}_x$ is the Fourier transform in $x$ and $E=1-h^2n^2$.  The equation for $u_{h,3}$ is transformed to the following one:
		\begin{align}\label{1Dreducedeq} 
			-h^2\partial_y^2w_{h,E}-Ew_{h,E}+a_{h,E}h^4\partial_y^4\widetilde{\chi}_0(\hbar D_y)w_{h,E}+ihb(y)w_{h,E}=r_{h,E},
		\end{align}
		where $a_{h,E}=\ov{r}_2(hn)\in\R$ and $r_{h,E}=\mathcal{F}_xg_h(n,\cdot)$. 
		Our goal is to prove the following one-dimensional resolvent estimate:
		\begin{prop}\label{1Dresolvent} 
			Let $\delta=\frac{1}{\nu+2}$.
			There exists $h_0>0$ and $C_0>0$, such that for all $0<h<h_0$ and $E\in\R$, the solution $w_{h,E}$ of \eqref{1Dreducedeq} satisfies
			\begin{align}\label{1Dresolventfinal} 
				\|w_{h,E}\|_{L^2(\T)}\leq C_0h^{-2-\delta}\|r_{h,E}\|_{L^2}.
			\end{align}
		\end{prop}
		As a direct consequence of Proposition \ref{1Dresolvent} and the Plancherel theorem, we deduce that $u_{h,3}=o_{L^2}(1)$, which implies $\psi_h^{(3)}=o_{L^2}(1)$, a contradiction. Below we prove Proposition \ref{1Dresolvent}:
		\begin{proof}
			Since the perturbation part is self-adjoint, we still have the a priori estimate
			\begin{align}\label{aprioriperburbed} 
				\|b^{\frac{1}{2}}w_{h,E}\|_{L^2}\lesssim h^{-\frac{1}{2}}\|r_{h,E}\|_{L^2}^{\frac{1}{2}}\|w_{h,E}\|_{L^2}^{\frac{1}{2}}.
			\end{align}
			
			First, there exists $C_1>0$ such that for all $0<h<1, E\in\R$, $$\|a_{h,E}h^2\partial_y^2\widetilde{\chi}_0(\hbar D_y)\|_{\mathcal{O}(L^2)}\lesssim h^{1-\delta}.
			$$
			Therefore, for $0<h\leq h_0\ll 1$, the operator $\mathcal{L}_{h,E}=1-a_{h,E}h^2\partial_y^2\widetilde{\chi}_0(\hbar D_y)$ is invertible and we denote by $\mathcal{L}_{h,E}^{-1}$ its inverse. Note that $\mathcal{L}_{h,E}^{-1}$ is also a  Fourier multiplier on $L^2(\T)$, self-adjoint, positive, and
			$$ \mathcal{L}_{h,E}^{-1}=\mathrm{Id}+a_{h,E}h^2\partial_y^2\widetilde{\chi}_0(\hbar D_y)+\mathcal{O}_{\mathcal{L}(L^2)}(h^{2(1-\delta)}).
			$$
			Applying $\mathcal{L}_{h,E}^{-1}$ to the equation of $w_{h,E}$, we obtain that
			\begin{align*}
				-h^2\partial_y^2w_{h,E}-E\mathcal{L}_{h,E}^{-1}w_{h,E}+ih\mathcal{L}_{h,E}^{-1}(b(y)w_{h,E})=\mathcal{L}_{h,E}^{-1}r_{h,E}.
			\end{align*} 
			By symbolic calculus and the fact that $\delta\leq \frac{1}{4}$ (since $\nu\geq 4$), we have
			\begin{align}\label{eq:whEnew}
				-h^2\partial_y^2w_{h,E}-E\mathcal{L}_{h,E}^{-1}w_{h,E}+ihb(y)w_{h,E}=\widetilde{r}_{h,E}:=\mathcal{L}_{h,E}^{-1}r_{h,E}-(\mathcal{L}_{h,E}^{-1}-\mathrm{Id})(ihb(y)w_{h,E}).
			\end{align}
			Here,
			\begin{align}\label{tilderhE}  
				\|\widetilde{r}_{h,E}\|_{L^2}\lesssim \|r_{h,E}\|_{L^2}+h^{2-\delta}\|b^{\frac{1}{2}}w_{h,E}\|_{L^2}\lesssim \|r_{h,E}\|_{L^2}+h^{\frac{3}{2}-\delta}\|r_{h,E}\|_{L^2}^{\frac{1}{2}}\|w_{h,E}\|_{L^2}^{\frac{1}{2}}.
			\end{align}
			Multiplying by $\ov{w}_{h,E}$, integrating and taking the real part, we get
			\begin{align}\label{apriorinew} 
				\big|\|h\partial_yw_{h,E}\|_{L^2}^2-E\|\mathcal{L}_{h,E}^{-\frac{1}{2}}w_{h,E}\|_{L^2}^2\big|\leq \|\widetilde{r}_{h,E}\|_{L^2}\|w_{h,E}\|_{L^2}.
			\end{align}
			If $E\geq K_0h^{1-\delta}$ and $w_{h,E}=\widetilde{\chi}_0(\hbar D_y)w_{h,E}$, we have
			\begin{align*} 
				\|w_{h,E}\|_{L^2}^2\sim \|\mathcal{L}_{h,E}^{-\frac{1}{2}}w_{h,E}\|^2 & \lesssim E^{-1}\|h\partial_yw_{h,E}\|_{L^2}^2+E^{-1}\|\widetilde{r}_{h,E}\|_{L^2}\|w_{h,E}\|_{L^2}\\[0.2cm]
				& \lesssim  K_0^{-1}\|w_{h,E}\|_{L^2}^2+K_0^{-1}h^{-(1-\delta)}\|\widetilde{r}_{h,E}\|_{L^2}\|w_{h,E}\|_{L^2}. 
			\end{align*}
			For $K_0$ large enough, we get
			\begin{align}\label{range1} 
				\|w_{h,E}\|_{L^2}\lesssim h^{-\frac{1-\delta}{2}}\|\widetilde{r}_{h,E}\|_{L^2}^{\frac{1}{2}}\|w_{h,E}\|_{L^2}^{\frac{1}{2}},\quad \text{ when }E\geq K_0h^{1-\delta}.
			\end{align}
			Plugging into \eqref{tilderhE} and using Young's inequality, we obtain \eqref{1Dresolventfinal} when $E\geq K_0h^{1-\delta}$.
			
			From now on we assume that $E<K_0h^{1-\delta}$. By \eqref{apriorinew},
			\begin{align}\label{apriorinew1} \|h\partial_yw_{h,E}\|_{L^2}\leq CE^{\frac{1}{2}}\|w_{h,E}\|_{L^2}+\|\widetilde{r}_{h,E}\|_{L^2}^{\frac{1}{2}}\|w_{h,E}\|_{L^2}^{\frac{1}{2}}.
			\end{align}
			Applying Lemma \ref{Holder} to the equation
			$$ -h^2\partial_y^2w_{h,E}-Ew_{h,E}+ihb(y)w_{h,E}=E(\mathcal{L}_{h,E}^{-1}-1)w_{h,E}+\widetilde{r}_{h,E},
			$$
			we get
			\begin{align*}
				\|w_{h,E}\|_{L^2}\lesssim h^{-2-\delta}\|\widetilde{r}_{h,E}\|_{L^2}+h^{-2-\delta}\|E(\mathcal{L}_{h,E}^{-1}-1)w_{h,E}\|_{L^2}.
			\end{align*}
			From the expansion of $\mathcal{L}_{h,E}^{-1}$ and \eqref{apriorinew1}, we deduce that
			$$ \|w_{h,E}\|_{L^2}\lesssim h^{-2-\delta}\|\widetilde{r}_{h,E}\|_{L^2}+h^{-2-\delta}\cdot Eh\hbar^{-1}(E^{\frac{1}{2}}\|w_{h,E}\|_{L^2}+\|\widetilde{r}_{h,E}\|_{L^2}^{\frac{1}{2}}\|w_{h,E}\|_{L^2}^{\frac{1}{2}}).
			$$
			Therefore, if $E\leq K_1 h^{1+\delta}$ for some $K_1$ large enough, we get 
			$$ \|w_{h,E}\|_{L^2}\lesssim h^{2-\delta}\|\widetilde{r}_{h,E}\|_{L^2}+h^{-\frac{1+\delta}{2}}\|\widetilde{r}_{h,E}\|_{L^2}^{\frac{1}{2}}\|w_{h,E}\|_{L^2}^{\frac{1}{2}}.
			$$
			By Young's inequality, we obtain \eqref{1Dresolventfinal} in this case.

			Finally, we assume that $K_0h^{1-\delta}\geq E>K_1h^{1+\delta}$. In this case, we are able to use the geometric control estimate. Indeed, denote by $\lambda=h^{-1}E^{\frac{1}{2}} (\gtrsim h^{-\frac{1-\delta}{2}})$, then
			$$ -\partial_y^2w_{h,E}-\lambda^2 w_{h,E}=Eh^{-2}(\mathcal{L}_{h,E}^{-1}-1)w_{h,E}-ih^{-1}b(y)w_{h,E}+h^{-2}\widetilde{r}_{h,E}.
			$$
			From the geometric control estimate (Lemma \ref{estimate:GCC}), we get
			\begin{align}\label{1Dresolventfinalrange} 
				\|w_{h,E}\|_{L^2}& \lesssim \|b^{\frac{1}{2}}w_{h,E}\|_{L^2}+\frac{1}{\lambda}\|ih^{-1}b(y)w_{h,E}+\lambda^2(\mathcal{L}_{h,E}^{-1}-\mathrm{Id}-a_{h,E}h^2\partial_y^2\widetilde{\chi}_0(\hbar D_y))w_{h,E} \|_{L^2}\notag \\[0.2cm]
				& \quad + \lambda^2\|h^2\partial_y^2\widetilde{\chi}_0(\hbar D_y)w_{h,E}\|_{H^{-1}}
				+\lambda^{-1}h^{-2}\|\widetilde{r}_{h,E}\|_{L^2}.
			\end{align}
			Note that (using \eqref{aprioriperburbed} ) $$\lambda^{-1}h^{-1}\|b^{\frac{1}{2}}w_{h,E}\|_{L^2}\lesssim h^{-\frac{1+\delta}{2}}\|b^{\frac{1}{2}}w_{h,E}\|_{L^2}\lesssim h^{-\frac{2+\delta}{2}}\|r_{h,E}\|_{L^2}^{\frac{1}{2}}\|w_{h,E}\|_{L^2}^{\frac{1}{2}},$$
			and
			$$\lambda^{-1}\|\lambda^2(\mathcal{L}_{h,E}^{-1}-1-a_{h,E}h^2\partial_y^2\widetilde{\chi}_0(\hbar D_y))w_{h,E}\|_{L^2}\lesssim \lambda h^{2-2\delta}\|w_{h,E}\|_{L^2}\lesssim h^{\frac{3-5\delta}{2}}\|w_{h,E}\|_{L^2}, $$ since $\lambda=h^{-1}E^{\frac{1}{2}}\leq h^{-\frac{1+\delta}{2}}$. For the last term, we have
			$$ \lambda^2\|h^2\partial_y^2\widetilde{\chi}_0(\hbar D_y)w_{h,E}\|_{H^{-1}}\lesssim \lambda^2h^2\hbar^{-1}\|w_{h,E}\|_{L^2}\lesssim h^{\frac{1-3\delta}{2}}\|w_{h,E}\|_{L^2}.
			$$
			Plugging these estimates into \eqref{1Dresolventfinalrange}, using \eqref{apriorinew} and Young's inequality, we obtain \eqref{1Dresolventfinal}.
			This completes the proof of Proposition \ref{1Dresolvent}.

		\end{proof}

	\section{Construction of subelliptic quasimodes}
		\label{s:quasimodes}
		
		This section is devoted to construct quasimodes in the subelliptic regime. In Section \ref{s:subelliptic_quasimodes} we construct quasimodes microlocalized outside the damping region, leading to the lower bounds in Theorems  \ref{t:EGCC} and Theorem \ref{t:main_theorem}. While in Section \ref{s:damped_quasimodes} we construct quasimodes microlocalized within the damping region. In particular, Theorem \ref{t:quasimodes_2} below will provide the lower bound in Theorem \ref{t:sharp_in_2}. 
		
		\subsection{Quasimodes microlocalized outside the damping region}
		\label{s:subelliptic_quasimodes}
		
		In this section, we prove the optimality of the inequality:
		\begin{align}\label{resolvent4_2} 
			\|u\|_{L^2}\leq \|b^{1/2}u\|_{L^2}+\frac{C}{h^2}\|(h^2\Delta_{G}+1)u\|_{L^2}
		\end{align}
		which yields the lower bound in case (1) assuming that $\operatorname{supp}(b) \cap \{x = 0 \}\neq \mathbb{T}_y$, and in case (2):
		\begin{equation}
			\label{e:lower_bound_from_se_quasimodes} 
			\|(h^2\Delta_G+1-ihb)^{-1}\|_{\mathcal{L}(L^2)}\geq O(h^{-2}).
		\end{equation}
		
		\begin{thm}
			\label{t:quasimodes_outside_damping}
			Assume that $\operatorname{supp}(b) \cap \{ x = 0 \} \neq \{0\}_x \times \mathbb{T}_y$. Then there exists a normalized sequence $(\psi_h) \subset L^2(M)$ such that
			$$
			(- h^2 \Delta_G + ihb - 1)\psi_h = O(h^2), \quad h \in (0,1],
			$$
			and $\Vert b^{1/2} \psi_h \Vert_{L^2(M)} = O(h^\infty)$.
		\end{thm}
		
		From Theorem \ref{t:quasimodes_outside_damping} we obtain inmediately the lower bound \eqref{e:lower_bound_from_se_quasimodes}.
		
		\noindent The idea of the proof of Theorem \ref{t:quasimodes_outside_damping} is based on the study of the time-dependent equation 
		\begin{equation}
			\label{e:time_dependent_equation_free}
			\big( ih^2 \partial_t - h^2 \Delta_G \big) u_h(t,x,y) = 0,
		\end{equation}
		for some initial data $u_h^0 = u_h(0) \in L^2(M)$ with $\Vert u^0_h \Vert_{L^2(M)} = 1$ and microlocal support away from the set $\omega = \{ b > 0 \}$, making rigorous the formal formula
		\begin{equation}
			\label{e:formal_formuka_free}
			\psi_h \sim  \int_\R e^{- \frac{i t \alpha_h }{h^2}} u_h(t) dt, \quad \alpha_h = 1,
		\end{equation}
		to find our quasimode $(\psi_h,\alpha_h )$.
		
		We first assume that $M= \R^2$, and let $b$ denote the periodic extension of the damping term into $\R^2$. This will be sufficient since our quasimode $\psi_h$ will be microlocalized into a compact set of $(-\pi,\pi)_x \times (-\pi,\pi)_y$. We will make use of a suitable finite subspace of $L^2(M)$ to construct a particular solution to \eqref{e:time_dependent_equation_free} in terms of an orthonormal basis of this subspace. We will also use this basis to conjugate the operator $-h^2\Delta_G$ on this subspace into diagonal form. To this aim, let us define, for $k \in \mathbb{N}_0$:
		\begin{equation}
			\label{e:elements}
			\Psi_k(x,\eta;h) := \frac{h \chi_1(h^2\eta)}{\Vert \chi_1 \Vert_{L^2(\R)}}  \varphi_{k}(\eta ,x), \quad (x,\eta) \in \R^2,
		\end{equation}
		where $\chi_1$ is defined by \eqref{e:cut_off_1} and $\varphi_{k}(\eta,x)$ is the $k$-Hermite function given by
		\begin{equation}
			\label{e:ansatz}
			\varphi_0(\eta,x) = \eta^{1/4} e^{- \frac{\eta x^2}{2}}, \quad \varphi_{k}(\eta,x) = e^{ \frac{ \eta x^2}{2}} \left( \frac{ \partial_x}{\sqrt{\eta}} \right)^{k} \left( \eta^{1/4} e^{- \eta x^2} \right), \quad k \in \mathbb{N}_0. 
		\end{equation}
		Assume that $ (0,y_0) \notin \operatorname{supp}(b)$. For the sake of simplicity, assume moreover that $y_0 \in (\alpha,\beta)$ where $(\alpha,\beta)$ is defined in case (2). 
		
		Now, for any $N \geq 3$ and $k \in \mathbb{N}_0$, we consider the values $\mu_{k} = \mu_k(\eta) = \lambda_k + O(\eta^{-1})$ (depending also on $N$) which will be defined by \eqref{e:complete_eigenvalues} below, consisting of quasi-eigenvalues of the truncation of the operator $\eta^{-1}( -\partial_x^2 + \eta^2V(x))$ when $V(x)$ is expanded by Taylor up to finite order $N$ near $x = 0$, and where $\lambda_k = 2k+1$ are the eigenvalues of $-\partial_x^2 + x^2$. We make the following ansatz for a  particular solution $u_h(t,x,y)$  to \eqref{e:time_dependent_equation_free}:
		\begin{align}
			\label{e:solution_time_dependent_free}
			u_h(t,x,y)  & =  \sum_{ k \in \mathbb{N}_0 } \int_\R  \overline{\mathfrak{u}}_{0,k}(\eta) \Psi_{k}(x,\eta;h)  e^{ i( y - y_0 +  \mu_0 t) \eta} d\eta
		\end{align}
		for certain unknowns  $\overline{\mathfrak{u}}_{0,k}(\eta) \in \mathbb{C}$, and with inital data
		\begin{align}
			\label{e:initial_data_free}
			u_h(0,x,y) & =  \sum_{k \in \mathbb{N}_0}  \int_\R  \overline{\mathfrak{u}}_{0,k}(\eta) \Psi_{k}(x,\eta;h) e^{ i( y - y_0) \eta} d\eta.
		\end{align}
		Set also, for $k \in \mathbb{N}_0$,
		$$
		\Phi_{k}(t,x,y;h) := \int_\R \Psi_k(t,x,\eta;h) e^{i (y - y_0 + \mu_0 t) \eta} d\eta.
		$$
		Recall that $\{ \varphi_{k}(\eta,x), \lambda_k \eta \}$ is an orthornomal basis of eigenfunctions for the harmonic oscillator $-\partial_x^2 + \eta^2 x^2$ (for $\eta >0$)  on $L^2(\R_x)$. This implies that
		\begin{equation}
			\label{e:self-adjoint_equation}
			(i h^2 \partial_t - h^2 \Delta_{G_0} )  \int_\R \Psi_{k}(x,\eta;h) e^{ i( y - y_0 +  \lambda_k t) \eta} d\eta = 0, \quad k \in \mathbb{N}_0.
		\end{equation}
		Notice moreover that $\{ \Psi_k(x,\eta;h) \}$ is an orthogonal set in $L^2(\R_x \times \R_\eta)$ of normalized functions. Similarly, notice that the set $\{ \Phi_{k}(t, \cdot,\cdot;h) \}_{k \in \mathbb{N}_0}$ is  orthogonal in $L^2(\R_x \times \R_y)$ and consisting of normalized functions. 
		
		In view of \eqref{e:self-adjoint_equation}, the study of the propagation equation \eqref{e:time_dependent_equation} relies on understanding the effect of the remainder term $V_0 (x)\partial_y^2$, where $V_0(x) := V(x) - x^2$, on $u_h(t,x,y)$. In this respect, let $N \geq 3$, we expand $V_0(x)$ by Taylor near $x=0$:
		\begin{align*}
			V_0(x) & = \sum_{l=3}^{N} \frac{V^{(l)}(0)}{l!} x^{l} + \frac{x^{N+1}}{N!} \int_0^1 (1-s)^{N} V^{(N+1)}(sx) ds \\[0.2cm]
			& =: \sum_{l=3}^{N} \upsilon_l x^l + R_N^V(x),
		\end{align*}
		and denote $\mathcal{V} := \sum_{l=3}^{N} \mathcal{V}_l$, where $\mathcal{V}_l = \upsilon_lx^l \partial_y^2$, and $\mathcal{R}_N^V := R_N^V(x) \partial_y^2$.
		
		The next step consists in conjugating equation \eqref{e:time_dependent_equation_free} into normal form up to order $N$ to obtain a particular (quasi-)solution in the form \eqref{e:solution_time_dependent_free}. To this aim, we truncate the operator $V_0(x)\partial_y^2$ into $\mathcal{V}$ and we look at the evolution equation:
		\begin{equation}
			\label{e:modified_equation_free}
			ih^2 \partial_t \Phi_{0}(t,x,y;h) + \mathcal{U}_h \left(  - h^2 \Delta_{G_0} - h^2 \Pi_N^h \mathcal{V} \Pi_N^h  \right) \mathcal{U}_h^* \Phi_{0}(t,x,y;h) = O(h^{N-1}),
		\end{equation}
		where $\Pi_N^h = \Pi_N^h(D_y)$ is the projection operator onto the subspace generated by $\Psi_k$, that is:
		\begin{equation}
			\label{e:truncation}
			\Pi_N^h(\eta) \phi(x,\eta) = \sum_{k=0}^{3N} \langle \phi, \Psi_k(\cdot, \cdot ;h) \rangle_{L^2(\R_x \times \R_\eta)} \Psi_k(\eta,x,h),
			\end{equation}
	where $N \geq 3$ will be chosen later on, and $\mathcal{U}_h = \Id + \mathcal{O}(h)$ is a unitary operator that commutes with $\partial_t$ and diagonalizes the operator $- h^2 \Delta_{G_0} - h^2 \Pi_N^h \mathcal{V} \Pi_n^h$ (see Section \ref{s:normal_form_reduction} below). Then $u_h(t,x,y) := \mathcal{U}_h^* \Phi_{0}(t,x,y;h)$ solves the equation
		\begin{equation}
			\label{e:truncated_equation_final}
			ih^2 \partial_t u_h(t,x,y) +  \left(  - h^2 \Delta_{G_0} - h^2 \Pi_N^h \mathcal{V} \Pi_N^h \right) u_h(t,x,y) = O(h^{N-1}),
		\end{equation}
		and this solution will be sufficient to construct the desired quasimode.
		
		\subsubsection{Normal form reduction} 
		\label{s:normal_form_reduction}
		
		We now explain how to put the operator  $ - h^2 \Delta_{G_0} - h^2 \Pi_N^h \mathcal{V} \Pi_N^h$ into normal form. First, we study how the operator $ -h^2\Pi_N^h \mathcal{V} \Pi_N^h$ acts on each $\Phi_{k}$. To this aim, notice that 
		$$
		\partial_y^2 \, e^{i(y-y_0+ \mu_0 t)\eta} = - \eta^2 \, e^{i(y-y_0+\mu_0 t) \eta}.
		$$
		Hence, computing $-h^2\Pi_N^h \mathcal{V} \Pi_N^h \, \Phi_{k}(x,y,t;h)$ relies on calculating, for $3 \leq l \leq N$,
		\begin{equation}
			\label{e:monomial_x_action}
			x^{l} \eta^2 \Psi_k(x,\eta;h) =  \frac{ h   \eta ( x \sqrt{\eta})^l \chi_1(h^2\eta) \varphi_k(\eta,x)}{\eta^{\frac{l}{2}-1} \Vert \chi_1 \Vert_{L^2}}, \quad 0 \leq k \leq 3N.
		\end{equation}
		By elementary properties of Hermite polynomials, we see that
		\begin{align}
			\label{e:action_hermite}
			(x\sqrt{\eta})^l \Psi_{k_1}(x,\eta;h) = \sum_{\vert k_1 - k_2 \vert \leq l} \mathbf{a}^l_{ k_1 k_2} \Psi_{k_2}(x,\eta;h),
		\end{align}
		for some $\mathbf{a}^l_{k_1 k_2} \in \mathbb{Q}$. This shows, in particular, that the matrix $\mathbf{A}_l = (\mathbf{a}^l_{k_1k_2})$ is band diagonal with bandwidth equal to $l$.

		\begin{lemma}
			\label{l:normal_form}
			For every $N \geq 3$ there exists a unitary operator $\mathcal{U}_h = \mathcal{U}_h(D_y) \in \mathcal{L}(L^2(\R_x))$ such that $\Id - \mathcal{U}_h = \mathcal{O}_{\mathcal{L}(L^2)}(h)$ is of finite rank $3N + 1$, and 
			\begin{equation}
				\label{e:normal_form_operators}
			 \mathcal{U}_h \big( - h^2 \Delta_{G_0} - h^2 \Pi_N^h \mathcal{V} \Pi_N^h \big) \mathcal{U}_h^* = - h^2 \Delta_{G_0} - h^2 \Pi_N^h \mathcal{D} \Pi_N^h + O(h^{N-1}),
			\end{equation}
			for some $\mathcal{D} = \mathcal{D}(D_y)$ satisfying $[ -h^2 \Delta_{G_0}, \mathcal{D}] = 0$.
		\end{lemma}
		
		\begin{proof}
	 Notice that the operator $-h^2 \Delta_{G_0}$ is already diagonal in the basis $\{ \Phi_k \}_{0\leq k \leq 3N}$, since it verifies $-h^2\Delta_{G_0} \Phi_{k} =  h^2 \lambda_{k} D_y \Phi_{k}$. Moreover, in view of \eqref{e:monomial_x_action} and \eqref{e:action_hermite}, the operator $\Pi_N^h \mathcal{V} \Pi_N^h$ is of finite rank and its matrix in the basis $\{ \Phi_k \}$ is band diagonal with bandwidth equal to $N$. As we will see, the proof then reduces to the construction of a normal form in finite dimension (see Lemma \ref{l:finite_dimension_averaging}). Indeed, to solve the conjugation equation \eqref{e:normal_form_operators}, we only need to compute the matrices of the operators $\Pi_N^h\mathcal{D} \Pi_N^h$ and $\mathcal{U}_h$ in the basis $\{\Phi_k\}$; that is, we aim at finding a unitary matrix $\mathfrak{U}_N = \mathfrak{U}_N(\eta) = (\mathfrak{u}_{k_1k_2}(\eta))$ and a diagonal matrix  $\mathbf{D}(\eta) = (\mathbf{d}_k(\eta))$ so that
			\begin{align*}
				\mathbf{d}_k(\eta) & := \langle \mathcal{D}(\eta) \varphi_k(\cdot,\eta), \varphi_k(\cdot,\eta) \rangle_{L^2(\R_x)}, \hspace*{1.6cm} 0 \leq k \leq 3N, \\[0.2cm]
				\mathfrak{u}_{k_1k_2}(\eta) & := \langle \mathcal{U}_h(\eta) \varphi_{k_2}(\eta,\cdot), \varphi_{k_1}(\eta,\cdot) \rangle_{L^2(\R_x)}, \hspace*{1.2cm} 0\leq k_1,k_2 \leq 3N.
			\end{align*}
			More precisely, we make the ansatz $\mathbf{d}_k(\eta) =  \sum_{l=3}^N \eta^{1-\frac{l}{2}} \lambda_k^l$ for certain unknowns $\lambda_k^l \in \R$ and, recalling the notation $\mathbf{A}_l = (\mathbf{a}^l_{k_1k_2} )$ for the matrix given by \eqref{e:action_hermite}, we claim that for solving equation \eqref{e:normal_form_operators} it is sufficient to solve the finite dimensional problem:
			\begin{equation}
			\label{e:matrix_equation_normal_form}
			\mathfrak{U}_N \left(  D_{\lambda_k} +  \sum_{l=3}^{N} \eta^{1-\frac{l}{2}} \upsilon_l \mathbf{A}_l \right) \mathfrak{U}_N^* =   \left( D_{\lambda_k} +  \sum_{l=3}^{N} \eta^{1-\frac{l}{2}} D_{\lambda^l_{k}}  \right) + O(\eta^{1- \frac{N+1}{2}}),
			\end{equation}
			where $D_{\lambda_k}=\mathrm{diag}(\lambda_0,\cdots,\lambda_{3N})$ and $D_{\lambda_k^l} = \mathrm{diag}(\lambda_0^l, \ldots, \lambda_{3N}^l)$ for $l=3, \ldots, N$. To show the claim, by writing
			$$
			h^2 \eta^{2- \frac{N+1}{2}} = h^{N-1} (h^2 \eta)^{2- \frac{N+1}{2}},
			$$
		we observe, in view of \eqref{e:monomial_x_action}, that 
			$$
			h^2 \eta^{2 - \frac{N+1}{2}} \left( h \chi_1(h^2 \eta) \varphi_k(\eta,x) \right) = h^{N-1} (h^2 \eta)^{2- \frac{N+1}{2}} \left( h \chi_1(h^2 \eta) \varphi_k(\eta,x) \right) = O(h^{N-1})
			$$ 
			in $L^2(\R_x \times \R_\eta)$, which justifies the remainder in \eqref{e:normal_form_operators} from the remainder in \eqref{e:matrix_equation_normal_form} (recall also \eqref{e:monomial_x_action} for the extra factor $\eta$). We finally extend $\mathcal{U}_h : \Pi_N^h L^2(\R_x \times \R_y) \oplus (\operatorname{Id} - \Pi_N^h)L^2(\R_x \times \R_y) \to L^2(\R_x \times \R_y)$ by the identity in the second component.
			
Equation \eqref{e:matrix_equation_normal_form} can be solved by Lemma \ref{l:finite_dimension_averaging} since $\lambda_{k+1} - \lambda_k = 2$ for every $k \in \mathbb{N}_0$, with $\epsilon = \eta^{-1/2}$.  Observe moreover, by the proof of Lemma \ref{l:finite_dimension_averaging}, that the unitary matrix $\mathfrak{U}_N$ is obtained as
			 $$
			 \mathfrak{U}_N = e^{i \epsilon^N F_N} \cdots e^{i \epsilon F_1},
			 $$ 
where, in particular, by \eqref{e:solution_cohomological_finite},
$$
(F_1)_{k_1 k_2} = \frac{1 - \delta_{k_1 k_2}}{\lambda_{k_1} - \lambda_{k_2}} \mathbf{a}^3_{k_1,k_2}.
$$
Expanding $\mathfrak{U}_N = \operatorname{Id} + i \epsilon F_1 + i \epsilon^2 F_2 + \frac{1}{2} (i \epsilon F_1)^2 +  \cdots$, using that the bandwidth of $\mathbf{A}_l$ is $l \geq 3$ and formula \eqref{e:remainder_taylor_normal_form}, which can be iterated, we see that the term of order $\epsilon^{l-2}$ in this expansion is a band diagonal matrix of bandwidth smaller than or equal to $3(l-2)$. This, in particular, implies that $\mathfrak{u}_{k_1,k_2}(\eta)$ is a polynomial in $\eta^{-1/2}$, and gives us the estimate
			\begin{align}
				\label{e:unitary_matrix_element}
				\vert \mathfrak{u}_{k_1,k_2}(\eta) \vert \leq C \eta^{- \frac{1}{2} \left \lceil \frac{ \vert k_1 - k_2 \vert}{3} \right \rceil}, \quad k_1, k_2 \leq 3N.
			\end{align}
Notice, in particular, that $\mathfrak{U}_N = \Id + O(\eta^{-\frac{1}{2}})$, which implies that $\mathcal{U}_h = \Id + \mathcal{O}_{\mathcal{L}(L^2)}(h)$.
		\end{proof}
		
		\begin{lemma}
			\label{l:third_eigenvalue}
			For any $0 \leq k \leq 3N$, we have $\lambda_k^3 = 0$.
		\end{lemma}
		
		\begin{proof}
			It is sufficient to observe, in view of the proof of Lemma \ref{l:finite_dimension_averaging} and the  parity properties of the Hermite functions, that, for $0 \leq k \leq 3N$,
			$$
			\lambda_k^3 =  \upsilon_3 \mathbf{a}_{kk}^3 =  \upsilon_3 \big \langle (x \sqrt{\eta})^3 \varphi_{k}(x,\eta), \varphi_{k}(x,\eta) \big \rangle_{L^2(\R_x)} = 0.
			$$
		\end{proof}
		We next define, for $\lambda_k^l$ given in the proof of Lemma \ref{l:normal_form},
		\begin{equation}
			\label{e:complete_eigenvalues}
			\mu_k = \mu_k(\eta) := \lambda_k + \sum_{l=4}^{N} \eta^{1-\frac{l}{2}} \lambda_k^l.
		\end{equation}
	
		\begin{equation}
			\label{e:diagonal_equation_part_2}
			\big( ih^2 \partial_t - h^2 \Delta_{G_0} - h^2 \Pi_N^h \mathcal{D}(D_y) \Pi_N^h \big)\Phi_{0}(t,x,y;h) = 0.
		\end{equation}
		
		Moreover, by the construction $\mathcal{U}_h$ in Lemma \ref{l:normal_form}, where recall that $\mathcal{U}_h = \operatorname{Id} + O(h)$, the function $u_h=\mathcal{U}_h^*\Phi_0$  takes the form
				\begin{align} 
			\label{e:solution_time_dependent'}
			u_h(t,x,y) & =  \sum_{k = 0}^{3N}   \int_\R \overline{\mathfrak{u}}_{0,k}(\eta)  \Psi_{k}(x,\eta;h) e^{i (y - y_0 + \mu_{0} t) \eta} d\eta \\[0.2cm]
			& =\Phi_{0}(t,x,y;h) + O_{L^2}(h). \notag 
		\end{align}
		
		\begin{proof}[Proof of Theorem \ref{t:quasimodes_outside_damping}] We first perform the analysis on $\R^2$, and recall the notation $\Delta_{G_0} := \partial_x^2 + x^2 \partial_y^2$. By the non-stationary phase analysis, we see that there exist $T>0$ small enough, and a sequence $(u_h)$ of solutions to \eqref{e:truncated_equation_final}, taking the form \eqref{e:solution_time_dependent'},  such that  
			$$ \|u_h \|_{L^2(\R^2)}=1,\quad \int_{-T}^T\int_{\R^2}b(x,y)|u_h(t,x)|^2dxdt=O(h^{\infty}).
			$$
			Indeed, since in \eqref{e:solution_time_dependent'},  $y_0 \in (\alpha,\beta)$, one verifies by integrating by parts that for $\omega=\{b > 0 \}$ and $T > 0$ satisfying $[y_0-T,y_0+T] \subset  (\alpha,\beta)$,
			\begin{equation} 
			\label{e:integral_T_T}
			\int_{-T}^T\|u_h(t,x,y)\|_{L^2(\omega)}^2dt=O(h^{\infty}).
			\end{equation}
			Setting $v_h(t)=\chi(t)u_h(t)$ for some fixed $\chi \in \mathcal{C}_c^\infty(-T,T)$, we have
			$$
		(ih^2 \partial_t - h^2 \Delta_{G_0} - h^2 \Pi_N^h \mathcal{V} \Pi_N^h )v_h=h^2f_h(t),
			$$
			where $f_h(t)=i\chi'(t)u_h(t)$. Taking the Fourier transform in $t$, we obtain
			\begin{equation}
			\label{e:principal_remainder_quasimode}
			 (-h^2\tau - h^2\Delta_{G_0} - h^2 \Pi_N^h \mathcal{V} \Pi_N^h)\widehat{v}_h(\tau)=h^2\widehat{f}_h(\tau).
		\end{equation}	
			We claim that, uniformly in $|\tau - \frac{1}{h^2}|\leq C_0$ and $h\ll 1$,
			$$ \|\widehat{v}_h(\tau)\|_{L^2}\sim h, \quad \|\widehat{f}_h(\tau)\|_{L^2}\sim h, \quad  \|b^{1/2}\widehat{v}_h(\tau)\|_{L^2}=O(h^{\infty}).
			$$
		 By  \eqref{e:integral_T_T}, Minkowski and Cauchy-Schwarz, we see that for any $\tau\in\R$,
			\begin{align*}\|b^{1/2}\widehat{v}_h(\tau)\|_{L^2(\R^2)} & =\Big\|\int_{-T}^Tb^{1/2}\chi(t)u_h(t)\mathrm{e}^{-it\tau}dt
			\Big\|_{L^2(\R^2)} \\[0.2cm]
			 & \leq C\int_{-T}^T\|u_h(t)\|_{L^2(\omega)}dt \\[0.2cm]
			 & \leq CT^{1/2}\|u_h\|_{L^2((-T,T)\times\omega)} \\[0.2cm]
			 & =O(h^{\infty}).
			\end{align*}
			Moreover, since $v_h=\chi(t)u_h(t)$ and $f_h(t)=i\chi'(t)u_h(t)$, we have explicitly that 
			\begin{equation} 
				\label{e:explicit_expression_v_h}
				\widehat{v}_h(\tau,x,y)=h\int_{\R} \frac{ \chi_1(h^2\eta)}{\Vert \chi_1 \Vert_{L^2}}e^{i(y-y_0)\eta}\varphi_0(x,\eta)\widehat{\chi}(\tau -\eta)d\eta \times (1 + O(h))\end{equation}
			and
			$$ \widehat{f}_h(\tau,x,y)=ih\int_{\R}\frac{ \chi_1(h^2\eta)}{\Vert \chi_1 \Vert_{L^2}}e^{i(y-y_0)\eta}\varphi_0(x,\eta)\widehat{\chi'}(\tau -\eta)d\eta \times (1 + O(h)).
			$$
			Once the claim is justified, we can choose $\tau_h$ as desired, that is, $\tau_h = 1/h^2$, and let $\psi_h := h^{-1}\widehat{v}_h(\tau_h)$, hence  we have $\|\psi_h\|_{L^2(\R^2)}\sim 1$,  and $\|\psi_h\|_{L^2(\omega)}=O(h^{\infty})$. Moreover, we also have
			\begin{align*}
			 (-h^2\tau_h - h^2\Delta_G )\psi_h & = - h^2\big( (\Id - \Pi_N^h) \mathcal{V} + \mathcal{R}_N^V \big)  \psi_h + O(h^2) + O(h^{N-1}) \\[0.2cm]
				& =: r_h(x,y) + O(h^2)+ O(h^{N-1}),
			\end{align*}
		where the $O(h^2)$ comes from \eqref{e:principal_remainder_quasimode} and the $O(h^{N-1})$ from \eqref{e:normal_form_operators}.
			Observe that, by \eqref{e:unitary_matrix_element}, \eqref{e:action_hermite}, and the explicit construction of $\psi_h$,
			\begin{align*}
				h^2(\Id - \Pi_N^h)   \mathcal{V}\psi_h & \\[0.2cm]
				&  \hspace*{-2cm}  = \upsilon_3 h \sum_{k = 0}^{3N}  \int_\R  \overline{\mathfrak{u}}_{0,k}(\eta)   (\Id - \Pi_N^h(\eta)) (h^2 \eta)^{2-\frac{3}{2}}  (x \sqrt{\eta})^3 \Psi_k(x,\eta;h) e^{i (y - y_0) \eta} \widehat{\chi}(\tau - \eta) d\eta +O(h^{N+1})\\[0.2cm]
				& \hspace*{-2cm} =  O(h^N),
			\end{align*}
where we observe again that the operator $(x \sqrt{\eta})^3$ has band diagonal matrix of bandwidth $3$, so the operator $(\operatorname{Id} - \Pi_N^h(\eta))$ acts only on basis elements $\Psi_k$ with $3N \geq k \geq 3N - 3$. For these elements, by \eqref{e:unitary_matrix_element},
$$
\vert \overline{\mathfrak{u}}_{0,k} \vert \leq C \eta^{- \frac{1}{2} (N-1)},
$$
which, writing it as $\eta^{-\frac{1}{2}(N-1)} = h^{N-1} (h^2 \eta)^{-\frac{1}{2}(N-1)}$ and together with the factor $h$ in front of the sum, gives us the size $O(h^N)$. By similar arguments we see that the higher order terms in the Taylor expansion in $\mathcal{V}$, up to $l = N$, produce remainders of smaller size, which justifies the remainder $O(h^{N+1})$. Moreover, by reasoning in the same way, we obtain:
			\begin{align*}
				h^2\mathcal{R}_N^V \psi_h & = \int_\R  R_N^V(x)  \eta^2 \Psi_0(x,\eta;h) e^{i (y - y_0) \eta} \widehat{\chi}(\tau - \eta) d\eta \times (1 + O(h)) \\[0.2cm]
				& = h^{N-1}  \int_\R  r_N^V(x)  (h^2\eta)^{2- \frac{N+1}{2}} (x \sqrt{\eta})^{N+1} \Psi_0(x,\eta;h) e^{i (y - y_0) \eta} \widehat{\chi}(\tau - \eta) d\eta \times (1 + O(h))\\[0.2cm]
				& =  O(h^{N-1}),
			\end{align*}
			where $r_N^V(x) = \frac{1}{N!} \int_0^1 (1-s)^{N} V^{(N+1)}(sx) ds$. Taking $h^2\tau_h= 1$ and $N = 3$, we prove that $\Vert r_h \Vert_{L^2} = O(h^2)$ and thus the desired optimality. 
			
			To justify the claim, we study the integral for a Schwartz function $g\in\mathcal{S}(\R)$:
			\begin{align*}
				F(\tau,x,y)=h\int_{\R}\frac{ \chi_1(h^2\eta)}{\Vert \chi_1 \Vert_{L^2}}e^{iy\eta}\varphi_0(x,\eta)\widehat{g}(\tau-\eta)d\eta. 
			\end{align*}
			By Plancherel,
			\begin{align*}
				\|F(\tau)\|_{L^2(\R^2)}^2 & = h^2\int_{\R^2}\frac{ |\chi_1(h^2\eta)|^2}{\Vert \chi_1 \Vert^2_{L^2}}|\varphi_0(x,\eta)\widehat{g}(\tau-\eta)|^2d\eta dx \\[0.2cm]
				& = \int_{\R} \frac{ |\chi_1(w)|^2}{\Vert \chi_1 \Vert^2_{L^2}} \left \vert \widehat{g}\left( \tau-\frac{w}{h^2}\right) \right \vert^2d w \\[0.2cm]
				& =\int_{\R} \frac{ |\chi_1(w)|^2}{\Vert \chi_1 \Vert^2_{L^2}} \left |\widehat{g}\left( \tau-\frac{1}{h^2}-\frac{w-1}{h^2}\right) \right |^2d w \\[0.2cm]
				& = h^2 \int_{|\zeta|\leq \frac{1}{h^2}} \frac{ |\chi_1\big(1+h^2\zeta\big)|^2}{\Vert \chi_1 \Vert^2_{L^2}}|\widehat{g}(\tau-h^{-2}-\zeta)|^2d\zeta.
			\end{align*}
			Thanks to  $|\tau-h^{-2}|\leq C_0$, we have $\|F(\tau)\|_{L^2(\R^2)}^2 \sim h^2\|\widehat{g}\|_{L^2(\R)}^2$. This verifies the claim. 
			
			Notice that, in the case $M = \mathbb{T}^2$, it is sufficient to define
			$$
			\psi_{\mathbb{T}^2,h} := \chi(x,y) \psi_h^{\pm},
			$$
			with $\chi \in \mathcal{C}_c^\infty(\R^2)$ such that $\operatorname{supp} \chi \subset \{ \vert x \vert \leq 2\epsilon \} \times (\alpha+2\epsilon,\beta-2\epsilon)$, for $\epsilon > 0$ small so that 
			$$
			\epsilon < \min  \{ \beta - y_0 - T, y_0 - T- \alpha \} ,
			$$ 
			and $\chi = 1$ on $\{ \vert x \vert \leq \epsilon \} \times  (\alpha+\epsilon,\beta-\epsilon)$, since we have that $\Vert (1 - \chi) \psi_h \Vert_{L^2(\R^2)} = O(h^\infty)$. Then $\psi_{\mathbb{T}^2,h}$ is the desired quasimode. This completes the proof of Theorem \ref{t:quasimodes_outside_damping}.
			
		\end{proof} 

		\subsection{Quasimodes microlocalized within the damping region}
		\label{s:damped_quasimodes}
		
		In this section we construct quasimodes in case (2) with assumption \eqref{e:particular_b}, and in case (3) with assumption \eqref{e:finite_type_condition}, generalizing the construction of the previous section by propagating solutions to \eqref{e:time_dependent_equation} towards the damping region.
		
		\begin{thm}
			\label{t:quasimodes}
			Assume that $b \in W^{1,\infty}(\mathbb{T}^2)$ satisfies \eqref{e:particular_b} with $\nu > 1$. Then, there exist $C_0 = C_0(\nu) > 0$, $C_1 = C_1(\nu) > 0$ such that, for any 
			$$
			0 \leq \beta_h \leq C_1 h \log \left( \frac{1}{h} \right), \quad h \in (0,1],
			$$
			there exists a quasimode $(\psi_h^\pm) \subset L^2(\mathbb{T}^2)$ for $P_h=-h^2\Delta_G+ihb(y)$ satisfying:
			$$
			\big( P_h - 1 - i h \beta_h \big) \psi^\pm_h = R(h),\quad \Vert \psi_h^\pm \Vert_{L^2} = 1,
			$$
			with
			$$
			r(h) := \Vert R(h) \Vert_{L^2} = C_0h^2 \exp \left( -\frac{C_0\beta_h}{h} \right), \quad  \text{as } h \to 0^+.
			$$
			In addition, let $\chi_0 \in \mathcal{C}_c^\infty(\R)$ be given by \eqref{e:cut_off_0} define, for $R \gg 1$, $\Upsilon_R(\eta) := \chi_0(\eta/R) - \chi_0(R\eta)$. Then:
			\begin{equation}
				\label{e:subelliptic_accumulation}
				\lim_{R \to \infty} \lim_{h \to 0^+} \left \Vert \Upsilon_R ( h^2 D_y ) \psi_h^\pm  \right \Vert_{L^2(\mathbb{T}^2)} = 1,
			\end{equation}
			and, if $\beta_h/h\rightarrow\infty$, 
			$$
			\vert \psi_h^\pm \vert^2 \rightharpoonup^\star \delta_{(0,\pm y_0)}, \quad \text{as } h \to 0^+.
			$$
		\end{thm}
		
		\begin{remark}
			The constant $C_1 > 0$ is taken so that:
			\begin{equation}
				\label{e:C_1_choice}
				h^3 \leq C_0h^2 \exp \left( - C_0 C_1 \log h^{-1} \right).
			\end{equation}
			The result could be extended to deal with any constant $C_1 > 0$, obtaining quasimodes of width $r(h) = O(h^N)$ for any $N > 0$. This would require to extend the normal form construction of Section \ref{s:normal_form_reduction} to average also the damping term. We prefer to avoid this generalization for the sake of simplicity and since condition \eqref{e:C_1_choice} will be sufficient for Theorem \ref{t:quasimodes_2} below to show the lower bound of Theorem \ref{t:sharp_in_2}.
		\end{remark}
		
		\begin{remark}
			The result remains valid in case \textnormal{(2)} with condition \eqref{e:B-H_condition}, with $\{\pm y_0 \}$ replaced by $\{a,b \}$.
		\end{remark}
		
		The next result gives the construction of quasimodes in case (3). In this case, the quasimodes are microlocalized near $y_0$.
		\begin{thm}
			\label{t:quasimodes_2}
			Assume that $b \in W^{1,\infty}(\mathbb{T}^2)$ satisfies \eqref{e:particular_b} with $\nu > 1$. Then, there exist $C_0 = C_0(\nu) > 0$, $C_1 = C_1(\nu) > 0$, and $h_0 = h_0(C_1) > 0$, such that, for any 
			\begin{align*}
				0 \leq \beta_h \leq C_1 \left( h \log \frac{1}{h} \right)^{\frac{\nu}{\nu+1}}, \quad h \in (0,h_0],
			\end{align*}
			there exists a quasimode $(\psi_h)$ for $P_h=-h^2\Delta_G+ihb(y)$ satisfying:
			$$
			\big( P_h - 1 - i h \beta_h \big) \psi_h = R(h),\quad \Vert \psi_h \Vert_{L^2} = 1,
			$$
			where 
			$$
			r(h) := \Vert R(h) \Vert_{L^2} =  C_0h^{2 -\frac{1}{\nu+1}} \exp \left( -\frac{C_0\beta_h^{(\nu+1)/\nu}}{h} \right), \quad  \text{as } h \to 0^+.
			$$
			In addition, let $\chi_0 \in \mathcal{C}_c^\infty(\R)$ be given by \eqref{e:cut_off_0}, and let, for $R \gg 0$, $\Upsilon_R(\eta) = \chi_0(\eta/R) - \chi_0(R\eta)$. Then:
			$$
			\lim_{R \to \infty} \lim_{h \to 0^+} \left \Vert \Upsilon_R (h^2 D_y) \psi_h  \right \Vert_{L^2(\mathbb{T}^2)} = 1,
			$$
			and:
			$$
			\vert \psi_h \vert^2 \rightharpoonup^\star \delta_{(0,y_0)}, \quad \text{as } h \to 0^+.
			$$
		\end{thm}
		
		\begin{remark}
			The constant $C_1(\nu)$ is chosen to satisfy \eqref{e:C_1_choice}. 
		\end{remark}
		
		From Theorem \ref{t:quasimodes_2}, we obtain in particular the following corollary, showing the lower bound in Theorem \ref{t:sharp_in_2}:
		
		\begin{cor} With the hypothesis of Theorem \ref{t:sharp_in_2}, there exist $h_0>0$ sufficiently small and $C_0 >0$ such that, for all $0<h\leq h_0$, we have:
			\begin{align*}
				C_0^{-1} h^{-2 + \frac{1}{\nu+1}}  \leq  \|(-h^2\Delta_G+ihb-1)^{-1}\|_{\mathcal{L}(L^2)}.
			\end{align*}
		\end{cor}
		
\begin{proof} 
Taking $(\psi_h)$ given by Theorem \ref{t:quasimodes_2} (recall that $\Vert \psi_h \Vert_{L^2} = 1$), with $\beta_h \equiv 0$, we see that
$$
(-h^2 \Delta_G + ih b - 1) \psi_h = C_0 h^{2-\frac{1}{\nu + 1}} w_h,
$$		
with $w_h = O_{\mathcal{L}(L^2)}(1)$. This implies that
$$
\Vert (-h^2 \Delta_G + ih b - 1)^{-1} \Vert \geq \Vert (-h^2 \Delta_G + ih b - 1)^{-1} w_h \Vert = C_0^{-1} h^{-2 + \frac{1}{ \nu + 1}}.
$$
\end{proof}

		\noindent The idea of the proof of Theorems \ref{t:quasimodes} and \ref{t:quasimodes_2} relies on the study of the time-dependent equation 
		\begin{equation}
			\label{e:time_dependent_equation}
			\big( ih^2 \partial_t - h^2 \Delta_G + i hb(y) \big) u_h(t,x,y) = 0,
		\end{equation}
		for some initial data $u_h^0 = u_h(0) \in L^2(\T^2)$ with $\Vert u_h \Vert_{L^2} = 1$, and microlocal support intersecting $\overline{\omega}$, making rigorous the formal formula
		\begin{equation}
			\label{e:formal_formuka}
			\psi_h \sim  \int_\R e^{- \frac{i t (  \alpha_h + i h \beta_h)}{h^2}} u_h(t) dt,
		\end{equation}
		to find our quasimode $(\psi_h,\alpha_h + i h \beta_h)$. We use here a similar approach to the one given in \cite{Ar20}. 
		
		\subsubsection{Adapted orthonormal basis}
		
		We first deal with the whole space $\R^2$ instead of $\mathbb{T}^2$. We use the orthonormal set $\{ \Psi_k \} \subset L^2(\R_x \times \R_\eta)$ defined by \eqref{e:elements}, which is adapted to the sub-elliptic regime $  \vert D_y \vert \sim h^{-2}$. We use this orthonormal set to put the operator $-h^2\Delta_G$ in diagonal normal form (given in Section \ref{s:normal_form_reduction}) so that we can decompose the solution to the time-dependent equation in this basis. We only deal with the case  $ D_y \sim h^{-2}$, giving the construction of $\psi_h^-$; the case for negative $D_y$ can be handled in a similar way to construct $\psi_h^+$. 
		
		We make the following ansatz for a solution $u_h(t,x,y)$  to \eqref{e:time_dependent_equation}:
		\begin{align}
			\label{e:solution_time_dependent}
			u_h(t,x,y)  & =   c(t;h) \sum_{k \in \mathbb{N}_0} \int_\R \overline{\mathfrak{u}}_{0,k}(\eta) \Psi_{k}(x,\eta;h)  e^{ i( y + y_0 +  \mu_0 t) \eta} d\eta,
		\end{align}
		and with inital data
		\begin{align}
			\label{e:initial_data}
			u_h(0,x,y) & =    \sum_{k\in\mathbb{N}_0} \int_\R  \overline{\mathfrak{u}}_{0,k}(\eta) \Psi_{k}(x,\eta;h) e^{ i( y + y_0) \eta} d\eta.
		\end{align}
		Notice that this state is centered at the initial point $-y_0 < 0$, and that the classic flow $-y_0 -    \lambda_0t$ enters the damping region $\omega = \{ b > 0 \}$ for $t > 0$. We use this initial data to construct $\psi_h^-$ via a microlocal version of \eqref{e:formal_formuka}. In the symmetric case ($\eta < 0$), one can mimic the constracution from the point $y_0 > 0$; in this case the classic flow $y_0 -\mathrm{sign}(\eta)\lambda_0 t  $ enters the damping region for $t > 0$, and one can use this initial data to construct $\psi_h^+$ using \eqref{e:formal_formuka}. This convention of signs ensures that the solution $u_h(t,x,y)$ decays as entering in the damping region for $t > 0$.

		\subsubsection{Time-dependent equation}
		
		In view of \eqref{e:self-adjoint_equation}, the study of the propagation equation \eqref{e:time_dependent_equation} relies on understanding the effect of the remainder term $V_0 (x)\partial_y^2$, where $V_0(x) := V(x) - x^2$, and the damping term $-ihb(y)$ on $u_h(t,x,y)$. In this respect, for any $k \in \mathbb{N}_0$, we expand $b(y)$ by Taylor near $y_t^0 := -y_0 -  \lambda_0 t$ up to first order:
		\begin{align*}
			b(y) & = b(y_t^0) + (y-y_t^0) \int_0^1 b'( (1-s)y^0_t + s y) ds   \\[0.2cm]
			& =: b(y^0_t) +  R_1(t,y-y^0_t).
		\end{align*}
		This allows us to write the damping function $b$ as a sum of operators: 
		$$
		b=  \mathcal{B}_0(t) + \mathcal{R}_1(t),
		$$ 
		where $\mathcal{B}_0(t) := b(y_t^0)$ and $\mathcal{R}_1(t) := R_1(t,y-y_t^0)$. In particular, for the unitary operator $\mathcal{U}_h$  constructed in Lemma \ref{l:normal_form}, we have $[\mathcal{U}_h , \mathcal{B}_0(t)] = 0$.
		We next focus on the study of equation \eqref{e:time_dependent_equation}. To facilitate the computations, instead of looking strictly at equation \eqref{e:time_dependent_equation}, we truncate the operators $\mathcal{B}_0(t)$ and $\mathcal{V}_l$, and we look at the evolution equation (in normal form):
		\begin{equation}
			\label{e:modified_equation}
			\left( ih^2 \partial_t  + \mathcal{U}_h \left(  - h^2 \Delta_{G_0} - h^2 \Pi_N^h \mathcal{V} \Pi_N^h + i h \Pi_{N}^h \mathcal{B}_0(t)  \Pi_N^h \right) \mathcal{U}_h^* \right) \big( c(t;h) \Phi_{0}(t,x,y;h) \big) = O(h^{N-1}),
		\end{equation}
		where $N \geq 3$ will be chosen later on, and $\mathcal{U}_h = \Id + O(h)$ puts the selfadjoint operator $- h^2 \Delta_{G_0} - h^2 \Pi_N^h \mathcal{V}$ in normal form (see Section \ref{s:normal_form_reduction}).
		
		Then $u_h(t,x,y) := c(t;h) \mathcal{U}_h^* \Phi_{0}(t,x,y;h)$ solves the equation
		\begin{equation}
			\label{e:original_equation}
			ih^2 \partial_t u_h(t,x,y) +  \left(  - h^2 \Delta_{G_0} - h^2 \Pi_N^h \mathcal{V} \Pi_N^h + i h \Pi_{N}^h \mathcal{B}_0(t) \Pi_N^h \right) u_h(t,x,y) = O(h^{N-1}).
		\end{equation}
		
		We start by showing the following propagation result for the solution $u_h(t,x,y)$ of \eqref{e:modified_equation} given by \eqref{e:solution_time_dependent}. Let us define the function
		\begin{align}
			\mathbf{b}_0(t) := \int_0^t b( -y_0 - s) ds, \quad t \in \R.
		\end{align}
		\begin{lemma}
			\label{p:propagation_result}
			Assume that $b$ satisfies \eqref{e:particular_b}. Let $N \geq 3$. Let $\rho> 0$ be defined by \eqref{e:particular_b}. Then for any $0 < T_0 < \rho$ there exists a unique solution $u_h(t,x,y)$ to \eqref{e:original_equation} of the form \eqref{e:solution_time_dependent} with coefficient $c(t;h)$ satisfying $c(t;h) = c(0;h)$ for $t \in [-T_0, 0]$ and for $0 \leq t \leq T_0$, $c(t;h) = e^{- \frac{\mathbf{b}_0(t)}{h} }$.
		\end{lemma}

		\begin{remark}
			Lemma \ref{p:propagation_result} remains valid with assumption \eqref{e:finite_type_condition}. We state Lemma \ref{p:propagation_result_2} below precising the minor changes needed in this case. 
		\end{remark}

		\begin{proof}  Using \eqref{e:diagonal_equation_part_2}, we deduce the differential equation for $c(t;h)$:
			\begin{align}
				\label{e:matrix_equation}
				i h^2 \dot{c}(t;h) & = - i h  b(y_t^0) c(t;h).
			\end{align}
			Then the claim holds by using that $c(0;h) = 1$ and integrating from $0$ to $t$.
			
		\end{proof}

		\begin{proof}[Proof of Theorem \ref{t:quasimodes}]
			We first perform the analysis on $\R^2$. This is essentially sufficient, since our quasimode will be $O(h^\infty)$ away from a compact set inside $(-\pi,\pi)_x \times (-\pi,\pi)_y$. 
		
		\noi
		$\bullet$ {\bf Step 1: Profile of the quasimodes.}
			
			Let $u_h = c(t;h) \mathcal{U}_h^*\Phi_0(t,x,y;h)$. Let $T_0 < \rho$. By \eqref{e:solution_time_dependent} and Lemma \ref{p:propagation_result}, we have, for every $t \in [-T_0,T_0]$,
			\begin{align}
				u_h(t;x,y) & = \sum_{k = 0}^{3N}   c(t;h) \int_\R \overline{\mathfrak{u}}_{0,k}(\eta)  \Psi_{k}(x,\eta;h) e^{i (y + y_0 + \mu_{0} t) \eta} d\eta \notag \\[0.2cm]
				\label{e:profile_time_dependent}
				& = h e^{-\frac{\mathbf{b}_0(t)}{h} } \int_\R \frac{ \chi_1(h^2 \eta)}{\Vert \chi_1 \Vert_{L^2}} e^{i (y + y_0 + \mu_0 t) \eta} \varphi_0(\eta,x) d\eta \times( 1 +  O(h)),
			\end{align}
		where recall that $\overline{\mathfrak{u}}_{0,0}(\eta) = 1 + O(\eta^{-\frac{1}{2}})$ and $\overline{\mathfrak{u}}_{0,k}(\eta) = O(\eta^{-\frac{1}{2}})$ for $k > 1$. Let us fix a bump function $g \in \mathcal{C}_c^\infty(\R)$ satisfying $\operatorname{supp} g \subset (-T_0, T_0)$ and $g(t) = 1$ for $t \in (-\frac{T_0}{2}, \frac{T_0}{2})$; and, for $t \in [-T_0, T_0]$, we define:
			\begin{align}
			\label{e:good_functions}
				v_h(t,x,y) := g(t) e^{  \frac{t\beta_h}{h}} u_h(t,x,y); \quad
				f_h(t,x,y) := g'(t) e^{  \frac{t\beta_h}{h}} u_h(t,x,y).
			\end{align}
			We also define weight functions:
			\begin{align*}
				\Lambda_h(t) := C_h g(t) e^{\frac{ t \beta_h - \mathbf{b}_0(t) }{h} };  \quad \Lambda_h^\dagger(t) := C_h g'(t) e^{\frac{ t \beta_h - \mathbf{b}_0(t) }{h} },
			\end{align*}
where $C_h$ will be chosen as a normalizing constant.

\begin{lemma}\label{Lambda_h}
Let $C_h$ be such that $\|\Lambda_h\|_{L^2(\R)}=1$. Assume that $\gamma_h:=\frac{\beta_h}{h}=O\big(\log\big(\frac{1}{h}\big)\big)$ as $h\rightarrow 0^+$, we have for $|\tau|\gg h^{-1},$ 
$$ |\widehat{\Lambda}_h(\tau)|\leq C|\tau|^{-2}.
$$
Moreover, if $\gamma_h\rightarrow+\infty$ as $h\rightarrow 0^+$, the normalized constant $C_h$ satisfies
\begin{align}\label{C_h} C_h^{-2}=\frac{1}{2\gamma_h}+o(\gamma_h^{-1}),\; h\rightarrow 0^+.
\end{align}
In particular, $\vert \Lambda_h(t) \vert^2 \rightharpoonup^\star c \, \delta_0$ for some constant $c>0$.
\end{lemma}
\begin{proof}
By definition of $\Lambda_h$ and $\mathbf{b}_0(t)$, we have
$$ \Lambda_h(t)=C_hg(t)\exp\Big(t\gamma_h-\frac{t^{\nu+1}}{(1+\nu)h}\mathbf{1}_{t\geq 0}\Big).
$$
Notice that 
\begin{equation}
	\label{e:estimate_on_the_left}
	\int_{\R}  g(t)^2 \exp \left( \frac{2 t \beta_h}{h} - \frac{2 t^{\nu+1}}{(1+ \nu)h} \mathbf{1}_{t \geq 0}(t) \right) dt \gtrsim \int_{-T_0}^0 \exp  \left( \frac{2t \beta_h}{h} \right) dt = \frac{h}{2\beta_h} \Big( 1- e^{-\frac{2T_0 \beta_h}{h}} \Big),
\end{equation}
provided that $\beta_h > 0$. Otherwise, if $\beta_h = 0$, then the left-hand-side of \eqref{e:estimate_on_the_left} is bounded from below by a positive constant. Then $C_h^2 \lesssim  1 + \beta_h h^{-1}\leq O(\log(h^{-1}))$. 
	Now for $|\tau|\gg h^{-1}$, integration by parts yields
\begin{align*} |\widehat{\Lambda}_h(\tau)|&  = C_h \left |\int_{\R} \frac{g(t)}{\gamma_h-t^{\nu}/h-i\tau}\frac{\partial}{\partial t}\exp\Big(t\gamma_h-\frac{t^{\nu+1}}{(\nu+1)h}\mathbf{1}_{t\geq 0}-it\tau\Big)dt\right | \notag \\[0.2cm]  & \lesssim\frac{C_h}{|\tau|}\int_{\R}(|g'(t)|+|g(t)|) \exp\Big(t\gamma_h-\frac{t^{\nu+1}}{(\nu+1)h}\mathbf{1}_{t\geq 0}\Big)dt\notag \\[0.2cm]
	 & \leq O(|\tau|^{-1}),
\end{align*}
since 
\begin{align*}
 C_h \int_0^\infty  \exp\Big(t\gamma_h-\frac{t^{\nu+1}}{(\nu+1)h}\Big)dt & \leq C\log(h^{-1}) h^{\frac{1}{\nu+1}} \int_0^\infty \frac{ s^{-\frac{\nu}{\nu+1}}}{\nu+1} \exp \left(- \frac{s^{\nu+1}}{\nu +1} \right)ds + O(1) \\[0.2cm]
  & = O(1).  
\end{align*}
 Applying integration by parts one more time, we get the desired bound $O(|\tau|^{-2})$.

Next we assume that $\gamma_h\rightarrow\infty$. We split the integral $\int_{\R}|g(t)\exp\Big(t\gamma_h-\frac{t^{\nu+1}}{(1+\nu)h}\mathbf{1}_{t\geq 0}\Big)|^2dt$ as $I_-+I_{+,1}+I_{+,2}$, where
\begin{align*}
I_{-} & :=\int_{-\infty}^0|g(t)|^2e^{2t\gamma_h}dt, \\[0.2cm]
I_{+,1} & :=\int_0^{(N_h\gamma_hh)^{\frac{1}{\gamma}}}|g(t)|^2\exp\Big(2t\gamma_h-\frac{2t^{\nu+1}}{(\nu+1)h} \Big)dt,\\[0.2cm]
I_{+,2} & :=\int_{(N_h\gamma_hh)^{\frac{1}{\gamma}}}^{\infty}|g(t)|^2\exp\Big(2t\gamma_h-\frac{2t^{\nu+1}}{(\nu+1)h} \Big)dt,
\end{align*}
and $N_h=\log\big(\frac{1}{h}\big).$ Note that for $0\leq t\leq (N_h\gamma_hh)^{\frac{1}{\nu}}$, $\exp\Big(2t\gamma_h-\frac{2t^{\nu+1}}{(\nu+1)h}\Big)\sim 1$, we have $I_{+,1}=O((N_h\gamma_hh)^{\frac{1}{\nu}})=o(\gamma_h^{-1})$. For $t>(N_h\gamma_hh)^{\frac{1}{\nu}}$, we have $\exp\Big(2t\gamma_h-\frac{2t^{\nu+1}}{(\nu+1)h}\Big)\leq e^{-2(N_h-1)\gamma_h t}$, then we deduce that $I_{+,2}=O((\gamma_hN_h)^{-1})=o(\gamma_h^{-1})$. Since $I_-=\frac{|g(0)|^2}{2\gamma_h}+o(\gamma_h^{-1})$, we obtain \eqref{C_h}. The same analysis yields $\vert \Lambda_h(t) \vert^2 \rightharpoonup^\star c \, \delta_0$ for some constant $c>0$. 
\end{proof}
		
Now we continue the proof of Theorem \ref{t:quasimodes}.	In view of \eqref{e:profile_time_dependent}, our candidate $\psi_h$ to be a quasimode now is:
			\begin{align}
			\label{e:candidate_quasimode_within}
				\psi_h(x,y) & := \frac{C_h}{h} \int_\R g(t) e^{-\frac{ i t}{h^2}( \alpha_h + i h \beta_h)}  u_h(t,x,y) dt  \\[0.2cm]
				& =  \int_\R \frac{ \chi_1(h^2 \eta)}{\Vert \chi_1 \Vert_{L^2}} \varphi_0(\eta,x) \widehat{\Lambda}_h\left( \frac{\alpha_h}{h^2} - \mu_0  \eta  \right) e^{i (y+y_0) \eta} d\eta \times ( 1+ O(h)), \notag
			\end{align}
			where the last equality holds due to Lemma \ref{p:propagation_result}. We next consider the following identity splitting the action of the damped-wave operator on the time dependent solution $u_h$:
			\begin{equation}
			\label{e:expansion}
			\begin{array}{rl}
			(-h^2\Delta_G + i h b) \, u_h(t) & = \displaystyle \big( -h^2\Delta_{G_0} - h^2\Pi_N^h \mathcal{V} \Pi_N^h + ih \Pi_N^h \mathcal{B}_0(t) \Pi_N^h \big) u_h(t) \\[0.3cm]
			& \displaystyle \quad - \big(  h^2(\operatorname{Id} - \Pi_N^h)\mathcal{V} - ih (\operatorname{Id} - \Pi_N^h)\mathcal{B}_0(t) + h^2\mathcal{R}_N^V - ih \mathcal{R}_1(t) \big) u_h(t).
			\end{array}
			 \end{equation}
With this identity, definition \eqref{e:candidate_quasimode_within} of $\psi_h$,	equation \eqref{e:original_equation}, and integration by parts in $t$ together with \eqref{e:good_functions}, we obtain:
			\begin{align*}
				(h^2 \Delta_G - i hb(y) + \alpha_h + i h \beta_h)\psi_h & = \frac{-i h C_h}{h} \int_\R \big(   (\operatorname{Id} - \Pi_N^h ) \mathcal{B}_0(t) + \mathcal{R}_1(t) \big)  e^{- \frac{i t \alpha_h}{h^2}} v_h(t,x,y) dt \\[0.2cm]
				& \quad + \frac{h^2 C_h}{h} \int_\R  \big( (\Id - \Pi_N^h) \mathcal{V} + \mathcal{R}_N^V \big) e^{- \frac{i t \alpha_h}{h^2}} v_h(t,x,y) dt \\[0.2cm]
				& \quad + \frac{-ih^2 C_h}{h} \int_\R  e^{-\frac{ i t \alpha_h}{h^2}  }   f_h(t,x,y) dt  + O(h^{N-1})\\[0.2cm]
				& =: r_h^1(x,y) + r_h^2(x,y) + r_h^3(x,y)+ O(h^{N-1}).
			\end{align*}
The remainder term $O(h^{N-1})$ comes from equation \eqref{e:original_equation}, and it will be justified once we prove that the quasimode $\psi_h$ is normalized for a suitable choice of the constant $C_h$. To prove that $\psi_h$ is the desired quasimode, we will first estimate the norm $\Vert \psi_h \Vert_{L^2}$, and then we will estimate the remaining terms $r_h^1$, $r_h^2$, and $r_h^3$. 
			We take:
			\begin{align}
				\label{e:choice_constants}
				\alpha_h = 1, \quad 0 \leq \beta_h \leq  C_1  h \log \frac{1}{h},
			\end{align}
			for $C_1  > 0$ chosen later, and take the constant $C_h$ to normalize  $\Vert \Lambda_h \Vert_{L^2(\R)} = 1$. 
		
\noi
$\bullet$ {\bf Step 2: Weak convergence of $|\psi_h|^2$. }
	
			First we estimate the norm of $\psi_h$. By Plancherel, we first have the upper bound:
			\begin{align*}
				\Vert \psi_h \Vert_{L^2(\R^2)}^2  & =  \int_\R \frac{\chi_1(h^2 \eta)^2}{\Vert \chi_1 \Vert^2_{L^2}} \Big|\widehat{\Lambda}_h\left( \frac{1}{h^2} - \mu_0 \eta \right)\Big|^2 d\eta  \\[0.2cm]
				 & = \frac{1}{h^2} \int_\R \frac{ \chi_1( \eta)^2}{\Vert \chi_1 \Vert^2_{L^2}} \Big|\widehat{\Lambda}_h\left(  \frac{1 - \eta}{h^2} \right)\Big|^2 d\eta \times (1 +  o(1)) \\[0.2cm]
				 &  = \int_{1 + h^2 \tau  \geq 0} \frac{\chi_1( 1+ h^2 \tau )^2}{\Vert \chi_1 \Vert^2_{L^2}} |\widehat{\Lambda}_h(-\tau)|^2 d\tau \times (1 + o(1))  \\[0.2cm]
				 &  \lesssim \Vert \Lambda_h \Vert_{L^2(\R)}^2,
			\end{align*}
			where we have used that $\mu_0(\eta) = 1 + O(\eta^{-1})$. On the other hand, by \eqref{C_h}, we also have the lower bound:
			\begin{align*}
				\Vert \psi_h \Vert_{L^2(\R^2)}^2 & \gtrsim \int_{\frac{3}{4} \leq 1 + h^2 \tau \leq \frac{3}{2}} |\widehat{\Lambda}_h\left( -\tau \right)|^2 d\tau  \\[0.2cm]
				 & \gtrsim \int_{-1/ 4h^2}^{-1/2h^2} |\widehat{\Lambda}_h(\tau)|^2 d\tau \\[0.2cm]
				 & \sim \|\Lambda_h\|_{L^2(\R)}^2.
			\end{align*}
Therefore,
	$$ \|\psi_h\|_{L^2(\R^2)}^2\gtrsim \|\widehat{\Lambda}_h\|_{L^2(\R)}^2\sim 1.
	$$
			Thus we have shown that, modulo adjusting the constant $C_h$, $\Vert \psi_h \Vert_{L^2(\R^2)} = \Vert \Lambda_h \Vert_{L^2(\R)} = 1$. Moreover, assuming that $\beta_h/h\rightarrow\infty$, by \eqref{C_h}, we have $\vert \Lambda_h \vert^2 \rightharpoonup^\star c \, \delta_0$ for some constant $c>0$. Next, let $a \in \mathcal{C}_c^\infty(\R^2)$, denoting $a_{y_0}(x,y) := a(x,y-y_0)$ and $\Theta_h(x,\eta) := \widehat{\Lambda}_h \left(\frac{1}{h^2} - \eta \right) \chi_1(h^2\eta) \varphi_0(x,\eta)$, we 
			also have, modulo $o(1)$ error,
		
			\begin{align}
				\int_{\R^2} a(x,y) \vert \psi_h(x,y) \vert^2 dx dy  &  \equiv  \int_{\R^3} \mathcal{F}_y(a_{y_0})(x,\eta- \xi) \Theta_h (x,\eta) \ov{\Theta_h(x,\xi)}  d\xi d\eta dx \notag \\[0.2cm]
				 & \equiv \int_{\R^3} \mathcal{F}_y(a_{y_0})(x,\tau- \tau') \widehat{\Lambda}_h(-\tau) \ov{\widehat{\Lambda}_h(-\tau')}\frac{ \chi_1(1+h^2\tau)\chi_1(1+h^2\tau')}{\Vert \chi_1 \Vert^2_{L^2}}  \notag \\[0.2cm]
				 & \quad \quad  \times  \varphi_0(x,\tau+h^{-2})\varphi_0(x,\tau'+h^{-2}) d\tau d\tau' dx \notag   \\[0.2cm]
				 & \equiv c \int_{\R^2} (\mathcal{F}_y{a}_{y_0})(0,\tau-\tau')\widehat{\Lambda}_h(-\tau)\ov{\widehat{\Lambda}_h(-\tau')}\frac{ \chi_1(1+h^2\tau)\chi_1(1+h^2\tau')}{\Vert \chi_1 \Vert^2_{L^2}}\notag  \\[0.2cm]
				 & \quad \quad \times \frac{(\tau+h^{-2})^{1/4}(\tau'+h^{-2})^{1/4} }{(\tau+\tau'+2h^{-2} )^{1/2}} d\tau d\tau'.  \label{last} 
			\end{align}
		To treat the last integral, we first observe that, thanks to Lemma \ref{C_h}, we may restrict the integral on $|\tau|\leq Ch^{-1}$ and $|\tau'|\leq Ch^{-1}$. Then
		$$ \chi_1(1+h^2\tau)\chi_1(1+h^2\tau')\cdot \frac{(\tau+h^{-2})^{1/4}(\tau'+h^{-2})^{1/4} }{(\tau+\tau'+2h^{-2} )^{1/2}}=\chi_1(1)^2\cdot \frac{1}{\sqrt{2}}+o(1),
		$$
		uniformly for $|\tau|\leq Ch^{-1}$, as $h\rightarrow 0^+$. Therefore, modulo $o(1)$ error, we can replace the last integral \eqref{last} by
		$$ c'\int_{\R^2}(\mathcal{F}_ya_{y_0})(0,\tau-\tau')\widehat{\Lambda}_h(-\tau)\widehat{\Lambda}_h(\tau')d\tau d\tau',
		$$
		which equals to
		\begin{align*} c'\int_{\R^2}(\mathcal{F}_ya_{y_0})(0,-\widetilde{\tau})\widehat{\Lambda}_h(\widetilde{\tau}-\tau' )\widehat{\Lambda}_h(\tau')d\widetilde{\tau}  d\tau' & =c'\int_{\R^2}(\mathcal{F}_ya_{y_0})(0,-\widetilde{\tau}) ( \mathcal{F}_y \Lambda_h^2)(\widetilde{\tau})d\widetilde{\tau},
		\end{align*}		
		and the last term converges to $c'a_{y_0}(0) = c'a(-y_0)$ as $h\rightarrow 0^+$, thanks to the fact that $\Lambda_h^2 \rightharpoonup^\star \delta$. This proves that $|\psi_h^2 | \rightharpoonup^\star  \delta_{(0,-y_0)}$ as $h \to 0^+$ after adjustying the normalizing constant.

\noi
$\bullet$ {\bf Step 3: Control the remainders $r_h^1,r_h^2,r_h^3$.} 
		
			Repeating the previous argument with $\Lambda^\dagger_h$ instead of $\Lambda_h$, we have: 
			\begin{align*}
				\Vert r_h^3 \Vert_{L^2(\R^2)} & \sim h^2 \Vert \Lambda^\dagger_h \Vert_{L^2(\R)} \\[0.2cm]
				& \sim h^2 \left( C_h^2  \int_{ [-T_0,\frac{T_0}{2}] \cup [\frac{T_0}{2}, T_0]}  \exp \left( \frac{2 s \beta_h}{h} - \frac{2 s^{\nu+1}}{(1+ \nu)h} \mathbf{1}_{s \geq 0}(s) \right) ds \right)^{\frac{1}{2}} \\[0.2cm]
				& \lesssim h^2 \exp \left( -\frac{C_0\beta_h}{h} \right),
			\end{align*}
			for some positive constant $C_0$ depending only on $\nu$. 
			
			On the other hand, notice that  $(\operatorname{Id} - \Pi_N^h ) \mathcal{B}_0(t) u_h(t) = 0$. Moreover, using that $\lambda_0 = 1$, and $c(t;h)= e^{-\frac{\mathbf{b}_0(t)}{h}}$, Lemma \ref{l:third_eigenvalue}, and denoting $\widetilde{c}(t;h) := e^{\frac{t \beta_h}{h}} c(t;h)$,
			we deduce:
	\begin{align*}
				\mathcal{R}_1(t) v_h(t,x,y) & \\[0.2cm]
				& \hspace*{-2cm}  = \sum_{ k=0}^{3N} \widetilde{c}(t;h) g(t)R_1(t,y-y_t^0) \int_\R  \overline{\mathfrak{u}}_{0,k}(\eta)  e^{i(y + y_0 + \mu_{0} t) \eta} \Psi_{k}(x,\eta;h) d \eta \\[0.2cm]
				&  \hspace*{-2cm} = - \sum_{ k =0}^{3N} \widetilde{c}(t;h)  \int_\R \int_0^1  b'((1-s)y_t^0 + sy)\frac{ \partial_\eta}{i}  \left( \overline{\mathfrak{u}}_{0,k}(\eta) \Psi_{k}(x,\eta;h) e^{it \eta (\mu_0 - \lambda_0)} \right)  e^{i(y -y_t^0 ) \eta} ds  d\eta.
			\end{align*}
		Since $\mu_0(\eta) - \lambda_0 = \eta^{-1} \lambda_0^4 + O(\eta^{-3/2})$, using estimate \eqref{e:unitary_matrix_element} and the elementary properties of the Hermite functions, $\partial_\eta \varphi_k(x,\eta) = \eta^{-1} \sum_{\vert k - k' \vert \leq 2} \vartheta_{k k'} \varphi_{k'}(x,\eta)$ for some $\vartheta_{k k'} \in \mathbb{Q}$, we deduce that
		the expression  
		$$ \partial_{\eta}\left( \overline{\mathfrak{u}}_{0,k}(\eta) \Psi_{k}(x,\eta;h) e^{it \eta (\mu_0 - \lambda_0)} \right)
		$$
		generates a linear combination of similar expressions of the form
		$$ 
	 O(h^2)\widetilde{\mathfrak{u}}_{k,k'}(\eta)\cdot h\widetilde{\chi}_1(h^2\eta)\varphi_{k'}(x,\eta)e^{it\eta(\mu_0-\lambda_0)},
		$$
		where $\widetilde{\mathfrak{u}}_{k',k}(\eta)$ is also a polynomial in $\eta^{-1/2}$, satisfying the same estimate as 
 \eqref{e:unitary_matrix_element}, and $\widetilde{\chi}_1\in C_c^{\infty}(\R_+)$. By the same estimate as for $\psi_h$, we obtain
	\begin{align*}
	 \frac{-ihC_h}{h}\int_0^1ds\int_{\R}dt \cdot\widetilde{c}(t;h)e^{-\frac{it}{h^2}}g(t)\int_{\R}d\eta\cdot  b'((1-s)y_t^0+sy )\widetilde{\mathfrak{u}}_{k',k}(\eta)\cdot h\widetilde{\chi}_1(h^2\eta)\varphi_{k'}(x,\eta)e^{it \eta(y+y_0+\mu_0)\eta} & \\[0.2cm]
	 & \hspace*{-16.5cm} = O_{L^2}(h).
		\end{align*}
	This implies that $\|r_h^1\|_{L^2(\R^2)}=O(h^3)$.

	 Next, as we showed in the proof of Theorem \ref{t:quasimodes_outside_damping}, we also have:
			$$
			h^2(\Id - \Pi_N^h) \mathcal{V} \psi_h  = O_{L^2}(h^N); \quad h^2\mathcal{R}_N^V \psi_h = O_{L^2}(h^{N-1}).
			$$
			Therefore, taking $N = 4$ we conclude, choosing $C_1 > 0$ sufficiently small satisfying \eqref{e:C_1_choice}:
			\begin{align*}
				\Vert r_h^j \Vert_{L^2(\R^2)} & = O(h^{3}) \leq C_0h^2 \exp \left( -\frac{C_0\beta_h}{h} \right), \quad j = 1,2,
			\end{align*}
			for $0 \leq \beta_h \leq C_1 h \log \left( h^{-1} \right)$. 
			
			Notice that, in the case $M = \mathbb{T}^2$, it is sufficient to define
			$$
			\psi_{\mathbb{T}^2,h}^{\pm} := \chi(x,y) \psi_h^{\pm},
			$$
			with $\chi \in \mathcal{C}_c^\infty(\R^2)$ such that $\operatorname{supp} \chi \subset \{ \vert x \vert \leq 2\epsilon \} \times \{ \vert y \vert < y_0 + \rho + 2\epsilon \}$, for $\epsilon > 0$ small, and $\chi = 1$ on $\{ \vert x \vert \leq \epsilon \} \times \{ \vert y \vert < y_0 + \rho + \epsilon \}$, since we have that $\Vert (1 - \chi) \psi_h^{\pm} \Vert_{L^2(\R^2)} = O(h^\infty)$. Then $\psi_{\mathbb{T}^2,h}^{\pm}$ are the desired quasimodes. 
			
		Finally, to show \eqref{e:subelliptic_accumulation}, we observe that 
			$$
			\Upsilon_R \left( h^2  D_y  \right) \int_\R \Psi_k(x,\eta;h) e^{i (y + y_0 + \mu_0t)\eta} d\eta =  \int_\R  \Upsilon_R \left( h^2 \eta \right) \Psi_k(x,\eta;h) e^{i (y + y_0 + \mu_0t)\eta} d\eta,
			$$
			and that $(1 - \Upsilon_R (h^2 \eta)) \chi_1(h^2 \eta) = O(1/R)$. This implies that
			$$
			\Vert \Upsilon_R \left( h^2 D_y \right) \psi_h^{-} \Vert_{L^2(M)} = 1 + O(1/R),
			$$
			and the same holds for $\psi_h^+$. This concludes the proof.
		
		\end{proof}

		The proof of Theorem \ref{t:quasimodes_2} is similar. Assume now that $b$ satisfies \eqref{e:finite_type_condition}. In this case, we define
		$$
		\mathbf{b}_0(t) := \int_0^t b( y_0 + s) ds, \quad t \in [-T_0,T_0], \quad T_0 := \rho - \epsilon.
		$$
		Observe that
		$$
		\mathbf{b}_0(t) =  \frac{ \operatorname{sgn}(t) \vert t \vert^{\nu +1}}{\nu+1}, \quad -T_0 \leq t \leq T_0.
		$$
		We now have the following propagation result, which is completely analogous to Lemma \ref{p:propagation_result}:
		
		\begin{lemma}
			\label{p:propagation_result_2}
			Assume that $b$ satisfies \eqref{e:finite_type_condition}. Let $N \geq 3$. Let $\rho> 0$ be defined by \eqref{e:finite_type_condition}. There exists a unique solution $u_h(t,x,y)$ to \eqref{e:original_equation} of the form \eqref{e:solution_time_dependent} with coefficient $c(t;h) = e^{- \frac{\mathbf{b}_0(t)}{h} }$ for $t \in [-T_0,T_0]$.
		\end{lemma}

		\begin{proof}[Proof of Theorem \ref{t:quasimodes_2}]
			
			The proof is very similar to the one of Theorem \ref{t:quasimodes}. We first construct quasimodes on $\R^2$. By \eqref{e:solution_time_dependent} and Lemma \ref{p:propagation_result_2}, we have, for every $t \in [-T_0,T_0]$,
			$$
			u_h(t;x,y) = h e^{-\frac{\mathbf{b}_0(t)}{h} } \int_\R \frac{ \chi_1(h^2 \eta)}{\Vert \chi_1 \Vert_{L^2}} e^{i (y - y_0 + \mu_0 t) \eta} \varphi_0(\eta,x) d\eta + O(h).
			$$
			Define
			$$
			L_h := \left \lbrace \begin{array}{ll}
				\displaystyle \left( \frac{\beta_h}{h^{\frac{\nu}{1+\nu}}} \right)^{\frac{1}{\nu}}, & \text{if }  h^{\frac{\nu}{1+\nu}} < \beta_h \leq C_1 \left( h \log \frac{1}{h} \right)^{\frac{\nu}{\nu+1}}, \\[0.4cm]
				1, & \text{if } 0 \leq \beta_h \leq h^{\frac{\nu}{1+\nu}}.
			\end{array} \right.
			$$
			Take $h_0 = h_0(C_1) > 0$ such that $T_0 \geq 2 h^{\frac{1}{1+\nu}} L_h$ for $h \in (0, h_0]$. We fix a bump function $g \in \mathcal{C}_c^\infty(\R)$ satisfying 
			\begin{equation}
			\label{e:support_g}
			\operatorname{supp} g \subset \left( -1, 3 \right), \quad g(t) = 1 \; \text{ for } t \in \left( -\frac{1}{2}, 2 \right),
			\end{equation}
			and define: $ g_h(t) = g(t h^{-\frac{1}{1+\nu}}L^{-1}_h)$. Set now:
			\begin{align}
			\label{e:neww_functions}
				v_h(t,x,y) := g_h(t) e^{  \frac{t\beta_h}{h}} u_h(t,x,y); \quad
				f_h(t,x,y) := g_h'(t) e^{  \frac{t\beta_h}{h}} u_h(t,x,y).
			\end{align}
			We also define weight functions:
			\begin{align*}
				& \Lambda_h(t) := C_h g_h(t) e^{\frac{ t \beta_h - \mathbf{b}_0(t) }{h} }; \hspace*{1.4cm}   \Lambda_h^\dagger(t) := C_h g_h'(t) e^{\frac{ t \beta_h - \mathbf{b}_0(t) }{h} }.
			\end{align*}
			Our candidate $\psi_h$ to be a quasimode is defined by:
			\begin{align}
			\label{e:quasimode_definicion}
				\psi_h(x,y) & := \frac{C_h}{h} \int_\R g_h(t) e^{-\frac{ i t}{h^2}( \alpha_h + i h \beta_h)}  u_h(t,x,y) dt  \\[0.2cm]
				\label{e:with_remainder}
				& =  \int_\R \frac{ \chi_1(h^2 \eta)}{\Vert \chi_1 \Vert_{L^2}} \varphi_0(\eta,x) \widehat{\Lambda}_h\left( \frac{\alpha_h}{h^2} - \mu_0 \eta \right) e^{i (y-y_0) \eta} d\eta  \times ( 1 + O(h)).
			\end{align}
			where the last equality holds due to Lemma \ref{p:propagation_result_2}. We next use  definition \eqref{e:quasimode_definicion} of $\psi_h$, \eqref{e:expansion},  \eqref{e:original_equation}, and integration by parts in $t$ together with \eqref{e:neww_functions}, to obtain:
			\begin{align*}
				(h^2 \Delta_G - i hb(y) + \alpha_h + i h \beta_h)\psi_h & = \frac{-i h C_h}{h} \int_\R  \big(  (\operatorname{Id} - \Pi_N^h ) \mathcal{B}_0(t)+ \mathcal{R}_1(t) \big) e^{- \frac{i t \alpha_h}{h^2}} v_h(t,x,y) dt   \\[0.2cm]
				& \quad + \frac{h^2 C_h}{h} \int_\R \big( \Id - \Pi_N^h) \mathcal{V} +  \mathcal{R}_N^V \big)  e^{- \frac{i t \alpha_h}{h^2}} v_h(t,x,y) dt \\[0.2cm]
				& \quad + \frac{h^2 C_h}{h} \int_\R  e^{-\frac{ i t \alpha_h}{h^2}  }   f_h(t,x,y) dt   + O(h^{N-1})\\[0.2cm]
				& =: r_h^1(x,y) + r_h^2(x,y) + r_h^3(x,y)+  O(h^{N-1}).
			\end{align*}
			To prove the claim of the theorem, we will estimate the norm $\Vert \psi_h \Vert_{L^2}$ and prove that it is asymptotically normalized for a suitable choice of the constant $C_h$; then we will estimate the remaining terms $r_h^1$, $r_h^2$, and $r_h^3$. First of all we compute the norm of $\Vert \Lambda_h \Vert_{L^2}$:
			\begin{align*}
				\Vert \Lambda_h \Vert^2_{L^2(\R)} & = C_h^2 \int_{\R}  g_h(t)^2 \exp \left( \frac{2 t \beta_h}{h} - \frac{2 \operatorname{sgn}(t)  \vert t \vert^{\nu+1}}{(1+ \nu)h} \right) dt \\[0.3cm]
				& = C_h^2 h^{\frac{1}{1+\nu}} \int_{\R}  g\left( \frac{s}{L_h} \right)^2 \exp \left( \frac{2 s \beta_h}{h^{\frac{\nu}{1+\nu}}} - \frac{2 \operatorname{sgn}(s)  \vert s \vert^{\nu+1}}{1+ \nu} \right) ds.
			\end{align*}
			We next take:
			\begin{align}
				\label{e:choice_constants_2}
				\alpha_h = 1, \quad 0 \leq \beta_h \leq  C_1 \left( h \log \frac{1}{h} \right)^{\frac{\nu}{\nu+1}},
			\end{align}
			for $C_1 > 0$, and take the constant $C_h$ to normalize  $\Vert \Lambda_h \Vert_{L^2(\R)} = 1$. We obtain by \cite[Prop. 4.2]{Ar20},
			\begin{equation}
				\label{e:normalizing_constant}
				C_h^2 \lesssim h^{-\frac{1}{1+\nu}}\left( 1 + \beta_h h^{-\frac{\nu}{1+\nu}} \right) \exp \left( - \frac{C_0 \beta_h^{\frac{1+\nu}{\nu}}}{h} \right),
			\end{equation}
			for some $C_0 > 0$ depending only on $\nu$. For completeness, we include the proof here. Assume first that $ \beta_h \geq h^{\nu/(1+\nu)}$. Then the function $\exp \left( \frac{2s\beta_h}{h^{\frac{\nu}{1+\nu}}} - \frac{2\operatorname{sgn}(s) \vert s \vert^{\nu+1} }{1+\nu} \right)$ reaches its maximum (for $s > 0$) at $L_h$. Moreover,
			\begin{equation}
				\label{e:estimate_exotic_integral}
				\exp \left( \frac{2s\beta_h}{h^{\frac{\nu}{1+\nu}}} - \frac{2s^{\nu+1} }{1+\nu} \right) \geq \exp \left( \frac{2 s \nu}{1+\nu} \cdot \frac{\beta_h}{h^{\frac{\nu}{1+\nu}}} \right), \quad \text{for} \quad 0 \leq s \leq L_h.
			\end{equation}
			Then there exists $C_0 = C_0(\nu) > 0$ such that
			\begin{align*}
				\int_\R  g(s/L_h)^2 \exp \left( \frac{2s\beta_h}{h^{\frac{\nu}{1+\nu}}} - \frac{2\operatorname{sgn}(s) \vert s \vert^{\nu+1} }{1+\nu} \right) ds & \geq \int_0^{L_h} \exp \left( \frac{2 s \nu}{1+\nu} \cdot \frac{\beta_h}{h^{\frac{\nu}{1+\nu}}} \right) ds \\[0.2cm]
				& = \frac{h^{\frac{\nu}{1+\nu}}}{\beta_h} \int_0^{\frac{ \beta_h^{\frac{1+\nu}{\nu}}}{h}} \exp \left( \frac{2 s \nu}{1+\nu} \right) ds \\[0.2cm]
				& \geq \frac{C_0 h^{\frac{\nu}{1+\nu}}}{\beta_h}  \exp \left(  \frac{C_0 \beta_h^{\frac{1+\nu}{\nu}}}{h}\right).
			\end{align*}
			Otherwise, if $0 \leq \beta_h \leq h^{\nu/(1+\nu)}$, there exists $c_0 > 0$ such that
			$$
			\int_\R  g(s/L_h)^2 \exp \left( \frac{2s\beta_h}{h^{\frac{\nu}{1+\nu}}} - \frac{2\operatorname{sgn}(s) \vert s \vert^{\nu+1} }{1+\nu} \right) ds \geq c_0 > 0.
			$$
			Then, using that $\Vert \Lambda_h \Vert_{L^2(\R)} = 1$, \eqref{e:normalizing_constant} holds true.
			
			Following the proof of Theorem \ref{t:quasimodes}, we have  $\Vert \psi_h \Vert_{L^2(\R^2)}  \sim \Vert \Lambda_h \Vert_{L^2(\R)} \sim1$, hence we can adjust the normalizaing constant $C_h$ to get $\Vert \psi_h \Vert_{L^2(\R^2)} = 1$. Similarly, the fact that $\vert \psi_h \vert^2 \rightharpoonup^\star \delta_{(0,y_0)}$ follows in the same way  as the one given in the proof of Theorem \ref{t:quasimodes}, provided that $\vert \Lambda_h(t) \vert^2 \rightharpoonup^\star c \, \delta_0$ for some $c > 0$.
			
Moreover, recall that  that by \eqref{e:support_g}  and the definition of $g_h$, $$
\operatorname{supp}\Lambda_h^\dagger \subset \Big( -L_h h^{\frac{1}{1+\nu}} , - \frac{1}{2} L_h h^{\frac{1}{1+\nu}}  \Big) \cup \big( 2L_h h^{\frac{1}{1+\nu}} ,3 L_h h^{\frac{1}{1+\nu}} \big).
$$ 		
This means $\Lambda_h^\dagger$ (which controls the norm of $f_h$ and hence of $r_h^3$) is supported on the tails of $\Lambda_h$, where is expected to be very small. Indeed, assuming $\beta_h > h^{\frac{\nu}{1+\nu}}$, the function $\exp \left( \frac{2 s \beta_h}{h^{\frac{\nu}{1+\nu}}} + \frac{2 \vert s \vert^{\nu+1}}{(1+ \nu)} \right)$ reaches its minimum for $s < 0$ at $-L_h$, and satisfies
			$$
			\int_{-L_h \leq s \leq -\frac{L_h}{2}} \exp \left( \frac{2 s \beta_h}{h^{\frac{\nu}{1+\nu}}} + \frac{2  \vert s \vert^{\nu+1}}{(1+ \nu)} \right) ds \leq \exp \left( - \frac{C_{\nu} \beta_h^{\frac{1 + \nu}{\nu}}}{h} \right) \int_{-L_h \leq s \leq -L_h/2} ds,
			$$ 
			while $\exp \left( \frac{2 s \beta_h}{h^{\frac{\nu}{1+\nu}}} - \frac{2 \vert s \vert^{\nu+1}}{(1+ \nu)} \right)$ reaches its maximum for $s > 0$ at $L_h$, and satisfies
			$$
			\int_{2L_h \leq s \leq 3 L_h} \exp \left( \frac{2 s \beta_h}{h^{\frac{\nu}{1+\nu}}} - \frac{2 \vert s \vert^{\nu+1}}{(1+ \nu)} \right) ds \leq \exp \left( - \frac{c_{\nu} \beta_h^{\frac{1 + \nu}{\nu}}}{h} \right)  \int_{2L_h \leq s \leq 3 L_h}  ds.
			$$
			This implies, using \eqref{e:normalizing_constant}, that 
			\begin{align*}
			 \Vert  r_h^3 \Vert_{L^2(\R^2)} & \sim h^2\left(  \Vert \Lambda^\dagger_h \Vert_{L^2(\R)} + O(h) \right)  \\[0.2cm]
				& \lesssim h^{2- \frac{1}{1+\nu}} \left( \frac{ C_h^2 h^{\frac{1}{1+\nu}}}{L_h^2}  \int_{\left( -L_h,-\frac{L_h}{2} \right) \cup (2L_h,3L_h)} \exp \left( \frac{2 s \beta_h}{h^{\frac{\nu}{1+\nu}}} - \frac{2 \operatorname{sgn}(s) \vert s \vert^{\nu+1}}{(1+ \nu)} \right) ds \right)^{\frac{1}{2}}  \\[0.2cm]
				& \lesssim h^{2- \frac{1}{1+\nu}} \exp \left( -\frac{C_0\beta_h^{\frac{\nu+1}{\nu}}}{h} \right).
			\end{align*}

			On the other hand, the estimates for $\mathcal{R}_N^V \psi_h $, $(\Id - \Pi_h^N) \mathcal{V} \psi_h$, and $\mathcal{R}_1(t)v_h$, given in the proof of Theorem \ref{t:quasimodes} remain valid. Therefore, taking $N \geq 4$, we conclude that
			\begin{align*}
				\Vert r_h^j \Vert_{L^2(\R^2)} & \leq O(h^3) \leq C_0 h^{2- \frac{1}{1+\nu}} \exp \left( -\frac{C_0\beta_h^{\frac{\nu+1}{\nu}}}{h} \right), \quad j =1,2.
			\end{align*}
			The remaining part of the proof is completely analogous to the one given in the proof of Theorem \ref{t:quasimodes}. In particular, to deal with the case $ \mathbb{T}^2$, it is sufficient to multiply our quasimode $\psi_h$ defined on $\R^2$ by a cut-off function supported near $(x,y) = (0,y_0)$, since $\psi_h$ is $O(h^\infty)$ away from this point.
		\end{proof}
		
\section{Construction of quasimodes in the horizontal trapped regime}
\label{s:compact_quasimodes}

In this section, we finish the proof of Theorem \ref{t:sharp_in_1} by showing the lower bound of the resolvent estimate $\|(h^2\Delta_G+1-ihb(y))^{-1}\|_{\mathcal{L}(L^2)}$, with damping $b(y)$ given by \eqref{e:particular_b}. The idea is to take the saturated quasimodes for the Laplace operator in \cite{K19} and use the normal form to cook up the desired quasimodes for the Baouendi-Grushin operator. In order to control the errors appearing in the normal form analysis, rather than using the construction in \cite{K19} as a blackbox, we should keep track of the regularity (anisotropic) of the $\T^2$ quasimodes.
\subsection{Review of the construction of $\T^2$ quasimodes}

In this section, we review the construction in \cite{K19} and give some specific estimates for the $\T^2$ quasimodes. The original idea for the construction dates back to the appendix in \cite{AL14} by Nonnenmacher.

Recall that $b(y)=(|y|-|y_0|)_+^{\nu}$ near $|y|=|y_0|$, we consider the ansatz
$$ u_{k}(x,y):=\mathrm{e}^{ikx}v_{h_k}(y),
$$ 
with $0<h_k\lesssim \frac{1}{|k|}$ to be specified later. We fix $\delta=\frac{1}{\nu+2}$ in this section. As before, we will drop the dependence in $k$ and write simply $h=h_k, v_h=v_{h_k}$. Plugging into the equation $$(-h^2\Delta-1+ihb(y))u_h=O_{L^2}(h^{2+\delta}),$$ we would like $v_h$ to satisfy
\begin{align}\label{quasimodev_h}
	-h^2\partial_y^2v_h+ihb(y)v_h=(1-h^2k^2)v_h+O_{L^2}(h^{2+\delta}),\quad \|v_h\|_{L^2}\sim 1.
\end{align}
For a reason which will become clearer later, it is sufficient to solve the following eigenvalue problem:
\begin{align}\label{EigenPb} 
	-h^2\partial_y^2v_h+ihb(y)v_h=\lambda_h^2v_h
\end{align}
with $\lambda_h=Ch+O(h^{1+\delta})$. We consider the even eigenfunction $v_h(y)=v_h(|y|)$ with
\begin{align*}
	v_h(y)=\begin{cases} 
		& \!\!\!\!\!\!v_{h,l}(y),\; 0\leq y<y_0,\\
		& \!\!\!\!\!\!v_{h,r}(y),\; y_0\leq y<\pi
	\end{cases}
\end{align*}
where 
$ v_{h,l}(y)=\cos\big(\frac{\lambda_h}{h}y\big).
$ The right function $v_{h,r}$ should satisfy the equation
\begin{align}\label{eq:vhr} 
	-h^2\partial_y^2v_{h,r}+ih(y-y_0)_+^{\nu}v_{h,r}=\lambda_h^2v_{h,r},\; y_0<y<y_0+\rho.
\end{align}
In order to ensure $v_h\in H^2$, the function $v_{h,r}$ should satisfy the compatibility condition
\begin{align}\label{compatibility} 
	v_{h,r}(y_0)=\cos\big(\frac{\lambda_hy_0}{h}\big),\quad v'_{h,r}(y_0)=-\frac{\lambda_h}{h}\sin\big(\frac{\lambda_hy_0}{h}\big).
\end{align} 
Note that the mass in the damping region is very small compared with the total mass, and the amplitude of the transmitted mass should have the same size as the reflected mass, so we expect that $|v_{h,r}'(y_0)|\gg |v_{h,r}(y_0)|$. Since $\frac{\lambda_h}{h}$ is of size $O(1)$, the principal part of the argument $\frac{\lambda_h y_0}{h}$ must belong to $\pi\big(l+\frac{1}{2}\big), l\in\Z$. Therefore, we take the following ansatz for $\lambda_h$:
\begin{align}\label{lambdahansatz} 
	\lambda_h=\frac{\pi\big(l+\frac{1}{2}\big)h}{y_0}+O(h^{1+\delta}).
\end{align}

Next, for $\theta\in\C$, sufficiently close to $0$, we denote $F(y;\theta)$ the $H^1(\R_+)$ solution of the Neumann problem
\begin{align}\label{variational}
	\begin{cases}
		&\!\!\!\!-F''(y)+iy^{\nu}F-\theta F=0,\; y>0\\
		&\!\!\!\! F'(0)=1,
	\end{cases} 
\end{align} 
then $v_{h,r}(y)$ takes the form $h^{\delta}\alpha_hF\big(\frac{y-y_0}{h^{\delta}};\theta_h\big)$, with $\theta_h=h^{-\frac{2(\nu+1)}{\nu+2}}\lambda_h^2$, and a constant $O(1)=\alpha_h\in\C$ to be determined in order to match the compatibility condition \eqref{compatibility}.

Let $\mu_0$ be the lowest Neumann eigenvalue of the operator $-\partial_y^2+y^{\nu}$ on $L^2(\R_+)$, i.e. with vanishing derivative at the boundary $y=0$.
Let $F_0$ be the Dirichlet trace $F(y;0)|_{y=0}$. It was shown  that $\mu_0>0$ (Lemma 4.1 of \cite{K19}). Moreover, by using the implicit function theory (Lemma 4.2 of \cite{K19}), there exists a uniform constant $C_0>0$, such that for all $|\eta|\leq \frac{\mu_0}{2}$, the (unique) solution $F(y;\theta)$ of the Neumann problem \eqref{variational} satisfies
\begin{align}\label{nonvanishing} 
	\frac{1}{C_0}\leq |F(0;\theta)|\leq C_0.
\end{align}
In particular, $F_0=F(0;0)\neq 0$.  

Now we are ready to solve \eqref{eq:vhr} with the compatibility condition \eqref{compatibility}. In view of \eqref{lambdahansatz}, we make precise the ansatz of $\lambda_h$ as
$$ \lambda_h=\frac{\pi\big(l+\frac{1}{2}\big)h}{y_0}+\gamma_h h^{1+\delta},\quad O(1)=\gamma_h\in\C. 
$$
Plugging into \eqref{compatibility} with $$v_{h,r}(y)=h^{\delta}\alpha_hF\big(\frac{y-y_0}{h^{\delta}};\theta_h\big),$$
we require
\begin{align}\label{compatibility2}& \alpha_hF(0;\eta_h)=(-1)^{l+1}h^{-\delta}\sin(\gamma_hy_0h^{\delta}),\notag \\ 
	& \alpha_h=(-1)^{l+1}\big[\frac{\pi\big(l+\frac{1}{2}\big)}{y_0}+\gamma_hh^{\delta}\big]\cos(\gamma_hy_0h^{\delta}).  
\end{align}
Since $|\theta_h|\sim h^{\frac{2}{\nu+2}}$, by Taylor expansion of sin and cos, the leading term of $\gamma_h$ should be $\gamma_0:=\frac{\pi\big(l+\frac{1}{2}\big)F_0}{y_0^2}$ and the leading term for $\alpha_h$ should be $\alpha_0=(-1)^{l+1}\frac{\pi\big(l+\frac{1}{2}\big)}{y_0}$. Fix the number $l$, by using the implicit function theorem (or fixed-point argument), the solution $(\alpha_h,\gamma_h)$ to \eqref{compatibility2} exists for $0\leq h\ll 1$ and
$$ |(\alpha_h,\gamma_h)-(\alpha_0,\gamma_0)|\lesssim h^{\delta}.
$$
For the detailed argument, we refer to Lemma 4.3 and Lemma 4.4 of \cite{K19}.

Finally, we take $h=h_k=\frac{y_0}{\sqrt{k^2y_0^2+\pi^2(l+\frac{1}{2})^2 }}$, then $1-h_k^2k^2=\lambda_{h_k}^2+O(h_k^{2+\delta})$, hence the function 
\begin{align}\label{T2quasimodes} 
	u_h(x,y)=e^{ikx}v_{h}(y)=e^{ikx}\big[\cos\big(\frac{\lambda_h y}{h}\big)\mathbf{1}_{|y|\leq |y_0|}+h^{\delta}\alpha_hF\big(\frac{|y|-|y_0|}{h^{\delta}};\eta_h\big) \mathbf{1}_{\pi>|y|>|y_0|}\big]
\end{align}
is the desired $\T^2$ quasimode, satisfying
\begin{align}\label{quasimodeflat} -h^2\partial_x^2u_h-h^2\partial_y^2u_h-u_h+ihb(y)u_h=O_{L^2}(h^{2+\delta}).
\end{align}
Moreover, we have the estimates:
\begin{prop}\label{estimatesT2quasimodes} 
	Let $u_h$ is given by \eqref{T2quasimodes}. There exists a uniform constant $C_1>0$, such that
	\begin{itemize}
		\item[$\mathrm{(a)}$] $\frac{1}{C _1}\leq \|u_h\|_{L^2}\leq C_1$.
		\medskip
		
		\item[$\mathrm{(b)}$]
		$\|h^j\partial_y^ju_h\|_{L^2}\leq C_1h^{j(1-\delta)+\frac{\delta}{2}}$ and $\|h^j\partial_x^ju_h\|_{L^2}\leq C_1$, for $j=1,2$.
		\medskip
		
		\item[$\mathrm{(c)}$] $\|b'(y)\partial_yu_h\|_{L^2}\leq C_1h^{\delta}$.
	\end{itemize}
\end{prop}
We need a Lemma:
\begin{lemma}\label{AprioriF} 
	Assume that $\nu>4$.  Let $\mu_0$ be the least eigenvalue of the operator $\mathcal{A}_{\nu}=-\partial_y^2+y^{\nu}$ on $L^2(\R_+)$ with the Neumann boundary condition. Then there exists a uniform constant $C>0$, such that for all $|\theta|\leq \frac{\mu_0}{4}$, the solution $F(y;\theta)$ of \eqref{variational} satisfies
	$$\|y^{\nu}F'\|_{L^2(\R_+)}+ \|F\|_{H^2(\R_+)}\leq C.
	$$
\end{lemma}
\begin{proof}
	Take $G\in C_c^{\infty}([0,\infty))$ such that $G(0)=0$ and $G'(0)=1$. Consider $\widetilde{F}:=F-G$, then
	$$ -\widetilde{F}''+iy^{\nu}\widetilde{F}-\theta \widetilde{F}=W,
	$$
	with $W=G''-iy^{\nu}G+\theta G$. Multiplying by $\ov{\widetilde{F}}$ and doing the integration by part, we get
	\begin{align}\label{integrationbypartformula}
		\int_0^{\infty}|\widetilde{F}'(y)|^2dy+i\int_0^{\infty}y^{\nu}|\widetilde{F}(y)|^2dy-\theta\int_0^{\infty}|\widetilde{F}(y)|^2dy=\int_0^{\infty}W(y)\ov{\widetilde{F}}(y)dy.
	\end{align}
	Taking the real part and imaginary part, we get
	$$ \|\widetilde{F}'\|_{L^2(\R_+)}^2\leq |\Re\theta|\|\widetilde{F}\|_{L^2(\R_+)}^2+\|W\|_{L^2(\R_+)}\|\widetilde{F}\|_{L^2(\R_+)},
	$$
	and
	$$ \|y^{\frac{\nu}{2}}\widetilde{F}\|_{L^2(\R_+)}^2\leq |\Im \theta|\|\widetilde{F}\|_{L^2(\R_+)}^2+\|W\|_{L^2(\R_+)}\|\widetilde{F}\|_{L^2(\R_+)}.
	$$
	Adding two inequalities above, using $\mathcal{A}_{\nu}-\mu_0\geq 0$ and the fact that $|\theta|\leq \frac{\mu_0}{4}$, we get
	$$ \mu_0\|\widetilde{F}\|_{L^2(\R_+)}^2\leq \frac{\mu_0}{2}\|\widetilde{F}\|_{L^2(\R_+)}^2+2\|W\|_{L^2(\R_+)}\|\widetilde{F}\|_{L^2(\R_+)}.
	$$
	As $\|W\|_{L^2(\R_+)}$ is an absolute constant, This shows that 
	\begin{align}\label{tildeF1}
	\|\widetilde{F}\|_{L^2(\R_+)}^2+\|\partial_y\widetilde{F}\|_{L^2(\R_+)}^2+\|y^{\frac{\nu}{2}}\widetilde{F}\|_{L^2(\R_+)}^2\leq C.
	\end{align}
	and consequently,
	$$	\|F\|_{L^2(\R_+)}^2+\|\partial_yF\|_{L^2(\R_+)}^2+\|y^{\frac{\nu}{2}}F\|_{L^2(\R_+)}^2\le C.
	$$
	Set $F_1=\widetilde{F}'$, and $\widetilde{F}_1:=F_1-G_1$ for some $G_1\in C_c^{\infty}([0,\infty))$ satisfying  $G_1'(0)=F_1'(0)=\widetilde{F}''(0)=\theta \widetilde{F}(0)+W(0)$, we have
	$$ -\widetilde{F}_1''+iy^{\nu}\widetilde{F}_1-\theta \widetilde{F}_1=W_1,
	$$
	where $W_1=W'+G_1''-iy^{\nu}G_1+\theta G_1-i\nu y^{\nu-1}\widetilde{F}$. Now we replace $\widetilde{F}$ by $\widetilde{F}_1$ and $W$ by $\widetilde{W}_1$ in \eqref{integrationbypartformula} and
	do the same manipulation as in the last paragraph, we obtain the bound
	$$ \|\widetilde{F}_1\|_{L^2(\R_+)}^2+\|\partial_y\widetilde{F}_1\|_{L^2(\R_+)}^2+\|y^{\frac{\nu}{2}}\widetilde{F}_1\|_{L^2(\R_+)}^2\leq C+\Big|\int_0^{\infty}\nu y^{\nu-1}\widetilde{F}(y)\ov{\widetilde{F}_1}(y)dy \Big|.
	$$
	Note that $\widetilde{F}_1=\widetilde{F}'-G_1$, by doing the integration by part,  the last term on the right hand side is bounded by
	$$ C\|y^{\frac{\nu-1}{2}}\widetilde{F}\|_{L^2(\R_+)}^2+C\|\widetilde{F}\|_{L^2(\R_+)}\|y^{\nu-1}G_1\|_{L^2(\R_+)}
	$$
	which is uniformly bounded, thanks to \eqref{tildeF1}. In particular, we obtain that $\|F''\|_{L^2(\R_+)}$ is uniformly bounded. By the equation of $F$ and the definition of $W_1$, we deduce that $y^{\nu}F, W_1$ are also uniformly bounded in $L^2$.
Finally, to obtain the uniform bound for $y^{\nu}F_1'$, we repeat the analysis above by considering $F_2=\widetilde{F}_1'$ and taking the derivative of the equation for $\widetilde{F}_1$. The same manipulation yields the uniform boundedness of $\|F_2'\|_{L^2(\R_+)}$, and we get consequently the uniform boundedness of $\|y^{\nu}\widetilde{F}_1\|_{L^2(\R_+)}$ by equation $-\widetilde{F}_1''+iy^{\nu}\widetilde{F}_1-\theta\widetilde{F}_1=W_1$.	
	 The proof of Lemma \ref{AprioriF} is now complete.

\end{proof}

\begin{proof}[Proof of Proposition \ref{estimatesT2quasimodes}]
	The assertion (a) is trivial. For (b), since $|k|\sim \frac{1}{h}$, we have $\|h^j\partial_x^ju_h\|_{L^2}\lesssim 1$ for $j=1,2$. 
	Since
	$$ |\partial_y^ju_h|\lesssim 1+\mathbf{1}_{|y|>|y_0|}h^{-(j-1)\delta}|(\partial_y^jF)\big(\frac{|y|-|y_0|}{h^{\delta}};\theta_h\big)|,
	$$
	by Lemma \ref{AprioriF} and change of variables, we deduce that $\|h^j\partial_y^ju_h\|_{L^2}\lesssim h^{j-(j-\frac{1}{2})\delta}$. 
	Finally, since
	$ \partial_yu_h=F'\big(\frac{|y|-|y_0|}{h^{\delta}};\theta_h\big)
	$ on supp$(b')$, we have from change of variable and Lemma \ref{AprioriF} that
	$$ \|b'(y)\partial_yu_h\|_{L^2}\lesssim \big\|(|y|-|y_0|)_+^{\nu-1}F'\big(\frac{|y|-|y_0|}{h^{\delta}};\theta_h\big)\big\|_{L^2}\lesssim h^{\frac{\delta}{2}}\cdot h^{\delta(\nu-1)}=h^{\frac{\nu-1/2}{\nu+2}}\leq h^{\delta},
	$$
	since $\delta=\frac{1}{\nu+2}$ and  $\nu>4$. This completes the proof of Proposition \ref{estimatesT2quasimodes}.
	
\end{proof}

\subsection{Proof of the lower bound in Theorem \ref{t:sharp_in_1}}

To prove the lower bound, we need to inverse the normal form reduction in Proposition \ref{prop:reduction}. Fortunately, since the $\T^2$ quasimodes are essentially localized at $|D_y|\leq O(h^{-\frac{\delta}{2}})$, the error terms appearing in the conjugation formula are much easier to control.

Without loss of generality, we may assume that $M=\frac{1}{2}\int_{-\pi}^{\pi}V(x)dx=1$. Recall the operator $F_{h,0}:=\Op_h^{\w}(q_0)\chi(\hbar D_y)D_y^2$ constructed in the proof of Proposition \ref{prop:reduction}, with
$$ q_0(x,\xi)=\frac{\chi_1(\xi)}{2\xi}\int_{-\pi}^x(V(z)-1)dz.
$$
Let $u_h$ be the $\T^2$ quasimodes given explicitly by \eqref{T2quasimodes}. We define
$ \widetilde{u}_h:=e^{-ihF_{h,0}}u_h,
$ and $\|\widetilde{u}_h\|_{L^2}\sim 1$ since $e^{-ihF_{h,0}}$ is unitary.
It remains to show that $\widetilde{u}_h$ satisfies the equation
\begin{align}\label{claimfinal} (-h^2\Delta_G-1+ihb(y))\widetilde{u}_h=O_{L^2}(h^{2+\delta}).
\end{align}
By \eqref{normalform1}, 
\begin{align*}
	e^{ihF_{h,0}}(-h^2\Delta_G-1+ihb(y))\widetilde{u}_h& =e^{ihF_{h,0}}(-h^2\Delta_G-1+ihb(y))e^{-ihF_{h,0}}u_h\\[0.2cm]
	& =(h^2D_x^2+V(x)h^2D_y^2-1+ihb(y))u_h+ih[F_{h,0},h^2D_x^2]u_h\notag\\
	& \quad -\frac{h^2}{2}[F_{h,0},[F_{h,0},h^2D_x^2]]u_h
	+ih[F_{h,0},V(x)^2h^2D_y^2]u_h
	\\[0.2cm]
	& \quad -h^2[F_{h,0},b(y)]u_h\notag -\frac{ih^3}{2}[F_{h,0},[F_{h,0},b]]u_h+O_{L^2}(h^{2+\delta}),
\end{align*}
where we use the fact that $\delta>\frac{1}{4}$ since $\nu> 4$. Following the same computation as \eqref{error1}, \eqref{doublecommutator}, \eqref{error2}, \eqref{error3}, \eqref{Lh1}, and using the a priori estimate
$\|b^{\frac{1}{2}}u_h\|_{L^2}=O(h^{\frac{1+\delta}{2}})$ from \eqref{quasimodeflat},
we get 
\begin{align*}
	e^{ihF_{h,0}}(-h^2\Delta_G-1+ihb(y))\widetilde{u}_h=& (-h^2\Delta-1+ihb(y))u_h+iR_{1,h}u_h+R_{2,h}u_h
	\\& + O_{L^2}(h^{2+\delta}),
\end{align*} 
where $R_{1,h},R_{2,h}$ are given by \eqref{R1},\eqref{R2}.
By Proposition \ref{estimatesT2quasimodes}, the terms
$$ iR_{1,h}u_h,\quad R_{2,h}u_h
$$
are all bounded by $O(h^{2+\delta})$ in $L^2(\T^2)$. This verifies \eqref{claimfinal} and the proof of Theorem \ref{t:sharp_in_1} is now complete.

		\appendix
\section{Special symbol classes and quantizations}\label{AppendixA}
In this section, we will collect some special symbolic calculus needed in Section \ref{subellipticPropagation}. First, let us recall that a symbol $\mathbf{a}\in S^{\rho,\delta}(\R^d)$ is a smooth function on $\R_{z,\zeta}^{2d}$ such that
		$$ |\partial_z^{\alpha}\partial_{\zeta}^{\beta}\mathbf{a}(z,\zeta)|\leq C_{\alpha,\beta}(1+|\zeta|)^{\rho|\alpha|-\delta|\beta|},\;\forall \alpha,\beta\in \mathbf{N}^{d}.
		$$
		For a symbol $\mathbf{a}$, we define the Weyl quantization $\mathrm{Op}_{1,\R^d}^w(\mathbf{a})$ via the formula
		\begin{align}\label{QuantizationRd} \mathrm{Op}_{1,\R^d}^\w(\mathbf{a})f(z):=\frac{1}{(2\pi)^d}\iint_{\R^{2d}}\mathbf{a}\big(\frac{z+z'}{2},\zeta\big)\mathrm{e}^{i(z-z')\cdot\zeta}f(z')d\zeta dz',
		\end{align}
		for any Schwartz function $f\in\mathcal{S}(\R^d)$. When $z=(x,y)\in\R^{2m},\zeta=(\xi,\eta)\in\R^{2m}$, we denote by
		$ \Op_{1}^{\w,(x,\xi)}(\mathbf{a}),\Op_1^{\w,(y,\eta)}(\mathbf{a})
		$ the partial quantization of $\mathbf{a}$ with respect to $(x,\xi)$ and $(y,\eta)$ variables, respectively. Note that we can write
		$$ \mathrm{Op}_1^{\w}(\mathbf{a})=\mathrm{Op}_{1}^{\w,(y,\eta)}(\Op_{1}^{\w,(x,\xi)}(\mathbf{a})).
		$$
		This means that we quantize the $\mathcal{L}(L^2(\R_x^n))$-valued symbol $\mathrm{Op}_1^{\w,(x,\xi)}(\mathbf{a})=\mathbf{a}^{\w}(x,D_x;y,\eta)$ with domain $(y,\eta)\in\R^{2n}$.
		
		Recall the following version of the Calderon-Vaillancourt theorem for $S^{0,0}$ symbols:
		\begin{prop}[\cite{Hw87}, Theorem 2]\label{CaVai} 
			Let $\mathbf{a}\in S^{0,0}(\R^{d})$. Then
			$$ \|\Op_{1,\R^d}^\w(\mathbf{a})\|_{\mathcal{L}(L^2(\R^d))}\leq C_0\sum_{|\alpha|+|\beta|\leq 2d}\sup_{\R^{2d}}|\partial_z^{\alpha}\partial_{\zeta}^{\beta}\mathbf{a}(z,\zeta)|,
			$$
			where $C_0>0$ is a uniform constant.		
		\end{prop}
	By rescaling $u(z)\mapsto \widetilde{u}(\widetilde{z}):=h^{\frac{d}{4}}u(h^{\frac{1}{2}}\widetilde{z})$, we obtain that
	$$ \mathrm{Op}_{h,\R^d}^{\w}(\mathbf{a})u(z)=h^{-\frac{d}{4}}\mathrm{Op}_{1,\R^d}^{\w}(\mathbf{a}_{h}(\widetilde{z},\zeta))\widetilde{u}(\widetilde{z}),
	$$
	where $\mathbf{a}_h(\widetilde{z},\widetilde{\zeta})=\mathbf{a}(h^{\frac{1}{2}}\widetilde{z},h^{\frac{1}{2}}\widetilde{\zeta})$. Consequently, we have
	\begin{cor}[Theorem 5.1,\cite{Z12}]\label{AppcorA2}
		Suppose that $\mathbf{a}\in S^{0,0}(\R^d)$. Then
		$$ \|\mathrm{Op}_{h,\R^d}^{\w}(\mathbf{a})\|_{\mathcal{L}(L^2)}\leq C_0\sup_{\R^{2d}}|a(z,\zeta)|+O(h^{\frac{1}{2}}),
		$$
		where $C_0$ is an absolute constant.
	\end{cor}
\subsection{Explicit quantization on $M_0=\R_x\times\T_y$}

In this subsection, we follow the procedure of Chapter 5, Section 3 of \cite{Z12} to define quantization for partially periodic symbols.
Denote by $\iota: L^2(M_0)\rightarrow \mathcal{S}'(\R^2)$ the identification of a function in $L^2(M_0)$ as a tempered distribution in $\mathcal{S}'(\R^2)$. In coordinate, we can write
$$ (\iota f)(x,y)=\sum_{k\in\Z}f(x,y-2\pi k)\mathbf{1}_{\T_y}(y-2\pi k),
$$
where $\mathbf{1}_{\T_y}$ stands for the restriction to the fundamental domain $(-\pi,\pi)$ of $\T_y$. 
Denote by $\mathcal{S}_{\text{per}}'(\R^2)$ the subspace of partially $2\pi$-periodic distributions (in $y$ variable) which is the image of $\iota$.
Note that
\begin{align}\label{identity} 
	(\iota\circ\mathbf{1}_{\T_y})|_{\mathcal{S}_{\text{per}}'}=\mathrm{Id},\; \text{ on } \mathcal{S}_{\mathrm{per}}'.   
\end{align}
Symbols $\mathbf{a}$ on $T^*M_0$ can be identified as partially periodic symbols on $T^*\R^2$, namely $\mathbf{a}\in C^{\infty}(\R^4)$, bounded as well as its derivatives and satisfying
$$ \mathbf{a}(x,y,\xi,\eta)=\mathbf{a}(x,y+2k\pi,\xi,\eta),\quad \forall k\in\Z,\; \forall (x,y,\xi,\eta).
$$ 
The quantization $\mathrm{Op}_1^\w(\cdot)$ on $M_0$ is naturally defined by
\begin{align}\label{quantizationperiodic}  \mathrm{Op}_1^\w(\cdot)=\mathbf{1}_{\T_y}\mathrm{Op}_{1,\R^2}^\w\iota,
\end{align}
where the notation $\Op_{1,\R^2}^\w$ is to emphasize the quantization is on $\R^2$, given by the formula \eqref{QuantizationRd}.

The quantization formula \eqref{quantizationperiodic} can be expanded explicitly as
\begin{align}\label{decompositionperiodic} \mathrm{Op}_1^\w(\mathbf{a})u(x,y)=\sum_{k\in\Z}\mathbf{A}_ku(x,y), 
\end{align}
where
$$ \mathbf{A}_ku(x,y):=\frac{1}{(2\pi )^2}\int_{\R^2}\int_{\R\times\T}\mathbf{a}\big(\frac{x+x'}{2},\frac{y+y'}{2},\xi,\eta\big)\mathrm{e}^{i(x-x')\xi+i(y-y'+2\pi k)\eta}u(x',y')dx'dy'd\xi d\eta,
$$
or equivalently,
\begin{align}\label{Ak} 
	\mathbf{A}_ku(x,y)=\mathbf{1}_{\T_y}\tau_{-2\pi k}\mathrm{Op}_{1,\R^2}^{\w}(\mathbf{a})\mathbf{1}_{\T_y},
\end{align}
where $\tau_{y_0}f(x,y)=f(x,y-y_0)$ is the translation operator.
If there exist $\sigma\geq 0$, $M\geq 2$  such that some $h$-dependent family of symbols $\mathbf{a}_h$ satisfies additionally 
		$|\partial_{\eta}^{m}\mathbf{a}_h|\lesssim h^{m\sigma}$ for any $m\leq M$ (the situation of $\sigma>0$ appears when considering the semiclassical quantization), then  in $\mathrm{Op}_1^\w(\mathbf{a}_h)$
		we may only consider the contribution $\mathbf{A}_{k}$ for $|k|\leq 1$. Indeed, for $|k|\geq 2$, we have, for any $y,y'\in\T$, $|y-y'+2\pi k|\geq (|k|-1)2\pi$. By writing $\mathrm{e}^{i(y-y'+2\pi k)}$ as $\frac{\partial_{\eta}^m(\mathrm{e}^{i(y-y'+2\pi k)} )}{(i(y-y'+2\pi k))^m }$ and using integration by parts, we deduce that $$\|\mathbf{A}_k\|_{\mathcal{L}(L^2(M_0))}=O(h^{m\sigma}\langle k\rangle^{-m}).$$
		Consequently, we have the following version of the Calderon-Vaillancourt theorem:
\begin{cor}[Calderon-Vaillancourt]\label{CV}
	There exists $C_0>0$, such that for all $a\in C^{\infty}(T^*M_0)$, 
	$$ \|\Op_1^\w(a)\|_{\mathcal{L}(L^2(M_0))}\leq C_0\sum_{|\alpha|\leq 4}\sup_{T^*M_0}|\partial^{\alpha}a|.
	$$
\end{cor}


\subsection{Special symbol classes associated to the second microlocalization}
\begin{definition}
	Assume that $h,\epsilon,R$ are sequences of parameters\footnote{As usual, here we drop the subindex for parameters.} satisfying the following asymptotic:
	\begin{align}\label{asymptoticparameters} 
		h\rightarrow 0,\;\epsilon\rightarrow 0,\;R\rightarrow\infty,\; \frac{h}{\epsilon}\rightarrow 0. 
	\end{align}
	The $(h,\epsilon,R)$ parameter-dependent symbol class $\mathbf{S}_{1,1;1^-,0}^0(\R^6)$ consists of families of smooth functions $\mathbf{a}(x,x_1,y,y_1,\xi,\eta;h,\epsilon,R)$, satisfying the following hypotheses:
	\begin{itemize}
		\item[$\mathrm{(i)}$] There exists $K>0$ such that uniformly in parameters $h,\epsilon,R$, $\mathbf{a}(x,x_1,y,y_1,\xi,\eta;h,\epsilon,R)\equiv 0$ when $$\sqrt{(x+x_1)^2+(y+y_1)^2}> K\; \text{ or } \; \sqrt{\xi^2+\eta^2}> Kh^{-1}.
		$$   
		\item[$\mathrm{(ii)}$] For any $k,k_1,m,m_1,l,j$, there exists $C_{k,k_1,m,m_1,l,j}>0$, such that for all $h,\epsilon,R$, the following estimate holds
		\begin{align}\label{condition} \sup_{\R^6}|\partial_x^k\partial_{x_1}^{k_1}\partial_y^m\partial_{y_1}^{m_1}\partial_{\xi}^l\partial_{\eta}^j\mathbf{a}|\leq C_{k,k_1,m,m_1,l,j}\cdot	\big(\frac{\epsilon}{h}\big)^{k+k_1}\cdot h^{l+j}.
		\end{align}
	\end{itemize}
	Similarly,  $\mathbf{a}=\mathbf{a}(x,y,\xi,\eta;h,\epsilon,R)$ belongs to $\mathbf{S}_{1,1;1^-,0}^0(\R^4)$ if the analogue of the hypotheses $\mathrm{(i)},\mathrm{(ii)}$ hold without variables $x_1,y_1$.
\end{definition}
When there is no risk of confusing, we will not display the dependence in $h,\epsilon,R$ explicitly for symbols in $\mathbf{S}_{1,1;1^-,0}^0$.
Moreover, when the symbol $\mathbf{a}\in \mathbf{S}_{1,1;1^-,0}^0(\R^4)$ is $2\pi$-periodic in $y$ variable, we denote by $\mathbf{a}\in \mathbf{S}_{1,1;1^-,0}^0(T^*M_0)$. We have the following $L^2$ boundedness property for the quantization of this class of symbols:
\begin{prop}\label{L2boundedness} 
	Assume that $\mathbf{a}\in \mathbf{S}_{1,1;1^-,0}^0(\R^6)$ and let $T_{\mathbf{a}}$ be the linear operator on $\mathcal{S}'(\R^2)$ with the Schwartz kernel
	$$ K_{\mathbf{a}}(x,x_1,y,y_1)=\frac{1}{(2\pi)^2}\iint_{\R^2}\mathbf{a}(x,x_1,y,y_1,\xi,\eta)\mathrm{e}^{i(x-x_1)\xi+i(y-y_1)\eta}d\xi d\eta,
	$$
	where $x,x_1,y,y_1\in\R$.
	Then $T_{\mathbf{a}}\in\mathcal{L}(L^2(\R^2))$ uniformly in $h,\epsilon$ and $R$ obeying the asymptotic \eqref{asymptoticparameters}. More precisely,
	$$ \|T_{\mathbf{a}}\|_{\mathcal{L}(L^2(\R^2))}\leq C_0\sup_{\R^6}|\mathbf{a}|+C_{\mathbf{a}}\big(\epsilon^{\frac{1}{2}}+h^{\frac{1}{2}}\big),
	$$
	where the first constant $C_0$ is independent of $\mathbf{a}$.
	Consequently, if $\mathbf{a}\in \mathbf{S}_{1,1;1^-,0}^0(\R^4)$ and is $2\pi$-periodic in $y$, then
	$$ \|\mathrm{Op}_1^\w(\mathbf{a})\|_{\mathcal{L}(L^2(M_0))}\leq C_0\sup_{\R^4}|\mathbf{a}|+C_{\mathbf{a}}\big(\epsilon^{\frac{1}{2}}+h^{\frac{1}{2}}\big).
	$$
\end{prop}
\begin{proof}
	Consider the scalings $x=\frac{h}{\epsilon^{1/2}}X, y=h^{1/2}Y$ and $\xi=\frac{\epsilon^{1/2}}{h}\Xi, \eta=h^{-1/2}\Theta$ and for $f,g\in L^2(\R^2)$, we denote by $$f(x,y)=\frac{\epsilon^{1/4}}{h^{3/4}}F(X,Y),\quad  g(x,y)=\frac{\epsilon^{1/4}}{h^{3/4}}G(X,Y).
	$$
	Note that we have $\|f\|_{L^2(\R^2)}=\|F\|_{L^2(\R^2)}$ and $\|g\|_{L^2(\R^2)}=\|G\|_{L^2(\R^2)}$. Direct computation yields
	$$ \langle T_{\mathbf{a}}f,g\rangle_{L^2(\R^2)}=\langle T_{\widetilde{\mathbf{a}}}F,G\rangle_{L^2(\R^2)},
	$$
	where 
	$$ \widetilde{\mathbf{a}}(X,X_1,Y,Y_1,\Xi,\Theta)=\mathbf{a}\left(\frac{h}{\epsilon^{1/2}}X,\frac{h}{\epsilon^{1/2}}X_1,h^{1/2}Y,h^{1/2}Y_1,\frac{\epsilon^{1/2}}{h}\Xi,h^{-1/2}\Theta \right)
	$$
	with $X_1=\frac{\epsilon^{1/2}}{h}x_1, Y_1=h^{-1/2}y_1$.
	By the assumption on $\mathbf{a}$, we verify that
	\begin{align}\label{boundderivative}  \sup_{\R^6}|\partial_{X}^{k}\partial_{X_1}^{k_1}\partial_Y^{m}\partial_{Y_1}^{m_1}\partial_{\Xi}^{l}\partial_{\Theta}^{j}\widetilde{\mathbf{a}}|\lesssim_{\mathbf{a}} \epsilon^{\frac{k+k_1+l}{2}}h^{\frac{m+m_1+j}{2}}.
	\end{align}
	Then by a variant version of the Calder\'on-Vaillancourt theorem (Theorem 2.8.1 of \cite{Mar}), we deduce that $T_{\widetilde{\mathbf{a}}}$ are uniformly bounded on $L^2(\R^2)$. This means that
	\begin{align*}
		|\langle T_{\mathbf{a}}f,g\rangle_{L^2(\R^2)}| & =|\langle T_{\widetilde{\mathbf{a}}}F,G\rangle_{L^2(\R^2)}|
		\leq \Big(C_0\sup_{\R^6}|\mathbf{a}|+C_{\alpha}\sum_{1\leq |\alpha|\leq N_0 }\sup_{\R^6}|\partial^{\alpha}\widetilde{\mathbf{a}}|\Big) \|f\|_{L^2}\|g\|_{L^2},
	\end{align*}
	where the first constant $C_0$ does not depend on $\mathbf{a}$.
	By \eqref{boundderivative} and duality, we obtain the desired bound for $\|T_{\mathbf{a}}\|_{\mathcal{L}(L^2(\R^2))}$.  For $\mathbf{a}\in\mathbf{S}_{1,1;1^-,0}^0(\R^4)$, $2\pi$-periodic in $y$, the extension to $\Op_1^\w(\mathbf{a})$ on $L^2(M_0)$ follows simply from the decomposition
	\eqref{decompositionperiodic} and \eqref{Ak}. The proof of Proposition \ref{L2boundedness} is now complete. 
\end{proof}


Next, we need the following Lemma to derive a G$\mathring{\mathrm{a}}$rding type inequality for symbols in $\mathbf{S}_{1,1;1^-,0}^0$:
\begin{lemma}\label{composition1}
	Let $\mathbf{a},\mathbf{b}\in \mathbf{S}_{1,1;1^-,0}^0(T^*M_0)$. Then there exists $C_1>0$, such that uniformly in $h,\epsilon,R$ satisfying the asymptotic \eqref{asymptoticparameters},
	$$ \Vert\Op_1^\w(\mathbf{a}\mathbf{b})-\Op_1^\w(\mathbf{a})\Op_1^\w(\mathbf{b}) \Vert_{\mathcal{L}(L^2(M_0))}\leq C_1(h+\epsilon).
	$$
\end{lemma}
\begin{proof}
	We first prove the estimate for $\mathbf{a},\mathbf{b}\in \mathbf{S}_{1,1;1^-,0}^0(\R^4)$.
	Recall that $$\mathrm{Op}_{1,\R^2}^\w(\mathbf{a})\mathrm{Op}_{1,\R^2}^\w(\mathbf{b})=\mathrm{Op}_{1,\R^2}^\w(\mathbf{c}),$$ with
	$$ \mathbf{c}(z,\zeta)=\frac{1}{\pi^4}\int_{\R^8}\mathrm{e}^ {-2i\sigma(z_1,\zeta_1;z_2,\zeta_2)}\mathbf{a}(z+z_1,\zeta+\zeta_1)\mathbf{b}(z+z_2,\zeta+\zeta_2)dz_1dz_2 d\zeta_1 d\zeta_2,
	$$
	where
	$$ \sigma(z_1,\zeta_1;z_2,\zeta_2)=\zeta_1\cdot z_2-z_1\cdot\zeta_2.
	$$
	With $X=(z,\zeta), X_j=(z_j,\zeta_j)\in\R^4$, $j=1,2$, $Z=(X_1,X_2)\in\R^8$, we can write
	$$ \mathbf{c}(X)=\frac{1}{\pi^4}\int_{\R^8}\mathrm{e}^{i\langle QZ,Z\rangle}\mathcal{A}(X,Z)dZ,
	$$
	where $\mathcal{A}(X,Z)=\mathbf{a}(X+X_1)\mathbf{b}(X+X_2)$, and
	$$ Q=\left( 
	\begin{matrix}
		0 &0 &0 &-\mathrm{Id}_2\\
		0 &0 &\mathrm{Id}_2 &0\\
		0 &\mathrm{Id}_2 &0 &0\\
		-\mathrm{Id}_2 &0 &0 &0
	\end{matrix}
	\right).
	$$
	Note that $Q^{-1}=Q$ and $\langle QZ,Z\rangle=-2\sigma(X_1,X_2)$.
	Direct computation yields 
	\begin{align}\label{c-ab}
		\mathbf{c}(X)-\mathbf{a}(X)\mathbf{b}(X)=-\frac{1}{i\pi^4}\int_0^1tdt\int_{\R^8}\mathrm{e}^{i\langle QZ,Z\rangle}(\mathcal{D}_Z\mathcal{A})(X,tZ)dZ,
	\end{align}
	where $\mathcal{D}_Z=D_{\zeta_1}\cdot D_{z_2}-D_{z_1}\cdot D_{\zeta_2}$. Indeed, since\footnote{The integral $\int_{\R^8}e^{i\langle QZ,Z\rangle}dZ$ should be understood as an oscillatory integral, in the sense that
	$$ \int_{\R^8}e^{i\langle QZ,Z\rangle}dZ=\lim_{\epsilon\rightarrow 0^+}\int_{\R^8}e^{i\langle QZ,Z\rangle-\epsilon|Z|^2}.
    $$		}
	$$ \frac{1}{\pi^4}\int_{\R^8}\mathrm{e}^{i\langle QZ,Z\rangle}dZ=1,
	$$
			 by Taylor expansion $$\mathcal{A}(X,Z)=\mathcal{A}(X,0)+\int_0^1Z\cdot(\nabla_Z\mathcal{A})(X,tZ)dZ$$	
	and the simple observation that $Z=\frac{Q^{-1}}{2i}\nabla_Z(\mathrm{e}^{i\langle QZ,Z\rangle})$, we have
	\begin{align*}
	\mathbf{c}(X)-\mathbf{a}(X)\mathbf{b}(X)=&\mathcal{A}(X,Z)-\mathcal{A}(X,0)\\
	=&\frac{1}{\pi^4}\int_0^1dt\int_{\R^8}\frac{1}{2i}\nabla_Z(\mathrm{e}^{i\langle QZ,Z\rangle})\cdot (Q^{-1}\nabla_Z\mathcal{A})(X,tZ)dZ\\
	=&-\frac{1}{2i\pi^4}\int_0^1tdt\int_{\R^8}\mathrm{e}^{i\langle QZ,Z\rangle }
	(\nabla_Z\cdot(Q^{-1}\nabla_Z)\mathcal{A})(X,tZ)dZ.
	\end{align*}
Noting that $Q^{-1}=-Q$ and $\nabla_Z\cdot Q\nabla_Z=-2\mathcal{D}_Z$, we obtain \eqref{c-ab}.

	Note that $\mathcal{D}_Z\mathcal{A}(X,tZ)$ is a linear combination of the terms
	$$ \epsilon \cdot \widetilde{\mathbf{a}}(X+tX_1)\widetilde{\mathbf{b}}(X+tX_2),
	$$
	with $\widetilde{\mathbf{a}},\widetilde{\mathbf{b}}\in\mathbf{S}_{1,1;1^-,0}^0(\R^4)$.
	It suffices to show that for $t\in(0,1]$,
	$$ \mathbf{d}_t(X):=\int_{\R^8}\mathrm{e}^{i\langle QZ,Z\rangle}\widetilde{\mathbf{a}}(X+tX_1)\widetilde{\mathbf{b}}(X+tX_2)dX_1dX_2\in\mathbf{S}^0_{1,1;1^-,0}.
	$$
	Since the derivatives on $z,\zeta$ will fall on $\widetilde{\mathbf{a}},\widetilde{\mathbf{b}}$, by definition of the symbol class, we only need to show that 
	\begin{align}\label{bdd} 
		|\mathbf{d}_t(X)|\lesssim 1,	
	\end{align}
	uniformly in $h,\epsilon,R$ and $t\in(0,1]$. Denote by
	$$ \widetilde{\mathcal{F}}\widetilde{\mathbf{a}}(z,\theta):=\int_{\R^4}\widetilde{\mathbf{a}}(z,\zeta)\mathrm{e}^{-i\zeta\cdot\theta}d\zeta, \quad \widetilde{\mathcal{F}}\widetilde{\mathbf{b}}(z,\theta):=\int_{\R^4}\widetilde{\mathbf{b}}(z,\zeta)\mathrm{e}^{-i\zeta\cdot\theta}d\zeta
	$$
	the partial Fourier transform, then we have
	$$ \mathbf{d}_t(z,\zeta)=\int_{\R^4}\mathrm
	{e}^{-2i(z_1-z_2)\cdot\zeta}(\widetilde{\mathcal{F}}\widetilde{\mathbf{a}})(z+t^2z_1,2z_2)
	(\widetilde{\mathcal{F}}\widetilde{\mathbf{b}})(z+t^2z_2,-2z_1)dz_1 dz_2.
	$$
	Since $\widetilde{\mathbf{a}},\widetilde{\mathbf{b}}\in \mathbf{S}_{1,1;1^-,0}^0$, it follows from the integration by part that
	$$ |\widetilde{\mathcal{F}}\widetilde{\mathbf{a}}(z,\theta)|\lesssim_N\frac{h^{N}}{|\theta|^N},\quad |\widetilde{\mathcal{F}}\widetilde{\mathbf{b}}(z,\theta)|\lesssim_N\frac{h^{N}}{|\theta|^N} 
	$$
	for all $N\geq 1$ and $|\theta|\geq 1$. This leads to \eqref{bdd}. Once we have established \eqref{bdd}, applying Proposition \ref{L2boundedness}, we obtain that
	\begin{align}\label{compositionR2}
		\|\mathrm{Op}_{1,\R^2}^{\w}(\mathbf{ab})-\mathrm{Op}_{1,\R^2}^{\w}(\mathbf{a})\mathrm{Op}_{1,\R^2}^{\w}(\mathbf{b})\|_{\mathcal{L}(L^2)}\leq C(h+\epsilon).
	\end{align}

	Next we extend the above estimate to symbols that are $2\pi$- periodic in $y$.
	Take $f,g\in L^2(M_0)$, we have
	\begin{align*}
		\big\langle\big(\Op_1^\w(\mathbf{a}\mathbf{b})-\Op_1^\w(\mathbf{a})\Op_1^\w(\mathbf{b}) \big)f,g \big\rangle_{L^2(M_0)}=\big\langle \big(\Op_{1,\R^2}^\w(\mathbf{a}\mathbf{b})-\Op_{1,\R^2}^\w(\mathbf{a})\Op_{1,\R^2}^\w(\mathbf{b})\big)\iota f, \mathbf{1}_{\T_y}\iota g
		\big\rangle_{L^2(\R^2)},
	\end{align*}
	where we used \eqref{identity}. By self-adjointness, the above quantity equals to
	$$ \big\langle \iota f, \big(\Op_{1,\R^2}^\w(\mathbf{a}\mathbf{b})-\Op_{1,\R^2}^\w(\mathbf{b})\Op_{1,\R^2}^\w(\mathbf{a})\big)\mathbf{1}_{\T_y}\iota g
	\big\rangle_{L^2(\R^2)}.
	$$
	Recall that $\iota=\sum_{k\in\Z}\tau_{2\pi k}(\mathbf{1}_{\T_y}\cdot)$, from the discussion below \eqref{Ak}, we have for all $|k|\geq 2$,
	$$ \big\langle \tau_{-2\pi k}(\mathbf{1}_{\T_y}f),  \big(\Op_{1,\R^2}^\w(\mathbf{a}\mathbf{b})-\Op_{1,\R^2}^\w(\mathbf{b})\Op_{1,\R^2}^\w(\mathbf{a})\big)\mathbf{1}_{\T_y}\iota g
	\big\rangle_{L^2(\R^2)}=O(h^{\infty}\langle k\rangle^{-\infty})\|f\|_{L^2(M_0)}\|g\|_{L^2(M_0)}.
	$$
	When $|k|\leq 1$, by \eqref{compositionR2}, 
	$$ \big\langle \tau_{-2\pi k}(\mathbf{1}_{\T_y}f),  \big(\Op_{1,\R^2}^\w(\mathbf{a}\mathbf{b})-\Op_{1,\R^2}^\w(\mathbf{b})\Op_{1,\R^2}^\w(\mathbf{a})\big)\mathbf{1}_{\T_y}\iota g
	\big\rangle_{L^2(\R^2)}=C(h+\epsilon)\|f\|_{L^2(M_0)}\|g\|_{L^2(M_0)}.
	$$
	Adding up all $k\in\Z$, we obtain that
	\begin{align*}
		\|\mathrm{Op}_{1}^{\w}(\mathbf{ab})-\mathrm{Op}_{1}^{\w}(\mathbf{a})\mathrm{Op}_{1,\R^2}^{\w}(\mathbf{b})\|_{\mathcal{L}(L^2)}\leq C(h+\epsilon).
	\end{align*}
	By duality, the desired result then follows and we complete the proof of Lemma \ref{composition1}. 
\end{proof}
Consequently, we have:
\begin{cor}\label{Garding} 
	Let $\mathbf{a}\in\mathbf{S}_{1,1;1^-,0}^0(T^*M_0)$, then there exists $C_1>0$, such that
	$$ \big\langle\Op_1^{\w}(\mathbf{a}^2)f,f\big\rangle_{L^2(M_0)}\geq -C_1(h+\epsilon)\|f\|_{L^2(M_0)}^2.
	$$
\end{cor}

\section{Exact computations for commutators}\label{AppendixB}

In this part, we collect some computations for commutators needed in Section \ref{subellipticPropagation}. Below, all computations will be done only for operators acting on $L^2(\R^2)$, since all the commutators are of the form $[P,\Op_1^\w(q)]$, with some differential operator $P$ and a symbol $q$ that is periodic in $y$, thanks to the identification \eqref{quantizationperiodic}.  
First we prove the commutator formula below, used to prove the propagation theorem for the scalar-valued second semiclassical microlocal measure:
\begin{lemma}\label{commutatorformula}
	Let $q(z,\zeta)\in\mathcal{S}(\R^4)$. We have the following formula
	\begin{align}\label{commutatorformula1}
		[-h^2\Delta_G,\Op_1^{\w}(q)]=&\frac{h^2}{i}\Op_1^\w(2\xi\partial_xq+2V(x)\eta\partial_yq-2x\eta^2\partial_{\xi}q)+\frac{h^2}{2i}\Op_1^{\w}(x\partial_y^2\partial_{\xi}q)\notag \\
		+&\frac{h^2}{i}\Op_1^\w\big(V_1'(x)(\frac{1}{4}\partial_y^2\partial_{\xi}q-\eta^2\partial_{\xi}q) \big)-\frac{h^2}{24i}\Op_1^\w\big(V_1'''(x)(\frac{1}{4}\partial_y^2\partial_{\xi}^3q-\eta^2\partial_{\xi}^3q) \big)-\mathcal{R}_h(q),
	\end{align}
	where $V_1(x)=V(x)-x^2=O(x^3)$ and the operator $\mathcal{R}_h(q)$ has the Schwartz kernel
	$$\mathcal{K}_q(z,z')=\frac{h^2}{(2\pi)^2}\int_{\R^2}\big[\widetilde{V}_1(x,x')(\frac{1}{4}\partial_y^2\partial_{\xi}^4q-\eta^2\partial_{\xi}^4q)\big(\frac{z+z'}{2},\zeta\big)+\widetilde{V}_2(x,x')\frac{\eta}{i}(\partial_y\partial_{\xi}^2q)\big(\frac{z+z'}{2},\zeta\big)\big]e^{i(z-z')\cdot\zeta}d\zeta,
	$$
	with
	$$ \widetilde{V}_1(x,x')=\frac{1}{16}\int_{-1}^1dt\int_0^tdt_1\int_0^{t_1}dt_2\int_0^{t_2}V_1^{(4)}\big(\frac{x+x'}{2}+t\frac{x-x'}{2}\big)dt_3
	$$
	and
	$$\widetilde{V}_2(x,x')=\frac{1}{4}\int_0^1dt\int_{-t}^tV''\big(\frac{x+x'}{2}+t_1\frac{x-x'}{2}\big)dt_1.
	$$
\end{lemma}

\begin{rem}
	In our applications of the formula, apart from the first operator on the right hand side, the others are all viewed as remainders.
\end{rem}

\begin{proof}
	The proof follows from a direct computation. Write
	$$ [-h^2\Delta_G,\Op_1^{\w}(q)]=[h^2D_x^2,\Op_1^{\w}(q)]+V(x)[h^2D_y^2,\Op_1^{\w}(q)]+[V(x),\Op_1^{\w}(q)]h^2D_y^2, 
	$$ we have for any $f\in\mathcal{S}(\R^2)$,
	\begin{align*}
		[-h^2\Delta_G,\mathrm{Op}_1^{\w}(q)]f(z) & \\[0.2cm]
		& \hspace*{-3cm} =-ih^2\mathrm{Op}_1^\w(2\xi\partial_xq)f-ih^2\mathrm{Op}_1^{\w}(2V(x)\eta\partial_yq)f\\[0.2cm]
		& \hspace*{-3cm} \quad -\frac{ih^2}{(2\pi)^2}\iint_{\R^4}\big(V(x)-V\big(\frac{x+x'}{2}\big)\big)2\eta(\partial_yq)\big(\frac{z+z'}{2},\zeta\big)\mathrm{e}^{i(z-z')\cdot\zeta}f(z')dz'd\zeta 
		\\[0.2cm]
		& \hspace*{-3cm} \quad  -\frac{h^2}{(2\pi)^2}\iint_{\R^4}(V(x)-V(x'))[\frac{1}{4}\partial_y^2q-i\eta\partial_yq-\eta^2q]\big(\frac{z+z'}{2},\zeta\big)\mathrm{e}^{i(z-z')\cdot\zeta}f(z')dz'd\zeta.
	\end{align*}
	After arrangement, we have
	$$ [-h^2\Delta_G,\Op_1^{\w}(q)]=\frac{h^2}{i}\Op_1^\w(2\xi\partial_xq+2V(x)\eta\partial_yq)-\mathcal{T}_1-\mathcal{T}_2,
	$$
	where $\mathcal{T}_1,\mathcal{T}_2$ have Schwartz kernels
	\begin{align*}
		&\mathcal{K}_1(z,z')=\frac{h^2}{(2\pi)^2}\int_{\R^2}\big(V(x)-V(x')\big)\big[\frac{1}{4}\partial_y^2q-\eta^2q \big]\big(\frac{z+z'}{2},\zeta\big)\mathrm{e}^{i(z-z')\cdot\zeta}d\zeta,\\ &\mathcal{K}_2(z,z')=\frac{ih^2}{(2\pi)^2}\int_{\R^2}\big[V(x)+V(x')-2V\big(\frac{x+x'}{2}\big)\big]\eta(\partial_yq)\big(\frac{z+z'}{2},\zeta\big)\mathrm{e}^{i(z-z')\cdot\zeta}d\zeta.
	\end{align*}
	Since $V(x)=x^2+V_1(x)$, $V_1(x)=O(|x|^3)$, we have
	\begin{align*}
		\mathcal{T}_1f(z) \\[0.2cm]
		& \hspace*{-1cm} =\frac{h^2}{(2\pi)^2}\iint_{\R^4}[(x+x')(x-x')+V_1(x)-V_1(x')][\frac{1}{4}\partial_y^2q-\eta^2q]\big(\frac{z+z'}{2},\zeta\big)\mathrm{e}^{i(z-z')\cdot\zeta}f(z')dz'd\zeta\\
		& \hspace*{-1cm} = ih^2\mathrm{Op}_1^\w(2x\cdot(\frac{1}{4}\partial_y^2\partial_{\xi}q-\eta^2\partial_{\xi}q))f\\[0.2cm]
		& \hspace*{-1cm} \quad  + \frac{ih^2}{(2\pi)^2}\iint_{\R^4}\Big(\int_{-1}^1\frac{1}{2}V_1'\big(\frac{x+x'}{2}+t\frac{x-x'}{2}\big)dt\Big)[\frac{1}{4}\partial_y^2\partial_{\xi}q-\eta^2\partial_{\xi}q]\big(\frac{z+z'}{2},\zeta\big)\mathrm{e}^{i(z-z')\cdot\zeta}f(z')dz'd\zeta.
	\end{align*}
	We can further write
	\begin{align*}
		\frac{1}{2}\int_{-1}^1V_1'\big(\frac{x+x'}{2}+t\frac{x-x'}{2}\big)dt& =V_1'\big(\frac{x+x'}{2}\big)+\frac{(x-x')^2}{24}V_1'''\big(\frac{x+x'}{2}\big)\\[0.2cm]
		& \quad + \frac{(x-x')^3}{16}\int_{-1}^1dt\int_0^{t}dt_1\int_{0}^{t_1}dt_2\int_{0}^{t_2}V_1^{(4)}\big(\frac{x+x'}{2}+t\frac{x-x'}{2}\big)dt_3,
	\end{align*}
	hence
	$$ \mathcal{T}_1f=ih^2\Op_1^\w\big((2x+V_1'(x))(\frac{1}{4}\partial_y^2\partial_{\xi}q-\eta^2\partial_{\xi}q)\big)-\frac{ih^2}{24}\Op_1^\w\big(V_1'''(x)(\frac{1}{4}\partial_y^2\partial_{\xi}^3q-\eta^2\partial_{\xi}^3q)\big)+\mathcal{R}_1,
	$$
	where $\mathcal{R}_1$ has the Schwartz kernel
	\begin{align*} \mathcal{K}_{\mathcal{R}_1}(z,z')=\frac{h^2}{(2\pi)^2}\int_{\R^2}&\Big(\frac{1}{16}\int_{-1}^1dt\int_0^tdt_1\int_0^{t_1}dt_2\int_0^{t_2}V_1^{(4)}\big(\frac{x+x'}{2}+t\frac{x-x'}{2}\big)dt_3
		\Big)\\
		\times &(\frac{1}{4}\partial_y^2\partial_{\xi}^4q-\eta^2\partial_{\xi}^4q)\big(\frac{z+z'}{2},\zeta\big)\mathrm{e}^{i(z-z')\cdot\zeta}d\zeta.
	\end{align*}
	For $\mathcal{T}_2$, by Taylor expansion and integration by part, we can write its kernel as
	\begin{align*}
		\mathcal{K}_2(z,z')=\frac{h^2}{4i(2\pi)^2}\int_{\R^2}\Big(\int_0^1dt\int_{-t}^tV''\big(\frac{x+x'}{2}+t_1\frac{x-x'}{2}\big)dt_1\Big)\eta(\partial_y\partial_{\xi}^2q)\big(\frac{z+z'}{2},\zeta\big)\mathrm{e}^{i(z-z')\cdot\zeta}d\zeta,
	\end{align*}
	which is part of the remainder. We complete the proof of Lemma \ref{commutatorformula}.
\end{proof}
The following formula will be used to prove the propagation formula for the operator-valued second semiclassical microlocal measures:
\begin{lemma}\label{commutatorformula2}
	Let $q(z,\zeta)\in\mathcal{S}(\R^4)$. We have the following formula
	\begin{align*}
		[-\partial_x^2-V(hx)h^2\partial_y^2,\Op_1^\w(q)]& = \Op_1^{\w,(y,\eta)}\big(\big[-\partial_x^2+x^2h^4\eta^2,\Op_1^{\w,(x,\xi)}(q)\big]_{L^2(\R_x)}\big)\\[0.2cm]
		& \quad + \Op_1^{\w,(y,\eta)}\big(\big[(V(hx)-h^2x^2)h^2\eta^2,\Op_1^{\w,(x,\xi)}(q) \big]_{L^2(\R_x)}\big)\\[0.2cm]
		& \quad + \frac{h^2}{i}\Op_1^{\w}(2V(hx)\eta\partial_yq)+\frac{h^3}{4i}\Op_1^\w\big(V'(hx)\partial_y^2\partial_{\xi}q)
		+\mathcal{R}_q, 
	\end{align*}
	where the operator $\mathcal{R}_q$ has the Schwartz kernel
	$$ \mathcal{K}_q(z,z')=\frac{h^4}{(2\pi)^2}\int_{\R^2}\big[\widetilde{V}_3(hx,hx')\frac{1}{4}\partial_y^2\partial_{\xi}^2q\big(\frac{z+z'}{2},\zeta\big)+i\widetilde{V}_4(hx,hx')\eta\partial_y\partial_{\xi}^2q\big(\frac{z+z'}{2},\zeta\big) \big]\mathrm{e}^{i(z-z')\cdot\zeta}d\zeta,
	$$
	where
	$$ \widetilde{V}_3(hx,hx')=\frac{1}{4}\int_{-1}^1dt\int_0^tV''\big(h\frac{x+x'}{2}+t_1h\frac{x-x'}{2}\big)dt_1,
	$$
	and
	$$ \widetilde{V}_4(hx,hx')=\frac{1}{4}\int_0^1dt\int_{-t}^{t}V''\big(h\frac{x+x'}{2}+t_1h\frac{x-x'}{2}\big)dt_1.
	$$
\end{lemma}
\begin{proof}
	The proof follows from a direct computation. We write
	\begin{align*}
		[D_x^2+V(hx)h^2D_y^2,\Op_1^\w(q)]=&[D_x^2,\Op_1^\w(q)]+V(hx)[h^2D_y^2,\Op_1^\w(q)]+[V(hx),\Op_1^\w(q)]h^2D_y^2.
	\end{align*}
	We observe that $$[D_x^2,\Op_1^\w(q)]=\Op_1^{\w,(y,\eta)}([D_x^2,\Op_1^{\w,(x,\xi)}(q)]_{L^2(\R_x)}),\quad V(hx)[h^2D_y^2,\Op_1^{\w}(q)]=\frac{h^2}{i}V(hx)\Op_1^\w(2\eta\partial_yq),$$
	and the third operator
	$[V(hx),\Op_1^\w(q)]h^2D_y^2$ has the Schwartz kernel
	$$ \frac{h^2}{(2\pi)^2}\int_{\R^2}(V(hx)-V(hx'))[\frac{1}{4}D_y^2q-\eta D_yq+\eta^2q]\big(\frac{z+z'}{2},\zeta\big)\mathrm{e}^{i(z-z')\cdot\zeta}d\zeta.
	$$
	Since $\frac{h^2}{(2\pi)^2}\int_{\R^2}(V(hx)-V(hx'))\eta^2q\big(\frac{z+z'}{2},\zeta\big)\mathrm{e}^{i(z-z')\cdot\zeta}d\zeta$ is the Schwartz kernel of the operator $$\Op_1^{\w,(y,\eta)}([V(hx)h^2\eta^2,\Op_1^{\w,(x,\xi)}(q)]_{L^2(\R_x)}),$$
	we get
	\begin{align}
		[D_x^2+V(hx)h^2D_y^2,\Op_1^{\w}(q)]=&\Op_1^{\w,(y,\eta)}\big([D_x^2+V(hx)h^2\eta^2,\Op_1^{\w,(x,\xi)}(q) ]_{L^2(\R_x)}\big)+\frac{h^2}{i}\Op_1^{\w}(2V(hx)\eta\partial_yq)\notag \\
		+&\mathcal{T}, \label{T} 
	\end{align}
	where the operator $\mathcal{T}$ has the Schwartz kernel 
	\begin{align*}
		\mathcal{K}_{\mathcal{T}}(z,z')=&\frac{h^2}{(2\pi)^2}\int_{\R^2}(V(hx)-V(hx'))[\frac{1}{4}D_y^2q-\eta D_yq]\big(\frac{z+z'}{2},\zeta\big)\mathrm{e}^{i(z-z')\cdot\zeta}d\zeta\\
		&+\frac{h^2}{(2\pi)^2}\int_{\R^2}\big(V(hx)-V\big(h\frac{x+x'}{2}\big)\big)\frac{2}{i}\eta(\partial_yq)\big(\frac{z+z'}{2},\zeta\big)\mathrm{e}^{i(z-z')\cdot\zeta}d\zeta,
	\end{align*}
	or equivalently,
	\begin{align*}
		\mathcal{K}_{\mathcal{T}}(z,z')=&\underbrace{\frac{h^2}{(2\pi)^2}\int_{\R^2}(V(hx)-V(hx'))\frac{1}{4}(D_y^2q)\big(\frac{z+z'}{2},\zeta\big)\mathrm{e}^{i(z-z')\cdot\zeta}d\zeta}_{\mathcal{K}_{\mathcal{T}}^{(1)}(z,z') }\\
		& +\underbrace{\frac{h^2}{(2\pi)^2}\int_{\R^2}\big[V(hx)+V(hx')-2V\big(h\frac{x+x'}{2}\big) \big]\eta (D_yq)\big(\frac{z+z'}{2},\zeta\big)\mathrm{e}^{i(z-z')\cdot\zeta}d\zeta}_{\mathcal{K}_{\mathcal{T}}^{(2)}(z,z') }.
	\end{align*}
	By Taylor expansion and doing the integration by part, we can further write
	\begin{align*}
		\mathcal{K}_{\mathcal{T}}^{(1)}(z,z')& \\[0.2cm]
		 & \hspace*{-1cm} =-\frac{h^3}{(2\pi)^2}\int_{\R^2}V'\big(h\frac{x+x'}{2}\big)\frac{1}{4}(D_y^2D_{\xi}q)\big(\frac{z+z'}{2},\zeta\big)\mathrm{e}^{i(z-z')\cdot\zeta}d\zeta\\[0.2cm]
		& \hspace*{-1cm} \quad +\frac{h^4}{(2\pi)^2}\int_{\R^2}\Big(\frac{1}{4}\int_{-1}^1dt\int_{0}^tV''\big(h\frac{x+x'}{2}+t_1h\frac{x-x'}{2} \big)dt_1\Big)\frac{1}{4}(D_y^2D_{\xi}^2)q\big(\frac{z+z'}{2},\zeta\big)\mathrm{e}^{i(z-z')\cdot\zeta}d\zeta,
	\end{align*}
	and
	\begin{align*}
		\mathcal{K}_{\mathcal{T}}^{(2)}:=\frac{h^4}{(2\pi)^2}\int_{\R^2}\Big(\frac{1}{4}\int_0^1dt\int_{-t}^{t}V''\big(h\frac{x+x'}{2}+t_1h\frac{x-x'}{2}\big)dt_1\Big)\eta\cdot (D_yD_{\xi}^2q)\big(\frac{z+z'}{2},\zeta\big)\mathrm{e}^{i(z-z')\cdot\zeta}d\zeta.
	\end{align*}
	By organizing terms, we complete the proof of Lemma \ref{commutatorformula2}.
\end{proof}

To end this section, we collect an elementary averaging Lemma for the one-dimensional harmonic oscillator:
\begin{lemma}\label{averagingLemma} 
	Denote by $H_0=D_x^2+x^2$ is the harmonic oscillator and $\mathcal{U}(t)=\mathrm{e}^{itH_0}$ is the associated propagator, then for any $k\in\N$,
	\begin{align}\label{averaging1}
		&\int_0^{2\pi}\mathcal{U}(t)^*x^{2k-1}\mathcal{U}(t)dt=0,\\
		&\int_0^{2\pi}\mathcal{U}(t)^*x^{2}\mathcal{U}(t)dt= \frac{1}{2}(D_x^2+x^2). \label{averaging2} 
	\end{align}
\end{lemma}
\begin{proof}
	Consider the ladder operators $\mathcal{L}_{\pm}:=x\mp iD_x$. These operators have properties:
	$$ [\mathcal{L}_+,\mathcal{L}_-]=-2,\quad\mathcal{L}_+\mathcal{L}_-=H_0-1,\quad \mathcal{L}_-\mathcal{L}_+=H_0+1
	$$
	and $\mathcal{L}_+:\mathcal{H}_n\mapsto \mathcal{H}_{n+1}$ while $\mathcal{L}_-:\mathcal{H}_n\rightarrow \mathcal{H}_{n-1}$, where $\mathcal{H}_n$ is the eigenspace of $H_0$ associated with the eigenvalue $2n+1$. Let $\Pi_n: L^2(\R)\rightarrow \mathcal{H}_n$ be the orthogonal projection. From this, we deduce that the averaging $\int_0^{2\pi}\mathcal{U}(t)^*(\cdots)\mathcal{U}(t)dt$ for a monomial composed of $\mathcal{L}_{\pm}$ is non-zero if and only if the number of $\mathcal{L}_+$ equals the number of $\mathcal{L}_-$.
	Since $x=\frac{1}{2}(\mathcal{L}_++\mathcal{L}_-)$, each monomial in the expansion of $x^{2k-1}$ has different numbers of $\mathcal{L}_+,\mathcal{L}_-$, thus \eqref{averaging1} holds. The second identity \eqref{averaging2} follows from the explicit expansion of $x^2=\frac{1}{4}(2H_0+\mathcal{L}_+^2+\mathcal{L}_-^2)$.
\end{proof}

\section{Subelliptic apriori estimates}

\begin{lemma}\label{lemmaB1}
	There exists $C_0>0$ such that for all $u\in H_G^2(\T^2)$,
	$$ \|\Delta_Gu\|_{L^2}\leq \|u\|_{\dot{H}_G^2}\leq C_0\|\Delta_G u\|_{L^2}+C_0\|u\|_{L^2}.
	$$
\end{lemma}
\begin{proof}
	The inequality $\|\Delta_G u \|_{L^2}\leq \|u\|_{\dot{H}_G^2}$ follows trivially by the triangle inequality, hence it suffices to prove the other.
	Let $\chi$ be a bump function which is equal to $1$ near $0$. We decompose $u=v+w$, with
	$$\quad v(x)=(1-\chi(x))u(x),\quad w(x)=\chi(x)u(x).
	$$
	Since $(-\Delta_Gu,u)=\|\nabla_Gu\|_{L^2}^2$, we deduce that
	for test functions $\varphi=1-\chi(x)$ or $\chi(x)$, 
	$$ \|\Delta_G(\varphi u)\|_{L^2}\lesssim_{\varphi} \|\Delta_G u\|_{L^2}+\|u\|_{L^2}.
	$$
	Therefore, it suffices to show that
	$$ \|v\|_{\dot{H}_G^2}\lesssim \|\Delta_Gv\|_{L^2}+\|v\|_{L^2},\; \|w\|_{\dot{H}_G^2}\lesssim \|\Delta_Gw\|_{L^2}+\|w\|_{L^2}.
	$$
	Note that $\Delta_G$ is elliptic on supp$(1-\chi)\subset \T^2\setminus\{x=0\}$, from the support property of $v$, we deduce that $$\|v\|_{\dot{H}_G^2}\lesssim \|\Delta_G v\|_{L^2}+\|v\|_{L^2}.$$ 
	
	To estimate $\|w\|_{\dot{H}_G^2}$, we expand $w$ as Fourier series in $y$, i.e. $w=\sum_{n\in\Z}w_n(x)\mathrm{e}^{iny}$.  It suffices to show that uniformly in $n$,
			\begin{align}\label{appendixB1}
				\||n|V^{1/2}(x)\partial_xw_n\|_{L_x^2}+\|\partial_x|n|V^{1/2}(x)w_n\|_{L_x^2}+ \|\partial_x^2w_n\|_{L_x^2}+\|n^2V(x)w_n\|_{L_x^2}\leq C_0\|\mathcal{L}_nw_n\|_{L_x^2}+\|w_n\|_{L_x^2}
			\end{align}
			where $\mathcal{L}_n=-\partial_x^2+n^2V(x)$ and $w_n$ are supported on supp$(\chi)$.  The estimate is trivial when $n=0$, so below we assume that $n\neq 0$, and without loss of generality, we assume that $n>0$.
		Set $f_n=\mathcal{L}_nw_n$ and consider the change of variable $z=n^{\frac{1}{2}}x$ and $\widetilde{w}_n(z)=w_n(x), \widetilde{f}_n(z)=f_n(x)$, we have
			$$ \widetilde{\mathcal{L}}_n \widetilde{w}_n=\widetilde{g}_n:=n^{-1}\widetilde{f}_n,
			$$
			where $\widetilde{\mathcal{L}}_n=-\partial_z^2+nW\big(\frac{z}{\sqrt{n}}\big)^2$. By rescaling, it suffices to show that
			\begin{align}\label{AppendixC1} 
				 \sqrt{n}\big\|\partial_zW\big(\frac{z}{\sqrt{n}}\big)\widetilde{w}_n\big\|_{L_z^2}+
				\sqrt{n}\big\|W\big(\frac{z}{\sqrt{n}}\big)\partial_z\widetilde{w}_n\big\|_{L_z^2}+
				\|\partial_z^2\widetilde{w}_n\|_{L_z^2}+n\big\|W\big(\frac{z}{\sqrt{n}}\big)^2\widetilde{w}_n\big\|_{L_z^2}\notag & \\[0.2cm] 
				& \hspace*{-10cm} \lesssim \|\widetilde{g}_n\|_{L_z^2}+\|\widetilde{w}_n\|_{L_z^2}.
			\end{align}
			Having in mind that $\sqrt{n}W\big(\frac{z}{\sqrt{n}}\big)\approx z$, the desired estimate is nothing but the a priori estimate for the elliptic equation $(-\partial_z^2+z^2)\widetilde{w}=\widetilde{g}$.
			
			We expand
			\begin{align}\label{tildegn} 
				\|\widetilde{g}_n\|_{L_z^2}^2=	\|\widetilde{\mathcal{L}}_n\widetilde{w}_n\|_{L_z^2}^2=\|\partial_z^2\widetilde{w}_n\|_{L_z^2}^2+\big\|nW\big(\frac{z}{\sqrt{n}}\big)^2\widetilde{w}_n\big\|_{L_z^2}^2-2n\Re\big(\partial_z^2\widetilde{w}_n,W\big(\frac{z}{\sqrt{n}}\big)^2\widetilde{w}_n\big)_{L_z^2}.
			\end{align}
		Integration by part yields
			\begin{align*}
				-2n\Re\big(\partial_z^2\widetilde{w}_n,W\big(\frac{z}{\sqrt{n}}\big)^2\widetilde{w}_n\big)_{L_z^2}=&2\big\|\sqrt{n}W\big(\frac{z}{\sqrt{n}}\big)\partial_z\widetilde{w}_n\big\|_{L_z^2}^2+2\Re\big(\partial_z\widetilde{w}_n,\big[\partial_z,nW\big(\frac{z}{\sqrt{n}}\big)^2\big]\widetilde{w}_n\big)_{L_z^2}.
			\end{align*}
			Note that the second term containing the commutator can be bounded from below by 
			$$ -4\big\|\sqrt{n}W\big(\frac{z}{\sqrt{n}}\big)\partial_z\widetilde{w}_n\big\|_{L_z^2}\big\|W'\big(\frac{z}{\sqrt{n}}\big)\widetilde{w}_n\big\|_{L_z^2}\geq -C\big\|\sqrt{n}W\big(\frac{z}{\sqrt{n}}\big)\partial_z\widetilde{w}_n\big\|_{L_z^2}\|\widetilde{w}_n\|_{L_z^2}.
			$$
			Plugging into \eqref{tildegn} and using Young's inequality $AB\leq \epsilon A^2+\frac{4}{\epsilon}B^2$, we deduce that
			$$ \|\widetilde{g}_n\|_{L_z^2}^2\geq \|\partial_z^2\widetilde{w}_n\|_{L_z^2}^2
			+\big\|nW\big(\frac{z}{\sqrt{n}}\big)\widetilde{w}_n\big\|_{L_z^2}^2+\big\|\sqrt{n}W\big(\frac{z}{\sqrt{n}}\big)\partial_z\widetilde{w}_n\big\|_{L_z^2}^2-C'\|\widetilde{w}_n\|_{L_z^2}^2.
			$$
			Since
			$$ \sqrt{n}\big\|\partial_zW\big(\frac{z}{\sqrt{n}}\big)\widetilde{w}_n\big\|_{L_z^2}\leq+
			\sqrt{n}\big\|W\big(\frac{z}{\sqrt{n}}\big)\partial_z\widetilde{w}_n\big\|_{L_z^2}+C\|\widetilde{w}_n\|_{L_z^2},
			$$
			this implies \eqref{AppendixC1}. The proof of Lemma \ref{lemmaB1} is complete. 
	
	
	
\end{proof}

\section{Equivalence to the semiclassical resolvent estimate}
\label{a:borichev_tomilov}

Let us recall the classical theorem of Borichev-Tomilov:
\begin{prop}[\cite{BoT}]\label{thm:resolvent}
	The following statements are equivalent:
	\begin{align*}
		&(a) \hspace{0.3cm}		\big\|(i\lambda-\dot{\mathcal{A}})^{-1}\big\|_{\mathcal{L}(\dot{\mathscr{H}})}\leq C|\lambda|^{\frac{1}{\alpha}} \quad  \text{ for all }\lambda\in\R, \quad |\lambda|\geq 1;\\
		&(b) \hspace{0.3cm}
		\Vert \mathrm{e}^{t \dot{\mathcal{A}}} \dot{\mathcal{A}}^{-1} \Vert_{\mathcal{L}(\dot{\mathscr{H}})} =  O(t^{-\alpha}).
	\end{align*}
	
\end{prop}
Denoting $\Pi_0$ the spectral projector of $\mathcal{A}$ on $\ker \mathcal{A}$, since $\mathrm{e}^{t \mathcal{A}} = \mathrm{e}^{t \dot{\mathcal{A}}} (\Id - \Pi_0) + \Pi_0$, we have, if $(a)$ or $(b)$ hold true, that the semigroup $\mathrm{e}^{t \mathcal{A}}$ is stable at rate $t^{-\alpha}$.
In what follows, $\mathcal{A}$ is given by \eqref{matrixform}, associated to the damping $b$. 
\begin{lemma}\label{equivalenceresolvent} 
	For sufficiently small $h>0$, the following resolvent estimates are equivalent:
	\begin{align*}
		&\mathrm{(a)}\hspace{0.3cm} \|(ih^{-1}-\mathcal{A})^{-1}\|_{\mathcal{L}(H_G^1\times L^2)}\leq C_1 h^{-\alpha};\\
		&\mathrm{(b)}\hspace{0.3cm}
		\|(-h^2\Delta_G-1\pm ihb)^{-1}\|_{\mathcal{L}(L^2)}\leq C_2h^{-\alpha-1}.
	\end{align*}
\end{lemma}
\begin{proof}
	Essentially, the proof is given in \cite{AL14}. For the sake of completeness, we provide the proof here.
	Denote by $U=(u,v)^t$ and $F=(f,g)^t$, then $(ih^{-1}-\mathcal{A})U=F$ is equivalent to
	$$ u=-ih(v+f),\quad (-h^2\Delta_G-1+ihb)v=ihg+h^2\Delta_Gf.
	$$
	The implication (a)$\implies$ (b) follows from making a special choice $(f,g)=(0,g)\in H_G^1\times L^2$. 
	
	To prove (b) $\implies$ (a), we first claim that:
	\begin{itemize}
		\item[(i)] $\|(-h^2\Delta_G-1+ihb)^{-1}\|_{L^2\rightarrow H_G^1}\lesssim h^{-\alpha-2}$;
		\item[(ii)]
		$\|(-h^2\Delta_G-1+ihb)^{-1}\|_{H_G^{-1}\rightarrow L^2}\lesssim h^{-\alpha-2}$.
	\end{itemize}
	Indeed, assume that 
	$$ (-h^2\Delta_G-1+ihb)w=r,
	$$
	from the energy identity 
	$$ \|h\nabla_G w\|_{L^2}^2-\|w\|_{L^2}^2=\Re(r,w)_{L^2},
	$$
	we deduce that 
	$$ \|h\nabla_G w\|_{L^2}\lesssim \|w\|_{L^2}+\|r\|_{L^2}^{1/2}\|w\|_{L^2}^{1/2}.
	$$
	The hypothesis (b) implies that $\|w\|_{L^2}\lesssim h^{-\alpha-1}\|r\|_{L^2}$, hence $\|w\|_{H_G^1}\lesssim h^{-\alpha-2}\|r\|_{L^2}$, and this verifies (i). Note that the argument above is also valid for $(-h^2\Delta_G-1-ihb(y))^{-1}$, which is the adjoint of $(-h^2\Delta_G-1+ihb(y))^{-1}$. By duality, we obtain (ii).

	Finally, from (b) and (ii), 
	\begin{align}\label{L2}
	\|v\|_{L^2}\lesssim h^{-\alpha-1}\|hg\|_{L^2}+h^{-\alpha-2}\|h^2\Delta_G f\|_{H_G^{-1}}\lesssim h^{-\alpha}\|(f,g)\|_{H_G^1\times L^2}. 
	\end{align}
	From the energy identity
	$$ \|h\nabla_G v\|_{L^2}^2-\|v\|_{L^2}^2=\Re\langle ihg+h^2\Delta_Gf,v\rangle_{H_G^{-1},H_G^1},
	$$
	hence
	$$ h^2\|v\|_{H_G^1}^2\lesssim \|v\|_{L^2}^2+\|ihg+h^2\Delta_G f\|_{H_G^{-1}}\|v\|_{H_G^1}.
	$$
	Consequently, $$ h\|v\|_{H_G^1}\lesssim 
	h^{-\alpha}\|(f,g)\|_{H_G^1\times L^2},$$
	thanks to \eqref{L2}.
	Finally, from $u=-ihv-ihf$, we deduce that $\|u\|_{H_G^1}\lesssim h^{-\alpha}\|(f,g)\|_{H_G^1\times L^2}$. This completes the proof of Lemma \ref{equivalenceresolvent} .

\end{proof}

\section{Averaging method in finite dimension}
\label{s:finite_dimensional_averaging}

In this part of Appendix, we prove the following well-known finite-dimensional averaging lemma, used in Section \ref{s:quasimodes}:
\begin{lemma}
	\label{l:finite_dimension_averaging} Let $D \in \R^{n \times n}$ be a diagonal matrix with entries $\lambda_1 < \cdots < \lambda_n$. Let $A_j \in \mathbb{\R}^{n\times n}$ for $j= 1, \ldots, N$ be self-adjoint matrices. Then for all $N\in\N$, there exist diagonal matrices $D_j\in\mathbb{R}^{n\times n}$ for all $j=1,\cdots, N$, such that for any sufficiently small $\epsilon$, there is a unitary matrix $\mathfrak{U}_N(\epsilon) \in \mathbb{C}^{n\times n}$, close to the identity, such that
	$$
	\mathfrak{U}_N \Big( D + \sum_{j=1}^N \epsilon^j A_j \Big) \mathfrak{U}_N^* = D + \sum_{j=1}^N \epsilon^j D_j + O(\epsilon^{N+1}).
	$$
\end{lemma}

\begin{proof}
	Write $\mathfrak{U}_1 = \mathrm{e}^{i \epsilon F_1}$ with $F_1 \in \R^{n\times n}$ to be chosen. Then, Taylor expansion gives
	\begin{align*}
	\mathfrak{U}_1 \Big( D + \sum_{j=1}^N \epsilon^j A_j \Big) \mathfrak{U}_1^* & = D + i \epsilon [F_1, D] + \epsilon A_1 + (i\epsilon)^2 \int_0^1 (1-s) e^{i s \epsilon F_1}[F_1, [F_1,D]] e^{-is \epsilon F_1} ds \\[0.2cm]
	& \quad + i \epsilon^2 \int_0^1 e^{is \epsilon F_1} [F_1, A_1] e^{- i s \epsilon F_1} ds+O(\epsilon^3).
	\end{align*}
	We choose $F_1$ and $D_1$ so that
	$$
	i [D,F_1] = A_1 - D_1.  
	$$
	This is possible since the eigenvalues of $D$ are distinct, by taking
	\begin{equation}
	\label{e:solution_cohomological_finite}
	(F_1)_{j_1 j_2} :=  \frac{1 - \delta_{j_1 j_2}}{i(\lambda_{j_1} - \lambda_{j_2})} (A_1)_{j_1 j_2}; \quad (D_1)_{j_1 j_2} := \delta_{j_1 j_2} (A_1)_{j_1j_2},
	\end{equation}
	where $\delta_{j_1 j_2}$ denotes the Kronocker delta.  Then we obtain:
	\begin{align}
	\label{e:remainder_taylor_normal_form}
	\mathfrak{U}_1 \left( D + \sum_{j=1}^N \epsilon^j A_j \right) \mathfrak{U}_1^* & = D + \epsilon D_1 + i \epsilon^2 \int_0^1 e^{i s \epsilon F_1} [F_1, s A_1 + (1-s) D_1] e^{- is \epsilon F_1} ds+O(\epsilon^3) \\[0.2cm]
	 & = D + \epsilon D_1 + O(\epsilon^2). \notag
	\end{align} 
	Iterating this procedure, we obtain the claim by defining $\mathfrak{U}_N := \mathrm{e}^{ i \epsilon^N F_n} \cdots \mathrm{e}^{i \epsilon F_1}$ for suitable self-adjoint matrices $F_1, \ldots, F_N$ and diagonal matrices $D_1, \ldots, D_N$.
\end{proof}

\section{Some black-box lemma}

We collect some known 1D resolvent estimates as black boxes. All will be used in Section \ref{normalformimproved}, when reducing the resolvent estimate to the one-dimensional setting. The first estimate is now well-known as the geometric control estimate:
\begin{lemma}[\cite{CoZuazua}]\label{estimate:GCC}
	Let $I\subset\T$ be a non-empty open set. Then there exists $C=C_I>0$, such that for any $v\in L^2(\T)$, $f_1\in L^2(\T), f_2\in H^{-1}(\T)$, $\lambda\geq 1$, if
	$$ (-\partial_x^2-\lambda^2)v=f_1+f_2,
	$$
	we have
	$$ \|v\|_{L^2(\T)}\leq \frac{C}{\lambda}\|f_1\|_{L^2(\T)}+C\|f_2\|_{H^{-1}(\T)}+C\|v\|_{L^2(I)}.
	$$
\end{lemma}
This result can be deduced from the one-dimensional uniform stabilization for the wave equation in \cite{CoZuazua}.
		The passage from the uniform stabilization to the resolvent estimate can be also found in Proposition 4.2 and Appendix A of \cite{Bu19}, or the proof of Proposition 1.4 of \cite{ChSun}.

The second estimate follows from the sharp resolvent estimate for the damped-wave operator on $\T^2$ with rectangular-shaped damping:
\begin{lemma}[\cite{DK20}, Formula (6)]\label{Holder} 
	There exists $h_0>0$, $C>0$ such that for all $0<h<h_0$, $E\in\R$, for any solution $v$ of
	$$ -h^2\partial_y^2v-Ev+ihb_2(y)v=f,
	$$
	we have
	$$ \|v\|_{L^2(\T)}\leq Ch^{-2-\frac{1}{\nu+2}}\|f\|_{L^2(\T)}.
	$$
\end{lemma}

The third estimate is the almost sharp resolvent estimate for the damped-wave operator on $\T^2$, proved in \cite{AL14} and (essentially) revisited in \cite{S21}:
\begin{lemma}[\cite{AL14}, Theorem 2.6]\label{AL} 
	There exists $h_0>0$, $C>0$ and $\delta_0=\delta_0(\sigma)$ such that for all $0<h<h_0$, for any solution $v$ of
	$$ -h^2\Delta v-v+ihb_1(y)v=f,
	$$
	we have
	$$ \|v\|_{L^2(\T^2)}\leq Ch^{-2-\delta_0}\|f\|_{L^2(\T^2)}.
	$$
\end{lemma}
The last estimate is a special case of the sharp resolvent estimate for the damped-wave operator on $\T^2$ for the narrowly undamped situation:
\begin{lemma}[\cite{LeL17}, Theorem 1.8]\label{LL} 
	There exists $h_0>0$, $C>0$, such that for all $0<h<h_0$, for any solution $v$ of
	$$ -h^2\Delta v-v+ihb_3(y)v=f,
	$$
	we have
	$$ \|v\|_{L^2(\T^2)}\leq Ch^{-2+\frac{2}{\nu+2}}\|f\|_{L^2(\T^2)}.
	$$
\end{lemma}
We used also intensively a commutator estimate for Lipschitz functions:
		\begin{lemma}\label{commutatorLip}
			Assume that $\kappa\in W^{1,\infty}(\R^d)$ and $a\in S^0(\R^{2d})$, then
			$$ \|[\Op_h^{\w}(a),\kappa ] \|_{\mathcal{L}(L^2(\R^d))}\leq Ch.
			$$
		\end{lemma}
		The proof of this Lemma is standard and can be found, for example, in Corollary (A.2) of \cite{BuS21}. The proof there applies to Weyl quantization as well. In various places of this article, we apply Lemma \ref{commutatorLip} to deduce that $[b^{1/2},\Op_h^{\w}(a)]=O_{\mathcal{L}(L^2)}(h)$ and $[b^{1/2},[b^{1/2},\Op_h^{\w}(a)]]=O_{\mathcal{L}(L^2)}(h^2)$, thanks to the hypothesis \eqref{e:B-H_condition}.

		\subsection*{Statement of competing Interests}
		The authors have no competing interests to declare that are relevant to the content of this article.

\end{document}